\documentclass[12pt]{amsart}

\usepackage{
	amsmath,
 	amsfonts,
  	amssymb,
 	amsthm,
	}

\usepackage{tikz-cd}
\usepackage{tikz}

\usepackage{rotating}

\usetikzlibrary{decorations.pathmorphing, decorations.pathreplacing}
\usetikzlibrary{decorations.pathreplacing,backgrounds,decorations.markings}

\usepackage{color}
\usepackage{mathtools}

\usepackage{hyperref}
\usepackage[capitalise]{cleveref}

\usepackage{a4wide}
\usepackage{tabularx}
\usepackage{stmaryrd}
\usepackage{bbm}

\theoremstyle{plain}
\newtheorem{thm}{Theorem}[section]

\newtheorem{lem}[thm]{Lemma}
\newtheorem{prop}[thm]{Proposition}

\theoremstyle{definition}
\newtheorem{rem}[thm]{Remark}

\newtheorem{exe}[thm]{Example}

\theoremstyle{definition}
\newtheorem{defn}[thm]{Definition}

\def\makeautorefname#1#2{\expandafter\def\csname#1autorefname\endcsname{#2}}
\makeautorefname{lem}{Lemma}%
\makeautorefname{prop}{Proposition}%
\makeautorefname{rem}{Remark}%
\makeautorefname{section}{Section}%

\newcommand{\cC}{\mathcal{C}}
\newcommand{\cD}{\mathcal{D}}

\newcommand{\cI}{\mathcal{I}}
\newcommand{\cM}{\mathcal{M}}
\newcommand{\cR}{\mathcal{R}}
\newcommand{\cZ}{\mathcal{Z}}

\newcommand{\bC}{\mathbb{C}}
\newcommand{\bP}{\mathbb{P}}

\newcommand{\bR}{\mathbb{R}}
\newcommand{\bZ}{\mathbb{Z}}

\newcommand{\un}{\mathbbm{1}}

\DeclareMathOperator{\HOM}{HOM}
\DeclareMathOperator{\Mod}{Mod}
\DeclareMathOperator{\Bimod}{Bimod}
\DeclareMathOperator{\BIMOD}{BIMOD}
\DeclareMathOperator{\OMoperator}{OM}

\DeclareMathOperator{\Hom}{Hom}
\DeclareMathOperator{\End}{End}

\DeclareMathOperator{\cone}{Cone}
\DeclareMathOperator{\br}{\mathbf p}
\DeclareMathOperator{\Lderiv}{L}

\DeclareMathOperator{\img}{img}

\newcommand{\vsim}{\rotatebox{90}{\(\sim\)}}
\newcommand{\vsimeq}{\rotatebox{90}{\(\simeq\)}}
\newcommand{\vsimeqop}{\rotatebox{-90}{\(\simeq\)}}

\newcommand{\BBimod}{\mathfrak{B}\text{imod}}

\newcommand{\id}{\text{Id}}

\newcommand{\Tau}{\mathrm{T}}

\definecolor{myblue}{rgb}{0,.5,1}
\definecolor{mygreen}{rgb}{.3,.75,.1}

\newcommand{\tikzdiag}[2][]{\tikz[#1,thick,baseline={([yshift=1ex+#2]current bounding box.center)}]}
\newcommand{\tikzdiagc}[1][]{\tikzdiag[#1]{-1ex}}
\newcommand{\tikzdiagh}[1][]{\tikzdiag[#1]{-2ex}}
\newcommand{\tikzdiagd}[1][]{\tikzdiag[#1]{0ex}}

\tikzset{
    partial ellipse/.style args={#1:#2:#3}{
        insert path={+ (#1:#3) arc (#1:#2:#3)}
    }
}

\usetikzlibrary{arrows.meta, 
		  decorations.markings}

\tikzstyle{giga thick}=[ line width=.7mm]
\tikzstyle{movieline}=[double,double distance=2pt]
\tikzstyle{moviedashed}=[dashed, line width=2pt]
\tikzstyle{rrect}=[rectangle, rounded corners=2pt,draw=black,fill=white,inner sep=3pt]

\tikzstyle{tikzdot}=[fill, circle, inner sep=2pt]

\tikzstyle{wall} = [
    draw,
    color=gray,
    postaction=decorate, 
    decoration={
        markings,
        mark=between positions 0.015 and 0.98 step 0.1072 with {\draw (0,0)--(-60:-3pt);}
    }
] 

\tikzstyle{leash} = [dashed, color=gray]
\tikzstyle{dleash} = [dashed, double, double distance=1pt, color=gray]

\tikzstyle{ds} = [ double, double=lightgray, double distance=1pt]
\tikzstyle{dt} = [dashed, thin]
\tikzstyle{dds} = [ds, dt]
\tikzstyle{vt} = [very thick]

\tikzset{>={stealth}}

\newcommand{\tikzbrace}[4]{\draw[decoration={brace,mirror,raise=-8pt},decorate]  (#1-.1,#3 -.35) -- node {#4} (#2+.1,#3-.35)}

\DeclareMathOperator{\kh}{Kh}
\DeclareMathOperator{\CSeq}{CSeq}
\DeclareMathOperator{\F}{F}

\newcolumntype{C}{>{\centering\arraybackslash}X}

\newcommand{\chcobcat}{\text{\textbf{ChCob}}}
\newcommand{\linchcobcat}{{R\chcobcat}}
\newcommand{\dotchcobcat}{{R\chcobcat}^2_{\bullet}}

\newcommand{\Moddg}{\Mod_{dg}}
\newcommand{\Bimoddg}{\Bimod_{dg}}
\newcommand{\Lotimes}{\otimes^{\Lderiv}}

\newcommand{\tqft}{\mathcal{F}}

\newcommand{\tangleSpace}{B}
\newcommand{\red}[1]{\widehat{#1}}
\newcommand{\redTangleSpace}{\red{\tangleSpace}}
\newcommand{\chcob}{W}
\newcommand{\assoc}{\alpha}
\newcommand{\gradC}{\mathcal{G}}
\newcommand{\shiftFunctCol}{\Phi}
\newcommand{\shiftFunct}[2][]{{\varphi^{#1}_{#2}}}
\newcommand{\shiftDom}[2][]{{D^{#1}_{#2}}}
\newcommand{\shiftCenter}[1]{\Sigma}
\newcommand{\compMap}[3][]{{\beta^{#1}_{#2,#3}}}
\newcommand{\vCompMap}[3][]{{\gamma^{#1}_{#2,#3}}}
\newcommand{\interCompMap}[5][]{{\Xi^{#1}_{\substack{#2,#4\\#3,#5}}}}
\newcommand{\commutMap}[5][]{{\tau^{#1}_{\substack{#2,#4\\#3,#5}}}}

\newcommand{\chdefect}[1]{\widehat \chi(#1)}
\newcommand{\eulerChar}{\chi}

\newcommand{\BIMdg}[3]{\BIMOD_{dg}^{#1}(#2,#3)}

\newcommand{\cKOM}{\mathcal{K}\OMoperator}

\newcommand{\und}[1]{{\underline{#1}}}

\newcommand{\neutralElement}{\mathbf{e}}

\newcommand{\simpleRootsSet}{\Pi}
\newcommand{\ladder}{\mathcal{L}}
\newcommand{\ladderToTangle}{\overline \ladder}
\newcommand{\UtoBfunctor}{\tqft \overline \ladder}

\newcommand{\oddKC}{\mathcal{U^-_{R}}}
\newcommand{\oddKCeven}{{}^{2}\mathcal{U^-_{R}}}
\newcommand{\BIMODH}{2\text{-}\BIMOD_q}

\newcommand{\bi}{\mathbf{i}}
\newcommand{\bj}{\mathbf{j}}
\newcommand{\bw}{\mathbf{w}}

\title{Odd Khovanov homology for tangles}
\author{Gr\'egoire Naisse}
\address{Max-Planck Institute for Mathematics\\
 Vivatsgasse 7 \\ 
53111 Bonn\\ 
Germany}
\email{gregoire.naisse@gmail.com}
\author{Krzysztof Putyra}
\address{Institut f\"ur Mathematik\\
Universit\"at Z\"urich\\
Winterthurerstrasse 190
CH-8057 Z\"urich
}
\email{krzysztof.putyra@math.uzh.ch}

\begin{document}


\begin{abstract}
We extend the covering of even and odd Khovanov link homology to tangles, using arc algebras. 
For this, we develop the theory of quasi-associative algebras and bimodules graded over a category with a 3-cocycle. 
Furthermore, we show that a covering version of a level 2 cyclotomic half 2-Kac--Moody algebra acts on the bicategory of quasi-associative bimodules over the covering arc algebras, relating our work to a construction of Vaz. 
\end{abstract}


\maketitle



\section{Introduction}\label{sec:intro}

In the seminal paper~\cite{khovanov00}, Khovanov constructed a homology theory for links that categorifies the Jones polynomial, in the sense that the graded Euler characteristic of the former coincides with the later.  
The construction relies on a 2d-TQFT obtained from the Frobenius algebra $\bZ[x]/(x^2)$. 
Khovanov extended his construction to tangles with an even number of endpoints in \cite{khovanovHn}. 
The invariant then takes the form, for a tangle $T$, of the homotopy type of a complex $\kh(T)$ of graded bimodules over \emph{arc algebras} $H^n$.  
It respects a gluing property, meaning that given a pair of tangles $T'$ and $T$, we have
$
\kh(T') \otimes_{H^n} \kh(T) \cong \kh(T'T).
$
Generalizations of arc algebras were introduced by Stroppel~\cite{stroppelalgebra} and independently by Chen--Khovanov~\cite{chenkhovanov}, allowing to extend the construction to any tangle. These algebras are given an in-depth study by Brundan and Stroppel in the series of paper~\cite{brundanstroppel1,brundanstroppel2,brundanstroppel3,brundanstroppel4}.

Ozsv\'ath, Rasmussen and Szab\'o constructed in~\cite{ORS} a distinct version of Khovanov link homology. It is based on a projective TQFT, defined only up to sign. 
They called it \emph{odd Khovanov homology} since it makes use of anticommutative variables instead of commutative ones. Their construction yields a distinct invariant of links from the usual `\emph{even}' Khovanov homology. 
Odd Khovanov homology also categorifies the Jones polynomial, and both construction agree modulo 2.  
Furthermore, even Khovanov homology can be thought as obtained from the geometry of the complex projective line $\bP^1(\bC)$~\cite{stroppelwebster12}, while odd Khovanov homology is obtained from the real one $\bP^1(\bR)$~\cite{oddspringer}.

The second author of this article built a framework~\cite{putyra14} that allows a cobordism-type interpretation of Ozsv\'ath--Rasmussen--Szab\'o projective TQFT. 
The idea is to enrich the notion of cobordisms with a \emph{chronology}, i.e. a framed Morse function that separates critical points~\cite{igusa}. 
These \emph{chronological cobordisms} form a 2-category $\chcobcat$ where the 2-morphisms are diffeotopies of the chronology maps. 
Then, the projective TQFT of~\cite{ORS} becomes a genuine 2-functor. 
In~\cite{putyra14}  is also introduced the notion of \emph{covering Khovanov homology} for links, coming from a complex of modules over the ground ring $R := \bZ[X,Y,Z^{\pm 1}]/(X^2=Y^2=0)$. Specializing $X=Y=Z=1$ recovers even Khovanov homology, and $X=Z=1$ and $Y=-1$ yields the odd one. 

In~\cite{naissevaz}, the first author and Vaz gave an odd version of the arc algebras using Ozsv\'ath--Rasmussen--Szab\'o projective TQFT. 
These \emph{odd arc algebras} are not associative, preventing the construction of an odd invariant for tangles by simply mimicking \cite{khovanovHn}. 
Indeed, while it is not hard to associate a chain complex of graded spaces to a tangle by considering all the ways one can close it, and do the usual procedure of~\cite{khovanov00,khovanovHn,ORS}, the non-associativity prevents to have a gluing property. . 
In particular, it is not clear how to define the tensor product of modules over a non-associative algebra. 
Following an unpublished idea of the second author and Shumakovitch, it is explained in~\cite{naissevaz} how the odd arc algebras can be made quasi-associative by grading them over some groupoid with an associator (i.e. a 3-cocycle). This notion extends  the definition of quasi-associative algebras of Albuquerque and Majid~\cite{quasialgebra}. 
Unfortunately, this quasi-associative structure is not enough to construct the desired bimodules that are needed for the construction of a tangle invariant as in~\cite{khovanovHn}. 
Another issue is that the maps used to form the complex $\kh(T)$ do not preserve the grading, thus requiring a notion of grading shift functor for quasi-associative bimodules. 

\smallskip

The present paper gives a solution to these problems, and we construct a covering version of Khovanov homology for tangles. 
Moreover, we also connect our construction to an tangle odd invariant introduced by Vaz in~\cite{pedrodd}. 

\subsection*{Main results and structure of the paper}

We first recall in \cref{sec:chcob} the notion of chronological cobordisms. We use them in \cref{sec:arcbimod} to give a covering version of the arc algebras  over $R$. 

In \cref{sec:gradcat}, we extend the definition of quasi-associative algebras and bimodules to structures graded over a category with an associator. 
We also introduce the notion of a \emph{shifting system}, which allows to construct grading shift functors for these objects. These can be interpreted as a generalization of parity shift functors in the theory of supergraded structures. 
We also explain how to construct `graded' categories of bimodules, where the maps posses non-trivial degree, by enriching to notion of shifting-system to a \emph{shifting 2-system}. This can be seen as a generalization of the notion of supercategories of superbimodules as in \cite{supermonoidal}. 
Finally, we develop basic notions of homological algebra for these structures. 

In \cref{sec:gradCat}, we construct a grading category $\gradC$ where the morphisms are given by a pair of a flat tangle and an element in $\bZ \times \bZ$. We show that the covering arc algebras are quasi-associative over $\gradC$. 
Note that our grading category is different from the one in~\cite{naissevaz}. 
We also construct a shifting-2-system over $\gradC$, where the shifts are given by a pair of a chronological cobordism together with a shift in the $\bZ \times \bZ$-grading. 
All this allows us in \cref{sec:tanglehomology} to mimic the construction of \cite{khovanovHn} in the quasi-associative context. 
We show it yields a tangle version of covering Khovanov homology, that agrees over a link with the one from~\cite{putyra14}, and respects the gluing property. 
Furthermore, by specializing the parameters, we get a tangle version of odd Khovanov homology from \cite{ORS}.  

In the last section, we relate our construction to Vaz tangle invariant~\cite{pedrodd}. This invariant is built using an odd version of a level 2 cyclotomic half 2-Kac--Moody algebra associated to $\mathfrak{gl}_n$, in the same spirit as Khovanov--Lauda~\cite{khovanovlauda3} and Rouquier~\cite{rouquier}, and odd categorification~\cite{oddnh, oddslt,supercat1,supercat2}. 
We extend the 2-Kac--Moody of~\cite{pedrodd} to a convering version over $R$.  
We show it
acts in a graded sense on the bicategory of quasi-associative bimodules over the covering arcs algebras, giving an odd version of the 2-action constructed in~\cite{brundanstroppel3} for the even case. 
Furthermore, after applying this  2-functor, the complex associated to a tangle in~\cite{pedrodd} coincides with ours. 
We conjecture that this 2-functor is fully faithful, and therefore, that Vaz invariant is equivalent to our version of odd Khovanov homology.

\subsection*{Acknowledgments}
The authors would like to thank Pedro Vaz for interesting discussions. 
The authors would also like to thank the Erwin Schr\"odinger International Institute for Mathematics and Physics for their hospitality. 
G.N. is grateful to the Max Planck Institute for Mathematics in Bonn for its hospitality and financial support. 
G.N. was a Research Fellow of the Fonds de la Recherche Scientifique - FNRS, under Grant no.~1.A310.16 while starting working on this project. 
%



%

\section{Chronological cobordims}\label{sec:chcob}

We recall the notion of chronological cobordisms and the 2-categories they form, following ~\cite{putyra14}. 

\subsection{Chronological cobordims}

Let $\Sigma_0$ and $\Sigma_1$ be two compact oriented $1$-manifolds embedded in $\bR^2$. An \emph{(embedded) cobordism} from $\Sigma_0$ to $\Sigma_1$ is  a compact surface $W \subset \bR^2 \times I$ such that $\partial W \cong \Sigma_0 \sqcup -\Sigma_1$, where $-\Sigma_1$ means we take $\Sigma_1$ with reverse orientation. 

Given a compact manifold $W$, a generic function $f : W \rightarrow I$ has only Morse critical points. A generic homotopy $f_t$ of Morse functions can have degenerate singularities of the simplest form, modeled by a cubic polynomial:
\[
q(x_1,\dots, x_n) = q(0) - x_1^2 - \cdots - x_k^2 + x_{k+1}^2 + \cdots + x_{n-1}^2 + x_n^3.
\]
Such a point $p$ is called an \emph{$A_2$-singularity}. It can be characterized by requiring $Hess_p(f)$ to have a $1$-dimensional kernel,  on wich $d^3f_p$ does not vanish. A function that has only Morse and $A_2$-singularities is called an \emph{Igusa function  \cite{igusa}}.

Choose a Riemannian metric on $W$, so that we can consider the Hessian of $f$ as a linear function $Hess_p(f) : T_pW \rightarrow T_pW$. Because it is symmetric, the tangent space $T_pW$ at a critical point $p$ decomposes into $T_pW = E_p^+ \oplus E_p^- \oplus N_p$, where $E_p^\pm$ is the positive or negative eigenspace, and $N_p$ is the nullspace of $Hess_p(f)$. 
A \emph{framing} on $f: W \rightarrow I$ is a choice of an orthonormal bases of the negative eigenspace $E^-_p$, extended by a vector $v \in N_p$ for which $d^3f_p(v,v,v) > 0$, if $p$ is degenerate.

\begin{defn}[Cf. \cite{putyra14}] An \emph{(embedded) chronological cobordism} $(W,\tau)$ is a surface $W \subset \bR^2 \times I$ with a framing $\tau$ on the restriction $\pi|_S$ of the canonical projection $\pi : \bR^2 \times I \rightarrow I$, such that
\begin{itemize}
\item $W$ is transverse to $\bR^2 \times \partial I$;
\item $\partial W \subset \bR^2 \times \partial I$ (it is a cobordism from $\partial W \cap \bR^2 \times \{0\}$ to $\partial W \cap \bR^2 \times \{1\}$); 
\item $\pi|_W$ is a Morse function with at most one critical point on each level.
\end{itemize}
A \emph{change of chronology} is a diffeotopy $H_t$ of $\bR^2 \times I$, such that each $\pi_{H_t(W)}$ is a generic homotopy of Morse functions, together with a smooth family of framings on $\pi|_{H_t(W)}$. 
\end{defn}

We say that two embedded chronological cobordisms are \emph{equivalent} if they can be related by a diffeotopy $H_t$ for which $\pi|_{H_t(W)}$ is separative Morse at every moment $t\in I$. This is for example the case when $\pi \circ H_t = \pi$ for all $t \in I$, that is when we deform the cobordism only along the horizontal plane. We refer to such diffeotopies as \emph{horizontal}. Another example is given by \emph{vertical} diffeotopies, that is when $H_t(p,z) = (p,h_t(z))$ for some diffeotopy $h_t$ on $I$. This means we deform uniformly the cobordism along the vertical axis. 

Two changes of chronology between equivalent cobordisms are \emph{equivalent} if, after composing with the equivalences of cobordisms, they are homotopic in the space of oriented Igusa functions.

\subsubsection{Locally vertical changes of chronology}

A important family of changes of chronology is given by the locally vertical changes of chronology. This means they are vertical inside some cylinders fixed in $\bR^2 \times I$. 

\begin{defn}[Cf. \cite{putyra14}]
Let $W$ be a chronological cobordism. 
Choose a family of disjoint disks $D_1, \dots, D_R$ in $\bR^2$, such that $W$ intersects each $\partial C_i$ in vertical lines, where $C_i := D_i \times I$. 
We say that a change of chronology $H_t$ on $W$ is \emph{locally vertical} with respect to the cylinders $C_i$ if it is vertical inside each $C_i$, but fixes all points outside them, except small annular neighborhoods of $\partial C_i$, in which we interpolate the two behaviors.
\end{defn}

The particularity of locally vertical change of chronologies is that they are unique up to homotopy.

\begin{prop}[{\cite[Proposition 4.4]{putyra14}}]\label{prop:locvertchange}
Let $H_t$ and $H'_t$ be two locally vertical changes of chronology with respect to the same cylinders. If $H_1 = H_1'$, then they are  homotopic in the space of framed diffeotopies. 
\end{prop}

\subsection{The 2-category $\chcobcat$}

From now on, except if specified otherwise, all cobordisms are embedded chronological cobordisms. 

 Given a cobordism $(W',\tau')$ from $B$ to $C$ and a cobordism $(W,\tau$) from $A$ to $B$, one can compose them by gluing the cobordisms together along $B$ (adding a collar if necessary, see~\cite{putyra14}[\S3]) and rescaling. Given in addition a change of chronology $H'$ on $(W',\tau')$ and a change of chronology $H$ on $(W,\tau)$, we can compose them \emph{horizontally}, giving a change of chronology $H' \circ H$ on $(W' \circ W, \tau' \circ \tau)$.
 
  Let $(W_1,\tau_1)$ and $(W_2,\tau_2)$ be cobordisms with a change of chronology $H : (W_1, \tau_1) \Rightarrow (W_2, \tau_2)$ and let and $(W_2',\tau_2')$ be a cobordism equivalent to $H(W_2,\tau_2)$. Let also $H' : (W_2', \tau_2') \Rightarrow (W_3, \tau_3)$ be a change of chronology on $(W_2',\tau_2')$. Then, we can \emph{vertically compose} $H'$ with $H$, giving a change of chronology $H' \star H : (W_1,\tau_1) \Rightarrow (W_3,\tau_3)$ as the composition $(W_1, \tau_1) \xrightarrow{H} (W_2, \tau_2) \xrightarrow{\simeq} (W_2, \tau_2') \xrightarrow{H'} (W_3, \tau_3)$. 

We usually draw cobordisms as diagrams read from bottom to top, and we draw the chosen framing as an arrow written next to the critical point. Composition on the left is translated by gluing from above.  

\begin{defn}
Let $\chcobcat$ be the (strict) 2-category consisting of
\begin{itemize}
 \item objects are finite disjoint collections of circles in $\bR^2$;
 \item 1-morphisms in $\Hom_{\chcobcat}(A,B)$ are classes of embedded chronological cobordisms from $A$ to $B$, up to equivalence;
 \item composition of $1$-morphism is given by gluing cobordisms;
 \item $2$-morphisms are given by classes of changes of chronology up to equivalence;
 \item vertical and horizontal compositions of $2$-morphisms are given by $\star$ and $\circ$ defined above, respectively.
\end{itemize}
\end{defn}

The 2-category $\chcobcat$ becomes Gray monoidal (that is a monoidal 2-category,  see~\cite[Appendix B.2]{putyra14}) when equipped with the horizontal monoidal product of cobordims given by \emph{right-then-left juxtaposition} of cobordism, pushing all critical points of the cobordism on the left to the top:
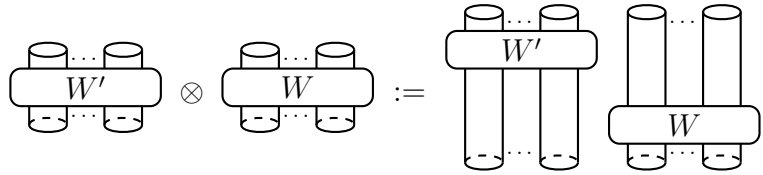
\[
\tikzdiagh[scale=.5]{
	\draw (0,0) .. controls (0,.-.25) and (1,-.25) .. (1,0);
	\draw[dashed] (0,0) .. controls (0,.25) and (1,.25) .. (1,0);
	\draw (0,0) -- (0,2);
	\draw (1,0) -- (1,2);
	\draw (0,2) .. controls (0,1.75) and (1,1.75) .. (1,2);
	\draw (0,2) .. controls (0,2.25) and (1,2.25) .. (1,2);
	\node at(1.5,.25) {\tiny $\dots$};
	\node at(1.5,1.75) {\tiny $\dots$};
	\draw (2,0) .. controls (2,.-.25) and (3,-.25) .. (3,0);
	\draw[dashed] (2,0) .. controls (2,.25) and (3,.25) .. (3,0);
	\draw (2,0) -- (2,2);
	\draw (3,0) -- (3,2);
	\draw (2,2) .. controls (2,1.75) and (3,1.75) .. (3,2);
	\draw (2,2) .. controls (2,2.25) and (3,2.25) .. (3,2);
	\filldraw [fill=white, draw=black,rounded corners] (-.5,.5) rectangle (3.5,1.5) node[midway] { $W'$};
}
\  \otimes \ 
\tikzdiagh[scale=.5]{
	\draw (0,0) .. controls (0,.-.25) and (1,-.25) .. (1,0);
	\draw[dashed] (0,0) .. controls (0,.25) and (1,.25) .. (1,0);
	\draw (0,0) -- (0,2);
	\draw (1,0) -- (1,2);
	\draw (0,2) .. controls (0,1.75) and (1,1.75) .. (1,2);
	\draw (0,2) .. controls (0,2.25) and (1,2.25) .. (1,2);
	\node at(1.5,.25) {\tiny $\dots$};
	\node at(1.5,1.75) {\tiny $\dots$};
	\draw (2,0) .. controls (2,.-.25) and (3,-.25) .. (3,0);
	\draw[dashed] (2,0) .. controls (2,.25) and (3,.25) .. (3,0);
	\draw (2,0) -- (2,2);
	\draw (3,0) -- (3,2);
	\draw (2,2) .. controls (2,1.75) and (3,1.75) .. (3,2);
	\draw (2,2) .. controls (2,2.25) and (3,2.25) .. (3,2);
	\filldraw [fill=white, draw=black,rounded corners] (-.5,.5) rectangle (3.5,1.5) node[midway] { $W$};
}
\ := \  \tikzdiagh[scale=.5]{
	\draw (0,0) .. controls (0,.-.25) and (1,-.25) .. (1,0);
	\draw[dashed] (0,0) .. controls (0,.25) and (1,.25) .. (1,0);
	\draw (0,0) -- (0,4);
	\draw (1,0) -- (1,4);
	\draw (0,4) .. controls (0,3.75) and (1,3.75) .. (1,4);
	\draw (0,4) .. controls (0,4.25) and (1,4.25) .. (1,4);
	\node at(1.5,.25) {\tiny $\dots$};
	\node at(1.5,3.75) {\tiny $\dots$};
	\draw (2,0) .. controls (2,.-.25) and (3,-.25) .. (3,0);
	\draw[dashed] (2,0) .. controls (2,.25) and (3,.25) .. (3,0);
	\draw (2,0) -- (2,4);
	\draw (3,0) -- (3,4);
	\draw (2,4) .. controls (2,3.75) and (3,3.75) .. (3,4);
	\draw (2,4) .. controls (2,4.25) and (3,4.25) .. (3,4);
	\filldraw [fill=white, draw=black,rounded corners] (-.5,2.5) rectangle (3.5,3.5) node[midway] { $W'$};
} \ \tikzdiagh[scale=.5]{
	\draw (0,0) .. controls (0,.-.25) and (1,-.25) .. (1,0);
	\draw[dashed] (0,0) .. controls (0,.25) and (1,.25) .. (1,0);
	\draw (0,0) -- (0,4);
	\draw (1,0) -- (1,4);
	\draw (0,4) .. controls (0,3.75) and (1,3.75) .. (1,4);
	\draw (0,4) .. controls (0,4.25) and (1,4.25) .. (1,4);
	\node at(1.5,.25) {\tiny $\dots$};
	\node at(1.5,3.75) {\tiny $\dots$};
	\draw (2,0) .. controls (2,.-.25) and (3,-.25) .. (3,0);
	\draw[dashed] (2,0) .. controls (2,.25) and (3,.25) .. (3,0);
	\draw (2,0) -- (2,4);
	\draw (3,0) -- (3,4);
	\draw (2,4) .. controls (2,3.75) and (3,3.75) .. (3,4);
	\draw (2,4) .. controls (2,4.25) and (3,4.25) .. (3,4);
	\filldraw [fill=white, draw=black,rounded corners] (-.5,.5) rectangle (3.5,1.5) node[midway] { $W$};
} 
\]
The unit is given by the empty cobordism. Given two changes of chronology $H : W_1 \Rightarrow W_2$ and $H' : W'_1 \Rightarrow W_2'$, we obtain a change of chronology $H' \otimes H : W_1' \otimes W_1 \Rightarrow W_2' \otimes W_2$ by setting
\[
(H' \otimes H) _t (x) := \begin{cases}
H'_t(x), & \text{ if $x\in W$}, \\
H_t(x), & \text{ if $x \in W'$}, \\
x, & \text{otherwise},
\end{cases}
\]
where the otherwise can occur because of the identity corbodisms we added below $W'$ and above $W$.

\smallskip

The $1$-morphisms in $\chcobcat$ (as Gray monoidal category)  are generated by the five elementary cobordisms:
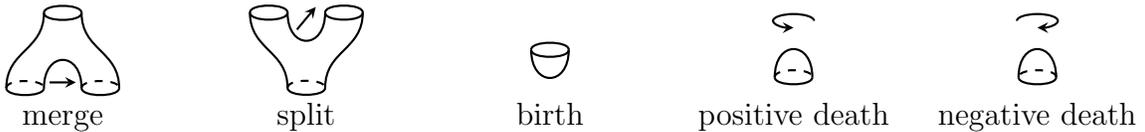
\begin{center}
\begin{tabularx}{\textwidth}{CCCCC}
\centering \tikzdiagh[scale=.5]{
	\draw (0,0) .. controls (0,1) and (1,1) .. (1,2);
	\draw (1,0) .. controls (1,1) and (2,1) .. (2,0);
	\draw (3,0) .. controls (3,1) and (2,1) .. (2,2);
	\draw (0,0) .. controls (0,-.25) and (1,-.25) .. (1,0);
	\draw[dashed] (0,0) .. controls (0,.25) and (1,.25) .. (1,0);
	\draw (2,0) .. controls (2,-.25) and (3,-.25) .. (3,0);
	\draw[dashed] (2,0) .. controls (2,.25) and (3,.25) .. (3,0);
	\draw (1,2) .. controls (1,1.75) and (2,1.75) .. (2,2);
	\draw (1,2) .. controls (1,2.25) and (2,2.25) .. (2,2);
	\draw [->] (1.15,.15) -- (1.85,.15);
} 
&
 \tikzdiagh[xscale=.5,yscale=-.5]{
	\draw (0,0) .. controls (0,1) and (1,1) .. (1,2);
	\draw (1,0) .. controls (1,1) and (2,1) .. (2,0);
	\draw (3,0) .. controls (3,1) and (2,1) .. (2,2);
	\draw (0,0) .. controls (0,-.25) and (1,-.25) .. (1,0);
	\draw (0,0) .. controls (0,.25) and (1,.25) .. (1,0);
	\draw (2,0) .. controls (2,-.25) and (3,-.25) .. (3,0);
	\draw (2,0) .. controls (2,.25) and (3,.25) .. (3,0);
	\draw[dashed] (1,2) .. controls (1,1.75) and (2,1.75) .. (2,2);
	\draw (1,2) .. controls (1,2.25) and (2,2.25) .. (2,2);
	\draw [->] (1.25,.45) -- (1.75,-.15);
} 
 &
 \tikzdiagc[scale=.5]{
	\draw (1,2) .. controls (1,1) and (2,1) .. (2,2);
	\draw (1,2) .. controls (1,1.75) and (2,1.75) .. (2,2);
	\draw (1,2) .. controls (1,2.25) and (2,2.25) .. (2,2);
} 
 &
 \tikzdiagh[scale=.5]{
	\draw (1,0) .. controls (1,1) and (2,1) .. (2,0);
	\draw (1,0) .. controls (1,-.25) and (2,-.25) .. (2,0);
	\draw[dashed] (1,0) .. controls (1,.25) and (2,.25) .. (2,0);
	\draw[->] (1.5,1.5) [partial ellipse=0:270:3ex and 1ex];
} 
 &
 \tikzdiagh[xscale=-.5,yscale=.5]{
	\draw (1,0) .. controls (1,1) and (2,1) .. (2,0);
	\draw (1,0) .. controls (1,-.25) and (2,-.25) .. (2,0);
	\draw[dashed] (1,0) .. controls (1,.25) and (2,.25) .. (2,0);
	\draw[->] (1.5,1.5) [partial ellipse=0:270:3ex and 1ex];
} 
\\
 merge 
&
 split
 &
 birth
 &
 positive death
 &
 negative death
\end{tabularx}
\end{center}
with a twist $\tikzdiagc[scale=.5]{
	\draw (0,0) .. controls (0,-.25) and (1,-.25) .. (1,0);
	\draw[dashed] (0,0) .. controls (0,.25) and (1,.25) .. (1,0);
	\draw (2,0) .. controls (2,-.25) and (3,-.25) .. (3,0);
	\draw[dashed] (2,0) .. controls (2,.25) and (3,.25) .. (3,0);
	\draw (0,0) .. controls (0,1) and (2,1) .. (2,2);
	\draw (1,0) .. controls (1,1) and (3,1) .. (3,2);
	\draw (2,0) .. controls (2,1) and (0,1) .. (0,2);
	\draw (3,0) .. controls (3,1) and (1,1) .. (1,2);
	\draw (0,2) .. controls (0,1.75) and (1,1.75) .. (1,2);
	\draw (0,2) .. controls (0,2.25) and (1,2.25) .. (1,2);
	\draw (2,2) .. controls (2,1.75) and (3,1.75) .. (3,2);
	\draw (2,2) .. controls (2,2.25) and (3,2.25) .. (3,2);
}$ acting as a strict symmetry (see \cite[Definition B.9]{putyra14}).

\begin{defn}
The \emph{$\bZ\times \bZ$-degree} of a cobordism $W$ is 
\[
\deg(W) := |W| := (\text{\#births} - \text{\#merges}, \text{\#death} - \text{\#splits}).
\] 
\end{defn} 

\subsection{Linearized $\linchcobcat$}

\begin{figure}
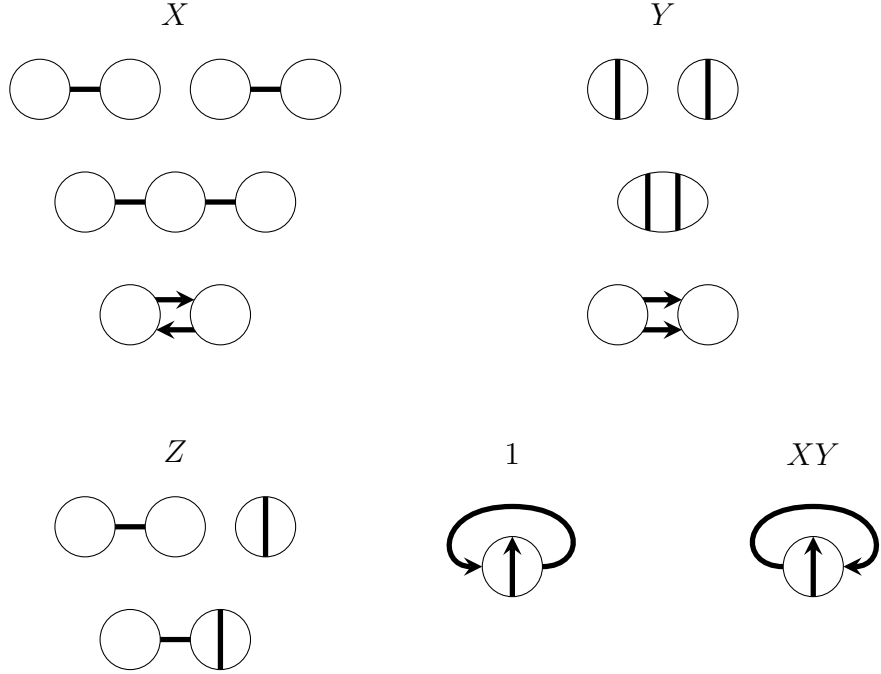

  \centering
\begin{tabularx}{.8\textwidth}{CC}\centering
\tikz{
  	\node at (0,0) {$X$};
  	\draw[giga thick] (-1.8,-1) -- (-.6,-1);
  	\draw[giga thick] (.6,-1) -- (1.8,-1);
  	\draw[fill=white] (-1.8,-1) circle (.4cm);
  	\draw[fill=white] (-.6,-1) circle (.4cm);
  	\draw[fill=white] (.6,-1) circle (.4cm);
  	\draw[fill=white] (1.8,-1) circle (.4cm);
  	\draw[giga thick] (-1.2,-2.5) -- (1.2,-2.5);
  	\draw[fill=white] (-1.2,-2.5) circle (.4cm);
  	\draw[fill=white] (0,-2.5) circle (.4cm);
  	\draw[fill=white] (1.2,-2.5) circle (.4cm);
  	\draw[giga thick,->] (-.6,-3.8) -- (.254,-3.8);
  	\draw[giga thick,->] (.6,-4.2) -- (-.254,-4.2);
  	\draw[fill=white] (-.6,-4) circle (.4cm);
  	\draw[fill=white] (.6,-4) circle (.4cm);
}
&
\tikz{
  	\node at (0,0) {$Y$};
  	\draw[fill=white] (-.6,-1) circle (.4cm);
  	\draw[fill=white] (.6,-1) circle (.4cm);
  	\draw[giga thick] (-.6,-1.4) -- (-.6,-.6);
  	\draw[giga thick] (.6,-1.4) -- (.6,-.6);
  	\draw[fill=white] (0,-2.5) ellipse (.6cm and .4cm);
	\begin{scope}
	   	\clip (0,-2.5) ellipse (.6cm and .4cm);
  		\draw[giga thick] (-.2,-2.9) -- (-.2,-2.1);
  		\draw[giga thick] (.2,-2.9) -- (.2,-2.1);
	\end{scope}
  	\draw[giga thick,->] (-.6,-3.8) -- (.254,-3.8);
  	\draw[giga thick,->] (-.6,-4.2) -- (.254,-4.2);
  	\draw[fill=white] (-.6,-4) circle (.4cm);
  	\draw[fill=white] (.6,-4) circle (.4cm);
}
\\
&
\\
&
\\
  \tikz{
  	\node at (0,0) {$Z$};
  	\draw[giga thick] (-1.2,-1) -- (0,-1);
  	\draw[fill=white] (-1.2,-1) circle (.4cm);
  	\draw[fill=white] (0,-1) circle (.4cm);
  	\draw[fill=white] (1.2,-1) circle (.4cm);
  	\draw[giga thick] (1.2,-1.4) -- (1.2,-.6);
  	\draw[giga thick] (-.6,-2.5) -- (.6,-2.5);
  	\draw[fill=white] (-.6,-2.5) circle (.4cm);
  	\draw[fill=white] (.6,-2.5) circle (.4cm);
  	\draw[giga thick] (.6,-2.9) -- (.6,-2.1);
  }
  &
  \tikz{
  	\node at (-2,0) {$1$};
  	\draw[fill=white] (-2,-1.5) circle (.4cm);
  	\draw[giga thick,->] (-2,-1.9) -- (-2,-1.1);
  	\draw[giga thick,->] (-1.6,-1.5) .. controls (-1,-1.5) and (-1,-.7) .. (-2,-.7)
  		.. controls (-3,-.7) and (-3,-1.5) .. (-2.4,-1.5);
  	\node at (2,0) {$XY$};
  	\draw[fill=white] (2,-1.5) circle (.4cm);
  	\draw[giga thick,->] (2,-1.9) -- (2,-1.1);
  	\draw[giga thick,->] (1.6,-1.5) .. controls (1,-1.5) and (1,-.7) .. (2,-.7)
  		.. controls (3,-.7) and (3,-1.5) .. (2.4,-1.5);
  	\draw[opacity=0] (.6,-2.5) circle (.4cm);
  }
\end{tabularx}
  \caption{Diagrams for elementary changes of chronologies, 
grouped by values of the function $\imath$. Thin lines are the input circles and thick 
arrows visualize saddle points. Orientations of the arrows are omitted if $\imath$ does 
not depend on them. For the $Z$ it means that we go from doing the merge then the split to the split then the merge. The other direction gives $Z^{-1}$. }
\label{fig:chchange}
\end{figure}

Let $R := \bZ[X,Y,Z^{\pm 1}]/ (X^2=Y^2=1)$.  To each elementary change of chronology, we assign a monomial in $R$ as given by \cref{fig:chchange}.
Any change of chronology can be decomposed as a composition of such elementary changes. Let $\imath$ be the map that associate to a change of chronology the resulting monomial after decomposing it and applying the above mentioned rule. From~\cite{putyra14}, we know it is well-defined and multiplicative in the sense that $\imath(H' \circ H) = \imath(H')\imath(H)$ and $\imath(H' \star H) = \imath(H')\imath(H)$. 

\begin{defn}
The \emph{linearized category of cobordims} $\linchcobcat$ is the $R$-linear monoidal category consisting of
\begin{itemize}
\item objects are the same as in $\chcobcat$;
\item morphisms are $R$-linear combinations of morphisms in $\chcobcat$ modulo the relation $W' = \imath(H)W$ for each change of chronology $H : W \rightarrow W'$;
\item the monoidal product is given by right-then-left juxtaposition of cobordisms. 
\end{itemize}
\end{defn}

Let $\lambda_R : (\bZ\times\bZ)^2 \rightarrow R$ be the bilinear map given by
\begin{equation}\label{eq:deflambdaR}
 \lambda_R\bigl((a',b'),(a,b)\bigr) := X^{a'a} Y^{b'b} Z^{a'b-b'a},
\end{equation}
where bilinear means $\lambda_R(x,y+z) = \lambda_R(x,y)\lambda_R(z)$ and $\lambda_R(x+y,z) = \lambda_R(x,z) \lambda_R(y,z)$.

It appears that $\linchcobcat$ admits a nice combinatorial description, with morphisms being $R$-linear combinations of cobordims modulo the following local relations:
\begin{align}\label{eq:reverseorientation}
 \tikzdiagh[scale=.5]{
	\draw (0,0) .. controls (0,1) and (1,1) .. (1,2);
	\draw (1,0) .. controls (1,1) and (2,1) .. (2,0);
	\draw (3,0) .. controls (3,1) and (2,1) .. (2,2);
	\draw (0,0) .. controls (0,-.25) and (1,-.25) .. (1,0);
	\draw[dashed] (0,0) .. controls (0,.25) and (1,.25) .. (1,0);
	\draw (2,0) .. controls (2,-.25) and (3,-.25) .. (3,0);
	\draw[dashed] (2,0) .. controls (2,.25) and (3,.25) .. (3,0);
	\draw (1,2) .. controls (1,1.75) and (2,1.75) .. (2,2);
	\draw (1,2) .. controls (1,2.25) and (2,2.25) .. (2,2);
	\draw [->] (1.15,.15) -- (1.85,.15);
}\   &= X \ \tikzdiagh[scale=.5]{
	\draw (0,0) .. controls (0,1) and (1,1) .. (1,2);
	\draw (1,0) .. controls (1,1) and (2,1) .. (2,0);
	\draw (3,0) .. controls (3,1) and (2,1) .. (2,2);
	\draw (0,0) .. controls (0,-.25) and (1,-.25) .. (1,0);
	\draw[dashed] (0,0) .. controls (0,.25) and (1,.25) .. (1,0);
	\draw (2,0) .. controls (2,-.25) and (3,-.25) .. (3,0);
	\draw[dashed] (2,0) .. controls (2,.25) and (3,.25) .. (3,0);
	\draw (1,2) .. controls (1,1.75) and (2,1.75) .. (2,2);
	\draw (1,2) .. controls (1,2.25) and (2,2.25) .. (2,2);
	\draw [<-] (1.15,.15) -- (1.85,.15);
} 
 &
 \tikzdiagh[xscale=.5,yscale=-.5]{
	\draw (0,0) .. controls (0,1) and (1,1) .. (1,2);
	\draw (1,0) .. controls (1,1) and (2,1) .. (2,0);
	\draw (3,0) .. controls (3,1) and (2,1) .. (2,2);
	\draw (0,0) .. controls (0,-.25) and (1,-.25) .. (1,0);
	\draw (0,0) .. controls (0,.25) and (1,.25) .. (1,0);
	\draw (2,0) .. controls (2,-.25) and (3,-.25) .. (3,0);
	\draw (2,0) .. controls (2,.25) and (3,.25) .. (3,0);
	\draw[dashed] (1,2) .. controls (1,1.75) and (2,1.75) .. (2,2);
	\draw (1,2) .. controls (1,2.25) and (2,2.25) .. (2,2);
	\draw [->] (1.25,.45) -- (1.75,-.15);
}\   &= Y \ \tikzdiagh[xscale=.5,yscale=-.5]{
	\draw (0,0) .. controls (0,1) and (1,1) .. (1,2);
	\draw (1,0) .. controls (1,1) and (2,1) .. (2,0);
	\draw (3,0) .. controls (3,1) and (2,1) .. (2,2);
	\draw (0,0) .. controls (0,-.25) and (1,-.25) .. (1,0);
	\draw (0,0) .. controls (0,.25) and (1,.25) .. (1,0);
	\draw (2,0) .. controls (2,-.25) and (3,-.25) .. (3,0);
	\draw (2,0) .. controls (2,.25) and (3,.25) .. (3,0);
	\draw[dashed] (1,2) .. controls (1,1.75) and (2,1.75) .. (2,2);
	\draw (1,2) .. controls (1,2.25) and (2,2.25) .. (2,2);
	\draw [<-] (1.25,.45) -- (1.75,-.15);
} 
 &
 \tikzdiagh[scale=.5]{
	\draw (1,0) .. controls (1,1) and (2,1) .. (2,0);
	\draw (1,0) .. controls (1,-.25) and (2,-.25) .. (2,0);
	\draw[dashed] (1,0) .. controls (1,.25) and (2,.25) .. (2,0);
	\draw[->] (1.5,1.5) [partial ellipse=0:270:3ex and 1ex];
}\   &= Y  \ \tikzdiagh[xscale=-.5,yscale=.5]{
	\draw (1,0) .. controls (1,1) and (2,1) .. (2,0);
	\draw (1,0) .. controls (1,-.25) and (2,-.25) .. (2,0);
	\draw[dashed] (1,0) .. controls (1,.25) and (2,.25) .. (2,0);
	\draw[->] (1.5,1.5) [partial ellipse=0:270:3ex and 1ex];
} 
 \\
 \label{eq:chcobrelbirthmerge}
 \tikzdiagh[scale=.5]{
	\draw (1,0) .. controls (1,.-.25) and (2,-.25) .. (2,0);
	\draw[dashed] (1,0) .. controls (1,.25) and (2,.25) .. (2,0);
	\draw (1,0) -- (1,4);
	\draw (2,0) -- (2,4);
	\draw (1,4) .. controls (1,3.75) and (2,3.75) .. (2,4);
	\draw (1,4) .. controls (1,4.25) and (2,4.25) .. (2,4);
}\   &= \  \tikzdiagc[scale=.5]{
	\draw (2,0) .. controls (2,.-.25) and (3,-.25) .. (3,0);
	\draw[dashed] (2,0) .. controls (2,.25) and (3,.25) .. (3,0);
	\draw (2,0) -- (2,2);
	\draw (3,0) -- (3,2);
	\draw (2,2) .. controls (2,1.75) and (3,1.75) .. (3,2);
	\draw[dashed] (2,2) .. controls (2,2.25) and (3,2.25) .. (3,2);
	\draw (0,2) .. controls (0,1) and (1,1) .. (1,2);
	\draw (0,2) .. controls (0,1.75) and (1,1.75) .. (1,2);
	\draw[dashed] (0,2) .. controls (0,2.25) and (1,2.25) .. (1,2);
	\draw (0,2) .. controls (0,3) and (1,3) .. (1,4);
	\draw (1,2) .. controls (1,3) and (2,3) .. (2,2);
	\draw (3,2) .. controls (3,3) and (2,3) .. (2,4);
	\draw (1,4) .. controls (1,3.75) and (2,3.75) .. (2,4);
	\draw (1,4) .. controls (1,4.25) and (2,4.25) .. (2,4);
	\draw [<-] (1.15,2.15) -- (1.85,2.15);
} 
&
 \tikzdiagh[scale=.5]{
	\draw (1,0) .. controls (1,.-.25) and (2,-.25) .. (2,0);
	\draw[dashed] (1,0) .. controls (1,.25) and (2,.25) .. (2,0);
	\draw (1,0) -- (1,4);
	\draw (2,0) -- (2,4);
	\draw (1,4) .. controls (1,3.75) and (2,3.75) .. (2,4);
	\draw (1,4) .. controls (1,4.25) and (2,4.25) .. (2,4);
}\   &=  \  \tikzdiagc[xscale=.5,yscale=-.5]{
	\draw (2,0) .. controls (2,.-.25) and (3,-.25) .. (3,0);
	\draw (2,0) .. controls (2,.25) and (3,.25) .. (3,0);
	\draw (2,0) -- (2,2);
	\draw (3,0) -- (3,2);
	\draw[dashed]  (2,2) .. controls (2,1.75) and (3,1.75) .. (3,2);
	\draw(2,2) .. controls (2,2.25) and (3,2.25) .. (3,2);
	\draw (0,2) .. controls (0,1) and (1,1) .. (1,2);
	\draw[dashed] (0,2) .. controls (0,1.75) and (1,1.75) .. (1,2);
	\draw (0,2) .. controls (0,2.25) and (1,2.25) .. (1,2);
	\draw[->] (.5,.5) [partial ellipse=-180:90:3ex and 1ex];
	\draw (0,2) .. controls (0,3) and (1,3) .. (1,4);
	\draw (1,2) .. controls (1,3) and (2,3) .. (2,2);
	\draw (3,2) .. controls (3,3) and (2,3) .. (2,4);
	\draw[dashed]  (1,4) .. controls (1,3.75) and (2,3.75) .. (2,4);
	\draw (1,4) .. controls (1,4.25) and (2,4.25) .. (2,4);
	\draw [->] (1.25,2.45) -- (1.75,1.85);
} 
 &
 \tikzdiagh[scale=.5]{
	\draw (1,0) .. controls (1,.-.25) and (2,-.25) .. (2,0);
	\draw[dashed] (1,0) .. controls (1,.25) and (2,.25) .. (2,0);
	\draw (1,0) -- (1,4);
	\draw (2,0) -- (2,4);
	\draw (1,4) .. controls (1,3.75) and (2,3.75) .. (2,4);
	\draw (1,4) .. controls (1,4.25) and (2,4.25) .. (2,4);
}\   &=   \  \tikzdiagc[xscale=-.5,yscale=-.5]{
	\draw (2,0) .. controls (2,.-.25) and (3,-.25) .. (3,0);
	\draw (2,0) .. controls (2,.25) and (3,.25) .. (3,0);
	\draw (2,0) -- (2,2);
	\draw (3,0) -- (3,2);
	\draw[dashed]  (2,2) .. controls (2,1.75) and (3,1.75) .. (3,2);
	\draw(2,2) .. controls (2,2.25) and (3,2.25) .. (3,2);
	\draw (0,2) .. controls (0,1) and (1,1) .. (1,2);
	\draw[dashed] (0,2) .. controls (0,1.75) and (1,1.75) .. (1,2);
	\draw (0,2) .. controls (0,2.25) and (1,2.25) .. (1,2);
	\draw[->] (.5,.5) [partial ellipse=-180:90:3ex and 1ex];
	\draw (0,2) .. controls (0,3) and (1,3) .. (1,4);
	\draw (1,2) .. controls (1,3) and (2,3) .. (2,2);
	\draw (3,2) .. controls (3,3) and (2,3) .. (2,4);
	\draw[dashed]  (1,4) .. controls (1,3.75) and (2,3.75) .. (2,4);
	\draw (1,4) .. controls (1,4.25) and (2,4.25) .. (2,4);
	\draw [->] (1.75,2.45) -- (1.25,1.85);
} 
\end{align}
and
\begin{align}
\label{eq:cobcommute1}
 \tikzdiagh[scale=.5]{
	\draw (0,0) .. controls (0,.-.25) and (1,-.25) .. (1,0);
	\draw[dashed] (0,0) .. controls (0,.25) and (1,.25) .. (1,0);
	\draw (0,0) -- (0,4);
	\draw (1,0) -- (1,4);
	\draw (0,4) .. controls (0,3.75) and (1,3.75) .. (1,4);
	\draw (0,4) .. controls (0,4.25) and (1,4.25) .. (1,4);
	\node at(1.5,.25) {\tiny $\dots$};
	\node at(1.5,3.75) {\tiny $\dots$};
	\draw (2,0) .. controls (2,.-.25) and (3,-.25) .. (3,0);
	\draw[dashed] (2,0) .. controls (2,.25) and (3,.25) .. (3,0);
	\draw (2,0) -- (2,4);
	\draw (3,0) -- (3,4);
	\draw (2,4) .. controls (2,3.75) and (3,3.75) .. (3,4);
	\draw (2,4) .. controls (2,4.25) and (3,4.25) .. (3,4);
	\node at(3.5,.25) {\tiny $\dots$};
	\node at(3.5,3.75) {\tiny $\dots$};
	\draw (4,0) .. controls (4,.-.25) and (5,-.25) .. (5,0);
	\draw[dashed] (4,0) .. controls (4,.25) and (5,.25) .. (5,0);
	\draw (4,0) -- (4,4);
	\draw (5,0) -- (5,4);
	\draw (4,4) .. controls (4,3.75) and (5,3.75) .. (5,4);
	\draw (4,4) .. controls (4,4.25) and (5,4.25) .. (5,4);
	\filldraw [fill=white, draw=black,rounded corners] (-.5,.5) rectangle (2.5,1.5) node[midway] { $W'$};
	\filldraw [fill=white, draw=black,rounded corners] (2.5,2.5) rectangle (5.5,3.5) node[midway] { $W$};
}\   &= \lambda_R(|W|,|W'|) \ \tikzdiagh[scale=.5]{
	\draw (0,0) .. controls (0,.-.25) and (1,-.25) .. (1,0);
	\draw[dashed] (0,0) .. controls (0,.25) and (1,.25) .. (1,0);
	\draw (0,0) -- (0,4);
	\draw (1,0) -- (1,4);
	\draw (0,4) .. controls (0,3.75) and (1,3.75) .. (1,4);
	\draw (0,4) .. controls (0,4.25) and (1,4.25) .. (1,4);
	\node at(1.5,.25) {\tiny $\dots$};
	\node at(1.5,3.75) {\tiny $\dots$};
	\draw (2,0) .. controls (2,.-.25) and (3,-.25) .. (3,0);
	\draw[dashed] (2,0) .. controls (2,.25) and (3,.25) .. (3,0);
	\draw (2,0) -- (2,4);
	\draw (3,0) -- (3,4);
	\draw (2,4) .. controls (2,3.75) and (3,3.75) .. (3,4);
	\draw (2,4) .. controls (2,4.25) and (3,4.25) .. (3,4);
	\node at(3.5,.25) {\tiny $\dots$};
	\node at(3.5,3.75) {\tiny $\dots$};
	\draw (4,0) .. controls (4,.-.25) and (5,-.25) .. (5,0);
	\draw[dashed] (4,0) .. controls (4,.25) and (5,.25) .. (5,0);
	\draw (4,0) -- (4,4);
	\draw (5,0) -- (5,4);
	\draw (4,4) .. controls (4,3.75) and (5,3.75) .. (5,4);
	\draw (4,4) .. controls (4,4.25) and (5,4.25) .. (5,4);
	\filldraw [fill=white, draw=black,rounded corners] (-.5,2.5) rectangle (2.5,3.5) node[midway] { $W'$};
	\filldraw [fill=white, draw=black,rounded corners] (2.5,.5) rectangle (5.5,1.5) node[midway] { $W$};
} 
 \\
\label{eq:cobcommute2}
 \tikzdiagh[scale=.5]{
	\draw (0,0) .. controls (0,.-.25) and (1,-.25) .. (1,0);
	\draw[dashed] (0,0) .. controls (0,.25) and (1,.25) .. (1,0);
	\draw (0,0) -- (0,4);
	\draw (1,0) -- (1,4);
	\draw (0,4) .. controls (0,3.75) and (1,3.75) .. (1,4);
	\draw (0,4) .. controls (0,4.25) and (1,4.25) .. (1,4);
	\node at(1.5,.25) {\tiny $\dots$};
	\node at(1.5,3.75) {\tiny $\dots$};
	\draw (2,0) .. controls (2,.-.25) and (3,-.25) .. (3,0);
	\draw[dashed] (2,0) .. controls (2,.25) and (3,.25) .. (3,0);
	\draw (2,0) -- (2,4);
	\draw (3,0) -- (3,4);
	\draw (2,4) .. controls (2,3.75) and (3,3.75) .. (3,4);
	\draw (2,4) .. controls (2,4.25) and (3,4.25) .. (3,4);
	\filldraw [fill=white, draw=black,rounded corners] (-.5,.5) rectangle (3.5,1.5) node[midway] { $W'$};
} \  \tikzdiagh[scale=.5]{
	\draw (0,0) .. controls (0,.-.25) and (1,-.25) .. (1,0);
	\draw[dashed] (0,0) .. controls (0,.25) and (1,.25) .. (1,0);
	\draw (0,0) -- (0,4);
	\draw (1,0) -- (1,4);
	\draw (0,4) .. controls (0,3.75) and (1,3.75) .. (1,4);
	\draw (0,4) .. controls (0,4.25) and (1,4.25) .. (1,4);
	\node at(1.5,.25) {\tiny $\dots$};
	\node at(1.5,3.75) {\tiny $\dots$};
	\draw (2,0) .. controls (2,.-.25) and (3,-.25) .. (3,0);
	\draw[dashed] (2,0) .. controls (2,.25) and (3,.25) .. (3,0);
	\draw (2,0) -- (2,4);
	\draw (3,0) -- (3,4);
	\draw (2,4) .. controls (2,3.75) and (3,3.75) .. (3,4);
	\draw (2,4) .. controls (2,4.25) and (3,4.25) .. (3,4);
	\filldraw [fill=white, draw=black,rounded corners] (-.5,2.5) rectangle (3.5,3.5) node[midway] { $W$};
} \   &= \lambda_R(|W|,|W'|) \  \tikzdiagh[scale=.5]{
	\draw (0,0) .. controls (0,.-.25) and (1,-.25) .. (1,0);
	\draw[dashed] (0,0) .. controls (0,.25) and (1,.25) .. (1,0);
	\draw (0,0) -- (0,4);
	\draw (1,0) -- (1,4);
	\draw (0,4) .. controls (0,3.75) and (1,3.75) .. (1,4);
	\draw (0,4) .. controls (0,4.25) and (1,4.25) .. (1,4);
	\node at(1.5,.25) {\tiny $\dots$};
	\node at(1.5,3.75) {\tiny $\dots$};
	\draw (2,0) .. controls (2,.-.25) and (3,-.25) .. (3,0);
	\draw[dashed] (2,0) .. controls (2,.25) and (3,.25) .. (3,0);
	\draw (2,0) -- (2,4);
	\draw (3,0) -- (3,4);
	\draw (2,4) .. controls (2,3.75) and (3,3.75) .. (3,4);
	\draw (2,4) .. controls (2,4.25) and (3,4.25) .. (3,4);
	\filldraw [fill=white, draw=black,rounded corners] (-.5,2.5) rectangle (3.5,3.5) node[midway] { $W'$};
} \ \tikzdiagh[scale=.5]{
	\draw (0,0) .. controls (0,.-.25) and (1,-.25) .. (1,0);
	\draw[dashed] (0,0) .. controls (0,.25) and (1,.25) .. (1,0);
	\draw (0,0) -- (0,4);
	\draw (1,0) -- (1,4);
	\draw (0,4) .. controls (0,3.75) and (1,3.75) .. (1,4);
	\draw (0,4) .. controls (0,4.25) and (1,4.25) .. (1,4);
	\node at(1.5,.25) {\tiny $\dots$};
	\node at(1.5,3.75) {\tiny $\dots$};
	\draw (2,0) .. controls (2,.-.25) and (3,-.25) .. (3,0);
	\draw[dashed] (2,0) .. controls (2,.25) and (3,.25) .. (3,0);
	\draw (2,0) -- (2,4);
	\draw (3,0) -- (3,4);
	\draw (2,4) .. controls (2,3.75) and (3,3.75) .. (3,4);
	\draw (2,4) .. controls (2,4.25) and (3,4.25) .. (3,4);
	\filldraw [fill=white, draw=black,rounded corners] (-.5,.5) rectangle (3.5,1.5) node[midway] { $W$};
} 
\end{align}
for all cobordisms $W'$ and $W$.

Note that in particular, two diffeomorphic cobordisms (with different chronology) are equal in $\Hom_\linchcobcat$  up to a scalar in $R$. However, this scalar is in general not unique since, for example:
\[
\tikzdiagc[scale=.5]{
	\draw (0,0) .. controls (0,1) and (1,1) .. (1,2);
	\draw (1,0) .. controls (1,1) and (2,1) .. (2,0);
	\draw (3,0) .. controls (3,1) and (2,1) .. (2,2);
	\draw (0,0) .. controls (0,-.25) and (1,-.25) .. (1,0);
	\draw[dashed] (0,0) .. controls (0,.25) and (1,.25) .. (1,0);
	\draw (2,0) .. controls (2,-.25) and (3,-.25) .. (3,0);
	\draw[dashed] (2,0) .. controls (2,.25) and (3,.25) .. (3,0);
	\draw (1,2) .. controls (1,1.75) and (2,1.75) .. (2,2);
	\draw (1,2) .. controls (1,2.25) and (2,2.25) .. (2,2);
	\draw [->] (1.15,.25) -- (1.85,.25);
	\draw[dashed] (1,-2) .. controls (1,-1.75) and (2,-1.75) .. (2,-2);
	\draw (1,-2) .. controls (1,-2.25) and (2,-2.25) .. (2,-2);
	\draw  (1,-2) .. controls (1,-1) and (0,-1) .. (0,0);
	\draw  (2,-2) .. controls (2,-1) and (3,-1) .. (3,0);
	\draw (1,0) .. controls (1,-1) and (2,-1) .. (2,0);
	\draw [->] (1.35,-.55) -- (1.75,.05);
}
\ = XY\ 
\tikzdiagc[scale=.5]{
	\draw (0,0) .. controls (0,1) and (1,1) .. (1,2);
	\draw (1,0) .. controls (1,1) and (2,1) .. (2,0);
	\draw (3,0) .. controls (3,1) and (2,1) .. (2,2);
	\draw (0,0) .. controls (0,-.25) and (1,-.25) .. (1,0);
	\draw[dashed] (0,0) .. controls (0,.25) and (1,.25) .. (1,0);
	\draw (2,0) .. controls (2,-.25) and (3,-.25) .. (3,0);
	\draw[dashed] (2,0) .. controls (2,.25) and (3,.25) .. (3,0);
	\draw (1,2) .. controls (1,1.75) and (2,1.75) .. (2,2);
	\draw (1,2) .. controls (1,2.25) and (2,2.25) .. (2,2);
	\draw [->] (1.15,.25) -- (1.85,.25);
	\draw[dashed] (1,-2) .. controls (1,-1.75) and (2,-1.75) .. (2,-2);
	\draw (1,-2) .. controls (1,-2.25) and (2,-2.25) .. (2,-2);
	\draw  (1,-2) .. controls (1,-1) and (0,-1) .. (0,0);
	\draw  (2,-2) .. controls (2,-1) and (3,-1) .. (3,0);
	\draw (1,0) .. controls (1,-1) and (2,-1) .. (2,0);
	\draw [->] (1.35,-.55) -- (1.75,.05);
}
\]

\subsection{Cobordisms with corners}

We define a \emph{chronological cobordism with corners} in the same fashion as a chronological cobordism except that:
\begin{itemize}
\item $W$ is contained in $I \times \bR \times I$;
\item $\partial W \subset I \times \bR \times \partial I \cup \partial I\times \bR \times I$;
\item $\partial W \cap \partial I\times \bR \times I = \sqcup_i \{p_i \times I\} $ for a finite collection of disjoint $p_i \in \partial I \times \bR \times \{0\}$;
\item $W$ is transverse to $I \times \bR \times \partial I$ and to $\partial I \times \bR \times I$;
\end{itemize}
Similarly, we ask for a change of chronology $H_t$ to restrict on $\partial I \cap \bR \cap I$ as $H_t(x,y,z) = (x,h_t(y),h'_t(z))$ for some diffeotopies $h_t, h'_t$ on $I$.
All this means the cobordism can have borders, but these are vertical, and a change of chronology preserves that. 

We say that two cobordisms with corners $W'$ and $W$ are \emph{horizontally composable} if 
\[
\#\{p_i \in  \{0\} \times I \times \{0\} \cap W' \} = \{p_i \in \{1\} \times I \times \{0\} \cap W \}.
\]
This means that $W'$ has the same number of boundaries on the right than $W$ on the left. Thus, in that situation we can horizontally compose $W'$ with $W$ by doing a right-then-left composition $W' \bullet W$, adding small curtains to smooth the composition if necessary, and rescaling. For example:
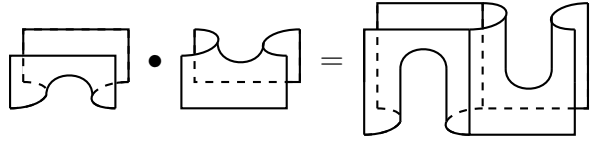
\[
\tikzdiagc[scale=.35]{
	\draw (.5,1) -- (.5,3) -- (4.5,3) -- (4.5,1);
	\draw (0,0) .. controls (1.5,0) and (2,1) .. (.5,1); 
	\draw (4,0) .. controls (2.5,0) and (3,1) .. (4.5,1); 
	\filldraw[fill=white, draw=white] (.5,.55)  rectangle  (4,2);
	\draw[dashed] (0,0) .. controls (1.5,0) and (2,1) .. (.5,1); 
	\draw[dashed] (4,0) .. controls (2.5,0) and (3,1) .. (4.5,1); 
	\draw (1.375,.5) .. controls (1.375,1.5) and (3.125,1.5) .. (3.125,.5);
	\draw[dashed] (.5,1) -- (.5,3) -- (4.5,3) -- (4.5,1);
	\draw (0,0) -- (0,2) -- (4,2) -- (4,0);
}
\ \bullet \ 
\tikzdiagc[scale=.35]{
	\draw (.5,3) -- (.5,1) -- (4.5,1) -- (4.5,3);
	\filldraw [fill=white, draw=white] (0,0) rectangle (4,2); 
	\draw (0,2) .. controls (1.5,2) and (2,3) .. (.5,3); 
	\draw (4,2) .. controls (2.5,2) and (3,3) .. (4.5,3); 
	\draw (1.375,2.5) .. controls (1.375,1.5) and (3.125,1.5) .. (3.125,2.5);
	\draw[dashed] (.5,3) -- (.5,1) -- (4.5,1) -- (4.5,3);
	\draw (0,2) -- (0,0) -- (4,0) -- (4,2);
}
\ = \ 
\tikzdiagc[scale=.35]{
	\draw (0,0) .. controls (1.5,0) and (2,1) .. (.5,1); 
	\draw (4,0) .. controls (2.5,0) and (3,1) .. (4.5,1); 
	\filldraw[fill=white, draw=white] (.5,.55)  rectangle  (4.5,2);
	\draw[dashed] (0,0) .. controls (1.5,0) and (2,1) .. (.5,1); 
	\draw[dashed] (4,0) .. controls (2.5,0) and (3,1) .. (4.5,1); 
	\draw (0,0) -- (0,4);
	\draw[dashed] (.5,1) -- (.5,5);
	\draw (1.375,.5) -- (1.375,2.5);
	\draw (3.125,.5) -- (3.125,2.5);
	\draw (.5,4) -- (.5,5) -- (4.5,5) -- (4.5,4);
	%
	%
	%
	%
	\draw (1.375,2.5) .. controls (1.375,3.5) and (3.125,3.5) .. (3.125,2.5);
	%
	\draw (0,2) -- (0,4) -- (4,4);
	%
	%
	%
	%
	\draw (8,1) -- (8.5,1) -- (8.5,5);
	%
	%
	%
	\draw (5.375,2.5) .. controls (5.375,1.5) and (7.125,1.5) .. (7.125,2.5);
	\draw[dashed] (4.5,5) -- (4.5,1) -- (8.5,1) -- (8.5,3);
	\draw (4,4) -- (4,0) -- (8,0) -- (8,2);
	\draw (4,4) .. controls (5.5,4) and (6,5) .. (4.5,5); 
	\draw (8,4) .. controls (6.5,4) and (7,5) .. (8.5,5); 
	\draw (8,2) -- (8,4);
	\draw (8.5,3) -- (8.5,5);
	\draw (5.375,2.5) -- (5.375,4.5);
	\draw (7.125,2.5) -- (7.125,4.5);
}
\]

\begin{rem}
Note that the $\bZ\times \bZ$-grading of cobordisms does not extend to  cobordisms with corners.
\end{rem}




\section{Arc algebras and bimodules}\label{sec:arcbimod}

We construct the non-associative covering arc algebras $H^n$ over $R$, as well as non-associative bimodules associated to flat tangles, using the procedure from \cite{khovanovHn} and the chronological TQFT from \cite{putyra14}.

\subsection{Flat tangles}

Recall that a \emph{flat tangle} is a tangle diagram with no crossing. 
Let $B_n^m$ be the set of classes of flat tangles from $2n$ fixed points on the horizontal line $\bR \times \{ 0 \} \subset \bR^2$ to $2m$ fixed points on $\bR \times \{1\} \subset \bR^2$, taken up to ambient isotopy that fixes the endpoints. Let $\bar{\phantom a} : B_m^n \rightarrow B_n^m$ be the map that takes the  mirror image along the horizontal line $\bR \times \{1/2\}$. 
When $m = 0$, we simply write $B^n$, and we refer to elements in $B^n$ as \emph{crossingless matchings}. We  put $ B^\bullet := \sqcup_{n \geq 0} B^n$, and we write $|a| = n$ whenever $a \in B^n$. There is a composition map
\[
B_n^{p} \times B_m^n \rightarrow B_m^p, \quad (t',t) \mapsto t't,
\]
where $t' t$ is given by gluing $t'$ on top of $t$, and rescaling. 
Let $1_n \in B_n^n$ be the identity tangle on $2n$ strands:
\[
1_n := 
\tikzdiag[xscale=.5]{0}{
	\draw (0,0) -- (0,1);
	\draw (1,0) -- (1,1);
	\node at(2,.5) {$\dots$};
	\draw (3,0) -- (3,1);
	\draw (4,0) -- (4,1);
	\tikzbrace{0}{4}{0}{$2n$};
}
\]
It is a neutral element for the composition map. 

\subsection{Surgery}\label{ssec:surgery}

For $a \in B^n, b\in B^m$ and $t \in B_n^m$, we have that $\bar bta$ is a closed 1-manifold. 
Let $1_{ba}(t) : \bar bta \rightarrow \bar bta$ be the identity cobordism given by $ \bar bta \times [0,1]$. We write $\un_t$ for the identity cobordism with corners $t \times [0,1]$. 
Given two flat tangles $t' \in B^{|c|}_{|b|}$ and $t \in B_{|a|}^{|b|}$, there is a unique  (up to homeomorphism) minimal cobordism $ \bar c t' b \bar b t a \rightarrow \bar c t' t a$ given by surgery over $b$, that is by contracting the symmetric arcs of $\bar b b$ using saddles. 
For each such cobordism, we choose the chronology given by adding the saddle points from right to left and orienting everything upward.
 We write $\chcob_{cba}(t',t) : \bar c t' b \bar b t a \rightarrow \bar c t' t a$ the resulting chronological cobordism.
 Note that $\chcob_{cba}(t',t)$ has Euler characteristic $-|b|$. 

\begin{exe}
Take 
\begin{align*}
a = c &= \ 
\tikzdiag[scale=.5]{1ex}{
	\draw (0,0) .. controls (0,-1) and (3,-1) .. (3,0);
	\draw (1,0) .. controls (1,-.5) and (2,-.5) .. (2,0);
}
&
b &= \ 
\tikzdiag[scale=.5]{1ex}{
	\draw (0,0) .. controls (0,-.5) and (1,-.5) .. (1,0);
	\draw (2,0) .. controls (2,-.5) and (3,-.5) .. (3,0);
}
\end{align*}
in $B^2$. Then, $\chcob_{cba}(1_2,1_2)$ is given by the movie:
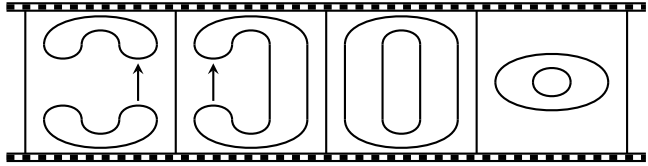
\begin{center}
\tikzdiagc[scale=.5]{
	\draw (.5,0) -- (.5,4);
	\node at(2.5,2) {\tikz[scale=.5,  thick]{
		\draw (1,1) .. controls (1,0) and (4,0) .. (4,1);
		\draw (2,1) .. controls (2,.5) and (3,.5) .. (3,1);
		\draw (1,1) .. controls (1,1.5) and (2,1.5) .. (2,1);
		\draw (3,1) .. controls (3,1.5) and (4,1.5) .. (4,1);
		\draw [->] (3.5,1.5) -- (3.5,2.5);
		\draw (1,3) .. controls (1,2.5) and (2,2.5) .. (2,3);
		\draw (3,3) .. controls (3,2.5) and (4,2.5) .. (4,3);
		\draw (1,3) .. controls (1,4) and (4,4) .. (4,3);
		\draw (2,3) .. controls (2,3.5) and (3,3.5) .. (3,3);
	}};
	\draw (4.5,0) -- (4.5,4);
	\node at(6.5,2) {\tikz[scale=.5,  thick]{
		\draw (1,1) .. controls (1,0) and (4,0) .. (4,1);
		\draw (2,1) .. controls (2,.5) and (3,.5) .. (3,1);
		\draw (1,1) .. controls (1,1.5) and (2,1.5) .. (2,1);
		\draw [->] (1.5,1.5) -- (1.5,2.5);
		\draw (3,1) -- (3,3);
		\draw (4,1) -- (4,3);
		\draw (1,3) .. controls (1,2.5) and (2,2.5) .. (2,3);
		\draw (1,3) .. controls (1,4) and (4,4) .. (4,3);
		\draw (2,3) .. controls (2,3.5) and (3,3.5) .. (3,3);
	}};
	\draw (8.5,0) -- (8.5,4);
	\node at(10.5,2) {\tikz[scale=.5,  thick]{
		\draw (1,1) .. controls (1,0) and (4,0) .. (4,1);
		\draw (2,1) .. controls (2,.5) and (3,.5) .. (3,1);
		\draw (1,1) -- (1,3);
		\draw (2,1) -- (2,3);
		\draw (3,1) -- (3,3);
		\draw (4,1) -- (4,3);
		\draw (1,3) .. controls (1,4) and (4,4) .. (4,3);
		\draw (2,3) .. controls (2,3.5) and (3,3.5) .. (3,3);
	}};
	\draw (12.5,0) -- (12.5,4);
	\node at(14.5,2) {\tikz[scale=.5,  thick]{
		\draw (1,1) .. controls (1,0) and (4,0) .. (4,1);
		\draw (2,1) .. controls (2,.5) and (3,.5) .. (3,1);
		\draw (1,1) .. controls (1,2) and (4,2) .. (4,1);
		\draw (2,1) .. controls (2,1.5) and (3,1.5) .. (3,1);
	}};
	\draw (16.5,0) -- (16.5,4);
	\draw[movieline] (0,0) -- (17,0);
	\draw[moviedashed] (0,0) -- (17,0);
	\draw[movieline] (0,4) -- (17,4);
	\draw[moviedashed] (0,4) -- (17,4);
}
\end{center}
\end{exe}

\subsection{Chronological TQFT}\label{sec:oddfunctor}

Let $\Mod_R$ be the category of $\bZ\times\bZ$-graded $R$-modules with maps that preserve the degree. 
Recall the  bilinear map $\lambda_R : (\bZ\times\bZ)^2 \rightarrow R$ given by
\[
 \lambda_R\bigl((a',b'),(a,b)\bigr) := X^{a'a} Y^{b'b} Z^{a'b-b'a},
\]
and note that $\lambda_R^{-1}(x,y) = \lambda_R(y,x)$. 
The category $\Mod_R$ is equipped with the usual notion of tensor product, but comes with a symmetry $\tau_{M,N} : M \otimes N \rightarrow N \otimes M$ given by 
\[
\tau_{M,N}(m \otimes n) := \lambda_R(\deg m, \deg n) n \otimes m.
\]
We picture this formula as:
\[
\tikzdiag{0}
{
	\draw (0,-.25) node[below] {$m$} -- (0,0) .. controls (0,.5) and (1,.5) .. (1,1);
	\draw (1,-.75) node[below] {$n$} -- (1,0)  .. controls (1,.5) and (0,.5) .. (0,1);
}
\ = \ 
\tikzdiag{0}
{
	\draw (1,-.25) node[below] {$m$} -- (1,0) -- (1,1);
	\draw (0,-.75) node[below] {$n$} -- (0,0) -- (0,1);
}
\ = \lambda_R(|m|,|n|) \ 
\tikzdiag{0}
{
	\draw (1,-.75) node[below] {$m$} -- (1,0) -- (1,1);
	\draw (0,-.25) node[below] {$n$} -- (0,0) -- (0,1);
}
\]

Let $A := Rv_+ \oplus  Rv_-$ be the free $R$-module generated by two elements $v_+$ and $v_-$ with degrees
\begin{align*}
\deg_R(v_+) &:= (1, 0),
&
\deg_R(v_-) &:= (0,-1).
\end{align*}
Then, we consider the chronological TQFT $\tqft : \linchcobcat \rightarrow \Mod_R$ from~\cite{putyra14} mapping the monoidal product of $\linchcobcat$ to the tensor product of $ \Mod_R$, the twist cobordism to the symmetry $\tau$, a single circle to $A$, and
\begin{align*}
\tqft\left(\tikzdiagh[scale=.5]{
	\draw (0,0) .. controls (0,1) and (1,1) .. (1,2);
	\draw (1,0) .. controls (1,1) and (2,1) .. (2,0);
	\draw (3,0) .. controls (3,1) and (2,1) .. (2,2);
	\draw (0,0) .. controls (0,-.25) and (1,-.25) .. (1,0);
	\draw[dashed] (0,0) .. controls (0,.25) and (1,.25) .. (1,0);
	\draw (2,0) .. controls (2,-.25) and (3,-.25) .. (3,0);
	\draw[dashed] (2,0) .. controls (2,.25) and (3,.25) .. (3,0);
	\draw (1,2) .. controls (1,1.75) and (2,1.75) .. (2,2);
	\draw (1,2) .. controls (1,2.25) and (2,2.25) .. (2,2);
	\draw [<-] (1.15,.15) -- (1.85,.15);
}\right) : A \otimes A \rightarrow A &:= 
\begin{cases}
v_+ \otimes v_+ \mapsto v_+, & v_+ \otimes v_- \mapsto v_-, \\
v_- \otimes v_- \mapsto 0, & v_- \otimes v_+ \mapsto XZ v_-,
\end{cases}
\\
\tqft\left(\tikzdiagh[xscale=.5,yscale=-.5]{
	\draw (0,0) .. controls (0,1) and (1,1) .. (1,2);
	\draw (1,0) .. controls (1,1) and (2,1) .. (2,0);
	\draw (3,0) .. controls (3,1) and (2,1) .. (2,2);
	\draw (0,0) .. controls (0,-.25) and (1,-.25) .. (1,0);
	\draw (0,0) .. controls (0,.25) and (1,.25) .. (1,0);
	\draw (2,0) .. controls (2,-.25) and (3,-.25) .. (3,0);
	\draw (2,0) .. controls (2,.25) and (3,.25) .. (3,0);
	\draw[dashed] (1,2) .. controls (1,1.75) and (2,1.75) .. (2,2);
	\draw (1,2) .. controls (1,2.25) and (2,2.25) .. (2,2);
	\draw [->] (1.25,.45) -- (1.75,-.15);
}\right) : A  \rightarrow A \otimes A &:= 
\begin{cases}
v_+ \mapsto v_- \otimes v_+ + YZ v_+ \otimes v_-,  &\\
v_- \mapsto v_- \otimes v_-, &
\end{cases}
\\
\tqft\left(\tikzdiagc[scale=.5]{
	\draw (1,2) .. controls (1,1) and (2,1) .. (2,2);
	\draw (1,2) .. controls (1,1.75) and (2,1.75) .. (2,2);
	\draw (1,2) .. controls (1,2.25) and (2,2.25) .. (2,2);
}\right) : R  \rightarrow A  &:= 
\begin{cases}
1 \mapsto v_+, & 
\end{cases}
\\
\tqft\left(\tikzdiagh[xscale=-.5,yscale=.5]{
	\draw (1,0) .. controls (1,1) and (2,1) .. (2,0);
	\draw (1,0) .. controls (1,-.25) and (2,-.25) .. (2,0);
	\draw[dashed] (1,0) .. controls (1,.25) and (2,.25) .. (2,0);
	\draw[->] (1.5,1.5) [partial ellipse=0:270:3ex and 1ex];
}\right) : A  \rightarrow R &:= 
\begin{cases}
v_+ \mapsto 0,  &\\
v_- \mapsto 1. &
\end{cases}
\end{align*}

\subsection{Dotted cobordisms}

There is also a notion of dotted cobordisms (see \cite[\S11]{putyra14}), where we allow the chronological cobordism (with corners or not) to be decorated by dots on the surface. A dot has degree $(-1,-1)$ and move freely on the surface at the cost of multiplying by scalars in $R$ whenever we exchange dots with critical points, respecting~\cref{eq:cobcommute1} and ~\cref{eq:cobcommute2}. Let $\dotchcobcat$ be the 2-category defined similarly as $\linchcobcat$ but where cobordisms are dotted, modulo the following \emph{dot relations}:
\begin{align}\label{eq:dotcob}
\tikzdiagh[xscale=-.5,yscale=.5]{
	\draw (1,0) .. controls (1,1) and (2,1) .. (2,0);
	\draw (1,0) .. controls (1,-.25) and (2,-.25) .. (2,0);
	\draw[dashed] (1,0) .. controls (1,.25) and (2,.25) .. (2,0);
	\draw[->] (1.5,1.5) [partial ellipse=0:270:3ex and 1ex];
	\draw (1,0) .. controls (1,-1) and (2,-1) .. (2,0);
	%
} \ &= 0,
&
\tikzdiagh[xscale=-.5,yscale=.5]{
	\draw (1,0) .. controls (1,1) and (2,1) .. (2,0);
	\draw (1,0) .. controls (1,-.25) and (2,-.25) .. (2,0);
	\draw[dashed] (1,0) .. controls (1,.25) and (2,.25) .. (2,0);
	\draw[->] (1.5,1.5) [partial ellipse=0:270:3ex and 1ex];
	\draw (1,0) .. controls (1,-1) and (2,-1) .. (2,0);
	%
	\node[tikzdot] at (1.5,.25) {};
} \ &= 1,
&
\tikzdiagh[scale=.5]{
	\draw (1,0) .. controls (1,.-.25) and (2,-.25) .. (2,0);
	\draw[dashed] (1,0) .. controls (1,.25) and (2,.25) .. (2,0);
	\draw (1,0) -- (1,4);
	\draw (2,0) -- (2,4);
	\draw (1,4) .. controls (1,3.75) and (2,3.75) .. (2,4);
	\draw (1,4) .. controls (1,4.25) and (2,4.25) .. (2,4);
}\ &= \ 
\tikzdiagh[xscale=-.5,yscale=.5]{
	\draw (1,0) .. controls (1,1) and (2,1) .. (2,0);
	\draw (1,0) .. controls (1,-.25) and (2,-.25) .. (2,0);
	\draw[dashed] (1,0) .. controls (1,.25) and (2,.25) .. (2,0);
	\draw[->] (1.5,1.5) [partial ellipse=0:270:3ex and 1ex];
	\node[tikzdot] at(1.5,.5){};
	\draw (1,4) .. controls (1,3) and (2,3) .. (2,4);
	\draw (1,4) .. controls (1,3.75) and (2,3.75) .. (2,4);
	\draw (1,4) .. controls (1,4.25) and (2,4.25) .. (2,4);
}
\ + \ 
\tikzdiagh[xscale=-.5,yscale=.5]{
	\draw (1,0) .. controls (1,1) and (2,1) .. (2,0);
	\draw (1,0) .. controls (1,-.25) and (2,-.25) .. (2,0);
	\draw[dashed] (1,0) .. controls (1,.25) and (2,.25) .. (2,0);
	\draw[->] (1.5,1.5) [partial ellipse=0:270:3ex and 1ex];
	\node[tikzdot] at(1.5,3.5){};
	\draw (1,4) .. controls (1,3) and (2,3) .. (2,4);
	\draw (1,4) .. controls (1,3.75) and (2,3.75) .. (2,4);
	\draw (1,4) .. controls (1,4.25) and (2,4.25) .. (2,4);
} \ ,
&
\tikzdiagh[xscale=1,yscale=1]{
	\draw (0,1) .. controls (.25,1.25) and (.75,1-.25) .. (1,1);
	\draw (0,0) .. controls (.25,.25) and (.75,-.25) .. (1,0);
	\draw (0,0) -- (0,1);
	\draw (1,0) -- (1,1);
	\node[tikzdot]at(.5,.25){};
	\node[tikzdot]at(.5,.75){};
}
\ &= 0,
\end{align}
and extending \cref{eq:cobcommute1} and \cref{eq:cobcommute2} to dots. 

The functor $\tqft$ extends to a functor $\tqft_\bullet :  \dotchcobcat \rightarrow \Mod_R$ with
\[
\tqft_\bullet\left( \ 
\tikzdiagh[scale=.5]{
	\draw (1,0) .. controls (1,.-.25) and (2,-.25) .. (2,0);
	\draw[dashed] (1,0) .. controls (1,.25) and (2,.25) .. (2,0);
	\draw (1,0) -- (1,2);
	\draw (2,0) -- (2,2);
	\draw (1,2) .. controls (1,1.75) and (2,1.75) .. (2,2);
	\draw (1,2) .. controls (1,2.25) and (2,2.25) .. (2,2);
	\node[tikzdot] at (1.5,1){};
} \ 
\right) : A \rightarrow A 
:= \begin{cases}
v_+ \mapsto v_-, &\\
v_- \mapsto 0. &
\end{cases}
\]

\subsection{Arc spaces and arc algebras}\label{sec:arcspaces}

Consider a flat tangle $t \in B_n^m$. Define the space
\begin{align*}
\tqft(t) &:= \bigoplus_{a\in B^n, b\in B^m} {_b}\tqft(t){_a},&
 {_b}\tqft(t){_a} &:= \tqft(\bar{b} t a) .
\end{align*}
Given another flat tangle $t' \in B_m^k$, define the composition map $\mu[t',t]$ as
\[
{_d}\tqft(t'){_c} \otimes_R {_b}\tqft(t){_a} \xrightarrow{\mu[t',t]} 0,
\]
whenever $c \neq b$, and as making the following diagram commutes otherwise:
\[
\begin{tikzcd}[column sep=10ex]
{_c}\tqft(t'){_b} \otimes_R {_b}\tqft(t){_a} \ar{rr}{\mu_{cba}[t',t]}  \ar[sloped]{d}{\simeq} && {_c}\tqft(t't){_a}  \\
\tqft(\bar{c} t' b) \otimes_R \tqft(\bar{b} t a)  \ar[sloped]{r}{\simeq} & \tqft(\bar{c}t'b \bar{b}ta)  \ar{r}{\tqft(\chcob_{cba}(t',t))} & \tqft(\bar{c}t'ta) \ar[sloped]{u}{\simeq}
\end{tikzcd}
\]
for $c \in B^k, b \in  B^m$ and $a \in B^n$. 

\begin{rem}
Note that the composition map $\mu[t',t]$ does not preserve the $\bZ\times\bZ$-grading. 
This will be fixed in \cref{sec:tanglehomology} using the tools from \cref{sec:gradcat}.
\end{rem}

\begin{defn}
The \emph{arc algebra $H^n$} is given by 
\[
H^n := \bigoplus _{a,b \in B^n} {_b}\tqft(1_n){_a},
\]
with multiplication $\mu[1_n, 1_n]$. 
\end{defn}

\begin{exe}
As a simple example, we can take $n=1$. Then, $H^1 \cong A$ with $\mu(v_+,v_+) = v_+$, $\mu(v_+,v_-) = v_-$, $\mu(v_-,v_+) = XZ v_-$ and $\mu(v_-,v_-) = 0$. 
\end{exe}

\begin{rem}
Note that $H^n$ is not associative in the proper sense (see~\cite{naissevaz} for an example with $n=2$) and also not unital. One of the main goal of~\cref{sec:gradcat} will be to construct a framework  such that it becomes associative and unital as an algebra object in the right monoidal category. 
\end{rem}




\section{$\cC$-graded structures}\label{sec:gradcat}

Let $\cC$ be a small category and $\Bbbk$ a unital commutative ring.
We write $\cC^{[n]} := \{ X_0 \xleftarrow{f_1} X_1 \xleftarrow{f_2} \cdots \xleftarrow{f_n} X_n \}$ for the set of paths of length $n$  in $\cC$.

\begin{defn}
A \emph{grading category} is pair $(\cC,\assoc)$ where $\assoc : \cC^{[3]} \rightarrow \Bbbk^\times$ is a 3-cocycle, meaning it respects
\[
d\assoc(g,h,k,l) := 
\assoc(h,k,l) \assoc(gh,k,l)^{-1} \assoc(g,hk,l) \assoc(g,h,kl)^{-1} \assoc(g,h,k) = 1,
\]
for all $g \circ h \circ k \circ l \in \cC^{[4]}$. 
We call $\assoc$ the \emph{associator}.
\end{defn}

We assume for the rest of the section that we have fixed a grading category $(\cC, \assoc)$, and a unital commutative ring $\Bbbk$.

\begin{defn}
A \emph{$\cC$-graded  $\Bbbk$-module} is a $\Bbbk$-module $M$ with a decomposition $M \cong \bigoplus_{g \in \cC} M_g$. 
Given $x \in M_g$, we write $|x| := g$. 
A \emph{graded map} $f : M \rightarrow N$ between two $\cC$-graded modules $M,N$ is a $\Bbbk$-linear map such that $f(M_g) \subset N_g$ for all $g \in \cC$. 
\end{defn}

We define the category $\Mod^\cC$ of $\cC$-graded $\Bbbk$-modules with objects being $\cC$-graded $\Bbbk$-modules and morphisms being graded maps.
It is a monoidal category with $M' \otimes M$ defined as
\begin{align*}
M' \otimes M &:=  \bigoplus_{g\in \cC} (M' \otimes M)_g,
&
(M' \otimes M)_g &:=  \bigoplus_{g = g_2 \circ g_1} M'_{g_2} \otimes_\Bbbk M_{g_1}.
\end{align*}
The coherence isomorphism $\alpha : (M'' \otimes M') \otimes M \xrightarrow{\simeq} M'' \otimes (M' \otimes M)$ is given by 
\[
(z \otimes y) \otimes x \mapsto \assoc(|z|,|y|,|x|) z \otimes (y \otimes x),
\] 
for homogeneous elements $x \in M''_{|x|}, y \in M'_{|y|}, z \in M_{|z|}$.
The unit object is 
\[
I_\cC := \bigoplus_{X \in Obj(\cC)} (\Bbbk)_{\id_X},
\]
 with left unitor $\lambda : I_\cC \otimes M \cong M$ given by
\begin{align*}
(k \otimes m) &\mapsto  \lambda(|k|,|m|) km, & \lambda(|k|,|m|) &:= \alpha(\id_Y,\id_Y, |m|)^{-1},
\intertext{
for $m \in M_g$,  $k \in (\Bbbk)_{\id_Y}$ and $g : X \rightarrow Y$. The right unitor $\rho$ is constructed similarly with 
}
(m \otimes k') &\mapsto \rho(|m|,|k'|) mk',  & \rho(|m|,|k'|) := \alpha(|m|,\id_X,\id_X), 
\end{align*}
for $k' \in  (\Bbbk)_{\id_X}$, so that
\begin{equation}\label{eq:unitorsrel}
\lambda(\id_Y,x)  = \rho(z,\id_Y) \assoc(z,\id_{Y},x) .
\end{equation}

\begin{rem}
Note that the definition of a grading category can be made more general by allowing arbitrary unitors respecting \cref{eq:unitorsrel}, and everything written below generalizes. 
\end{rem}

\subsection{$\cC$-graded algebras}

Having a monoidal category at thand, constructing classical algebraic structures becomes straightforward. 

\begin{defn}
A \emph{$\cC$-graded algebra} is a monoid object in $\Mod^\cC$. 
\end{defn}

Explicitly, it is $\cC$-graded $\Bbbk$-module $A \cong \bigoplus_{g \in \cC} A_g$ with a $\Bbbk$-linear multiplication map 
\[
\mu_A : A \otimes A \rightarrow A,
\]
usually written as $yx := \mu_A(y,x)$ for $x,y \in A$,
and a unit element $1_X \in A_{\id_X}$ for each $X \in Obj(\cC)$, 
respecting:
\begin{itemize}
\item $\mu_A$ is graded $\mu_A(A_{g'}, A_{g}) \subset A_{g' \circ g}$;
\item $\mu_A$ is associative $(zy)x = \assoc(|z|,|y|,|x|) z(yx)$;
\item $1_Y x = \lambda(\id_Y, |x|)  x$ and $x 1_X = \rho(|x|, \id_X)x$ for all $x \in A_{|x| : X \rightarrow Y}$.
\end{itemize}
In particular, $I_\cC$ is a $\cC$-graded algebra with multiplication induced by $\Bbbk$ and $xy = 0$ whenever $|x| \neq |y|$. 

\begin{defn}
Let $A_1, A_2$ be two $\cC$-graded algebras. An \emph{$A_2$-$A_1$-bimodule} is a bimodule object over the pair $(A_2,A_1)$ in $\Mod^\cC$. 
\end{defn}

Explicitly, it is a $\cC$-graded $\Bbbk$-module $M \cong \bigoplus_{g \in \cC} M_g$ with graded $\Bbbk$-linear left and right action maps 
\begin{align*}
\rho_L &: A_2 \otimes M \rightarrow M, &
\rho_R &: M \otimes A_1 \rightarrow M,
\end{align*}
written as $y \cdot m := \rho_L(y,m)$ and $m \cdot x := \rho_R(m,x)$ for $x \in A_1, y \in A_2$ and $m \in M$, 
 respecting:
\begin{itemize}
\item $(y'y) \cdot m = \assoc(|y'|,|y|,|m|) y' \cdot (y \cdot m)$;
\item $(m \cdot x') \cdot x = \assoc(|m|,|x'|,|x|) m \cdot (x'x)$;
\item $(y \cdot m) \cdot x = \assoc(|y|,|m|,|x|) y \cdot (m \cdot x)$;
\item $1_Y\cdot m = \lambda(\id_Y,|m|) m$ and $m \cdot 1_X = \rho(|m|,\id_X)m$ for all $m \in M_{|m| : X \rightarrow Y}$,
\end{itemize}
for all $y',y \in A_2$, $x',x \in A_1$ and $m \in M$. 
We obtain the notion of a left (resp. right) $A$-module as a $A$-$I_\cC$-bimodule (resp. $I_\cC$-$A$-bimodule).
A graded map of $A_2$-$A_1$-bimodules $f : M \rightarrow M'$ is a graded $\Bbbk$-linear map preserving the actions $f(y \cdot m) = y \cdot f(m)$ and $f(m \cdot x) = f(m) \cdot x$ for all $x \in A_1, m \in M$ and $y \in A_2$.
We write $\Bimod^\cC(A_2,A_1)$ for the category of $A_2$-$A_1$-bimodules with graded maps. 
For the particular case of $A$-$I_\cC$-modules (i.e. left $A$-modules) we write it as $\Mod^\cC(A)$. 

\smallskip

Take $M' \in \Bimod^\cC(A_3,A_2)$ and $M \in \Bimod^\cC(A_2, A_1)$. We first observe that $M' \otimes M$ has the structure of a $(A_3, A_1)$-bimodule with left and right actions defined by making the following diagrams commute:
\[
\begin{tikzcd}[column sep = -3ex]
A_3 \otimes (M' \otimes M) \ar{rr} \ar[swap]{rd}{\assoc^{-1}} && M' \otimes M \\
&(A_3 \otimes M') \otimes M \ar[swap]{ur}{\rho_L \otimes 1}&
\end{tikzcd}
\qquad
\begin{tikzcd}[column sep = -3ex]
(M' \otimes M) \otimes A_1 \ar{rr} \ar[swap]{rd}{\assoc} && M' \otimes M \\
& M' \otimes (M \otimes A_1) \ar[swap]{ur}{1 \otimes \rho_R}&
\end{tikzcd}
\]
that is
\begin{align*}
y \cdot (m' \otimes m) &:= \assoc(|y|,|m'|,|m|)^{-1} (y \cdot m') \otimes m, \\
(m' \otimes m) \cdot x &:= \assoc(|m'|,|m|,|y|) m' \otimes (m \cdot x), 
\end{align*}
for all homogeneous $x \in A_1$, $y \in A_3$, $m \in M$ and $m' \in M'$. 

\begin{defn}
The \emph{tensor product} of $M'$ with $M$ over $A_2$ is the coequalizer
\[
\begin{tikzcd}[row sep = 1ex]
(M' \otimes A_2) \otimes M \ar["\vsimeqop"',swap]{dd}{\assoc} \ar[pos=.35]{drr}{\rho_R \otimes 1} && & \\
&& M' \otimes M \ar{r} & M' \otimes_{A_2} M, \\
M ' \otimes (A_2 \otimes M)  \ar[swap,pos=.35]{urr}{1 \otimes \rho_L} & &  &
\end{tikzcd}
\]
 in $\Mod^\cC$.
\end{defn}

Explicitly, we have
\[
M' \otimes_{A_2} M := M' \otimes M / \bigl((m' \cdot x) \otimes m - \assoc(|m'|,|x|,|m|) m' \otimes (x \cdot m)\bigr).
\]
It is an $A_3$-$A_1$-bimodule with left and right actions induced by the ones on $M' \otimes M$. 

\smallskip

Given a map of $A_2$-$A_1$-bimodules $f : M \rightarrow N$ and another one of $A_3$-$A_2$-bimodules $f' :M' \rightarrow N'$ we obtain an induced map $(f' \otimes f) : (M' \otimes_{A_2} M) \rightarrow (N' \otimes_{A_2} N)$ given by $(f'\otimes f)(m' \otimes m) := f'(m') \otimes f(m)$. 

\smallskip

For $M \in \Bimod^\cC(A_2, A_1)$ we construct a functor 
\[
M \otimes_{A_1} - : \Mod^\cC(A_1) \rightarrow \Mod^{\cC}(A_2), \quad N \mapsto M \otimes_{A_1} N.
\]
Putting all these together, we obtain a bicategory $\BBimod^\cC$ with objects being $\cC$-graded algebras and $1$-hom spaces $\Hom(A_1,A_2) := \Bimod^\cC(A_2,A_1)$. The horizontal composition $\Hom(A_1,A_2) \times \Hom(A_2,A_3)$ is given by tensor product over $A_2$. 

\begin{exe}\label{ex:assocZ}
As an easy example we can consider the category $\cZ$ with one object $\star$ and $\Hom_\cZ(\star, \star) := \bZ$, with $\alpha=1$. Then, a $\cZ$-graded object is the same thing as a $\bZ$-graded one. 
\end{exe}

\subsection{$\cC$-grading shift}
A \emph{$\cC$-grading shift} $\shiftFunct{}$ is a collection of maps
\[
\shiftFunct{} = \{ \shiftFunct[YX]{} : \shiftDom[YX]{} \subset \Hom_\cC(X,Y) \rightarrow \Hom_\cC (X,Y)\}_{X,Y \in Obj(\cC)}.
\]
We write $\shiftFunct{}(g) := \shiftFunct[YX]{}(g)$ whenever $g \in \shiftDom[YX]{}$.

\begin{defn}
Let $\shiftCenter{S}(\cC)$ be the wide subcategory of $\cC$ given by
\begin{align*}
\Hom_{\shiftCenter{S}(\cC)}(X,Y) := 
\begin{cases}
\emptyset, &\text{whenever $X \neq Y$,} \\
Z\bigl(\End_\cC(X)\bigr), &\text{otherwise,}
\end{cases}
\end{align*}
where $Z$ denotes the center. 
\end{defn}

\begin{defn}
A \emph{$\cC$-shifting system} $S = (I,\shiftFunctCol
)$ is the datum of
\begin{itemize}
\item a monoid  $I$ with composition $\bullet : I \times I \rightarrow I$ and neutral element $\neutralElement \in I$;
\item  a collection of $\cC$-grading shifts $\shiftFunctCol = \{\shiftFunct{i}\}_{i \in I}$,
\end{itemize}
 respecting:
\begin{itemize}
\item $\shiftFunctCol$ contains the \emph{neutral shift} $\shiftFunct{\neutralElement}$ with $\shiftDom[YX]{\neutralElement} := \Hom_{\shiftCenter{S}(\cC)}(X,Y)$ and $\shiftFunct[YX]{\neutralElement} := \id_{\shiftDom[YX]{\neutralElement}}$; 
\item for each pair $\shiftFunct[ZY]{j}$ and $\shiftFunct[YX]{i}$ we have $\shiftDom[ZY]{j} \circ \shiftDom[YX]{i} \subset \shiftDom[ZX]{j \bullet i}$ and
\[
\begin{tikzcd}
\shiftDom[ZY]{j} \times \shiftDom[YX]{i}
\ar{r}{-\circ-}
\ar[swap]{d}{(\shiftFunct[ZY]{j}, \shiftFunct[YX]{i})}
&
\shiftDom[ZX]{j\bullet i}
\ar{d}{\shiftFunct[ZX]{j\bullet i}}
\\
\Hom_\cC(Y,Z) \times \Hom_\cC(X,Y)
\ar[swap]{r}{-\circ-}
&
\Hom_\cC(X,Z)
\end{tikzcd}
\]
commutes;
\item there is a subset $I_\id \subset I$ such that for each pair $X,Y \in Obj(\cC)$ there is a partition of $\Hom_\cC(X,Y) = \sqcup_{i \in I_\id} \shiftDom[XY]{i}$, and $\shiftFunct{i} = \id_{\shiftDom[YX]{i}}$ for all $i \in I_\id$;
\item if $I$ contains an absorbing element $0 \in I$, then $\shiftFunct{0}$ is a \emph{null shift} with $\shiftDom[XY]{0} = \emptyset$. 
\end{itemize}
\end{defn}

\begin{rem}\label{rem:neutralShift}
Note that the neutral shift $\shiftFunct{\neutralElement}$ preserves only $\shiftCenter{}(\cC)$. 
In general, we cannot hope to ask for it to preserve the whole hom-spaces. Indeed, suppose we have a non trivial shift $\shiftFunct{i} : \id_X \mapsto x$ where $x \neq \id_X \in \End(X)$, and consider an element $y \in \End(X)$ such that $x \circ y \neq y \circ x$. Then, we would have $\shiftFunct{\neutralElement \bullet  i}(y  \circ \id_X) = \shiftFunct{\neutralElement}(y) \circ \shiftFunct{i}(\id_X) = y \circ x$ and $\shiftFunct{i \bullet  \neutralElement}(\id_X \circ y) = \shiftFunct{i}(\id_X) \circ \shiftFunct{\neutralElement}(y) = x \circ y$, but we also want $\shiftFunct{\neutralElement \bullet  i}(y  \circ \id_X) = \shiftFunct{i}(y) = \shiftFunct{\neutralElement \bullet i}(\id_X \circ y)$, which is a contradiction. Since $y \notin Z(\End(X))$, we actually have $\shiftFunct{\neutralElement}(y) = 0$. 
\end{rem}

Let  $\beta = \{ \compMap[ZYX]{j}{i} \}$ be a collection of maps 
\[
\compMap[ZYX]{j}{i} : \shiftDom[ZY]{j} \times \shiftDom[YX]{i} \rightarrow \Bbbk^\times,
\]
for all $X,Y,Z \in Obj(\cC)$ and $i,j \in I$, such that $\compMap[ZYX]{\neutralElement}{\neutralElement} = 1$. 

\begin{defn}
We say that a $\cC$-shifting system $S$ 
 is \emph{compatible with $\assoc$ through $\beta$} if 
\begin{align}
\label{eq:shiftingsystcomp}
\begin{split}
\assoc(g'',g',g)&
\compMap[ZXW]{k\bullet j}{i}(g''g',g)
\compMap[ZYX]{k}{j}(g'',g') \\
&= 
\compMap[ZYW]{k}{j\bullet i}(g'', g'g)
\compMap[YXW]{j}{i}(g',g)
\assoc(\shiftFunct[ZY]{k}(g''), \shiftFunct[YX]{j}(g'), \shiftFunct[XW]{i}(g)),
\end{split}
\end{align}
for all $Z \xleftarrow{g''} Y \xleftarrow{g'} X \xleftarrow{g} W \in \shiftDom[ZY]{k} \circ \shiftDom[YX]{j} \circ \shiftDom[XW]{i}$ and $i,j,k \in I$. 
In that situation, we refer to $\compMap[ZYX]{j}{i}$ as \emph{compatibility maps}. 
Also, we write $\compMap{j}{i}(g',g) := \compMap[XYZ]{j}{i}(g',g)$ for all $Z \xleftarrow{g'} Y \xleftarrow{g} X \in \shiftDom[ZY]{j} \circ \shiftDom[YX]{i}$. 
\end{defn}

Let $S = (I, \{\shiftFunct{i}\}_{i \in I})$ be a $\cC$-shifting system compatible with $\assoc$. 
For each $i \in I$ we define the \emph{grading shift functor} $\shiftFunct{i} : \Mod^\cC \rightarrow \Mod^\cC$ as the identity on morphisms, and for each $\cC$-graded $\Bbbk$-module $M$ 
we put $\shiftFunct{i}(M) := \bigoplus_{g \in \shiftDom{i}} \shiftFunct{i}(M)_{\shiftFunct{i}(g)}$ where $\shiftFunct{i}(M)_{\shiftFunct{i} (g)} := M_g$. 
In other words, the grading shift functor sends elements in degree $g  \in D_i$ to elements in degree $\shiftFunct{i}(g)$, and elements in degree not in $\shiftDom{i}$ to zero. 
The \emph{identity shift functor} is given by $\shiftFunct{\id} := \bigoplus_{i \in I_\id} \shiftFunct{i}$. 
Note that $\shiftFunct{\id}(M) \cong M$. 

\begin{rem}
As already pointed out in \cref{rem:neutralShift}, in general we have $\shiftFunct{\id} \neq \shiftFunct{\neutralElement}$. 
\end{rem}
For $M,M' \in \Mod^\cC$ there is a canonical isomorphism
\[
\compMap{j}{i} : \shiftFunct{j}(M') \otimes \shiftFunct{i}(M) \xrightarrow{\simeq} \shiftFunct{j \bullet i}(M' \otimes M),
\quad m' \otimes m \mapsto \compMap{j}{i}(|m'|, |m|) m' \otimes m,
\]
for all homogeneous $m \in M, m' \in M'$. 
These isomorphisms are compatible with $\assoc$ in the sense that the diagram
\begin{equation}\label{eq:shiftingsystcompdiag}
\begin{tikzcd}[column sep = 15ex]
(\shiftFunct{k} M'' \otimes \shiftFunct{j} M') \otimes \shiftFunct{i}(M) \ar{r}{\assoc} \ar[swap]{d}{\compMap{k}{j} \otimes 1} 
 & \shiftFunct{k}(M'') \otimes (\shiftFunct{j} M' \otimes \shiftFunct{i} M) \ar{d}{1 \otimes \compMap{j}{i}}
  \\
\shiftFunct{k \bullet j}(M'' \otimes M') \otimes \shiftFunct{i}(M) \ar[swap]{d}{\compMap{k \bullet j}{i}}
& \shiftFunct{k}(M'') \otimes \shiftFunct{j \bullet i}(M' \otimes M) \ar{d}{\compMap{k}{j \bullet i}}
 \\
\shiftFunct{k \bullet j \bullet i}((M'' \otimes M') \otimes M) \ar[swap]{r}{\shiftFunct{k \bullet j \bullet i}\assoc} 
& \shiftFunct{k \bullet j  \bullet i} (M'' \otimes (M' \otimes M)),
\end{tikzcd}
\end{equation}
commutes 
for all $i,j,k \in I$ and $M,M',M'' \in \Mod^\cC$, thanks to~\cref{eq:shiftingsystcomp}.

\begin{rem}
For the sake of simplicity, we will often use $\widetilde I := I \sqcup \{\id\}$. Then, when we use a compatibility map $\compMap{\id}{i}$, we mean 
\[
\compMap{\id}{i}(g',g) := \compMap{j}{i}(g',g),
\]
for $g' \in \shiftDom{j}$ and $j \in I_\id$. In other words, $\compMap{\id}{i} = \sum_{j \in I_\id} \compMap{j}{i}$ where we interpret $\compMap{j}{i} (g',g) = 0$ whenever $g' \notin \shiftDom{j}$. 
 We will also write $\shiftFunct{\id \bullet i} := \bigoplus_{j \in I_\id} \shiftFunct{j \bullet i}$.
 The same applies for $\compMap{j}{\id}$. We also put $\compMap{\id}{\id} := 1$. 
\end{rem}

\begin{prop}
The compatibility map $\compMap{j}{i}$ defines a natural isomorphism
\[
\beta_{j,i} : \shiftFunct{j}(-) \otimes \shiftFunct{i}(-) \xrightarrow{\simeq} \shiftFunct{j \bullet i}(- \otimes -)
\]
of bifunctors for all $i,j \in \widetilde I$. 
\end{prop}

\begin{proof}
We need to check that
\[
\begin{tikzcd}
\shiftFunct{j}(M') \otimes \shiftFunct{i}(M) 
\ar{r}{\compMap{j}{i}}
\ar[swap]{d}{f' \otimes f}
&
\shiftFunct{j \bullet i}(M' \otimes M)
\ar{d}{f' \otimes f}
\\
\shiftFunct{j}(N') \otimes \shiftFunct{i}(N) 
\ar[swap]{r}{\compMap{j}{i}}
&
\shiftFunct{j \bullet i}(N' \otimes N)
\end{tikzcd}
\]
commutes for all $f : M \rightarrow N$ and $f' : M' \rightarrow N'$. This is immediate since the grading shift functors do not modify the morphisms, and $\compMap{j}{i}$ depends only on the degree, and $f$ and $f'$ preserve the degrees.
\end{proof}

\subsection{Shifting bimodules}\label{sec:shiftingbimod}

Consider a $\cC$-shifting system $(S = (I,\{\shiftFunct{i}\}_{i \in I}),\compMap{}{})$ compatible with $\assoc$. 
We assume that all $\cC$-graded algebras $A$ are supported on $\shiftCenter{S}(\cC)$ in the sense that $A_g = 0$ whenever $g \notin \Hom_{\shiftCenter{S}(\cC)}$. Thus, we have $\shiftFunct{\neutralElement}(A) \cong A$. 

\begin{defn}\label{defn:shiftedbim}
For $M \in  \Bimod^\cC(A_2, A_1)$ we define the \emph{shifted bimodule} $\shiftFunct{i}(M)$ as the $\cC$-graded $\Bbbk$-module $M$ shifted by $\shiftFunct{i}$ with left and right actions
given by making the diagrams
\[
\begin{tikzcd}
A_2  \otimes \shiftFunct{i}(M)  \ar[swap]{d}{\vsimeq} \ar{r}{\shiftFunct{i}\rho_L} & \shiftFunct{i}(M) 
\\
\shiftFunct{\neutralElement}(A_2)\otimes \shiftFunct{i}(M) \ar[swap]{d}{\compMap{\neutralElement}{i}} &
\\
\shiftFunct{\neutralElement \bullet i}(A_2 \otimes M) \ar{r}{\rho_L} & \shiftFunct{\neutralElement  \bullet  i} \ar[equal]{uu}(M)
\end{tikzcd}
\qquad
\begin{tikzcd}
\shiftFunct{i}(M) \otimes A_1 \ar[swap]{d}{\vsimeq} \ar{r}{\shiftFunct{i}\rho_R} & \shiftFunct{i}(M) 
\\
\shiftFunct{i}(M) \otimes \shiftFunct{\neutralElement }(A_1) \ar[swap]{d}{\compMap{i}{\neutralElement }} &
\\
\shiftFunct{i \bullet \neutralElement}(M \otimes A_1) \ar{r}{\rho_R} & \shiftFunct{i \bullet \neutralElement} \ar[equal]{uu}(M)
\end{tikzcd}
\]
commute. 
\end{defn}

Explicitly, it gives us
\begin{align*}
y \cdot \shiftFunct{i}(m) := \compMap{\neutralElement }{ i}(|y|,|m|) \shiftFunct{i}(y \cdot m), \\
\shiftFunct{i}(m) \cdot x := \compMap{i}{\neutralElement }(|m|,|x|) \shiftFunct{i}(m \cdot x),
\end{align*}
for all $x \in A_1,y \in A_2$ and $m \in M$. 

\begin{rem}
Note that we need to have $A_i$ supported in $\shiftCenter{S}(\cC)$-degree, otherwise we cannot identify $A_i$  with $\shiftFunct{\neutralElement }(A_i)$.
\end{rem}

\begin{prop}
The shifted bimodule $\shiftFunct{i}(M)$ is an $A_2$-$A_1$-bimodule. 
\end{prop}

\begin{proof}
It follows immediately from \cref{eq:shiftingsystcomp} together with the fact that $A_1$ and $A_2$ are supported  in $\shiftCenter{S}(\cC)$-degree, and $\compMap{\neutralElement}{\neutralElement} = 1$.
\end{proof}

Therefore, we obtain a shifting endofunctor 
\[
\shiftFunct{i} : \Bimod^\cC(A_2, A_1) \rightarrow \Bimod^\cC(A_2,A_1),
\]
 for each $i \in I$ induced by $\shiftFunct{i} : \Mod^\cC \rightarrow \Mod^\cC$.

\begin{prop}
Let $M' \in \Bimod^\cC(A_3,A_2)$ and $M \in \Bimod^\cC(A_2, A_1)$. Then, there is an isomorphism of $A_3$-$A_1$-bimodules 
\[
\compMap{j}{i} : \shiftFunct{j}(M') \otimes_{A_2} \shiftFunct{i}(M)  \xrightarrow{\simeq} \shiftFunct{j \bullet i} (M' \otimes_{A_2} M),
\]
induced by the canonical isomorphism $\compMap{j}{i} : \shiftFunct{j} (M') \otimes \shiftFunct{i} (M) \xrightarrow{\simeq} \shiftFunct{j \bullet i}(M' \otimes M)$. 
\end{prop}

\begin{proof}
We first show that $\compMap{j}{i}$ induces an isomorphism as $\Bbbk$-modules. Consider the following diagram:
\[
\begin{tikzcd}[column sep = -5ex]
\bigl( \shiftFunct{j}(M') \otimes A_2  \bigr) \otimes \shiftFunct{i}(M) \ar{drr}{\shiftFunct{j} \rho_R \otimes 1} \ar[swap]{ddd}{\compMap{j \bullet \neutralElement}{i}\circ(\compMap{j}{\neutralElement } \otimes 1)}  \ar{ddr}{\assoc } &&&&
\\
&& \shiftFunct{j}(M') \otimes \shiftFunct{i}(M) \ar{rr} \ar["\vsimeq"']{ddd}{\compMap{j}{i}} & \hspace{15ex} & \shiftFunct{j}(M') \otimes_{A_2} \shiftFunct{i}(M) \ar[dashed,swap]{ddd}{\vsimeq}
\\
&\shiftFunct{j}(M') \otimes \bigl( A_2 \otimes \shiftFunct{i}(M) \bigr) 
\ar["{1\otimes\shiftFunct{i}\rho_L}" description]{ur}
\ar["{\compMap{j}{\neutralElement \bullet i}\circ(1\otimes\compMap{\neutralElement }{i})}" description, near start]{ddd}
&&&
\\
 \shiftFunct{j \bullet \neutralElement \bullet i}\bigl( (M' \otimes A_2) \otimes M \bigr)  \ar["\rho_R \otimes 1" near end]{drr} \ar[swap]{ddr}{\assoc}&&&&
\\
&& \shiftFunct{j \bullet i}(M' \otimes M) \ar{rr} & \hspace{15ex}& \shiftFunct{j \bullet i}(M' \otimes_{A_2} M)
\\
&\shiftFunct{j \bullet \neutralElement \bullet i}\bigl(M' \otimes (A_2 \otimes M)\bigr) \ar{ur}{1 \otimes \rho_L} &&&
\end{tikzcd}
\]
The left square is equivalent to~\cref{eq:shiftingsystcompdiag} and thus commutes. The other two squares commute by definition of $\shiftFunct{i}\rho_L$ and $\shiftFunct{j}\rho_R$. Since $\compMap{j}{i}$ is an isomorphism, we obtain by universal property of the coequalizer an induced isomorphism. 

We show that the induced map $\compMap{j}{i}$ preserves the left action, the proof for the right action being similar. Consider the following diagram:
\[
\begin{tikzcd}[column sep = -5ex]
&\bigl( A_3 \otimes \shiftFunct{j}(M') \bigr) \otimes \shiftFunct{i}(M) \ar{rr}{\compMap{\neutralElement }{j}} & \hspace{12ex} & \shiftFunct{\neutralElement \bullet j}(A_3 \otimes M) \otimes \shiftFunct{i}(M) \ar{dr}{\rho_L \otimes 1} &
\\
A_3 \otimes \bigl( \shiftFunct{j}(M') \otimes \shiftFunct{i}(M) \bigr) \ar{ur}{\assoc^{-1}} \ar[swap]{d}{1 \otimes\compMap{j}{i}}  \ar{rrrr}{\rho_L} &&&& \shiftFunct{j}(M') \otimes \shiftFunct{i}(M) \ar{d}{\compMap{j}{i}}
\\
A_3 \otimes \shiftFunct{j \bullet i}(M' \otimes M) \ar{rrrr}{\rho_L} \ar[swap]{dr}{\compMap{\neutralElement }{j \bullet i}} &&&& \shiftFunct{j \bullet i}(M' \otimes M)
\\
&\shiftFunct{\neutralElement  \bullet  j \bullet i}\bigl(A_3 \otimes (M' \otimes M)\bigr) \ar[swap]{rr}{\assoc^{-1}} & \hspace{12ex}  & \shiftFunct{\neutralElement \bullet j \bullet i}\bigl((A_3\otimes M') \otimes M\bigr) \ar[swap]{ur}{\rho_L \otimes 1} &
\end{tikzcd}
\]
where the upper and lower parts commute by definition of the left action on a tensor product of bimodules. Furthermore, the outer part of the diagram commutes thanks to~\cref{eq:shiftingsystcompdiag}. Thus, we conclude the inner part commutes, and $\compMap{j}{i}$ preserves the left action. 
\end{proof}

\begin{exe}
As in \cref{ex:assocZ}, we can consider $\cZ$. In this case, we can take $I := \bZ$ (for the additive structure), $\shiftFunct{n} : \Hom_\cZ(\star,\star) \rightarrow  \Hom_\cZ(\star,\star), m \mapsto m+n$ for all $n \in \bZ$, and $\compMap{n}{m} := 1$ for all $m,n \in \bZ$. It matches trivially with the usual notion of $\bZ$-grading shift.  
\end{exe}

\begin{exe}\label{exe:supershift}
A more interesting example is given by looking at the category $\cZ_2$ with a single object $\star$ and $\Hom_{\cZ_2}(\star, \star) := \bZ/2\bZ$, and $\alpha = 1$. Then, we take $I := \{\neutralElement, \pi\} := \bZ/2\bZ$, with $\shiftFunct{\neutralElement} := \id$ and $\shiftFunct{\pi}(m) := m+1 \mod 2$. We also put $\compMap{\neutralElement}{\pi}(x,y) := (-1)^{x}$, $\compMap{\pi}{\neutralElement}(x,y) := 1$, and $\compMap{\pi}{\pi}(x,y) := (-1)^{x}$. Then, it corresponds with the classical notion of super-graded structures, with $\shiftFunct{\pi}$ being the parity shift. The sign appearing in the compatibility map $\compMap{\neutralElement}{\pi}$ is the same as in the usual isomorphism
\[
M' \otimes \Pi M \xrightarrow{\simeq} \Pi\bigl(M' \otimes M'),
\]
where $\Pi$ is the parity shift functor, see for example \cite{supermonoidal}.
\end{exe}

\begin{exe}\label{exe:lambdaRshift}
Another similar example is given by the category $\cZ^2$ with one object $\star$ and $\Hom_{\cZ^2}(\star,\star) := \bZ \times \bZ$, with $\alpha = 1$. We take $I := \bZ \times \bZ$, with $\shiftFunct{(n_2,n_1)}(x_2,x_1) := (x_2+n_2, x_1+n_1)$, and $\compMap{(n_2,n_1)}{(m_2,m_1)}((y_2,y_1),(x_2,x_1)) := \lambda_R((y_2,y_1), (m_2,m_1))$, where $\lambda_{R}$ is the bilinear map defined in \cref{eq:deflambdaR}. This constructs shifting functors on $\Mod_R$. 
\end{exe}

\begin{exe}
Both \cref{exe:supershift} and \cref{exe:lambdaRshift} can be generalized as follow. Consider a monoid $G$ and put $\cC_G$ to be the category with a single object $\star$ and $\Hom_{\cC_G}(\star,\star) := G$, and put $\alpha = 1$. Consider a map $\lambda_G : G \times G \rightarrow \Bbbk$ such that $\lambda_G(x,a) \lambda_G(xy,b) = \lambda_G(y,b)\lambda_G(x,ab)$ for all $x,y,a,b \in G$ (this happens whenever $\lambda_G$ is bilinear ; if we also ask that $\lambda_G(1,1) = 1$, then the two conditions are equivalent). Then, put $I := Z(G)$ (the center of $G$), $\shiftFunct{a}(x) := ax$, and $\compMap{b}{a}(y,x) := \lambda_G(y,a)$. 
\end{exe}

\begin{rem}
In \cref{sec:gradCat}, we will construct a more sophisticated grading category, with a non-trivial associator, and where $I$ is very different from the hom-spaces in the grading category. 
\end{rem}

\subsection{$\cC$-graded dg-modules}\label{sec:gradeddg}

There is a natural notion of $\cC$-graded dg-modules (differentially graded modules) that is almost identitical to the usual notion (see for example~\cite{keller}). Most results about dg-modules holds for $\cC$-graded dg-modules.

\begin{defn}
A \emph{$\cC$-graded dg-bimodule} $(M,d_M)$ over a pair of $\cC$-graded $\Bbbk$-algebras $A_2$ and $A_1$ is a $\bZ\times\cC$-graded $A_2$-$A_1$-bimodule  $M = \bigoplus_{n \in \bZ, g \in \cC} M_g^n$, where we call the $\bZ$-grading \emph{homological}, together with a differential $d_M$ such that: 
\begin{itemize}
\item $d_M(M_g^n) \subset M_g^{n+1}$;
\item $d_M(y \cdot m) =  y \cdot d_M(m)$;
\item $d_M(m \cdot x) = d_M(m) \cdot x$;
\item $d_M \circ d_M = 0$,
\end{itemize}
for all $y \in A_2, m\in M$ and $x \in A_1$. \\
A \emph{map of $\cC$-graded dg-bimodules} preserves both the $\cC$-grading and the homological grading, and commutes with the differentials. 
\end{defn}

We write $\Bimoddg^\cC(A_2, A_1)$ for the abelian category of $\cC$-graded $A_2$-$A_1$-dg-bimodules, and $\Moddg^\cC(A)$ for the abelian category of $\cC$-graded $A$-$I_\cC$-dg-bimodules. 
We also write $|m|_\cC := \deg_\cC(m) := g$ and $|m|_h :=\deg_h(m) := n$ for respectively the $\cC$-degree and homological degree of $m \in M_g^n$. 

\smallskip

The tensor product over $A_2$ of a dg-bimodule $(M',d_{M'}) \in \Bimoddg^\cC(A_3, A_2)$ with another one $(M,d_M) \in \Bimoddg^\cC(A_2, A_1)$ is defined as
\[
(M', d_{M'}) \otimes_{A_2} (M, d_M) := 
(M' \otimes_{A_2} M, d_{M'\otimes M}), 
\]
where $\deg_h(m' \otimes m) := \deg_h(m') + \deg_h(m)$, and 
\[
d_{M'\otimes M}(m' \otimes m) := d_{M'}(m') \otimes m + (-1)^{\deg_h(m')} m' \otimes d_M(m).
\]

Given an $A_2$-$A_1$-dg-bimodule $(M,d_M)$, its \emph{homology} is $H(M,d_M) := \ker(d_M)/\img(d_M)$, which is a $\bZ \times \cC$-graded bimodule. As usual, a map between two $\cC$-graded dg-bimodules $f : (M,d_M) \rightarrow (M', d_{M'})$ induces a map on homology $f^* : H(M, d_M) \rightarrow H(M', d_{M'})$, and  we say that $f$ is a \emph{quasi-isomorphism} whenever $f^*$ is an isomorphism. 
The \emph{derived category} $\cD^{\cC}(A)$ of a $\cC$-graded algebra $A$ is given by localizing the category  $\Moddg^\cC(A)$ along quasi-isomorphisms. We also define similarly the derived category of dg-bimodules $\cD^{\cC}(A_2,A_1)$. 

\smallskip

The \emph{homological shift functor} is the functor
\[
[1] : \Bimoddg^\cC(A_2, A_1) \rightarrow \Bimoddg^\cC(A_2, A_1),
\]
given by sending the dg-module $(M,d_M)$ to $(M[1], d_{M[1]})$ where
\begin{itemize}
\item the homological degree is shifted up by one, $(M[1])_g^h := M_g^{h-1}$;
\item $d_{M[1]} := -d_M$;
\item $M[1]$ inherits the left and right actions of $M$, that is $y \cdot (m)[1]  := (y \cdot m)[1]$ and $(m)[1] \cdot x := (m \cdot x)[1]$, 
\end{itemize}
and the identity on morphisms. 

\begin{rem}
There is a more general notion of $\cC$-graded dg-module over a $\cC$-graded dg-algebra, where the algebra itself carries a $\bZ$-grading and a differential. In that situation, the homological shifting functor would need to twist the left action as usual. 
\end{rem}

Let $f : (M,d_M) \rightarrow (M',d_{M'})$ be a map of $\cC$-graded dg-bimodules. The \emph{mapping cone of $f$} is 
\begin{align*}
\cone(f) &:= (M[-1] \oplus M', d_C), 
& d_C &:= 
\begin{pmatrix}
-d_M & 0 \\
f & d_{M'}
\end{pmatrix}
\end{align*}
It is a $\cC$-graded dg-bimodule that fits in a short exact sequence
\[
0 \rightarrow (M',d_{M'})   \xrightarrow{\imath_{M'}} \cone(f) \xrightarrow{\pi_M} (M,d_M)[-1] \rightarrow 0.
\]
The derived category  $\cD^{\cC}(A)$  is a triangulated category with translation functor given by the homological shift $[-1]$, and the distinguished triangles are isomorphic to triangles of the following form:
\[
(M,d_{M}) \xrightarrow{f} (M',d_{M'}) \xrightarrow{\imath_{m'}} \cone(f) \xrightarrow{\pi_M}  (M,d_M)[-1].
\]

\begin{defn}
We say that a $\cC$-graded dg-module over $A$ is \emph{relatively projective} if it is a direct summand of a direct sum of shifted copies (both in homological and in $\cC$-degree) of the free dg-module $(A,0)$. 
We say that a $\cC$-graded dg-module is \emph{cofibrant} if it is a direct summand in $\Moddg^\cC(A)$ of the inverse limit of a filtration
\[
0 = F_0 \subset F_1 \subset F_2 \subset \cdots \subset F_r \subset F_{r+1} \subset \cdots,
\]
where each $F_{r+1}/F_r$ is isomorphic to a relatively projective dg-module. 
\end{defn}

A nice property of cofibrant dg-modules is that taking a tensor product with such a dg-module preserves quasi-isomorphisms: given a quasi-isomorphism $f : M \xrightarrow{\sim} M'$ and a cofibrant dg-module $P$, then $f \otimes 1 : M \otimes_A P \xrightarrow{\sim} M' \otimes_A P$ is a quasi-isomorphism. By standard arguments in homological algebra (see~\cite{keller}), we obtain the following:

\begin{prop}
For any $\cC$-graded dg-module $M$, there exists a cofibrant dg-module $\br(M)$ with a surjective quasi-isomorphism
\[
\begin{tikzcd}
\br(M) \ar[twoheadrightarrow]{r}{\sim} & M.
\end{tikzcd}
\]
We call $\br(M)$ the \emph{bar resolution} of $M$, and the assignment $M \rightarrow \br(M)$ is functorial.
\end{prop}

This allows us to define for any $M \in \Bimoddg^\cC(A_2,A_1)$ the \emph{derived tensor product} functor
\[
M \Lotimes_A - : \cD^\cC(A_1) \rightarrow \cD^\cC(A_2), \quad M \Lotimes_{A_1} X := M \otimes_{A_1} \br(X).
\]
Whenever $M$ is cofibrant as right $A_1$-dg-module, then we have $M \otimes_{A_1} - \cong M \Lotimes_{A_1} -$.

\subsection{Homogeneous maps}\label{sec:homogeneousmap}

We use the same hypothesis as in \cref{sec:shiftingbimod}.

\begin{defn}
Let $M$ and $N$ be $\cC$-graded $A_2$-$A_1$-bimodules. We say that a (non-graded) map $f : M \rightarrow N$ is \emph{purely homogeneous} of degree $i \in \widetilde I$ if for all $m \in M$ we have
\begin{itemize}
\item $f(m) = 0$ whenever $|m| \notin \shiftDom{i}$;
\item $|f(m)| = \shiftFunct{i}(|m|)$;
\item $y \cdot f(m)  =  \compMap{\neutralElement}{i}(|y|,|m|) f(y \cdot m)$ for all $y \in A_2$;
\item $f(m) \cdot x  =  \compMap{i}{\neutralElement}(|m|,|x|) f(m \cdot x)$ for all $x \in A_1$.
\end{itemize}
A map $f : M \rightarrow N$ is \emph{homogeneous} if it is a finite sum $f =  \sum_{j \in J \subset \widetilde I} f_j$ of purely homogeneous maps $f_j : M \rightarrow N$ of degree $j$. 
\end{defn}

Note that if $f : M \rightarrow N$ is purely homogeneous of degree $i$, then the induced map $\bar f : \shiftFunct{i}(M) \rightarrow N$, where $\bar f(\shiftFunct{i} x) := f(x)$, is graded. 
Let $g : M' \rightarrow N'$ be a homogeneous map of $A_3$-$A_2$-bimodules and $f : M \rightarrow N$ be a homogeneous map of $A_2$-$A_1$-bimodules. We define the tensor map $g \otimes f := \sum_{j \in I} (f \otimes g)_j$ where
\[
(g \otimes f)_j (m'\otimes m) := \sum_{i' \bullet i) = j} \compMap{|g|}{|f|}(|m'|,|m|)^{-1} g_{i'}(m') \otimes f_i(m),
\]
for all homogeneous elements $m' \in M', m \in M$. 

\smallskip

We would like to define a (non-abelian) category $\BIMOD^\cC(A_2,A_1)$ of $\cC$-graded bimodules with homogeneous maps, hence having hom-spaces 
\[
\Hom_{\BIMOD^\cC(A_2,A_1)}(M,N) := \bigoplus_{i \in \widetilde I} \Hom_{\Bimod^\cC(A_2,A_1)}(\shiftFunct{i}(M), N).
\]
However, the composition of homogeneous maps is in general not homogeneous, so that $\BIMOD^\cC(A_2,A_1)$ is in general not a category. 


\subsection{$\cC$-shifting 2-system}\label{sec:2shiftingsystem}

In order to have that the composition of two purely homogeneous maps is purely homogeneous, we need to define a way to \emph{vertically} compose grading shifts. 

\begin{defn}
A \emph{$\cC$-shifting 2-system} $S = \{\cI, \shiftFunctCol\}$ is a $\cC$-shifting system $\{(\cI, \bullet, \neutralElement), \shiftFunctCol\}$ such that  $\cI$ is equipped with an associative \emph{vertical} composition map 
\[
 - \circ - : \cI \times \cI \rightarrow \cI,
 \]
  respecting
\begin{itemize}
\item $\neutralElement \circ \neutralElement = \neutralElement$;
\item $\shiftDom{j \circ i} =\shiftDom{i} \cap \shiftFunct{i}^{-1}(\shiftDom{j})$;
\item $\shiftFunct{j}|_{\shiftFunct{i}(\shiftDom{i}) \cap \shiftDom{j}} \circ \shiftFunct{i}|_{\shiftDom{j \circ i}} = \shiftFunct{j \circ i}$;
\item $\shiftFunct{(j' \circ i') \bullet (j \circ i)} = \shiftFunct{(j' \bullet j) \circ (i' \bullet i)}$ and they have the same domain, 
\end{itemize}
for all $j',j,i',i \in \cI$.
\end{defn}

As before, we need some compatibility maps. Thus, consider a collection of maps
\[
\vCompMap[YX]{j}{i} : \shiftDom[YX]{i} \rightarrow \Bbbk^\times,
\]
for all $j,i \in \cI$ and $Y,X \in \cC$, and where $\vCompMap[YX]{\neutralElement}{\neutralElement} = 1$. We also require that $\vCompMap[YX]{j}{i} = 1$ whenever $i$ or $j \in I_\id$. 
Consider as well a collection of invertible scalars
\[
\interCompMap[ZYX]{j'}{i'}{j}{i} \in \Bbbk^\times,
\]
respecting $\interCompMap[YX]{j'}{i'}{j}{i} = 1$ whenever $(j' \circ i') \bullet (j \circ i) = (j' \bullet j) \circ (i' \bullet i)$ (for example, whenever $j'=i'=\neutralElement$). Furthermore, if $j',i',j$ or $i$ is in $I_\id$, then we can exchange it with any other element of $I_\id$ and $\interCompMap[YX]{j'}{i'}{j}{i}$ stays the same. 
We will write $\interCompMap{j'}{i'}{j}{i}(g) := \interCompMap[YX]{j'}{i'}{j}{i}$ whenever $|g|_\cC \in \Hom_\cC(X,Y)$. 

\begin{defn}\label{def:shifting2systemcomp}
We say that a $\cC$-shifting 2-system $S$ is \emph{compatible with $\assoc$ through $(\beta, \gamma, \Xi)$} if the underlying $\cC$-shifting system is compatible with $\assoc$ through $\beta$ and 
\begin{equation}\label{eq:shiftingsystvcomp}
\compMap[ZYX]{j' \circ i'}{j \circ i}(g',g) 
\vCompMap[YX]{j'}{i'}(g') 
\vCompMap[ZX]{j}{i}(g) 
\interCompMap[ZYX]{j'}{i'}{j}{i}
=
\vCompMap[ZX]{j' \bullet j}{i' \bullet i}(g'g) 
\compMap[ZYX]{i'}{i}(g',g) 
\compMap[ZYX]{j'}{j'}(\shiftFunct{i'}(g'), \shiftFunct{i}(g)),
\end{equation}
for all $g' \in \shiftDom[ZY]{i} , g \in \shiftDom[YX]{i}$, 
and
\begin{equation}\label{eq:shiftingsystvcomp2}
\vCompMap[YX]{k \circ j}{i}(g)  \vCompMap[YX]{k}{j}(\shiftFunct{i}(g)) = \vCompMap[YX]{k}{j \circ i}(g) \vCompMap[YX]{j}{i}(g),
\end{equation}
for all $g \in  \shiftDom[YX]{i}$. 
\end{defn}

This allows us to construct natural isomorphisms
\begin{align*}
\shiftFunct{j} \circ \shiftFunct{i} &\xrightarrow{\simeq} \shiftFunct{j \circ i}, \\
\shiftFunct{(j' \circ i') \bullet (j \circ i) } &\xrightarrow{\simeq} \shiftFunct{(j' \bullet j) \circ (i' \bullet i)},
\end{align*}
in $\Mod^\cC$ given by 
\begin{align*}
\shiftFunct{j} \circ \shiftFunct{i}(M) &\rightarrow \shiftFunct{j \circ i}(M), & m &\mapsto \vCompMap{j}{i}(|m|) m, \\
\shiftFunct{(j' \circ i') \bullet (j \circ i) }(M) &\rightarrow\shiftFunct{(j' \bullet j) \circ (i' \bullet i)}(M), & m &\mapsto \interCompMap{j'}{i'}{j}{i}(|m|) m,
\end{align*}
for all homogenous element $m \in M$. 
Then, \cref{eq:shiftingsystvcomp} means the following diagram commutes:
\begin{equation}\label{eq:shiftingsystvcompdiag}
\begin{tikzcd}
\shiftFunct{j'} \circ \shiftFunct{i'}(M') \otimes \shiftFunct{j}\circ \shiftFunct{i}(M) 
\ar{r}{\compMap{j'}{j}}
\ar[swap]{d}{\vCompMap{j'}{i'} \otimes \vCompMap{j}{i}}
&
\shiftFunct{j' \bullet j}\bigl(\shiftFunct{i'} (M') \otimes \shiftFunct{i}(M) \bigr)
\ar{r}{\compMap{i'}{i}}
&
\shiftFunct{j' \bullet j} \circ \shiftFunct{i' \bullet i}(M' \otimes M)
\ar{d}{\vCompMap{j' \bullet j}{i' \bullet i}}
\\
\shiftFunct{j' \circ i'}(M') \otimes \shiftFunct{j \circ i}(M)
\ar[swap]{r}{\compMap{j' \circ i'}{j \circ i}}
&
\shiftFunct{(j' \circ i') \bullet (j \circ i) }(M' \otimes M)
\ar[swap]{r}{\interCompMap{j'}{i'}{j}{i}}
&
\shiftFunct{(j' \bullet j) \circ (i' \bullet i)}(M' \otimes M)
\end{tikzcd}
\end{equation}
for all $M',M \in \Mod^\cC$, and \cref{eq:shiftingsystvcomp2} the following one:
\begin{equation}\label{eq:shiftingsystvcompdiag2}
\begin{tikzcd}
\shiftFunct{k} \circ \shiftFunct{j} \circ \shiftFunct{i} (M)
\ar{r}{\vCompMap{j}{i}}
\ar[swap]{d}{\vCompMap{k}{j}}
&
\shiftFunct{k} \circ \shiftFunct{j \circ i}(M)
\ar{d}{\vCompMap{k}{j \circ i}}
\\
\shiftFunct{k \circ j} \circ \shiftFunct{i}(M)
\ar[swap]{r}{\vCompMap{k\circ j}{i}}
&
\shiftFunct{k \circ j \circ  i}(M)
\end{tikzcd}
\end{equation}
for all $M \in \Mod^\cC$. 

\smallskip

As before, we extend it to $\tilde \cI := \cI \sqcup \{\id\}$ where putting $\id$ in $\interCompMap[YX]{-}{-}{-}{-}$ means we can replace it by any element $j \in I_\id$. 
Since $\shiftFunct{\id} \circ \shiftFunct{i} \cong \shiftFunct{i} \cong \shiftFunct{i} \circ \shiftFunct{\id}$ by identity maps, we put $\id \circ i = i = i \circ \id$. We also extend $ \shiftFunct{(j' \bullet j) \circ (i' \bullet i)}$ for $j',j,i',i \in \tilde \cI$ by replacing any occurrence of $\id$ with a direct sum over elements in $I_\id$. 

\begin{prop}
The natural isomorphisms 
$\shiftFunct{j} \circ \shiftFunct{i} \xrightarrow{\simeq} \shiftFunct{j \circ i}$
 and 
$\shiftFunct{(j' \circ i') \bullet (j \circ i) } \xrightarrow{\simeq} \shiftFunct{(j' \bullet j) \circ (i' \bullet i)}$
 both induce natural isomorphisms between endofunctors of $\Bimod^\cC(A_2,A_1)$.
\end{prop}

\begin{proof}
We only need to show the natural isomorphisms respect the bimodule structure. It immediately follows from \cref{eq:shiftingsystvcompdiag}, and  using the fact that $\interCompMap{\neutralElement}{\neutralElement}{j}{i} = \interCompMap{j'}{i'}{\neutralElement}{\neutralElement} = 1$. 
\end{proof}

Furthermore, given such a system and two purely homogeneous maps $g : M' \rightarrow M''$ and $f : M \rightarrow M'$ in $\BIMOD^\cC(A_2,A_1)$ of degree $j$ and $i$ respectively, we define their \emph{$\cC$-graded composition} as 
\[
(g \circ_\cC f)(m) := \vCompMap{j}{i}(|m|)^{-1} (g \circ f)(m).
\]

\begin{prop}
Equipped with the $\cC$-graded composition, 
$\BIMOD^\cC(A_2,A_1)$ is an additive category. 
\end{prop}

\begin{proof}
First, we obtain that if $g$ has degree $j$ and $f$ has degree $i$, then $g \circ_\gradC f$ has degree $j \circ i$ in $\BIMOD^\cC(A_2,A_1)$. 
This follows from \cref{eq:shiftingsystvcomp} using the fact that $\interCompMap{\neutralElement}{\neutralElement}{j}{i} = \interCompMap{j'}{i'}{\neutralElement}{\neutralElement} = 1$. 
Then, we need to show that the $\cC$-graded composition is associative, which follows from \cref{eq:shiftingsystvcomp2}.  
\end{proof}

\begin{prop}\label{prop:tensorcomphomomaps}
Consider two pairs of purely homogeneous maps $f_1 : M_1 \rightarrow M_2, f_2 : M_2 \rightarrow M_3 \in \BIMOD^\cC(A_3,A_2)$ and $g_1 : N_1 \rightarrow N_2, g_2 : N_2 \rightarrow N_3 \in  \BIMOD^\cC(A_2,A_1)$. We have
\[
\bigl((g_2 \circ_\cC g_1) \otimes (f_2 \circ_\cC f_1)\bigr) (m \otimes_{A_2} n) = \interCompMap{|g_2|}{|g_1|}{|f_2|}{|f_1|}(|m| \bullet |n|)  \bigl((g_2 \otimes f_2) \circ_\cC (g_1 \otimes f_1) \bigr) (m \otimes_{A_2} n),
\]
for all $m \in M_1, n \in N_1$.
\end{prop}

\begin{proof}
It follows from \cref{eq:shiftingsystvcompdiag}.
\end{proof}

\begin{exe}
Consider again superstructures as in \cref{exe:supershift}. We put $j \circ i = j + i \mod 2$, and we have $\vCompMap[\star\star]{j}{i} := 1$ and $\interCompMap[\star\star]{j'}{i'}{j}{i} := (-1)^{i'j}$. Then, it coincides with the usual notion of composition of tensor product of supermaps. 
\end{exe}

\begin{exe}
Recall \cref{exe:lambdaRshift}. We put $\vCompMap[\star\star]{(n_1,n_2)}{(m_1,m_2)} := 1$ and $\interCompMap[\star\star]{(n'_1,n'_2)}{(m'_1,m'_2)}{(n_1,n_2)}{(m_1,m_2)} := \lambda_R((m_1',m_2'),(n_1,n_2))$. 
\end{exe}

\subsection{Graded commutativity}\label{sec:grshiftBIMOD}

Suppose $(\cC,\assoc)$ is equipped with a compatible $\cC$-shifting 2-system $\{\cI, \shiftFunctCol\}$. 

\begin{defn}
A \emph{commutativity system} on $\{\cI, \shiftFunctCol\}$ is a collection $\Tau := \{((j',i'), (j,i)) \in \cI^2 \times \cI^2\}$, such that 
\begin{itemize}
\item if $((j',i'), (j,i)) \in \Tau$ then $\shiftFunct{j'\circ i'} = \shiftFunct{j \circ i}$ and $((j,i),(j',i')) \in \Tau$;
\item if $((j_2',i'_2), (j_2,i_2)) \in \Tau, ((j_1',i'_1), (j_1,i_1)) \in \Tau$ then $((j_2' \bullet j_1',i'_2 \bullet i_1'), (j_2 \bullet j_1,i_2 \bullet i_1)) \in \Tau$.
\end{itemize}
\end{defn}

Consider a collection of scalars
\[
\commutMap[YX]{j'}{i'}{j}{i} \in\Bbbk^\times,
\]
for all $Y,X \in Obj(\cC)$ and $((j',i'), (j,i)) \in \Tau$, such that 
 $\commutMap[YX]{j'}{i'}{j}{i}= 1$ whenever $j' \circ i' = j \circ i$, and $(\commutMap[YX]{j'}{i'}{j}{i})^{-1} = \commutMap[YX]{j}{i}{j'}{i'}$ for all $((j',i'), (j,i)) \in \Tau$. 
If $((j',i'),(j,i)) \notin \Tau$, then we put $\commutMap{j'}{i'}{j}{i} := 0$. We write $\commutMap{j'}{i'}{j}{i}(g) := \commutMap[YX]{j'}{i'}{j}{i}$ whenever $g \in \Hom_\cC(X,Y)$. 
 
 \begin{defn}
 We say that the commutativity system $\Tau$ is \emph{compatible with $\{\assoc, \beta, \Xi\}$ through $\tau$}  if it respects the following compatibility condition:
\begin{equation}\label{eq:compCommut}
\compMap{j_2\circ i_2}{j_1\circ i_1}(g',g)
\commutMap{j'_2}{i'_2}{j_2}{i_2}(g')
\commutMap{j'_1}{i'_1}{j_1}{i_1}(g)
\interCompMap{j_2}{i_2}{j_1}{i_1}(g' \bullet g)
= 
\compMap{j_2'\circ i_2'}{j_1'\circ i_1'}(g',g)
\commutMap{j'_2 \bullet j_1'}{i'_2 \bullet i'_1}{j_2 \bullet j_1}{i_2 \bullet i_1}(g' \bullet g)
\interCompMap{j_2'}{i_2'}{j_1'}{i_1'}(g' \bullet g),
\end{equation}
for all  $((j_2',i'_2), (j_2,i_2)) , ((j_1',i'_1), (j_1,i_1)) \in \Tau$ and $g' \in \shiftDom{j_2 \circ i_2},g \in \shiftDom{j_1 \circ i_1}$. We also require that
\begin{equation}\label{eq:compCommut2}
	\commutMap[YX]{\id \bullet i}{j \bullet \id}{j \bullet \id}{\id \bullet i}
	= 
	(\interCompMap[YX]{\id}{j}{i}{\id})^{-1}
	\interCompMap[YX]{j}{\id}{\id}{i}.
\end{equation}
\end{defn}

For each $((j',i'), (j,i)) \in \Tau$ we obtain a natural isomorphism of functors
\[
 \shiftFunct{j' \circ i'} \xrightarrow{\simeq} \shiftFunct{j \circ i},
\]
in $\Mod^\cC$ given by 
\[
 \shiftFunct{j' \circ i'}(M) \rightarrow \shiftFunct{j \circ i}(M), \quad m \mapsto \commutMap{j'}{i'}{j}{i}(|m|) m,
\]
for all homogenous element $m \in M$.
Then, the compatibility condition ensures the following diagram commutes
\begin{equation}\label{eq:diagcompcommut}
\begin{tikzcd}[column sep = 6ex, row sep=6ex]
\shiftFunct{j_2' \circ i_2'}(M_2) \otimes \shiftFunct{j_1' \circ i_1'}(M_1)
\ar{r}{\compMap{j_2' \circ i_2'}{j_1' \circ i_1'}}
\ar[swap]{d}{\commutMap{j'_2}{i'_2}{j_2}{i_2} \otimes \commutMap{j'_1}{i'_1}{j_1}{i_1}}
&
\shiftFunct{(j_2' \circ i_2')\bullet(j_1' \circ i_1')}(M_2 \otimes M_1)
\ar{r}{\interCompMap{j_2'}{i_2'}{j_1'}{i_1'}}
&
\shiftFunct{(j_2' \bullet j_1') \circ (i_2' \bullet i_1')}(M_2 \otimes M_1)
\ar{d}{\commutMap{j'_2 \bullet j_1'}{i'_2 \bullet i'_1}{j_2 \bullet j_1}{i_2 \bullet i_1}}
\\
\shiftFunct{j_2 \circ i_2}(M_2) \otimes \shiftFunct{j_1 \circ i_1}(M_1)
\ar[swap]{r}{\compMap{j_2\circ i_2}{j_1\circ i_1}}
&
\shiftFunct{(j_2 \circ i_2) \bullet (j_1 \circ i_1)}(M_2 \otimes M_1)
\ar[swap]{r}{\interCompMap{j_2}{i_2}{j_1}{i_1}}
&
\shiftFunct{(j_2 \bullet j_1) \circ (i_2 \circ i_1)}(M_2 \otimes M_1)
\end{tikzcd}
\end{equation}
as well as the following one:
\begin{equation*}
\begin{tikzcd}
	\shiftFunct{(\id \bullet i) \circ (j \bullet \id)} (M)
	\ar[swap]{dd}{\commutMap{\id \bullet i}{j \bullet \id}{j \bullet \id}{\id \bullet i}}
	\ar{r}{\left(\interCompMap{\id}{j}{i}{\id}\right)^{-1}}
	&
	\shiftFunct{(\id \circ j) \bullet (i \circ \id)}(M)
	\ar[equals]{d}
	\\
	&
	\shiftFunct{j \bullet i}(M)
	\ar[equals]{d}
	\\
	\shiftFunct{(j \bullet \id) \circ (\id \bullet i)} (M)
	\ar[swap]{r}{\interCompMap{j}{\id}{\id}{i}}
	&
	\shiftFunct{(j \circ \id) \bullet (\id \circ i)}(M)
\end{tikzcd}
\end{equation*}

\begin{defn}
Given a $\cC$-shifting 2-system with a compatible commutativity system $\Tau$, we say that a diagram of purely homogeneous maps
\[
\begin{tikzcd}[ampersand replacement=\&]
M_{11} 
\ar{r}{f_{1*}} 
\ar[swap]{d}{f_{*1}}
\&
M_{12}
\ar{d}{f_{*2}}
\\
M_{21}
\ar[Rightarrow,shorten <= 2ex, shorten >= 2ex]{ur}{\tau}
\ar[swap]{r}{f_{2*}}
\&
M_{22}
\end{tikzcd}
\]
is \emph{$\cC$-graded commutative} if $((|f_{*2}|,|f_{1*}|),(|f_{2*}|,|f_{*1}|)) \in \Tau$ and 
\[
\commutMap{|f_{*2}|}{|f_{1*}|}{|f_{2*}|}{|f_{*1}|}.(f_{2*}\circ f_{*1}) = (f_{*2} \circ f_{1*}).
\]
We extend the definition linearly to homogeneous maps. 
\end{defn}

\begin{prop}\label{prop:fotimesgcommutes}
Given two $\cC$-graded commutative diagrams
\begin{align*}
\begin{tikzcd}[ampersand replacement=\&]
N_{11} 
\ar{r}{g_{1*}} 
\ar[swap]{d}{g_{*1}}
\&
N_{12}
\ar{d}{g_{*2}}
\\
N_{21}
\ar[swap]{r}{g_{2*}}
\ar[Rightarrow,shorten <= 2ex, shorten >= 2ex]{ur}
\&
N_{22}
\end{tikzcd}
&&
\text{and}
&&
\begin{tikzcd}[ampersand replacement=\&]
M_{11} 
\ar{r}{f_{1*}} 
\ar[swap]{d}{f_{*1}}
\&
M_{12}
\ar{d}{f_{*2}}
\\
M_{21}
\ar[swap]{r}{f_{2*}}
\ar[Rightarrow,shorten <= 2ex, shorten >= 2ex]{ur}
\&
M_{22}
\end{tikzcd}
\end{align*}
then the diagram 
\begin{equation*}
\begin{tikzcd}[ampersand replacement=\&]
N_{11} \otimes_A M_{11}
\ar{r}{g_{1*} \otimes f_{1*}} 
\ar[swap]{d}{g_{*1} \otimes f_{*1}}
\&
N_{12} \otimes_A M_{12}
\ar{d}{g_{*2} \otimes f_{*2}}
\\
N_{21} \otimes_A M_{21}
\ar[swap]{r}{g_{2*} \otimes f_{2*}}
\ar[Rightarrow,shorten <= 3ex, shorten >= 3ex]{ur}
\&
N_{22} \otimes_A M_{22}
\end{tikzcd}
\end{equation*}
is $\cC$-graded commutative. 
\end{prop}

\begin{proof}
It is a consequence of \cref{prop:tensorcomphomomaps} and \cref{eq:compCommut}.
\end{proof}

\begin{prop}
The natural isomorphism $ \commutMap{j'}{i'}{j}{i} : \shiftFunct{j' \circ i'} \xrightarrow{\simeq} \shiftFunct{j \circ i}$ induces a natural isomorphism as endofunctors over $\Bimod^\cC(A_2,A_1)$.
\end{prop}

\begin{proof}
It follows from the commutativity of \cref{eq:diagcompcommut} and the fact that $\commutMap{\neutralElement}{\neutralElement}{\neutralElement}{\neutralElement} = 1$, and $\interCompMap{\neutralElement}{\neutralElement}{j'}{i'} = \interCompMap{\neutralElement}{\neutralElement}{j}{i} = 1$ and $\interCompMap{j'}{i'}{\neutralElement}{\neutralElement} = \interCompMap{j}{i}{\neutralElement}{\neutralElement} = 1$. 
\end{proof}

We also define the grading shift functor 
\[
\shiftFunct{i} : \BIMOD^\cC(A_2,A_1) \rightarrow \BIMOD^\cC(A_2,A_1),
\]
as the usual grading shift by $i$ on objects, and acting on maps as
\[
\shiftFunct{i}(f)(m) := \commutMap{i}{|f|}{|f|}{i}(|m|) f(m),
\]
for $f : M \rightarrow N$ purely homogeneous and $m \in M$. Note that $\shiftFunct{i}(f) = 0$ whenever $((i,|f|),(|f|,i)) \notin \Tau$.

\begin{exe}
Still for superstructures as in \cref{exe:supershift}, we put $\Tau := \{((j',i'),(j,i)) | j'+i' \equiv j+i \mod 2\}$ with $\commutMap[\star,\star]{\pi}{\pi}{\pi}{\pi} := -1$ and $\commutMap[\star,\star]{j'}{i'}{j}{i} := 1$ otherwise. 
\end{exe}

\begin{exe}
Recall \cref{exe:lambdaRshift}. We take $\Tau := \{\bigl( ((n_1,n_2), (m_1,m_2)), ((m_1,m_2),(n_1,n_2)) \bigr))$. We put $\commutMap[\star,\star]{(n_1,n_2)}{(m_1,m_2)}{(m_1,m_2)}{(n_1,n_2)} :=\lambda_R((n_1,n_2),(m_1,m_2))$.
\end{exe}

\subsection{Dg-$\cC$-graded modules}

Suppose $(\cC,\assoc)$ is equipped with a compatible $\cC$-shifting 2-system $\{\cI, \shiftFunctCol\}$ with a commutativity system $\Tau$.

In \cref{sec:gradeddg}, we have introduced the notion of $\cC$-graded dg-(bi)module, which comes with a differential that preserves the $\cC$-grading. 
We now introduce a notion of dg-(bi)module with a differential that is $\cC$-homogeneous.  

\begin{defn}
A \emph{dg-$\cC$-graded bimodule} $(M,d_M)$ over a pair of $\cC$-graded $\Bbbk$-algebras $A_2$ and $A_1$ is a $\bZ\times\cC$-graded $A_2$-$A_1$-bimodule  $M = \bigoplus_{n \in \bZ, g \in \cC} M_g^n$,  with a differential $d_M = \sum_{j \in J \subset I} d_{M,j}$, where $J$ is finite, such that:
\begin{itemize}
\item $d_{M,j}(M_g^n) \subset M_{\shiftFunct{j}(g)}^{n+1}$ if $g \in D_j$ and $d_{M,j}(M_g^n)  = 0$ otherwise;
\item $d_{M,j}(y \cdot m) = \compMap{\neutralElement}{j}(|y|_\cC, |m|_\cC)^{-1} y \cdot d_{M,j}(m)$;
\item $d_{M,j}(m \cdot x) = \compMap{j}{\neutralElement}(|m|_\cC, |x|_\cC)^{-1} d_{M,j}(m) \cdot x$;
\item $d_M \circ d_M = 0$,
\end{itemize}
for all $y \in A_2, m\in M$ and $x \in A_1$. \\
A map of dg-$\cC$-graded bimodules is an homogeneous map of $\cC$-graded bimodules that preserves the homological grading and $\cC$-graded commutes with the differentials. 
\end{defn}

Let $\BIMdg{\cC}{A_2}{A_1}$ be the additive category of dg-$\cC$-graded $A_2$-$A_1$-bimodule. 
The homology is defined as before, but it is not $\cC$-graded anymore, and therefore not a bimodule. 
However, it is still a $\bZ$-graded space, and a map of dg-$\cC$-graded bimodules induces a $\bZ$-graded map in homology. 
 The homological shift functor is also constructed as before. 
Let $f : M \rightarrow M'$ be a map of dg-$\cC$-graded bimodules.  The \emph{mapping cone of $f$} is 
\begin{align*}
\cone(f) &:= (M[-1] \oplus M', d_C), 
& d_C &:= 
\begin{pmatrix}
- \commutMap{|d_{M'}|}{|f|}{|f|}{|d_M|} d_M & 0 \\
f & d_{M'}
\end{pmatrix}
\end{align*}
Note that $d_C$ is an homogeneous map, and it yields a complex since $f$ $\cC$-graded commutes with $d_M$ and $d_{M'}$.  The tensor product of $M' \in \BIMdg{\cC}{A_3}{A_2}$ with $M \in \BIMdg{\cC}{A_2}{A_1}$ is defined as before, except that
\begin{align*}
d_{M' \otimes M,j}&(m' \otimes m) := \\
 &(d_{M',j} \otimes 1)(m' \otimes m) + (-\commutMap{\id \bullet i}{i' \bullet \id}{i' \bullet \id}{\id \bullet i}(|m'|\bullet|m|))^{\deg_h(m')} (1 \otimes d_{M,j})(m' \otimes m).
\end{align*}
Note that since $\commutMap{\id \bullet i}{i' \bullet \id}{i' \bullet \id}{\id \bullet i} = 
	(\interCompMap{\id}{j}{i}{\id})^{-1}
	\interCompMap{j}{\id}{\id}{i}$
 the definition is independent from a choice of commutative system. Also, it gives a chain complex thanks to \cref{prop:tensorcomphomomaps}.

\begin{rem}
Note that $\BIMdg{\cC}{A_2}{A_1}$ is not a monoidal category in the usual sense. Indeed, the tensor product is not a bifunctor because of \cref{prop:tensorcomphomomaps}. We can think of it as a $\cC$-graded bifunctor. 
\end{rem}

We say that two maps $f,g : M \rightarrow M'$ are \emph{homotopic} $f \sim g$ if there is a map $h : M[-1] \rightarrow N$ in $\BIMOD^\cC_\bZ(A_2,A_1)$ (thus preserving the homological grading, homogeneous w.r.t. the $\cC$-grading and not necessarily commuting with the differentials) such that
\[
f -g = d_{M'} \circ h + \commutMap{|d_{M'}|}{|h|}{|h|}{|d_M|}   h \circ d_{M}.
\]
Note that obviously homotopic maps induce equivalent maps in homology. 
Also recall that a homotopy equivalence is pair of maps $f : M \rightarrow M'$ and $g : M' \rightarrow M$ such that $g \circ f \sim \id_M$ and $f \circ g \sim \id_{M'}$, thus inducing an isomorphism in homology. 

\begin{defn}
Let $\cKOM^{\cC}(A_2,A_1)$ be the additive category with the same objects as $\BIMdg{\cC}{A_2}{A_1}$ and hom-spaces given by
\[
\Hom_{\cKOM^{\cC}(A_2,A_1)}(M,N) := \Hom_{\BIMdg{\cC}{A_2}{A_1}}(M,N)/ \sim. 
\]
We refer to it as the \emph{$\cC$-graded homotopy category} of $A_2$-$A_1$-bimodules.
\end{defn}

%
%
%
%




\section{The grading category $\gradC$}\label{sec:gradCat}

Consider the subset $\redTangleSpace_n^m \subset \tangleSpace_n^m$ of reduced flat tangles (i.e. without free loop). Let 
\[
\red{\phantom{t}}: \tangleSpace_n^m \rightarrow \redTangleSpace_n^m
\]
 be the map that remove free loops. 
Define also the map  $s_{cba}(-,-) : \tangleSpace_{|b|}^{|c|} \times \tangleSpace_{|a|}^{|b|} \rightarrow \bZ^2$ for $a,b,c \in B^\bullet$  as given for $t' \in \tangleSpace_{|b|}^{|c|}$ and $t \in \tangleSpace_{|a|}^{|b|}$ by 
\[
s_{cba}(t',t) := \deg(\chcob_{cba}(t',t)),
\]
where $\chcob_{cba}(t',t)$ is the canonical cobordism defined in \cref{ssec:surgery}. 

\begin{lem}\label{lem:splitcounting}
For any triple of flat tangles $t'' \in  B_{|c|}^{|d|},t' \in B_{|b|}^{|c|}$ and $t \in B_{|a|}^{|b|}$, we have $s_{dca}(t'',t't)+s_{cba}(t',t) = s_{dcb}(t'',t') + s_{dba}(t''t',t)$. 
\end{lem}

\begin{proof}
It follows by minimality condition on the Euler characteristic of the canonical cobordisms $\chcob_{cba}(t',t)$. 
\end{proof}

\begin{defn}
We define the category $\gradC$ as: 
\begin{itemize}
\item objects are given by elements in $B^\bullet$;
\item hom-spaces are
\[
\Hom_{\gradC}(a,b) := \redTangleSpace_n^m \times \bZ^2,
\qquad
\id_a = (1_n, (n,0)),
\]
 for $a \in B^n$ and $b \in B^m$; 
\item composition is
\[
(t',p') \circ (t,p) := (\red{t't}, p+p'+s_{cba}(t',t) ),
\] 
for $(t,p) \in \Hom_{\gradC}(a,b)$ and $(t',p') \in \Hom_{\gradC}(b,c)$.
\end{itemize} 
\end{defn}

Note that the composition in $\gradC$ is associative thanks to \cref{lem:splitcounting}. 

We turn $\gradC$ into a grading category by equipping it with the associator $\assoc : \gradC^{[3]} \rightarrow R$ given
for
\[
d \xleftarrow{(t'',p'')} c \xleftarrow{(t',p')} b \xleftarrow{(t,p)} a \quad \in \gradC^{[3]},
\]
 by
\[
\assoc((t'',p''),(t',p'),(t,p)) := \assoc_1 \assoc_2,
\]
where 
\begin{equation}\label{eq:lambda1diag}
\begin{tikzcd}
& \bar d t'' t' b \otimes \bar b t a \ar{dr}{\chcob_{dba}(t''t',t)} & \\
\bar d t'' c  \otimes \bar c t' b  \otimes  \bar b t a \ar{ur}{\chcob_{dcb}(t'',t') \otimes 1_{ba}(t)} \ar[swap]{dr}{1_{dc}(t'') \otimes \chcob_{cba}(t',t) } && \bar d t'' t' t a \\
& \bar d t'' c \otimes \bar c t' t a \ar[phantom]{uu}{\quad\rotatebox{90}{$\Rightarrow$}\ \assoc_1} \ar[swap]{ur}{\chcob_{dca}(t'',t't)}
&
\end{tikzcd}
\end{equation}
in the sense that 
there exists a unique (up to homotopy) locally vertical change of chronology 
\[
H : \chcob_{dca}(t'',t't) \circ  \bigl( 1_{dc}(t'') \otimes \chcob_{cba}(t',t) \bigr) \Rightarrow \chcob_{dba}(t''t',t) \circ \bigl( \chcob_{dcb}(t'',t') \otimes 1_{ba}(t) \bigr),
\]
 and we take $\assoc_1 := \imath(H)$, 
%
%
%
%
and where
\[
\assoc_2 :=  \lambda_R( s_{cba}(t',t), p''), 
\]
with $\lambda_R$ being the bilinear map defined in \cref{eq:deflambdaR}.

\smallskip

In terms of pictures we can explain $\alpha$ as coming from
\[
\tikzdiag[xscale=1]{0}{
	\draw (0,1) node[below]{$(t'',p'')$}
		.. controls (0,1.5) and (.5,1.5) ..
		(.5,2)
		.. controls (.5,2.5) and (1.25,2.5) ..
		(1.25,3); 
	\draw (1,.5) node[below]{$(t',p')$}
		--
		(1,1)
		.. controls (1,1.5) and (.5,1.5) ..
		(.5,2);
	\draw (2, 0) node[below]{$(t,p)$}
		--
		(2,2)
		.. controls (2,2.5) and (1.25,2.5) ..
		(1.25,3);
	%
}
\ = \assoc_1 \ 
\tikzdiag[xscale=1]{0}{
	\draw (-.5,1) node[below]{$(t'',p'')$}
		--
		(-.5,2)
		.. controls (-.5,2.5) and (.25,2.5) ..
		(.25,3);
	\draw (.5,.5) node[below]{$(t',p')$}
		--
		(.5,1)
		.. controls (.5,1.5) and (1,1.5) ..
		(1,2);
	\draw (1.5,0) node[below]{$(t,p)$}
		--
		(1.5,1)
		.. controls (1.5,1.5) and (1,1.5) ..
		(1,2)
		.. controls (1,2.5) and (.25,2.5) ..
		(.25,3);
}
\ = \assoc_1 \assoc_2 \ 
\tikzdiag[xscale=1]{0}{
	\draw (-.5,2) node[below]{$(t'',p'')$}
		.. controls (-.5,2.5) and (.25,2.5) ..
		(.25,3);
	\draw (.5,.5) node[below]{$(t',p')$}
		.. controls (.5,1) and (1,1) ..
		(1,1.5)
		--
		(1,2)
		.. controls (1,2.5) and (.25,2.5) ..
		(.25,3);
	\draw (1.5,0) node[below]{$(t,p)$}
		--
		(1.5,.5)
		.. controls (1.5,1) and (1,1) ..
		(1,1.5)
		--
		(1,2);
}
\]
where the trivalent vertices represent the canonical cobordisms used to form the composition maps $\mu$. 

\begin{rem}
Note that the definition of $\assoc$ also make sense for non-reduced flat tangles. However, taking the reduced version of the tangle would give the same value $\assoc(t'',t',t) = \assoc(\widehat t'', \widehat t', \widehat t)$ since $\chcob_{cba}(t',t)$ is the identity on the extra loops in $t'$ and $t$. 
\end{rem}

\begin{prop}\label{prop:gradCassoc}
The map $\assoc : \gradC^{[3]} \rightarrow R$ defined above is a 3-cocycle.
\end{prop}

\begin{proof}
By definition, $d\alpha(g''',g'',g',g)$ computes the difference between taking the two possible path in the following diagram:
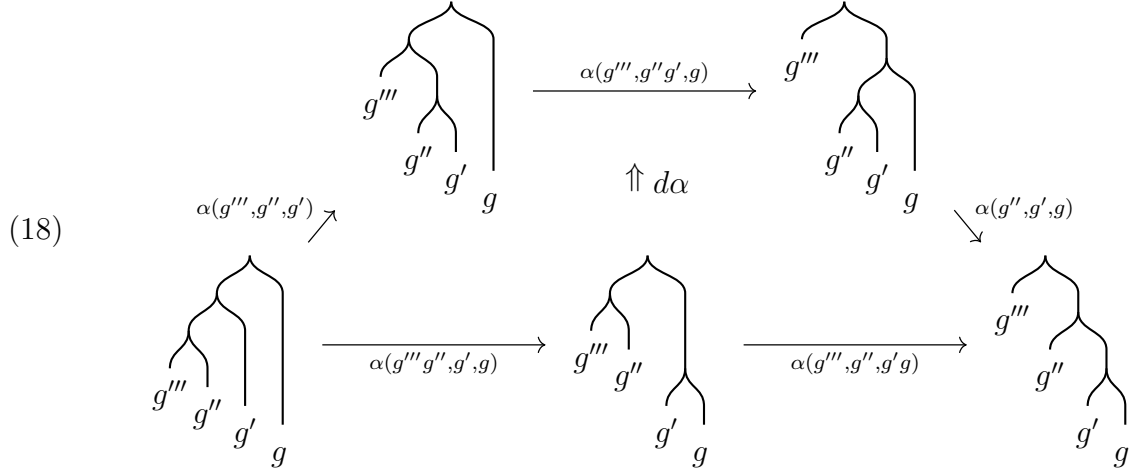
\begin{equation}\label{eq:dassocpent}
\begin{tikzcd}[column sep=1ex, row sep = .5ex]
&
\tikzdiag[scale=.5]{0}{
	\draw (0,2) node[below]{$g'''$}
		.. controls (0,2.5) and (.75,2.5) ..
		(.75,3)
		.. controls (.75,3.5) and (1.875,3.5) ..
		(1.875,4);
	\draw (1,.5) node[below]{$g''$}
		.. controls (1,1) and (1.5,1) .. 
		(1.5,1.5);
	\draw (2, 0) node[below]{$g'$}
		--
		(2,.5)
		.. controls (2,1) and (1.5,1) .. 
		(1.5,1.5)
		--
		(1.5,2)
		.. controls (1.5,2.5) and (.75,2.5) ..
		(.75,3);
	\draw (3,-.5) node[below]{$g\phantom{'}$}
		--
		(3,3)
		.. controls (3,3.5) and (1.875,3.5) ..
		(1.875,4);
}
\ar{rr}{\assoc(g''',g''g',g)}
&{}&
\tikzdiag[scale=.5]{0}{
	\draw (0,3) node[below]{$g'''$}
		.. controls (0,3.5) and (1.125,3.5) ..
		(1.125,4);
	\draw (1,.5) node[below]{$g''$}
		.. controls (1,1) and (1.5,1) .. 
		(1.5,1.5);
	\draw (2, 0) node[below]{$g'$}
		--
		(2,.5)
		.. controls (2,1) and (1.5,1) .. 
		(1.5,1.5)
		.. controls (1.5,2) and (2.25,2) ..
		(2.25,2.5);
	\draw (3,-.5) node[below]{$g\phantom{'}$}
		--
		(3,1.5)
		.. controls (3,2) and (2.25,2) ..
		(2.25,2.5)
		-- 
		(2.25,3)
		.. controls (2.25,3.5) and (1.125,3.5) ..
		(1.125,4);
}
\ar{dr}{\assoc(g'',g',g)}
&
\\
\tikzdiag[scale=.5]{0}{
	\draw (0,1) node[below]{$g'''$}
		.. controls (0,1.5) and (.5,1.5) ..
		(.5,2)
		.. controls (.5,2.5) and (1.25,2.5) ..
		(1.25,3)
		.. controls (1.25,3.5) and (2.125,3.5) ..
		(2.125,4);
	\draw (1,.5) node[below]{$g''$}
		--
		(1,1)
		.. controls (1,1.5) and (.5,1.5) ..
		(.5,2);
	\draw (2, 0) node[below]{$g'$}
		--
		(2,2)
		.. controls (2,2.5) and (1.25,2.5) ..
		(1.25,3);
	\draw (3,-.5) node[below]{$g\phantom{'}$}
		--
		(3,3)
		.. controls (3,3.5) and (2.125,3.5) ..
		(2.125,4);
}
\ar{ur}{\assoc(g''',g'',g')}
\ar[swap]{rr}{\assoc(g'''g'',g',g)}
&&
\tikzdiag[scale=.5]{0}{
	\draw (0,2) node[below]{$g'''$}
		.. controls (0,2.5) and (.5,2.5) ..
		(.5,3)
		.. controls (.5,3.5) and (1.5,3.5) ..
		(1.5,4);
	\draw (1,1.5) node[below]{$g''$}
		-- 
		(1,2)
	 	.. controls (1,2.5) and (.5,2.5) ..
	 	(.5,3);
	\draw (2, 0) node[below]{$g'$}
		.. controls (2,.5) and (2.5,.5) ..
		(2.5,1);
	\draw (3,-.5) node[below]{$g\phantom{'}$}
		--
		(3,0)
		.. controls (3,.5) and (2.5,.5) ..
		(2.5,1)
		--
		(2.5,3)
		.. controls (2.5,3.5) and (1.5,3.5) ..
		(1.5,4);
}
\ar[swap]{rr}{\assoc(g''',g'',g'g)}
 \ar[phantom]{u}{ \quad \rotatebox{90}{$\Rightarrow$}\ d\assoc}
&&
\tikzdiag[scale=.5]{0}{
	\draw (0,3) node[below]{$g'''$}
		.. controls (0,3.5) and (.875,3.5) ..
		(.875,4);
	\draw (1,1.5) node[below]{$g''$}
	 	.. controls (1,2) and (1.75,2) ..
	 	(1.75,2.5);
	\draw (2, 0) node[below]{$g'$}
		.. controls (2,.5) and (2.5,.5) ..
		(2.5,1);
	\draw (3,-.5) node[below]{$g\phantom{'}$}
		--
		(3,0)
		.. controls (3,.5) and (2.5,.5) ..
		(2.5,1)
		--
		(2.5,1.5)
		.. controls (2.5,2) and (1.75,2) ..
		(1.75,2.5)
		--
		(1.75,3)
		.. controls (1.75,3.5) and (.875,3.5) ..
		(.875,4);
}
\end{tikzcd}
\end{equation}
Since we are considering two locally vertical changes of chronology with same source and target, the following diagram of cobordisms commutes by \cref{prop:locvertchange}: 
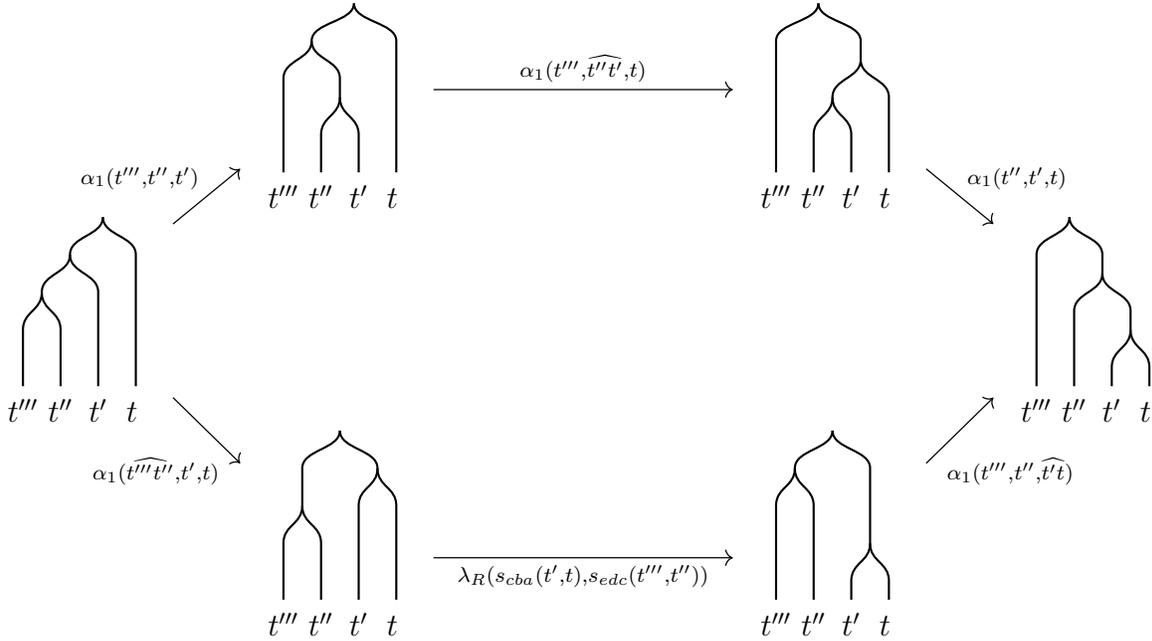
\[
\begin{tikzcd}[column sep=5ex, row sep = -2ex]
&
\tikzdiag[scale=.5]{0}{
	\draw (0,-.5)node[below]{$t'''$}
		--
		(0,2) 
		.. controls (0,2.5) and (.75,2.5) ..
		(.75,3)
		.. controls (.75,3.5) and (1.875,3.5) ..
		(1.875,4);
	\draw (1,-.5) node[below]{$t''$}
		--
		(1,.5) 
		.. controls (1,1) and (1.5,1) .. 
		(1.5,1.5);
	\draw (2, -.5) node[below]{$t'$}
		--
		(2,.5)
		.. controls (2,1) and (1.5,1) .. 
		(1.5,1.5)
		--
		(1.5,2)
		.. controls (1.5,2.5) and (.75,2.5) ..
		(.75,3);
	\draw (3,-.5) node[below]{$t\phantom{'}$}
		--
		(3,3)
		.. controls (3,3.5) and (1.875,3.5) ..
		(1.875,4);
}
\ar{rrrr}{\assoc_1(t''',\widehat{t''t'},t)}
&&{}&&
\tikzdiag[scale=.5]{0}{
	\draw (0,-.5) node[below]{$t'''$}
		--
		(0,3)
		.. controls (0,3.5) and (1.125,3.5) ..
		(1.125,4);
	\draw (1,-.5) node[below]{$t''$}
		--
		(1,.5)
		.. controls (1,1) and (1.5,1) .. 
		(1.5,1.5);
	\draw (2, -.5) node[below]{$t'$}
		--
		(2,.5)
		.. controls (2,1) and (1.5,1) .. 
		(1.5,1.5)
		.. controls (1.5,2) and (2.25,2) ..
		(2.25,2.5);
	\draw (3,-.5) node[below]{$t\phantom{'}$}
		--
		(3,1.5)
		.. controls (3,2) and (2.25,2) ..
		(2.25,2.5)
		-- 
		(2.25,3)
		.. controls (2.25,3.5) and (1.125,3.5) ..
		(1.125,4);
}
\ar{dr}{\assoc_1(t'',t',t)}
&
\\
\tikzdiag[scale=.5]{0}{
	\draw (0,-.5) node[below]{$t'''$}
		--
		(0,1)
		.. controls (0,1.5) and (.5,1.5) ..
		(.5,2)
		.. controls (.5,2.5) and (1.25,2.5) ..
		(1.25,3)
		.. controls (1.25,3.5) and (2.125,3.5) ..
		(2.125,4);
	\draw (1,-.5) node[below]{$t''$}
		--
		(1,1)
		.. controls (1,1.5) and (.5,1.5) ..
		(.5,2);
	\draw (2, -.5) node[below]{$t'$}
		--
		(2,2)
		.. controls (2,2.5) and (1.25,2.5) ..
		(1.25,3);
	\draw (3,-.5) node[below]{$t\phantom{'}$}
		--
		(3,3)
		.. controls (3,3.5) and (2.125,3.5) ..
		(2.125,4);
}
\ar{ur}{\assoc_1(t''',t'',t')}
\ar[swap,pos=.8]{dr}{\assoc_1(\widehat{t'''t''},t',t)}
&&&&&&
\tikzdiag[scale=.5]{0}{
	\draw (0,-.5) node[below]{$t'''$}
		--
		(0,3)
		.. controls (0,3.5) and (.875,3.5) ..
		(.875,4);
	\draw (1,-.5) node[below]{$t''$}
		--
		(1,1.5)
	 	.. controls (1,2) and (1.75,2) ..
	 	(1.75,2.5);
	\draw (2, -.5) node[below]{$t'$}
		--
		(2,0)
		.. controls (2,.5) and (2.5,.5) ..
		(2.5,1);
	\draw (3,-.5) node[below]{$t\phantom{'}$}
		--
		(3,0)
		.. controls (3,.5) and (2.5,.5) ..
		(2.5,1)
		--
		(2.5,1.5)
		.. controls (2.5,2) and (1.75,2) ..
		(1.75,2.5)
		--
		(1.75,3)
		.. controls (1.75,3.5) and (.875,3.5) ..
		(.875,4);
}
\\
&
\tikzdiag[scale=.5]{-3ex}{
	\draw (0,-.5) node[below]{$t'''$}
		--
		(0,1)
		.. controls (0,1.5) and (.5,1.5) ..
		(.5,2)
		-- 
		(.5,3)
		.. controls  (.5,3.5) and (1.5,3.5) ..
		(1.5,4);
	\draw (1,-.5) node[below]{$t''$}
		-- 
		(1,1)
	 	.. controls (1,1.5) and (.5,1.5) ..
	 	(.5,2);
	\draw (2, -.5) node[below]{$t'$}
		--
		(2,2)
		.. controls (2,2.5) and (2.5,2.5) .. 
		(2.5,3);
	\draw (3,-.5) node[below]{$t\phantom{'}$}
		--
		(3,2)
		.. controls (3,2.5) and (2.5,2.5) .. 
		(2.5,3)
		.. controls  (2.5,3.5) and (1.5,3.5) ..
		(1.5,4);
}
\ar[swap]{rrrr}{\lambda_R( s_{cba}(t',t), s_{edc}(t''',t''))}
&&{} 
&&
\tikzdiag[scale=.5]{-3ex}{
	\draw (0,-.5) node[below]{$t'''$}
		--
		(0,2)
		.. controls (0,2.5) and (.5,2.5) ..
		(.5,3)
		.. controls (.5,3.5) and (1.5,3.5) ..
		(1.5,4);
	\draw (1,-.5) node[below]{$t''$}
		-- 
		(1,2)
	 	.. controls (1,2.5) and (.5,2.5) ..
	 	(.5,3);
	\draw (2, -.5) node[below]{$t'$}
		--
		(2,0)
		.. controls (2,.5) and (2.5,.5) ..
		(2.5,1);
	\draw (3,-.5) node[below]{$t\phantom{'}$}
		--
		(3,0)
		.. controls (3,.5) and (2.5,.5) ..
		(2.5,1)
		--
		(2.5,3)
		.. controls (2.5,3.5) and (1.5,3.5) ..
		(1.5,4);
}
\ar[swap,pos=.2]{ur}{\assoc_1(t''',t'',\widehat{t't})}
&
\end{tikzcd}
\] 
Therefore, we have that the contribution of $\assoc_1$ in \cref{eq:dassocpent} is
\[
\text{(top)}  = \lambda_R\bigl(s_{cba}(t',t), s_{edc}(t''',t'')\bigr) \text{(bottom)}.
\] 

 For the $\assoc_2$ contribution, we have for the top that
\begin{align*}
&\lambda_R\bigl( s_{dcb}(t'',t'), p'''\bigr)
 \lambda_R\bigl(s_{dba}(t''t',t), p'''\bigr)
 \lambda_R\bigl(s_{cba}(t',t), p''\bigr)
 \\
 &\quad = 
 \lambda_R\bigl(s_{dcb}(t'',t') + s_{dba}(\widehat{t''t'},t), p'''\bigr)
 \lambda_R\bigl(s_{cba}(t',t), p''\bigr),
\end{align*}
and for the bottom that
\begin{align*}
&\lambda_R\bigl(s_{cba}(t',t), p'''+p''+s_{edc}(t''',t'')\bigr)
\lambda_R\bigl(s_{dca}(t'',\widehat{t't}), p'''\bigr)
\\
& \quad =
\lambda_R\bigl(s_{cba}(t',t), s_{edc}(t''',t'')\bigr)
\lambda_R\bigl(s_{cba}(t',t), p'''\bigr)
\lambda_R\bigl(s_{cba}(t',t), p''\bigr)
\lambda_R\bigl(s_{dca}(t'',\widehat{t't}), p'''\bigr)
\\
& \quad =
\lambda_R\bigl(s_{cba}(t',t), s_{edc}(t''',t'')\bigr)
\lambda_R\bigl(s_{dca}(t'',\widehat{t't}) + s_{cba}(t',t), p'''\bigr)
\lambda_R\bigl(s_{cba}(t',t), p''\bigr),
\end{align*}
using the bilinearity of $\lambda_R$. 
 By \cref{lem:splitcounting}, we conclude that the $\assoc_2$ contribution in \eqref{eq:dassocpent} is
\[
 \lambda_R\bigl(s_{cba}(t',t),s_{edc}(t''',t'')\bigr) \text{(top)} = \text{(bottom)}.
\]
Putting all these together, we conclude that $d\assoc = 0$. 
\end{proof}

\subsection{$\gradC$-grading shift}

Our goal is to construct a $\gradC$-shifting system compatible with $\assoc$.
First, we note that $\shiftCenter{}(\gradC)$ is the subcategory generated by identity tangles. Explicitly, for $a \in B^n$ and $b \in B^m$, we have
\[
\Hom_{\shiftCenter{S}(\gradC)}(a,b) =
\begin{cases}
\{(\id_n, p) | p \in \bZ\}, &\text{ if $n=m$,} \\
\emptyset, &\text{otherwise.}
\end{cases}
\]

For a 
chronological cobordism with corners $\chcob : t \rightarrow t'$, with $t,t' \in \tangleSpace_n^m$, and for $v \in \bZ^2$, we write $\chcob^v$ for the pair $(\chcob, v)$.
Then, we define the shifting map 
$\shiftFunct{\chcob^v}$ with domain $\shiftDom[ba]{\chcob^v} := \{ (\red t,p) \in \Hom_{\gradC}(a,b) | p \in \bZ^2\}$ and 
\[
	\shiftFunct{\chcob^v}(\red t,p) := \bigl(\red{t'}, p+v+\deg({1_{\bar b}\chcob 1_{a}})\bigr),
\] 
for $(\red t, p) \in \Hom_\gradC(a,b)$.

Given $\chcob_1 : t_1 \rightarrow t'_1$ and $\chcob_2 : t_2 \rightarrow t'_2$ with $t_1,t_1' \in \tangleSpace_n^m$ and $t_2,t_2' \in \tangleSpace_m^{n'}$ we obtain a cobordism $\chcob_2 \bullet \chcob_1 : t_2t_1 \rightarrow t_2't_1'$ by right-then-left horizontal composition. 
In this case we put $\chcob_2^{v_2} \bullet \chcob_1^{v_1} := (\chcob_2 \bullet \chcob_1)^{v_2 + v_1}$. Otherwise, we put $\chcob_2^{v_2} \bullet \chcob_1^{v_1} := 0$.  Then, we define the monoid $I :=  \{\chcob^v\}_{\chcob, v} \sqcup \{\neutralElement,0\}$ with composition $\bullet$ defined above, neutral element $\neutralElement$ and an absorbing element $0$. 
Finally, we take $I_\id := \{ (\un_t)^0 \}_{t \in \redTangleSpace_\bullet^\bullet } \subset I$, where we recall $\un_t$ is the identity cobordism on $t$. 

\begin{prop}
The datum of $S :=(I, \{\shiftFunct{\chcob^v}\}_{\chcob,v} \sqcup \{\shiftFunct{\neutralElement}, \shiftFunct{0}\})$ forms a $\gradC$-shifting system. 
\end{prop}

\begin{proof}
Straightforward.
\end{proof}

We construct compatibility maps for $S$ by putting 
\[
\compMap[cba]{\chcob_2^{v_2} }{\chcob_1^{v_1}}((\red t_2,p_2),(\red t_1,p_1)) := \beta_1 \beta_2  \beta_2' \beta_2'',
\]
with $\beta_1 \in R$ given by making the following diagram commute 
\begin{equation}\label{eq:diagbeta1}
\begin{tikzcd}
& \bar c t_2' b \bar b t_1' a \ar{dr}{\chcob_{cba}(t_2', t_1')} &
\\
\bar c t_2 b \bar b t_1 a \ar[swap]{dr}{\chcob_{cba}(t_2,t_1)}  \ar{ur}{1_{\bar c} \chcob_2 1_{b \bar b} \chcob_1 1_{a}}
& & 
\bar c t_2' t_1' a 
\\
& \bar c t't a \ar[swap]{ur}{1_{\bar c}(\chcob_2 \bullet \chcob_1) 1_a}  \ar[phantom]{uu}{\quad \rotatebox{90}{$\Rightarrow$}\ \beta_1} &
\end{tikzcd}
\end{equation}
in the same sense as in \cref{eq:lambda1diag}, 
and
\begin{align*}
\beta_2 &:= \lambda_R(s_{cba}(t_2',t_1'), v_2+v_1), \\ 
\beta_2' &:=  \lambda_R( \deg(1_{\bar c}\chcob_21_b), v_1), \\ 
\beta_2'' &:= \lambda_R(p_2, \deg(1_{\bar b}\chcob_1 1_a)+ v_1).  
\end{align*}
In terms of pictures, we think of it as
\begin{align}
\begin{split}\label{eq:pictdefbeta}
\tikzdiag{0}{
	\draw (0,1) node[below]{$(t_2,p_2)$}
		--
		(0,1.75) node[near end,rrect]{$\chcob_2$}
		--
		(0,2.5) node[near end,rrect]{$v_2$}
		.. controls (0,3) and (.5,3) ..
		(.5,3.5) -- (.5,4);
	\draw (1,-1) node[below]{$(t_1,p)$}
		--
		(1,-.25) node[near end,rrect]{$\chcob_1$}
		-- 
		(1,.5) node[near end,rrect]{$v_1$}
		-- 
		(1,2.5)
		.. controls (1,3) and (.5,3) ..
		(.5,3.5);
}
\ &= \beta_2'' \ 
\tikzdiag{0}{
	\draw (0,-.5) node[below]{$(t_2,p_2)$}
		--
		(0,1)
		--
		(0,1.75) node[near end,rrect]{$\chcob_2$}
		--
		(0,2.5) node[near end,rrect]{$v_2$}
		.. controls (0,3) and (.5,3) ..
		(.5,3.5) -- (.5,4);
	\draw (1,-1) node[below]{$(t_1,p)$}
		--
		(1,-.5)
		--
		(1,.25) node[near end,rrect]{$\chcob_1$}
		-- 
		(1,1) node[near end,rrect]{$v_1$}
		-- 
		(1,2.5)
		.. controls (1,3) and (.5,3) ..
		(.5,3.5);
}
\ = \beta_2' \beta_2'' \ 
\tikzdiag{0}{
	\draw (0,-.5) node[below]{$(t_2,p_2)$}
		--
		(0,.25)
		--
		(0,1)node[near end,rrect]{$\chcob_2$}
		--
		(0,1.75) 
		--
		(0,2.5) node[near end,rrect]{$v_2$}
		.. controls (0,3) and (.5,3) ..
		(.5,3.5) -- (.5,4);
	\draw (1,-1) node[below]{$(t_1,p)$}
		--
		(1,-.5)
		--
		(1,.25) node[near end,rrect]{$\chcob_1$}
		-- 
		(1,1) 
		-- 
		(1,1.75) node[near end,rrect]{$v_1$}
		-- 
		(1,2.5)
		.. controls (1,3) and (.5,3) ..
		(.5,3.5);
}
\\
\ &=  \beta_2\beta_2'\beta_2'' \
 \tikzdiag{0}{
	\draw (0,-.75) node[below]{$(t_2,p_2)$}
		--
		(0,-.25)
		--
 		(0,.5)  node[near end,rrect]{$\chcob_2$}
		.. controls (0,1) and (.5,1) ..
		(.5,1.5)  
		--
		(.5,2.5)  node[midway,rrect]{$v_2 + v_1$};
	\draw (1,-1.25) node[below]{$(t_1,p)$} 
		--
		(1,-1)
		--
		(1,-.25)  node[near end,rrect]{$\chcob_1$}
		-- 
		(1,.5)
		.. controls (1,1) and (.5,1) ..
		(.5,1.5);
}
\ =  \beta_1\beta_2\beta_2'\beta_2'' \
 \tikzdiag{0}{
	\draw (0,-1.5) node[below]{$(t_2,p_2)$}
		.. controls (0,-1) and (.5,-1) ..
		(.5,-.5)  
		--
		(.5,0)
		--
		(.5,.75) node[midway,rrect]{$\chcob_2\bullet\chcob_1$}
		--
		(.5,1.5)  node[midway,rrect]{$v_2+v_1$}
		--
		(.5,1.75);
	\draw (1,-2) node[below]{$(t_1,p)$} 
		--
		(1,-1.5)
		.. controls (1,-1) and (.5,-1) ..
		(.5,-.5);
}
\end{split}
\end{align}
The compatibility maps with the neutral shift $\shiftFunct{\neutralElement}$ are computed in a similar fashion, by treating it as the identity cobordism on the identity tangle, and with $v = 0$. 

\begin{rem}
We computed $\compMap{}{}$ by thinking of a degree shift as shifting first by the chronological cobordism $\chcob$, then shifting the $\bZ\times\bZ$-degree by $v$. We could also have done it in the reverse order, modifying only $\beta_2 := \lambda_R(s_{cba}(t_2,t_1), v_2+v_1)$ and $\beta_2' := (v_2, \chcob_1)$. It would give a different set of compatibility maps.
\end{rem}

\begin{lem} \label{lem:horcompdeg}
We have
\[
\deg(1_{\bar c}(\chcob_2 \bullet \chcob_1)1_a) +  s_{cba}(t_2,t_1) = \deg(1_{\bar c} \chcob_2 1_b) + \deg(1_{\bar b} \chcob_1 1_a) + s_{cba}(t_2',t_1'). 
\]
\end{lem}

\begin{proof}
There is a diffeomorphism
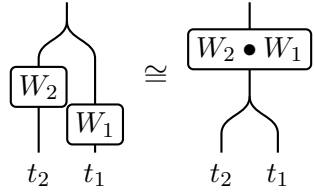
\[
\tikzdiag[yscale=.5,xscale=.75]{0}{
	\draw (1,-1.5) node[below]{\small $t_2$}
		--
		(1,-.5)
		-- 
		(1,.5)   node[near end,rrect]{\small $\chcob_2$}
		--
		(1,1)
		.. controls (1,1.5) and (1.5,1.5) ..
		(1.5,2);
	\draw (2,-1.5) node[below]{\small $t_1$}
		--
		(2, -.5)   node[near end,rrect]{\small $\chcob_1$}
		--
		(2,1)
		.. controls (2,1.5) and (1.5,1.5) ..
		(1.5,2)
		--
		(1.5,2.5);
}
\ \cong \ 
\tikzdiag[yscale=.5,xscale=.75]{0}{
	\draw (1,-1.5) node[below]{\small $t_2$}
		--
		(1,-1)
		.. controls (1,-.5) and (1.5,-.5) ..
		(1.5,0);
	\draw (2,-1.5) node[below]{\small $t_1$}
		--
		(2, -1)   
		.. controls (2,-.5) and (1.5,-.5) ..
		(1.5,0)
		--
		(1.5,2.5) node[midway,rrect]{\small $\chcob_2 \bullet \chcob_1$};
}
\]
between the cobordisms. Thus, they have the same degree.
\end{proof}

\begin{prop}\label{prop:assoccompbeta}
The construction above forms a $\gradC$-shifting system compatible with $\assoc$.
\end{prop}

\begin{proof}
We fix 
\[
d \xleftarrow{(\red t_3,p_3)} c \xleftarrow{(\red t_2,p_2)} b \xleftarrow{(\red t_1,p_1)} a \quad \in \gradC^{[3]}
\]
and $\chcob_i : t_i \rightarrow t_i'$ and $v_i \in \bZ^{2}$ for $i \in \{1,2,3\}$.  
Since it is clear from the context, we will write below $|\chcob_1|$ instead of $\deg(1_{\bar b}\chcob_1 1_a)$ and so on. 

We first look at the contribution of $\assoc_1$ and $\beta_1$ in the left and right side of \cref{eq:shiftingsystcomp}. 
By \cref{prop:locvertchange}, the following diagram of cobordisms commutes: 
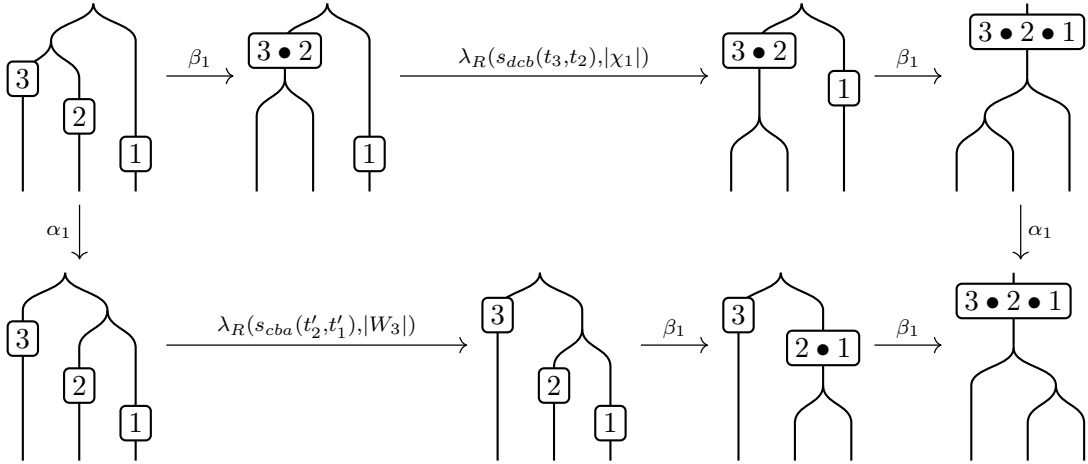
\[
\begin{tikzcd}[column sep=5ex]
\tikzdiag[yscale=.5,xscale=.75]{0}{
	\draw (0,-2)
		--
		(0,0)
		--
		(0,1)   node[near end,rrect]{\small $3$}
		.. controls (0,1.5) and (.5,1.5) ..
		(.5,2)
		.. controls (.5,2.5) and (1.25,2.5) ..
		(1.25,3);
	\draw (1,-2)
		--
		(1,-1)
		-- 
		(1,0)   node[near end,rrect]{\small $2$}
		--
		(1,1)
		.. controls (1,1.5) and (.5,1.5) ..
		(.5,2);
	\draw (2,-2)
		--
		(2, -1)   node[near end,rrect]{\small $1$}
		--
		(2,2)
		.. controls (2,2.5) and (1.25,2.5) ..
		(1.25,3);
}
\ar{r}{\beta_1}
\ar[swap]{d}{\alpha_1}
&
\tikzdiag[yscale=.5,xscale=.75]{0}{
	\draw (0,-2)
		--
		(0,0) 
		.. controls (0,.5) and (.5,.5) ..
		(.5,1)
		--
		(.5,2)node[midway,rrect]{\small $3\bullet2$}
		.. controls (.5,2.5) and (1.25,2.5) ..
		(1.25,3);
	\draw (1,-2)
		-- 
		(1,0)   
		--
		(1,0)
		.. controls (1,.5) and (.5,.5) ..
		(.5,1);
	\draw (2,-2)
		--
		(2, -1)   node[near end,rrect]{\small $1$}
		--
		(2,2)
		.. controls (2,2.5) and (1.25,2.5) ..
		(1.25,3);
}
\ar{rr}{\lambda_R(s_{dcb}(t_3,t_2), |\chi_1|)}
&
\hspace{2ex}
&
\tikzdiag[yscale=.5,xscale=.75]{0}{
	\draw (0,-2)
		--
		(0,-1) 
		.. controls (0,-.5) and (.5,-.5) ..
		(.5,0)
		--
		(.5,1)
		--
		(.5,2)node[midway,rrect]{\small $3\bullet2$}
		.. controls (.5,2.5) and (1.25,2.5) ..
		(1.25,3);
	\draw (1,-2)
		--
		(1,-1)
		.. controls (1,-.5) and (.5,-.5) ..
		(.5,0);
	\draw (2,-2)
		--
		(2, 0) 
		--
		(2,1)node[midway,rrect]{\small $1$}
		--
		(2,2)
		.. controls (2,2.5) and (1.25,2.5) ..
		(1.25,3);
}
\ar{r}{\beta_1}
&
\tikzdiag[yscale=.5,xscale=.75]{0}{
	\draw (0,-2)
		--
		(0,-1) 
		.. controls (0,-.5) and (.5,-.5) ..
		(.5,0)
		.. controls (.5,.5) and (1.25,.5) ..
		(1.25,1)
		--
		(1.25,3) node[midway,rrect]{\small $3\bullet2 \bullet 1$};
	\draw (1,-2)
		--
		(1,-1)
		.. controls (1,-.5) and (.5,-.5) ..
		(.5,0);
	\draw (2,-2)
		--
		(2, 0)
		.. controls (2,.5) and (1.25,.5) ..
		(1.25,1);
}
\ar{d}{\assoc_1}
\\
\tikzdiag[yscale=.5,xscale=.75]{0}{
	\draw (0,-2)
		--
		(0,1)
		--
		(0,1)   node[near end,rrect]{\small $3$}
		--
		(0,2)
		.. controls (0,2.5) and (.75,2.5) ..
		(.75,3);
	\draw (1,-2)
		--
		(1,-1)
		-- 
		(1,0)   node[near end,rrect]{\small $2$}
		--
		(1,1)
		.. controls (1,1.5) and (1.5,1.5) ..
		(1.5,2);
	\draw (2,-2)
		--
		(2, -1)   node[near end,rrect]{\small $1$}
		--
		(2,1)
		.. controls (2,1.5) and (1.5,1.5) ..
		(1.5,2)
		.. controls (1.5,2.5) and (.75,2.5) .. 
		(.75,3);
}
\ar{rr}{\lambda_R(s_{cba}(t_2',t_1'), |\chcob_3|)}
&
&
\tikzdiag[yscale=.5,xscale=.75]{0}{
	\draw (0,-1.5)
		--
		(0,1) 
		--
		(0,2.5)  node[near end,rrect]{\small $3$}
		.. controls (0,3) and (.75,3) ..
		(.75,3.5);
	\draw (1,-1.5)
		--
		(1,-.5)
		-- 
		(1,.5)   node[near end,rrect]{\small $2$}
		--
		(1,1)
		.. controls (1,1.5) and (1.5,1.5) ..
		(1.5,2);
	\draw (2,-1.5)
		--
		(2, -.5)   node[near end,rrect]{\small $1$}
		--
		(2,1)
		.. controls (2,1.5) and (1.5,1.5) ..
		(1.5,2)
		--
		(1.5,2.5)
		.. controls (1.5,3) and (.75,3) .. 
		(.75,3.5);
}
\ar{r}{\beta_1}
&
\tikzdiag[yscale=.5,xscale=.75]{0}{
	\draw (0,-1.5)
		--
		(0,1) 
		--
		(0,2.5)  node[near end,rrect]{\small $3$}
		.. controls (0,3) and (.75,3) ..
		(.75,3.5);
	\draw (1,-1.5)
		--
		(1,-.5)
		.. controls (1,0) and (1.5,0) ..
		(1.5,.5);
	\draw (2,-1.5)
		--
		(2, -.5) 
		.. controls (2,0) and (1.5,0) ..
		(1.5,.5)
		--
		(1.5,1)
		--
		(1.5,1.5)node[midway,rrect]{\small $2\bullet1$}
		--
		(1.5,2.5)  
		.. controls (1.5,3) and (.75,3) .. 
		(.75,3.5);
}
\ar{r}{\beta_1}
&
\tikzdiag[yscale=.5,xscale=.75]{0}{
	\draw (0,-1.5)
		--
		(0,.5) 
		.. controls (0,1) and (.75,1) ..
		(.75,1.5);
	\draw (1,-1.5)
		--
		(1,-.5)
		.. controls (1,0) and (1.5,0) ..
		(1.5,.5);
	\draw (2,-1.5)
		--
		(2, -.5) 
		.. controls (2,0) and (1.5,0) ..
		(1.5,.5)  
		.. controls (1.5,1) and (.75,1) .. 
		(.75,1.5)
		--
		(.75,3.5) node[midway,rrect]{\small $3\bullet2 \bullet 1$};
}
\end{tikzcd}
\]
where a box with label $i$ is $\chcob_i$. 
Therefore, we have
\begin{equation}\label{eq:compalpha1beta1}
\lambda_R(s_{dcb(t_3,t_2)}, |\chcob_1|) \times \text{(left)} =  \lambda_R(s_{cba}(t_2',t_1'), |\chcob_3|)  \times \text{(right)}.
\end{equation}

We now compute the contribution of $\alpha_2, \beta_2, \beta_2'$ and $\beta_2''$. 
The left side of \cref{eq:shiftingsystcomp} yields
\begin{align}
\begin{split} \label{eq:compatleft1}
&\lambda_R(s_{cba}(t_2,t_1), p_3) \\
\times &\lambda_R(s_{dba}(t_3't_2', t_1'), v_3+v_2+v_1) \lambda_R(|\chcob_3 \bullet \chcob_2|, v_1) \lambda_R( p_3+p_2+s_{dcb}(t_3,t_2), |\chcob_1|+v_1) \\
\times&\lambda_R(s_{dcb}(t_3',t_2'), v_3 + v_2) \lambda_R(|\chcob_3|,v_2) \lambda_R(p_3, |\chcob_2| + v_2),
\end{split}
\end{align}
and the right side yields
\begin{align}
\begin{split}  \label{eq:compatright1}
&\lambda_R(s_{dca}(t_3',t_2't_1'), v_3+v_2+v_1) \lambda_R(|\chcob_3|, v_2+v_1)\lambda_R(p_3, v_2+v_1+|\chcob_2 \bullet \chcob_1|) \\
\times & \lambda_R(s_{cba}(t_2',t_1'), v_2+v_1) \lambda_R(|\chcob_2|,v_1) \lambda_R(p_2, |\chcob_1|+v_1) \\
\times & \lambda_R(s_{cba}(t_2', t_1'), p_3+|\chcob_3| + v_3).
\end{split}
\end{align}
We use \cref{lem:horcompdeg} to get
\[
\lambda_R(|\chcob_3 \bullet \chcob_2|, v_1) = \lambda_R(|\chcob_3|+|\chcob_2|, v_1)\lambda_R(s_{dcb}(t_3',t_2'),v_1)\lambda_R^{-1}(s_{dcb}(t_3,t_2),v_1),
\]
in \cref{eq:compatleft1}. We also decompose $\lambda_R(|\chcob_1|+v_1, p_3+p_2+s_{dcb}(t_3,t_2))$ in \cref{eq:compatleft1}, getting a term $\lambda_R(s_{dcb}(t_3,t_2),v_1)$ that cancels with $\lambda_R^{-1}(s_{dcb}(t_3,t_2),v_1)$ above, so that remains
\begin{align}
\begin{split} \label{eq:compatleft2}
&\lambda_R(s_{cba}(t_2,t_1), p_3) \\
\times &\lambda_R(|\chcob_3|+|\chcob_2|+p_3+p_2, v_1) \\
\times &\lambda_R(s_{dba}(t_3't_2', t_1'), v_3+v_2+v_1)  \lambda_R( p_3+p_2+s_{dcb}(t_3,t_2), |\chcob_1|) \\
\times&\lambda_R(s_{dcb}(t_3',t_2'), v_3 + v_2+v_1) \lambda_R(|\chcob_3|,v_2) \lambda_R(p_3, |\chcob_2| + v_2),
\end{split}
\end{align}
in \cref{eq:compatleft1}. Playing a similar game on \cref{eq:compatright1}, we obtain
\begin{align}
\begin{split} \label{eq:compatright2}
&\lambda_R(p_3,|\chcob_2| + |\chcob_1|) \lambda_R(s_{cba}(t_2,t_1),p_3) \\
\times &\lambda_R(s_{dca}(t_3',t_2't_1'), v_3+v_2+v_1) \lambda_R(|\chcob_3|, v_2+v_1)\lambda_R(p_3, v_2+v_1) \\
\times & \lambda_R(s_{cba}(t_2',t_1'), v_3+v_2+v_1) \lambda_R(|\chcob_2|,v_1) \lambda_R(p_2, |\chcob_1|+v_1) \\
\times & \lambda_R(s_{cba}(t_2', t_1'),|\chcob_3| ),
\end{split}
\end{align}
where we also use the fact that $\lambda_R(x,y) = \lambda^{-1}_R(y,x)$. By \cref{lem:splitcounting}, we have
\begin{align*}
\lambda_R(s_{dba}(t_3't_2', t_1'),& v_3+v_2+v_1) \lambda_R(s_{dcb}(t_3',t_2'), v_3 + v_2+v_1 = \\
&\lambda_R(s_{dca}(t_3',t_2't_1'), v_3+v_2+v_1)\lambda_R(s_{cba}(t_2',t_1'), v_3+v_2+v_1),
\end{align*}
so that a careful examination of the remaining terms yields
\[
 \lambda_{R}(s_{cba}(t_2',t_1'), |\chcob_3|) \times \cref{eq:compatleft2} = \lambda_R(s_{dcb}(t_3,t_2), |\chcob_1|)  \times \cref{eq:compatright2}.
\]
Together with \cref{eq:compalpha1beta1}, it concludes the proof. 
\end{proof}

\subsection{Vertical composition}

We extend the  $\gradC$-shifting system to a  $\gradC$-shifting 2-system. For this, we define the vertical composion on $I$ as follows:
\[
\chcob_2^{v_2} \circ \chcob_1^{v_1} := 
\begin{cases}
 (\chcob_2 \circ \chcob_1)^{v_2+v_1}, &  \text{if $t_1' = t_2$,}  
 \\
 0, & \text{otherwise,}
 \end{cases}
\]
for $\chcob_1 : t_1 \rightarrow t_1'$ and $\chcob_2 : t_2 \rightarrow t_2'$. 
Then, we construct compatibility maps by setting 
\begin{align*}
\vCompMap[ba]{\chcob_2^{v_2}}{\chcob_1^{v_1}}((t, p)) := \lambda_R(\deg(1_{\bar b} W_2 1_a), v_1),
\intertext{and}
\interCompMap[cba]{{\chcob_2'}^{v_2'}}{{\chcob_2}^{v_2}}{{\chcob_1'}^{v_1'}}{{\chcob_1}^{v_1}}((t,p)) := \imath(H_a^c) \lambda_R\bigl(v_2 ,v_1' \bigr),
\end{align*}
where $H : (W_2' \circ W_2) \bullet (W_1' \circ W_1) \Rightarrow (W_2' \bullet W_1') \circ (W_2 \bullet W_1)$ is a locally vertical change of chronology. This is if the cobordism are vertically composable, otherwise the compatibility maps are zero. 

\begin{prop}
The $\gradC$-shifting 2-system $S$ defined above is compatible with $\assoc$ through $(\beta, \gamma, \Xi)$ also defined above.
\end{prop}

\begin{proof}
By \cref{prop:assoccompbeta} we already know the underlying shifting system is compatible with $\assoc$ through $\beta$. Thus, we only need to verify that \cref{eq:shiftingsystvcomp} and \cref{eq:shiftingsystvcomp2} both hold. The second one is straightforward using the linearity of $\lambda_R$. The first one can be proven using similar arguments as in \cref{prop:assoccompbeta}. We leave the details to the reader. 
\end{proof}

\subsection{Changes of chronology}

Consider a change of chronology $H : \chcob \Rightarrow \chcob'$. We can extend it trivially to a change of chronology ${}_bH_a: (1_{\bar b} \chcob 1_a) \Rightarrow (1_{\bar b} \chcob' 1_a) $. Therefore, we obtain a natural transformation $\shiftFunct{H} :  \shiftFunct{\chcob} \Rightarrow \shiftFunct{\chcob'}$ as $\Mod^{\gradC}$-functors, given by 
\[
\shiftFunct{H}(M) : \shiftFunct{\chcob}(M)  \rightarrow  \shiftFunct{\chcob'}(M), \quad \shiftFunct{\chcob}(m) \mapsto \shiftFunct{H}(|m|) \shiftFunct{\chcob'}(m),
\]
where $\shiftFunct{H}(|m|) :=\imath({}_bH_a)^{-1}$, 
for all $M \in \Mod^{\gradC}$ and $m \in M_{g : a \rightarrow b}$.

\begin{prop}\label{prop:nattrfromchangeofch}
For $M  \in \Bimod^\gradC(A_2,A_1)$, the map $\shiftFunct{H}(M) $ defined above is a map of bimodules. In particular, $\shiftFunct{H}$ gives a natural transformation of $ \Bimod^\gradC(A_2,A_1)$-functors.
\end{prop}

\begin{proof}
We will show a slightly stronger statement, that is the following diagram commutes:
\[
\begin{tikzcd}
\shiftFunct{\chcob_2}(M') \otimes \shiftFunct{\chcob_1}(M) \ar{r}{\beta_{\chcob_2\chcob_1}} \ar[swap]{d}{\shiftFunct{H_2} \otimes \shiftFunct{H_1}}
& 
\shiftFunct{\chcob_2 \bullet \chcob_1} (M' \otimes M) \ar{d}{\shiftFunct{H_2 \bullet H_1}}
\\
\shiftFunct{\chcob_2'}(M') \otimes \shiftFunct{\chcob_1'}(M)  \ar[swap]{r}{\compMap{\chcob_2'}{\chcob_1'}}
&
\shiftFunct{\chcob_2' \bullet \chcob_1'}(M' \otimes M)
\end{tikzcd}
\]
for all changes of chronologies $H_i : \chcob_i \rightarrow \chcob_i'$.
Take $m' \in M', m \in M$ with $|m'| \in \Hom_{\gradC}(b,c)$ and $|m| \in \Hom_{\gradC}(a,b)$. 
The only part of $\compMap{\chcob_2'}{\chcob_1'}(|m'|,|m|)$ that could differ from $\compMap{\chcob_2}{\chcob_1}(|m'|,|m|)$ is $\beta_1$.  Thus, we have
\begin{align*}
&\compMap{\chcob_2'}{\chcob_1'}(|m'|,|m|) 
\imath\bigl({{}_c(H_2\otimes H_1)}_a:  1_{\bar c} (\chcob_2 \otimes \chcob 1)_a \Rightarrow 1_{\bar c} (\chcob_2' \otimes\chcob') 1_a\bigr)
\\
&\quad =
\compMap{\chcob_2}{\chcob_1}(|m'|,|m|)
\imath\bigl({{}_c(H_2)}_b\otimes {{}_b(H_1)}_a : 1_{\bar c} \chcob_2 1_{b \bar b} \chcob 1_{a} \Rightarrow 1_{\bar c} \chcob_2'1_{b \bar b} \chcob_1' 1_{a} \bigr),
\end{align*}
by applying the change of chronologies on \cref{eq:diagbeta1}. 
But we also have
\begin{align*}
\imath\bigl({}_c{(H_2)}_b \otimes {}_b{(H_1)}_a \bigr)^{-1} &= \imath\bigl({}_c{(H_2)}_b\bigr)^{-1}\imath\bigl({}_b{(H_1)}_a\bigr)^{-1} =  \shiftFunct{H_2}\otimes \shiftFunct{H_1} (|m'| \otimes |m|), \\
\imath\bigl({}_c{(H_2\bullet H_1)}_a\bigr)^{-1} &= \shiftFunct{H_2 \bullet H_1}(|m'||m|),
\end{align*}
concluding the proof. 
\end{proof}

\begin{prop}
We have
\[
\shiftFunct{H'} \circ \shiftFunct{H} \cong \shiftFunct{H' \star H},
\]
and
\[
\shiftFunct{H'} \otimes \shiftFunct{H} \cong \shiftFunct{H' \circ H},
\]
for any pair of  changes of chronology $H : \chcob \Rightarrow \chcob'$ and $H' : \chcob' \Rightarrow \chcob''$.
\end{prop}

\begin{proof}
It follows immediately from $\imath({}_b(H')_a \star {}_bH_a) = \imath({}_b(H')_a)  \imath({}_bH_a) = \imath({}_b(H')_a \circ {}_bH_a)$.
\end{proof}

Take a cobordism with corners $\chcob : t \rightarrow t'$. 
Consider a minimal cobordism $\widehat \chcob : t \rightarrow t'$. We define 
\[
\chdefect{\chcob} := \left( \frac{ \eulerChar(\widehat \chcob) - \eulerChar(\chcob)}{2} , \frac{ \eulerChar(\widehat \chcob) - \eulerChar(\chcob)}{2}\right) \in \bZ\times\bZ,
\]
where $\eulerChar$ is the Euler characteristics. 

\begin{prop}\label{prop:shiftsimplification}
Let $\chcob : t \rightarrow t'$ be a cobordism. Let $\widehat \chcob : t \rightarrow t'$ be a minimal corbordism. 
There is a natural isomorphism
\[
\shiftFunct{\chcob^v} \cong \shiftFunct{\widehat \chcob^{v+\chdefect{\chcob}}}, 
\]
for all $v \in \bZ^2$.
\end{prop}

\begin{proof}
If $\chcob$ is not minimal, it means that there is a change of chronology $H : \chcob \Rightarrow \chcob'$ where $\chcob'$ is obtained from $\widehat \chcob$ by attaching tubes in the top region. These tubes locally look like:
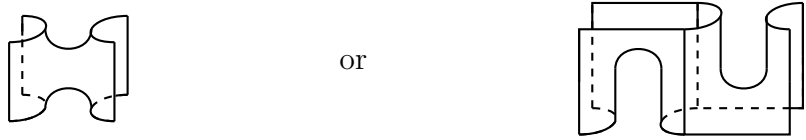
\begin{align*}
\tikzdiagc[scale=.35]{
	%
	\draw (0,0) .. controls (1.5,0) and (2,1) .. (.5,1); 
	\draw (4,0) .. controls (2.5,0) and (3,1) .. (4.5,1); 
	\filldraw[fill=white, draw=white] (.5,.55)  rectangle  (4,2);
	\draw[dashed] (0,0) .. controls (1.5,0) and (2,1) .. (.5,1); 
	\draw[dashed] (4,0) .. controls (2.5,0) and (3,1) .. (4.5,1); 
	\draw (1.375,.5) .. controls (1.375,1.5) and (3.125,1.5) .. (3.125,.5);
	%
	%
	%
	%
	%
	%
	\draw[yshift=1cm]  (0,2) .. controls (1.5,2) and (2,3) .. (.5,3); 
	\draw[yshift=1cm]  (4,2) .. controls (2.5,2) and (3,3) .. (4.5,3); 
	\draw[yshift=1cm]  (1.375,2.5) .. controls (1.375,1.5) and (3.125,1.5) .. (3.125,2.5);
	%
	%
	%
	%
	\draw (0,0) -- (0,3);
	\draw (.5,3.25) -- (.5,4);
	\draw[dashed] (.5,1) -- (.5,4);
	\draw (4,0) -- (4,3);
	\draw (4.5,1) -- (4.5,4);
}
&&
\text{or}
&&
\tikzdiagc[scale=.35]{
	\draw (0,0) .. controls (1.5,0) and (2,1) .. (.5,1); 
	\draw (4,0) .. controls (2.5,0) and (3,1) .. (4.5,1); 
	\filldraw[fill=white, draw=white] (.5,.55)  rectangle  (4.5,2);
	\draw[dashed] (0,0) .. controls (1.5,0) and (2,1) .. (.5,1); 
	\draw[dashed] (4,0) .. controls (2.5,0) and (3,1) .. (4.5,1); 
	\draw (0,0) -- (0,4);
	\draw[dashed] (.5,1) -- (.5,5);
	\draw (1.375,.5) -- (1.375,2.5);
	\draw (3.125,.5) -- (3.125,2.5);
	\draw (.5,4) -- (.5,5) -- (4.5,5) -- (4.5,4);
	%
	%
	%
	%
	\draw (1.375,2.5) .. controls (1.375,3.5) and (3.125,3.5) .. (3.125,2.5);
	%
	\draw (0,2) -- (0,4) -- (4,4);
	%
	%
	%
	%
	\draw (8,1) -- (8.5,1) -- (8.5,5);
	%
	%
	%
	\draw (5.375,2.5) .. controls (5.375,1.5) and (7.125,1.5) .. (7.125,2.5);
	\draw[dashed] (4.5,5) -- (4.5,1) -- (8.5,1) -- (8.5,3);
	\draw (4,4) -- (4,0) -- (8,0) -- (8,2);
	\draw (4,4) .. controls (5.5,4) and (6,5) .. (4.5,5); 
	\draw (8,4) .. controls (6.5,4) and (7,5) .. (8.5,5); 
	\draw (8,2) -- (8,4);
	\draw (8.5,3) -- (8.5,5);
	\draw (5.375,2.5) -- (5.375,4.5);
	\draw (7.125,2.5) -- (7.125,4.5);
}
\end{align*}
and we can always do a change of chronology so that we keep only the one on the right side.  
Note that in this case, the first saddle point is a split that spawns an isolated circle component, and the second one a merge. Therefore, the tube part  behaves like a degree shift by $(-1,-1)$ in $\bZ\times\bZ$ thanks to \cref{eq:cobcommute1}. 
Thus, we have a natural isomorphism given by the identity $\shiftFunct{\chcob'^v} \cong \shiftFunct{\widehat \chcob^{v+\chdefect{\chcob}}}$. 
Composition with the natural isomorphism $\shiftFunct{H}$ gives the desired isomorphism.  
\end{proof}

Note that, in particular, for any pair $(\chcob : t \rightarrow t', v)$, it means $\shiftFunct{\chcob^v}$ admits locally a left inverse (up to natural isomorphism) given by
\[
(\shiftFunct{\chcob^v})^{-1} := \shiftFunct{\overline \chcob^{(-v-\chdefect{\overline \chcob \circ \chcob})}},
\]
where $\overline \chcob$ is obtained from $\chcob$ by taking the mirror image with respect to the horizontal plane. Indeed, by \cref{prop:shiftsimplification}, we have
\[
(\shiftFunct{\chcob^v}^{-1})  \circ \shiftFunct{\chcob^v} \cong \shiftFunct{\un_{t}}.
\]
For example we have $\bigl(\shiftFunct{\tikzdiagh[scale=.2]{
	\draw (.5,1) -- (.5,3) -- (4.5,3) -- (4.5,1);
	\draw (0,0) .. controls (1.5,0) and (2,1) .. (.5,1); 
	\draw (4,0) .. controls (2.5,0) and (3,1) .. (4.5,1); 
	\filldraw[fill=white, draw=white] (.5,.55)  rectangle  (4,2);
	\draw[dotted] (0,0) .. controls (1.5,0) and (2,1) .. (.5,1); 
	\draw[dotted] (4,0) .. controls (2.5,0) and (3,1) .. (4.5,1); 
	\draw (1.375,.5) .. controls (1.375,1.5) and (3.125,1.5) .. (3.125,.5);
	%
	\draw (0,0) -- (0,2) -- (4,2) -- (4,0);
}^{(v_2,v_1)}}\bigr)^{-1} = \shiftFunct{{\tikzdiagh[scale=.2]{
	\draw (.5,3) -- (.5,1) -- (4.5,1) -- (4.5,3);
	\filldraw [fill=white, draw=white] (0,0) rectangle (4,2); 
	\draw (0,2) .. controls (1.5,2) and (2,3) .. (.5,3); 
	\draw (4,2) .. controls (2.5,2) and (3,3) .. (4.5,3); 
	\draw (1.375,2.5) .. controls (1.375,1.5) and (3.125,1.5) .. (3.125,2.5);
	%
	\draw (0,2) -- (0,0) -- (4,0) -- (4,2);
}}^{(1-v_2,1-v_1)}}$.

\begin{prop}\label{prop:removeloops}
Let $\chcob : t \rightarrow t'$ be a cobordism. Suppose $t$ contains a free loop $S$. Then there is a natural isomorphism
\[
\shiftFunct{\chcob^{(v_2,v_1)}} \cong \shiftFunct{{\chcob'}^{(v_2-1,v_1)}},
\quad
m \mapsto \lambda_R((1,0), \deg(1_{\bar b}W 1_a)) m,
\]
for $|m| \in \Hom_\gradC(a,b)$, and 
where $\chcob' : t \setminus S \rightarrow t'$ is given by gluing a birth under the free loop in $\chcob$. 
Similarly if $t'$ contains a free loop we have   $\shiftFunct{\chcob^{v_2,v_1}} \cong \shiftFunct{{\chcob''}^{(v_2,v_1-1)}}$ where $\chcob''$ is given by gluing a positive death, and the isomorphism is given by the identity map. 
\end{prop}

\begin{proof}
Gluing a birth change the degree of the cobordism $1_{\bar b} W 1_a$ by $(1,0)$ for all $a,b \in B^\bullet$. Because we added this element at the bottom and the compatibility maps are computed with the idea that the $\bZ^2$-grading shift $v$ is at the top of the cobordism (see \cref{eq:pictdefbeta}), we need to compensate by $\lambda_R((1,0), \deg(1_{\bar b}W 1_a))$. 
\end{proof}

\subsection{$\bZ\times\bZ$-grading shift}

We define the $\bZ\times\bZ$-grading shift functor:
\[
(-)\{v_2, v_1\} : \Mod^\gradC \rightarrow \Mod^\gradC,
\]
by setting
\[
M\{v_2, v_1\} := \bigoplus_{t \in B_\bullet^\bullet} \shiftFunct{\un_{t}^{(v_2,v_1)}}(M).
\]
It takes an element of degree $(t, (p_2,p_1))$ to an element of degree $(t,(p_2+v_2,p_1+v_1))$. 
When we write $\shiftFunct{\chcob^v} (M) \{v_1, v_2\}$, we will mean we first shift by $\chcob^v$ then by $(v_2,v_1)$, i.e. $\bigl(\shiftFunct{\chcob^v}(M)\bigr) \{v_2, v_1\}$. 
The $\bZ\times\bZ$-grading shift functors define a $\bZ\times\bZ$-grading on $\Bimod^\gradC(A_2,A_1)$.

\subsection{$\gradC$-graded commutativity} \label{sec:Gcommutativity}
We equip the $\gradC$-grading shifting 2-system $S$ with a commutativity system $\Tau$ given by all pairs $\{ ((W_2'^{v_2'},W_1'^{v_1'}),(W_2^{v_2},W_1^{v_1})) \}$ whenever there is a locally vertical change of chronology $H : W_2' \circ W_1' \Rightarrow W_2 \circ W_1$ and $v_2' = v_1, v_1' = v_2$. In that case, we put 
\[
\commutMap[ba]{W_2'^{v_2'}}{W_1'^{v_1'}}{W_2^{v_2}}{W_1^{v_1}} := \imath({}_bH_a)^{-1} \lambda_R(v_1,v_2).
\]

\begin{prop}
The commutativity system $\Tau$ defined above is compatible with the $\gradC$-grading shifting 2-system $S$, also defined above, through $\tau$.
\end{prop}

\begin{proof}
For the sake of simplicity, we assume that $v_2'=v_1'=v_2=v_1$. They can be added in the proof without much effort. 
Then, we have that \cref{eq:compCommut} holds from the commutativity of the following diagram:
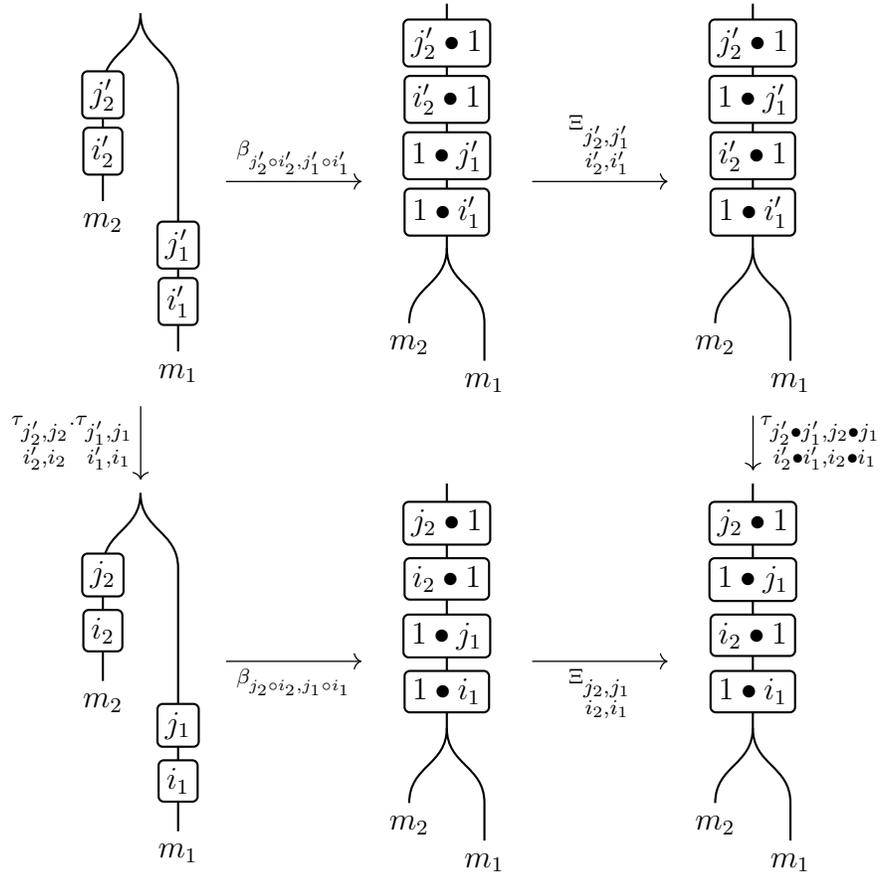
\[
\begin{tikzcd}[column sep = 10ex]
\tikzdiag{0}{
	\draw (0,1) node[below]{$m_2$}
		--
		(0,1.75) node[near end,rrect]{$i_2'$}
		--
		(0,2.5) node[near end,rrect]{$j_2'$}
		.. controls (0,3) and (.5,3) ..
		(.5,3.5) -- (.5,3.5);
	\draw (1,-1) node[below]{$m_1$}
		--
		(1,-.25) node[near end,rrect]{$i_1'$}
		-- 
		(1,.5) node[near end,rrect]{$j_1'$}
		-- 
		(1,2.5)
		.. controls (1,3) and (.5,3) ..
		(.5,3.5);
}
\ar{r}{\compMap{j_2' \circ i_2'}{j_1' \circ i_1'}}
\ar[swap]{d}{\commutMap{j'_2}{i'_2}{j_2}{i_2} . \commutMap{j'_1}{i'_1}{j_1}{i_1}}
&
 \tikzdiag{0}{
	\draw (0,-1.5) node[below]{$m_2$}
		.. controls (0,-1) and (.5,-1) ..
		(.5,-.5)  
		--
		(.5,.25 )node[midway,rrect]{$1 \bullet i_1'$}
		--
		(.5,1) node[midway,rrect]{$1 \bullet j_1'$}
		--
		(.5,1.75)  node[midway,rrect]{$i_2' \bullet 1$}
		--
		(.5,2.5) node[midway,rrect]{$j_2'\bullet 1$}
		-- 
		(.5,2.75);
	\draw (1,-2) node[below]{$m_1$} 
		--
		(1,-1.5)
		.. controls (1,-1) and (.5,-1) ..
		(.5,-.5);
}
\ar{r}{\interCompMap{j_2'}{i_2'}{j_1'}{i_1'}}
&
 \tikzdiag{0}{
	\draw (0,-1.5) node[below]{$m_2$}
		.. controls (0,-1) and (.5,-1) ..
		(.5,-.5)  
		--
		(.5,.25 )node[midway,rrect]{$1 \bullet i_1'$}
		--
		(.5,1) node[midway,rrect]{$i_2' \bullet 1$}
		--
		(.5,1.75)  node[midway,rrect]{$1 \bullet j_1'$}
		--
		(.5,2.5) node[midway,rrect]{$j_2'\bullet 1$}
		-- 
		(.5,2.75);
	\draw (1,-2) node[below]{$m_1$} 
		--
		(1,-1.5)
		.. controls (1,-1) and (.5,-1) ..
		(.5,-.5);
}
\ar{d}{\commutMap{j'_2 \bullet j_1'}{i'_2 \bullet i'_1}{j_2 \bullet j_1}{i_2 \bullet i_1}}
\\
\tikzdiag{0}{
	\draw (0,1) node[below]{$m_2$}
		--
		(0,1.75) node[near end,rrect]{$i_2$}
		--
		(0,2.5) node[near end,rrect]{$j_2$}
		.. controls (0,3) and (.5,3) ..
		(.5,3.5) -- (.5,3.5);
	\draw (1,-1) node[below]{$m_1$}
		--
		(1,-.25) node[near end,rrect]{$i_1$}
		-- 
		(1,.5) node[near end,rrect]{$j_1$}
		-- 
		(1,2.5)
		.. controls (1,3) and (.5,3) ..
		(.5,3.5);
}
\ar[swap]{r}{\compMap{j_2 \circ i_2}{j_1 \circ i_1}}
&
 \tikzdiag{0}{
	\draw (0,-1.5) node[below]{$m_2$}
		.. controls (0,-1) and (.5,-1) ..
		(.5,-.5)  
		--
		(.5,.25 )node[midway,rrect]{$1 \bullet i_1$}
		--
		(.5,1) node[midway,rrect]{$1 \bullet j_1$}
		--
		(.5,1.75)  node[midway,rrect]{$i_2 \bullet 1$}
		--
		(.5,2.5) node[midway,rrect]{$j_2\bullet 1$}
		-- 
		(.5,2.75);
	\draw (1,-2) node[below]{$m_1$} 
		--
		(1,-1.5)
		.. controls (1,-1) and (.5,-1) ..
		(.5,-.5);
}
\ar[swap]{r}{\interCompMap{j_2}{i_2}{j_1}{i_1}}
&
 \tikzdiag{0}{
	\draw (0,-1.5) node[below]{$m_2$}
		.. controls (0,-1) and (.5,-1) ..
		(.5,-.5)  
		--
		(.5,.25 )node[midway,rrect]{$1 \bullet i_1$}
		--
		(.5,1) node[midway,rrect]{$i_2 \bullet 1$}
		--
		(.5,1.75)  node[midway,rrect]{$1 \bullet j_1$}
		--
		(.5,2.5) node[midway,rrect]{$j_2\bullet 1$}
		-- 
		(.5,2.75);
	\draw (1,-2) node[below]{$m_1$} 
		--
		(1,-1.5)
		.. controls (1,-1) and (.5,-1) ..
		(.5,-.5);
}
\end{tikzcd}
\]
where a box with label $i$ represent the cobordism $W_i$. The commutativity comes from the fact the two path are given by applying $\imath$ on locally vertical changes of chronology with same outputs.
Also, \cref{eq:compCommut2} holds thanks to \cref{eq:cobcommute1} and \cref{eq:cobcommute2}. 
\end{proof}

\subsection{Symmetric structure}

The monoidal category $\Mod^\gradC$ admits a symmetry
\[
\tau_{M,N} : M \otimes N \rightarrow M \otimes N, \quad \tau_{M,N}(m \otimes n) := \lambda_R(|m|_R,|n|_R)  n \otimes m,
\]
where $|m|_R = p \in \bZ \times \bZ$ if $|m| = (t,p) \in \Hom_\gradC$. 

\begin{prop}
The symmetry $\tau_{M,N}$ turns $\Mod^\gradC$ into a symmetric monoidal category.
\end{prop}

\begin{proof}
The proof use similar arguments as in \cref{prop:gradCassoc}, picturing the symmetry as
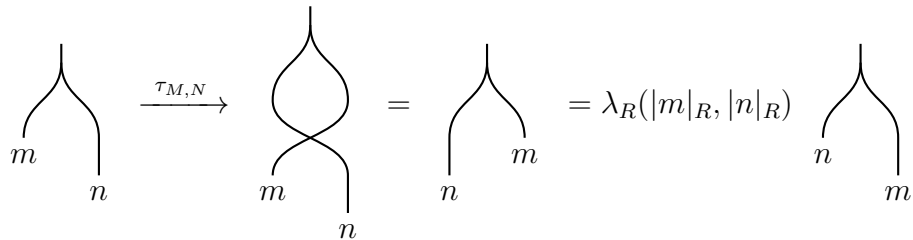
\[
 \tikzdiag{0}{
	\draw (0,-1.5) node[below]{$m$}
		.. controls (0,-1) and (.5,-1) ..
		(.5,-.5)  
		--
		(.5,-.25);
	\draw (1,-2) node[below]{$n$} 
		--
		(1,-1.5)
		.. controls (1,-1) and (.5,-1) ..
		(.5,-.5);
}
\ \xrightarrow{\ \tau_{M,N}\ }\ 
 \tikzdiag{0}{
	\draw (1,-3) node[below]{$n$}
		--
		(1,-2.5)
			.. controls (1,-2) and (0,-2) .. 
		(0,-1.5)
		.. controls (0,-1) and (.5,-1) ..
		(.5,-.5)  
		--
		(.5,-.25);
	\draw (0,-2.5) node[below]{$m$} 
			.. controls (0,-2) and (1,-2) ..
		(1,-1.5)
		.. controls (1,-1) and (.5,-1) ..
		(.5,-.5);
}
\ = \ 
\tikzdiag{0}{
	\draw (0,-2) node[below]{$n$}
		--
		(0,-1.5)
		.. controls (0,-1) and (.5,-1) ..
		(.5,-.5)  
		--
		(.5,-.25);
	\draw (1,-1.5) node[below]{$m$} 
		--
		(1,-1.5)
		.. controls (1,-1) and (.5,-1) ..
		(.5,-.5);
}
\ = \lambda_R(|m|_R, |n|_R) \ 
 \tikzdiag{0}{
	\draw (0,-1.5) node[below]{$n$}
		.. controls (0,-1) and (.5,-1) ..
		(.5,-.5)  
		--
		(.5,-.25);
	\draw (1,-2) node[below]{$m$} 
		--
		(1,-1.5)
		.. controls (1,-1) and (.5,-1) ..
		(.5,-.5);
}
\]
We leave the details to the reader.
\end{proof}

Whenever we have a symmetric monoidal category $(\cM, \otimes, \tau)$ , it makes sense to define the notion of center of a monoid object $A \in \cM$ (and it can be made more general, see \cite{centeralgebra}). Consider $\cZ(A)$ the category of morphisms $y : Y \rightarrow A$ in $\cM$ such that the diagram
\[
\begin{tikzcd}[row sep=2ex]
Y \otimes A 
\ar{r}{y \otimes 1}
\ar[swap]{dd}{\tau_{Y,A}}
& 
A \otimes A 
\ar{dr}{\mu}
&
\\
&& A
\\
A \otimes Y 
\ar[swap]{r}{1 \otimes y}
&
 A \otimes A
 \ar[swap]{ur}{\mu}
 &
\end{tikzcd}
\]
commutes. The \emph{center} of $A$ is the terminal object $Z(A) \rightarrow A \in \cZ(A)$. Note that, by universal property of the terminal object in $\cZ(A)$, $Z(A)$ is a commutative monoid object in $\cM$. 
Concretely, in the case of $\Mod^\cC$, it means the center of a $\gradC$-graded algebra $A$ is given by
\[
Z(A) := \{ z \in A | \mu(z \otimes x) =  \mu(\tau_{Z(A),A}(z\otimes x)), \forall x \in A \},
\]
as usual. In the case of $\gradC$, it translates to $
Z(A) := \{ z \in A | zx = \lambda_R(|z|_R, |x|_R) xz, \forall x \in A \}.$




\section{Tangle invariant}\label{sec:tanglehomology}

We first show that $H^n$ is a $\gradC$-graded algebra. Then, we explain how to adapt the construction from~\cite{khovanovHn} to our case. 

\subsection{$\gradC$-graded arc algebra and arc bimodules}

Recall the space $\tqft(t)$ from~\cref{sec:arcspaces}. We can think of it as an object in  $\Mod^\gradC$ where 
\[
\deg_{\gradC}(m) := \bigl( \red{t}, \deg_R(m) \bigr) \in \Hom_\gradC(a,b),
\]
for $m \in {_b}\tqft(t){_a}$.
Then, we have that the composition map $\mu[t',t]$ preserves the $\gradC$-grading. 

\begin{lem}\label{lem:muassoc}
We have
\[
\mu[t''t',t]\bigl( \mu[t'',t'](z,y) , x \bigr) = \assoc(|z|,|y|,|x|) \mu[t'',t't]\bigl(z, \mu[t',t](y,x)\bigr),
\]
for all $x \in \tqft(t), y \in \tqft(t'), z \in \tqft(t'')$. 
\end{lem}

\begin{proof}
This is immediate by construction of $\assoc$ in $\gradC$ and the fact that $\mu[t',t]$ is constructed using $\tqft(W_{cba}(t',t))$. 
\end{proof}

\begin{prop}\label{prop:Hnalgebra}
 $H^n$  is a unital, associative $\gradC$-graded $R$-algebra with units given by ${_a}1_a$. 
\end{prop}

\begin{proof}
The associativity is givenb by \cref{lem:muassoc}.  

For the unitality, we compute for all  $m \in \tqft(t)$ with $|m| = (t,p) \in \Hom_\gradC(a,b)$ that
\begin{align*}
\assoc(\id_b,\id_b,|c|) &= X^{|b|} X^{|b|} = 1, \\
\assoc(|m|,\id_a,\id_a) &= X^{|a|} \lambda_R((-|a|,0),p),
\end{align*} 
since the cobordisms involved consist only of merges, recalling that $\id_a = (1_{|a|}, (|a|,0))$. 
Furthermore, since all saddle points are oriented upward, we obtain that 
\begin{align*}
\mu[1_{|b|},t]({_b}1_b,m) &= m, 
\\
\mu[t,1_{|a|}]m,({_a}1_a) &= \lambda_R(p,(|a|,0)) X^{|a|}, 
\end{align*}
thanks to \cref{eq:chcobrelbirthmerge}. 
This is best explained by the following pictures:
\[
 \tikzdiag{0}{
	\draw (0,-1.5) node[below]{$\id_b$}
		.. controls (0,-1) and (.5,-1) ..
		(.5,-.5)  
		--
		(.5,0);
	\draw (1,-2) node[below]{$(t,p)$} 
		--
		(1,-1.5)
		.. controls (1,-1) and (.5,-1) ..
		(.5,-.5);
	\draw[<-] (.25,-1.25) -- (.75,-1.25);
}
\ = \ 
 \tikzdiag{0}{
 	\draw (0,-2) node[below]{$(t,p)$}  -- (0,0);
 }
\]
and
\[ 
\tikzdiag{0}{
	\draw (0,-1.5) node[below]{$(t,p)$}
		.. controls (0,-1) and (.5,-1) ..
		(.5,-.5)  
		--
		(.5,0);
	\draw (1,-2) node[below]{$\id_a$} 
		--
		(1,-1.5)
		.. controls (1,-1) and (.5,-1) ..
		(.5,-.5);
	\draw[<-] (.25,-1.25) -- (.75,-1.25);
}
\ = \ 
\lambda_R(p, (|a|,0))
\tikzdiag{0}{
	\draw (0,-2) node[below]{$(t,p)$}
		--
		(0,-1.5)
		.. controls (0,-1) and (.5,-1) ..
		(.5,-.5)  
		--
		(.5,0);
	\draw (1,-1.5) node[below]{$\id_a$} 
		--
		(1,-1.5)
		.. controls (1,-1) and (.5,-1) ..
		(.5,-.5);
	\draw[<-] (.25,-1.25) -- (.75,-1.25);
}
\ = \ 
X^{|a|}
\lambda_R(p, (|a|,0))
 \tikzdiag{0}{
 	\draw (0,-2) node[below]{$(t,p)$}  -- (0,0);
 }
\]
We conclude that $H^n$ is unital. 
\end{proof}

\begin{rem}
When specializing $X=Y=Z=1$ or $X=Z=1$ and $Y = -1$, the algebra $H^n$ coincides with the usual Khovanov arc algebra~\cite{khovanovHn} or the odd one~\cite{naissevaz}. 
Moreover, the center $Z(H^n)$ then coincides with usual notion of center or the notion of odd center from~\cite{naissevaz}. 
\end{rem}

For $t \in B_n^m$, the composition maps $\mu[1_m, t]$ and $\mu[t,1_n]$ turn $\tqft(t)$ into a $H^m$-$H^n$-bimodule in $\Mod^\gradC$, by the same arguments as in the proof of \cref{prop:Hnalgebra}.
Moreover, any cobordism with corners $\chcob : t \rightarrow t'$ induces a graded map 
\[
\shiftFunct{\chcob}(\tqft(t)) \xrightarrow{\tqft(\chcob)} \tqft(t').
\]
 Let $\otimes_n$ denotes the ($\gradC$-graded) tensor product $\otimes_{H^n}$. 

\begin{prop}\label{prop:projbimod}
The $H^m$-$H^n$-bimodule $\tqft(t)$ is projective as left $H^m$-module and as right $H^n$-module.
\end{prop}

\begin{proof}
The proof is essentially the same as in \cite[Proposition 3]{khovanovHn}. We leave the details to the reader. 
\end{proof}

\begin{prop}\label{prop:flattanglecompiso}
For $t' \in B_n^{m'}$ and $t \in B_m^n$ we have an isomorphism
\[
\tqft(t') \otimes_n \tqft(t) \cong \tqft(t't),
\]
induced by $\mu[t',t]$. 
\end{prop}

\begin{proof}
We first note that $\mu[t',t] : \tqft(t') \otimes_R \tqft(t)$ induces a map 
\[
\tqft(t') \otimes_n \tqft(t) \rightarrow \tqft(t't),
\]
by the universal property of the coequalizer, since by \cref{lem:muassoc} we have
\[
\mu[t',t] (m' \cdot x, m) = \assoc(|m'|,|x|,|m|) \mu[t',t](m', x \cdot m).
\]
The remaining of the proof is exactly the same as in~\cite[Theorem 1]{khovanovHn}. 
\end{proof}

\begin{lem}\label{lem:flattanglecompisoHex}
The following diagram
\[
\begin{tikzcd}
\bigl( \tqft(t'') \otimes_H \tqft(t') \bigr) \otimes_H \tqft(t) 
\ar{r}{\assoc}
\ar[swap]{d}{\mu[t'',t'] \otimes 1}
&
\tqft(t'') \otimes_H \bigl( \tqft(t') \otimes_H \tqft(t) \bigr)
\ar{d}{1 \otimes \mu[t',t]}
\\
\tqft(t'' t') \otimes_H \tqft(t)
\ar[swap]{d}{\mu[t''t', t]}
&
\tqft(t'') \otimes_H \tqft(t' t)
\ar{d}{\mu[t'',t't]}
\\
\tqft(t'' t' t)
\ar[equals]{r}
&
\tqft(t'' t' t)
\end{tikzcd}
\]
commutes for all $t'',t',t \in B_\bullet^\bullet$.
\end{lem}

\begin{proof}
Immediate by definition of $\assoc$ and $\mu[t',t]$. 
\end{proof}

\subsection{Tangle resolution}\label{sec:tangleres}

An $(n,m)$-tangle $T$ is a tangle in $\bR^2 \times I$ connecting $2m$ points on the bottom to $2n$-points on the top. A plane diagram of such a tangle is a generic projection of the tangle on $\bR \times I$, marking the order of superposition in the crossings. 

Given a crossing in a plane diagram of a tangle, one can resolve it in to possible ways:
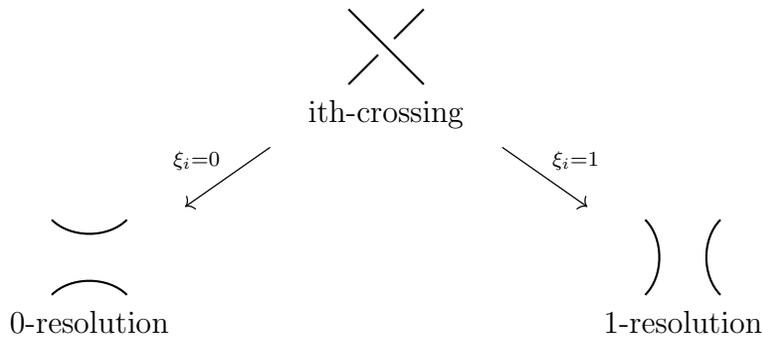
\[
\begin{tikzcd}
&
\ 
\tikzdiagc
{
	\draw(0,0) -- (1,1);
	\fill[fill=white] (.5,.5) circle (.15);
	\draw(0,1) -- (1,0);
	\node at (.5,-.5) {ith-crossing};
}
\ 
\ar[swap]{dl}{\xi_i = 0}
\ar{dr}{\xi_i = 1}
& 
\\
\tikzdiagc
{
	\draw (2.5,-1.5) .. controls (2.75,-1.75) and (3.25,-1.75) .. (3.5,-1.5);
	\draw (2.5,-2.5) .. controls  (2.75,-2.25) and (3.25,-2.25) .. (3.5,-2.5);
	\node at (3,-3) {0-resolution};
}
&&
\tikzdiagc
{
	\draw (-1.5,-1.5) .. controls (-1.75,-1.75) and (-1.75,-2.25) .. (-1.5,-2.5);
	\draw (-2.5,-1.5) .. controls (-2.25,-1.75) and (-2.25,-2.25) ..  (-2.5,-2.5);
	\node at (-2,-3) {1-resolution};
}
\end{tikzcd}
\]
A resolution of a plan diagram is given by resolving all its crossings.
It yields a flat tangle in $\tangleSpace_m^n$.  
Suppose $T$ has $k$ crossings, that we order by reading $T$ from bottom to top (we can suppose there are no pair of crossings at the same height). 
For each $\xi = (\xi_k, \dots, \xi_1) \in \{0,1\}^k$, we write $T_\xi \in \tangleSpace_m^n$ for the resolution of $T$ given by resolving the $i$-th crossing as given by $\xi_i$. We also write $\und 0 := (0, \dots, 0) \in \{0,1\}^k$ and $ \und 1 := (1, \dots, 1) \in \{0,1\}^k$.

\smallskip

We can suppose all crossings in $T$ are pivoted to look like above (or mirror), and we associate to them an arrow pointing upward or leftward:
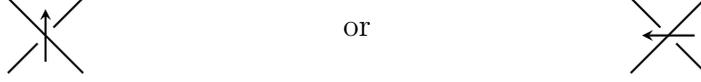
\begin{align*}
\tikzdiagc
{
	\draw(0,0) -- (1,1);
	\fill[fill=white] (.5,.5) circle (.15);
	\draw(0,1) -- (1,0);
	\draw[->] (.5,.15) -- (.5,.85);
}
&&
\text{or}
&&
\tikzdiagc
{
	\draw(0,1) -- (1,0);
	\fill[fill=white] (.5,.5) circle (.15);
	\draw(0,0) -- (1,1);
	%
	\draw[<-] (.15,.5) -- (.85,.5);
}
\end{align*}
For each $\xi_i = 0$ in $\xi$, we write $\xi+i := (\xi_k, \dots, \xi_{i+1}, 1, \xi_{i-1}, \dots, \xi_1)$. We construct a chronological cobordism $\chcob_{\xi,i} : T_\xi \rightarrow T_{\xi+i}$ by putting a saddle above the $0$-resolution of the $i$-th crossing, with orientation given by the arrow on the crossing. 
This defines a map:
\[
\shiftFunct{\chcob_{\xi,i}} \bigl( \tqft(T_\xi) \bigr)  \xrightarrow{\tqft(\chcob_{\xi,i})} \tqft(T_{\xi+i}).
\]
Note that we need to shift $\tqft(T_\xi)$ in order to get a graded map of bimodules. 

\smallskip

Let $|\xi| = \sum_{i=0}^k \xi_i$ be the weight of $\xi$. 
For each $\xi$, we define recursively a cobordism
\[
\chcob_\xi :=  \chcob_{\xi+\ell} \circ  \chcob_{\xi,\ell} ,
\]
where $\ell$ is the lowest integer such that $\xi_\ell = 0$, and $\chcob_{\und 1}$ is the identity. 
In other words, there is a unique (non-chronological) cobordism from $T_\xi$ to $T_\und 1$ given changing all $0$ to $1$, and we give it a chronology by stacking the saddles in the order given by reading 
the crossings in $T$ from bottom to top. 
\smallskip

We write 
\begin{align*}
C(T)_r &:= \bigoplus_{|\xi| = r} C(T)_\xi[r],
&
C(T)_\xi &:= \shiftFunct{\chcob_\xi} \bigl( \tqft(T_{\xi})\bigr),
\end{align*}
where we recall $[r] = [1]^r$ is a shift up by $r$ units in the homological degree. 
For each $\xi_j = 0$, we consider the change of chronology
\[
H_{\xi,j} :  \chcob_{\xi} \Rightarrow  \chcob_{\xi+j} \circ \chcob_{\xi,j},
\]
that consists in taking the saddle above the $j$-th crossing and pushing it to the bottom. 
This allows us to build a graded map of bimodules
\[
d_{\xi,j} :=  \tqft(\chcob_{\xi,j}) \circ \shiftFunct{H_{\xi,j}}\bigl(\tqft(T_{\xi})\bigr)  : C(T)_\xi \rightarrow C(T)_{\xi+j}. 
\]

\begin{lem}
The diagram 
\[
\begin{tikzcd}
&
C(T)_{\xi+i}
 \ar{dr}{d_{\xi+i, j}}&
\\
C(T)_{\xi}
\ar{ur}{d_{\xi,i}}
\ar[swap]{dr}{d_{\xi,j}}
&&
C(T)_{\xi+i+j}
\\
&
C(T)_{\xi+j} \ar[swap]{ur}{d_{\xi+j,i}}&
\end{tikzcd}
\]
commutes 
for all $\xi$ and $i,j$ such that $\xi_i=\xi_j=0$. 
\end{lem}

\begin{proof}
We have
\begin{align*}
d_{\xi+i,j} \circ d_{\xi,i} &=  \tqft(\chcob_{\xi+i,j} \circ \chcob_{\xi,i}) \circ \shiftFunct{H}(T_\xi),
\end{align*}
where
$
H : \chcob_{\xi} \Rightarrow \chcob_{\xi+i+j} \circ \chcob_{\xi+i,j} \circ \chcob_{\xi, i}
$
is a locally vertical change of chronology. 
Moreover, we have
\[
\tqft(\chcob_{\xi+i,j} \circ \chcob_{\xi,i})  = \tqft(\chcob_{\xi+j,i} \circ \chcob_{\xi,j})  \circ \shiftFunct{H'}(T_\xi),
\]
where
$
H' :  \chcob_{\xi+i+j} \circ \chcob_{\xi+j,i} \circ \chcob_{\xi, i} \Rightarrow \chcob_{\xi+i+j} \circ \chcob_{\xi+j,i} \circ \chcob_{\xi,j}.
$
Thus, 
\begin{align*}
d_{\xi+i,j} \circ d_{\xi,i} &= \tqft(\chcob_{\xi+j,i} \circ \chcob_{\xi,j}) \circ \shiftFunct{H'}(T_\xi) \circ \shiftFunct{H}(T_\xi)
\\
&= 
d_{\xi+j,i} \circ d_{\xi,j} \circ \shiftFunct{H''}(T_\xi) \circ \shiftFunct{H'}(T_\xi) \circ \shiftFunct{H}(T_\xi),
\end{align*}
for $H'' : \chcob_{\xi+i+j} \circ \chcob_{\xi+i,j} \circ \chcob_{\xi, i} \Rightarrow \chcob_{\xi}$. Since $H'' \star H' \star H : \chcob_\xi \Rightarrow \chcob_\xi$, by \cref{prop:locvertchange} it is homotopic to the identity change of chronology, and we obtain 
\[
 \shiftFunct{H''}(T_\xi) \circ \shiftFunct{H'}(T_\xi) \circ \shiftFunct{H}(T_\xi) = \shiftFunct{H'' \star H' \star H}(T_\xi) = \shiftFunct{\id_{\chcob_\xi}}(T_\xi) = \id,
 \]
  so that $d_{\xi+i,j} \circ d_{\xi,i} = d_{\xi+j,i} \circ d_{\xi,j}$.
\end{proof}

\begin{rem}
This is the first difference that appears between the usual construction in \cite{ORS, putyra14} and our framework of $\gradC$-graded modules. Indeed, in the references, they obtain maps that are commutative only up to some sign (or scalar) assignment, that is shown afterward to always exists. With our approach, the grading shift already assign the correct scalar to each map in order to have a commutative cube of resolutions. 
\end{rem}

Let $p(\xi,j) := \#\{ k \geq \ell > j | \xi_\ell = 1\}$ for all $\xi$ and $1 \leq j \leq k$.  
We construct a map
\[
d_r  : T_r \rightarrow T_{r+1},
\]
given by
\[
d_r|_{C(T)_{\xi}} :C(T)_{\xi} \rightarrow \bigoplus_{\{j | \xi_j = 0\}} C(T)_{\xi+j} \subset T_{r+1}, \qquad  d_r|_{C(T)_{\xi}}  := \sum_{\{j | \xi_j = 0\}} (-1)^{p(\xi,j)} d_{\xi,j},
\]
for all $|\xi| = r$.  Note that $d_{r+1} \circ d_r = 0$. We write
\[
\tqft(T) := \bigl(\bigoplus_{r=0}^{k} C(T)_r, d = \sum_r d_r\bigr) \in \Bimoddg^\gradC(H^n, H^m).
\]

\begin{rem}
Note that if $T$ is made of a single crossing, then 
\[
\tqft(T) \cong \cone\left(\shiftFunct{\chcob} \bigl(\tqft(T_{0})\bigr) \xrightarrow{\tqft(\chcob)} \tqft(T_1)\right)[1],
\]
where $\chcob$ is the saddle from the $0$-resolution of the crossing to its $1$-resolution. 
\end{rem}

\subsection{Composition of tangles}

The dg-bimodule $\tqft(T)$ is nicely behaved with respect to composition of tangles. 

\begin{prop}\label{prop:comptangles}
Let $T'$ be an $(m',n)$-tangle and $T$ an $(n,m)$-tangle. 
There is an isomorphism
\[
\tqft(T') \otimes_{n} \tqft(T) \xrightarrow{\simeq} \tqft(T'T),
\] 
induced by the composition maps $\mu[T'_{\xi'},T_\xi]$. 
\end{prop}

\begin{proof}
Let $k$ and $k'$ be the number of crossings in $T$ and $T'$ respectively. Take $\xi \in \{0,1\}^k$ and $\xi' \in \{0,1\}^{k'}$. Consider the sequence of isomorphisms
\begin{align*}
C(T')_{\xi'} \otimes_n C(T)_{\xi} 
&= 
\shiftFunct{\chcob_{\xi'}}\bigl(  \tqft(T'_{\xi'})\bigr) \otimes_n \shiftFunct{\chcob_{\xi}}\bigl(  \tqft(T_{\xi})\bigr) 
\xrightarrow{\compMap{\chcob_{\xi'}}{\chcob_{\xi}}}
 \shiftFunct{\chcob_{\xi'} \bullet \chcob_{\xi}}\bigl(\tqft(T'_{\xi'}) \otimes_n \tqft(T_\xi)\bigr) 
 \\
&\xrightarrow{{\mu[T'_{\xi'}, T_\xi]}}
 \shiftFunct{\chcob_{\xi'} \bullet \chcob_\xi}\bigl(\tqft(T'_{\xi'}T_\xi)\bigr) 
 \cong 
 \shiftFunct{\chcob_{\xi' \sqcup \xi}}\bigl(\tqft(T'_{\xi'}T_\xi)\bigr) 
  = C(T'T)_{\xi' \sqcup \xi},
\end{align*}
where the last isomorphism comes from the fact that $\chcob_{\xi'}\bullet \chcob_\xi \cong \chcob_{\xi' \sqcup \xi}$. 
It is an isomorphism thanks to \cref{prop:flattanglecompiso}. 
Thus, we only need to show that both diagrams
\begin{align}\label{eq:diagscomptangles}
\begin{tikzcd}[ampersand replacement=\&]
C(T')_{\xi'} \otimes_n C(T)_\xi 
\ar{r}{\mu}
\ar[swap]{d}{d_{\xi',i} \otimes 1}
\& 
C(T'T)_{\xi'\sqcup\xi}
\ar{d}{d_{\xi' \sqcup \xi, i+k}}
\\
C(T')_{\xi'+i} \otimes_n C(T)_\xi 
\ar[swap]{r}{\mu}
\& 
C(T'T)_{\xi'\sqcup\xi+(i+k)}
\end{tikzcd}
&&
\begin{tikzcd}[ampersand replacement=\&]
C(T')_{\xi'} \otimes_n C(T)_\xi 
\ar{r}{\mu}
\ar[swap]{d} {1 \otimes d_{\xi,j}}
\& 
C(T'T)_{\xi'\sqcup\xi}
\ar{d} {d_{\xi' \sqcup \xi, j}}
\\
C(T')_{\xi'} \otimes_n C(T)_{\xi+j} 
\ar[swap]{r}{\mu}
\& 
C(T'T)_{\xi'\sqcup\xi+j}
\end{tikzcd}
\end{align}
commute.

In this proof, we write:
\begin{align*}
C_0& := \tqft(T_\xi), 
&
C'_0 &:= \tqft(T'_{\xi'}), 
&
C_0'' &:= \tqft(T'_{\xi'} T_\xi),
&
C_1' &:= \tqft(T_{\xi'+i}),
&
C_1'' &:= \tqft(T'_{\xi'+i}T_\xi),
\\
\shiftFunct{0} &:= \shiftFunct{\chcob_\xi},
&
\shiftFunct{0}' &:= \shiftFunct{\chcob_{\xi'}},
&
\shiftFunct{0}'' &:= \shiftFunct{\chcob_{\xi' \sqcup \xi}},
&
\shiftFunct{1}' &:= \shiftFunct{\chcob_{\xi'+i}},
&
\shiftFunct{1}'' &:=\shiftFunct{\chcob_{(\xi' + i) \sqcup \xi}},
\end{align*}
and also:
\begin{align*}
\shiftFunct{i}' &:= \shiftFunct{\chcob_{\xi',i}},
&
\tqft_i' &:= \tqft(\chcob_{\xi',i}),
&
\tqft_{i+k}'' &:= \tqft(\chcob_{\xi' \sqcup \xi, i+k}).
\end{align*}
We also write $\shiftFunct{1}' \bullet \shiftFunct{0} := \shiftFunct{\chcob_{\xi'+i} \bullet \chcob_\xi}$ etc. 
Note that $\shiftFunct{0}''  = \shiftFunct{0}'\bullet \shiftFunct{0}$ and $\shiftFunct{1}'' = \shiftFunct{1}' \bullet \shiftFunct{0}$. 

Then, the diagram on the left of \cref{eq:diagscomptangles} factorizes as:
\[
\begin{tikzcd}[row sep=8ex, column sep = 7ex]
\shiftFunct{0}'(C'_0) \otimes_n \shiftFunct{0}(C_0)
\ar{r}{\shiftFunct{H_{\xi',i}} \otimes 1}
\ar[swap]{dd}{\compMap{\chcob_{\xi'}}{\chcob_\xi}}
&
\shiftFunct{1}' \circ \shiftFunct{i}' ( C_0') \otimes_n  \shiftFunct{0}(C_0)
\ar{r}{\tqft_i' \otimes 1}
\ar[swap]{d}{\compMap{\chcob_{\xi'+i}}{\chcob_{\xi}}}
\ar[phantom]{dr}{f_{12}}
&
\shiftFunct{1}'(C_1') \otimes_n \shiftFunct{0}(C_0)
\ar{d}{\compMap{\chcob_{\xi'+i}}{\chcob_\xi}}
\\
{} \ar[phantom]{r}{f_{11}}
&
\shiftFunct{1}' \bullet \shiftFunct{0}\bigl( \shiftFunct{i}'(C_0') \otimes_n C_0) \bigr)
\ar{r}{\tqft_i' \otimes 1}
\ar[swap]{d}{\compMap{\chcob_{\xi',i}}{1}} 
\ar[phantom]{dr}{f_{22}}
&
\shiftFunct{1}' \bullet \shiftFunct{0}\bigl( C_1' \otimes_n C_0 \bigr)
\ar[equals]{d}
\\
\shiftFunct{0}' \bullet \shiftFunct{0}(C_0' \otimes_n C_0)
\ar{r}{\shiftFunct{H}}
\ar[equals]{d}
\ar[phantom]{dr}{f_{31}}
&
(\shiftFunct{1}' \bullet \shiftFunct{0})\circ(\shiftFunct{i}' \bullet \shiftFunct{1_{T_\xi}}) \bigl(C_0' \otimes_n C_0 \bigr)
\ar{r}{\tqft_i' \otimes 1}
\ar[equals]{d}
\ar[phantom]{dr}{f_{32}}
&
\shiftFunct{1}' \bullet \shiftFunct{0}\bigl( C_1' \otimes_n C_0 \bigr)
\ar[equals]{d}
\\
\shiftFunct{0}''(C_0' \otimes_n C_0)
\ar{r}{\shiftFunct{H}}
\ar[swap]{d}{\mu}
\ar[phantom]{dr}{f_{41}}
&
\shiftFunct{1}'' \circ (\shiftFunct{i}' \bullet  \shiftFunct{1_{T_\xi}})\bigl(C_0' \otimes_n C_0 \bigr)
\ar{r}{\tqft_i' \otimes 1}
\ar[swap]{d}{\mu}
\ar[phantom]{dr}{f_{42}}
&
\shiftFunct{1}''\bigl(C_1' \otimes _n C_0\bigr)
\ar{d}{\mu}
\\
\shiftFunct{0}''(C_0'')
\ar{r}{\shiftFunct{H}}
&
\shiftFunct{1}''\circ(\shiftFunct{i}' \bullet  \shiftFunct{1_{T_\xi}})(C_0'')
\ar{r}{\tqft_{i+k}''}
&
\shiftFunct{1}''(C_1'')
\end{tikzcd}
\]
where $H : \chcob_{\xi'} \bullet \chcob_\xi \Rightarrow (\chcob_{\xi'+i}\bullet\chcob_{\xi})\circ(\chcob_{\xi',i}\bullet 1)$ is a locally vertical change of chronology. We claim the exterior square commutes. 
We can suppose we are applying the maps on some element $x' \otimes x$ with $|x'| = (T'_{\xi'}, p') \in \Hom_\gradC(b,c)$ and $|x| =  (T_\xi, p) \in \Hom_{\gradC}(a,b)$. We compute each face of the diagram:
\begin{itemize}
\item For the face $f_{11}$, we first observe that the contribution from $\beta_1$ in the definition of $\compMap{\chcob_{\xi'}}{\chcob_\xi}$ together with $\shiftFunct{H}$ computes 
$\imath\bigl( \chcob_{cba}({T'_{\und 1},T_{\und 1}}) \circ (1_{\bar c}\chcob_{\xi'}1_b \otimes  1_{\bar b}\chcob_{\xi} 1_a) \Rightarrow  (1_{\bar c}\chcob_{\xi'+i} \bullet \chcob_{\xi} 1_a) \circ (1_{\bar c}\chcob_{\xi',i} \bullet 1_{T_\xi} 1_a) \circ \chcob_{cba}({T'_{\xi'}, T_\xi})\bigr))$. 
Moreover, the contribution from $\beta_1$ in $\compMap{\chcob_{\xi'+i}}{\chcob_\xi}$ and in $\compMap{\chcob_{\xi',i}}{1}$ together with $\shiftFunct{H_{\xi',i}}$ compute the same change of chronology. Therefore only remains the contribution from $\beta_2, \beta_2'$ and $\beta_2''$. Both $\beta_2$ and $\beta_2'$ are zero in all cases, and $\beta_2''$ is zero for  $\compMap{\chcob_{\xi',i}}{1}$. Then, we observe that the contribution of $\beta_2''$ is the same in both $\compMap{\chcob_{\xi'+i}}{\chcob_\xi}$ and $\compMap{\chcob_{\xi'}}{\chcob_\xi}$, being $\lambda_R(p', \deg(1_{\bar b}\chcob_{\xi} 1_a)$. Thus, $f_{11}$ commutes. 
\item The face $f_{12}$ commutes by functoriality of the bifunctor $\compMap{\chcob_{\xi'+i}}{\chcob_\xi}$. 
\item The face $f_{22}$ commutes up to the factor 
\begin{align*}
&\compMap{\chcob_{\xi',i}}{1}(|x'|,|x|)^{-1} \\
&= \imath\bigl(  (1_{\bar c}\chcob_{\xi',i} \bullet 1_{T_\xi} 1_a) \circ \chcob_{cba}(T'_{\xi'},T_\xi) \Rightarrow \chcob_{cba}(T'_{\xi'+i},T_\xi) \circ (1_{\bar c}\chcob_{\xi',i} 1_b \otimes 1_{\bar b} 1_{T_\xi} 1_a) \bigr)^{-1},
\end{align*}
in the sense that the top path equal the bottom path times this factor.
\item The faces $f_{31}$ and $f_{32}$ trivially commute. 
\item The face $f_{41}$ commutes by naturality of $\shiftFunct{H}$. 
\item The face $f_{42}$ commutes up to the factor 
\[
\imath\bigl(  1_{\bar c}\chcob_{\xi' \sqcup \xi, i+k} 1_a \circ \chcob_{cba}(T'_{\xi'}, T_\xi)  \Rightarrow \chcob_{cba}(T'_{\xi'+i}, T_\xi) \circ (1_{\bar c} \chcob_{\xi',i} 1_b \otimes 1_{\bar b} 1_{T_\xi} 1_a) \bigr),
\]
since both path are related by this locally vertical change of chronology. 
\end{itemize}
Since $\chcob_{\xi' \sqcup \xi, i+k} = \chcob_{\xi',i} \bullet 1_{T_\xi}$, we obtain that the contributiuon of $f_{42}$ cancels with the one of $f_{22}$. All other faces being commutative, we conclude the outer square commutes. 

The commutativity of the second square follows from a similar argument, with a few additional subtleties. 
First, the computation of $f_{11}$ differs a little bit, but is still commutative thanks to the fact that:
\begin{align*}
\lambda_{R}(p', \deg(1_{\bar b} \chcob_{\xi} 1_a )) &= \lambda_R(p', \deg(1_{\bar b} \chcob_{\xi+j} 1_a ) + \deg(1_{\bar b} \chcob_{\xi,j} 1_a ))  
\\
&= \lambda_R(p', \deg(1_{\bar b} \chcob_{\xi+j} 1_a ))   \lambda_R(p', \deg(1_{\bar b} \chcob_{\xi,j} 1_a )).
\end{align*}
Secondly, for $f_{22}$, we obtain
\begin{align*}
\compMap{1_{T'_{\xi'}}}{\chcob_{\xi,j}} =
&\imath( H'' ) \lambda_R(p', \deg(1_{\bar b}\chcob_{\xi,j}1_a)),
\end{align*}
for some change of chronology $H'' :   (1_{\bar c}1_{T'_{\xi'}} \bullet \chcob_{\xi,j}1_a) \circ \chcob_{cba}(T'_{\xi'},T_\xi) \Rightarrow \chcob_{cba}(T'_{\xi'},T_{\xi+j}) \circ (1_{\bar c}1_{T'_{\xi'}} 1_b \otimes 1_{\bar b} \chcob_{\xi,j} 1_a)$. The face $f_{42}$ is computed as
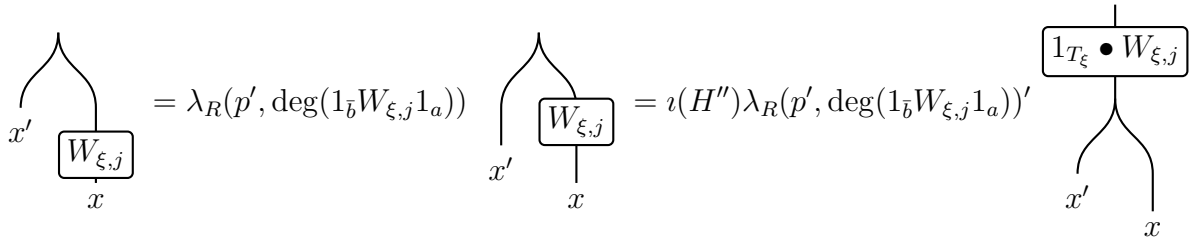
\[
 \tikzdiag{0}{
	\draw (0,-.25) node[below]{$x'$}
		--
		(0,-.25)
		.. controls (0,.25) and (.5,.25) ..
		(.5,.75);
	\draw (1,-1.25) node[below]{$x$} 
		--
		(1,-.75)  node[near end,rrect]{$\chcob_{\xi,j}$}
		--
		(1,-.25) 
		.. controls (1,.25) and (.5,.25) ..
		(.5,.75);
}
\ = \lambda_R(p', \deg(1_{\bar b}\chcob_{\xi,j}1_a)) \ 
 \tikzdiag{0}{
	\draw (0,-.75) node[below]{$x'$}
		--
		(0,-.25)
		.. controls (0,.25) and (.5,.25) ..
		(.5,.75) ;
	\draw (1,-1.25) node[below]{$x$} 
		--
		(1,-1)
		--
		(1,-.25)  node[near end,rrect]{$\chcob_{\xi,j}$}
		.. controls (1,.25) and (.5,.25) ..
		(.5,.75);
}
\ =  \imath(H'') \lambda_R(p', \deg(1_{\bar b}\chcob_{\xi,j}1_a))' \
 \tikzdiag{0}{
	\draw (0,-1.5) node[below]{$x'$}
		.. controls (0,-1) and (.5,-1) ..
		(.5,-.5)  
		--
		(.5,.75) node[midway,rrect]{$1_{T_{\xi}}\bullet\chcob_{\xi,j}$};
	\draw (1,-2) node[below]{$x$} 
		--
		(1,-1.5)
		.. controls (1,-1) and (.5,-1) ..
		(.5,-.5);
}
\]
thus giving a factor cancelling with the one of $f_{22}$. This ends the proof. 
\end{proof}

\subsection{Reidemeister moves}

Write $\otimes_H := \otimes_{\bigoplus_{n\geq 0} H^n}$. 
Let $T$ be an oriented $(n,m)$-tangle. We can decomposes it as a composition of elementary tangles $T = T_{k} \cdots T_1$. If $T_i$ is a cup or cap elementary tangle, then we put $\kh(T_i) := \tqft(T_i)$. If $T_i$ is a crossing, then we associate to it the dg-bimodule given by the following rule:
\begin{align*}
\kh\left( 
\tikzdiagh[scale=.75]
{
	\draw[dotted] (.5,.5) circle(0.707);
	\draw[->](0,0) -- (1,1);
	\fill[fill=white] (.5,.5) circle (.15);
	\draw[->](1,0) -- (0,1);
}
 \right) 
\ &:= \ 
\cone\left( 
 (\shiftFunct{ \tikzdiagh[scale=.2]{
	\draw (.5,1) -- (.5,3) -- (4.5,3) -- (4.5,1);
	\draw (0,0) .. controls (1.5,0) and (2,1) .. (.5,1); 
	\draw (4,0) .. controls (2.5,0) and (3,1) .. (4.5,1); 
	\filldraw[fill=white, draw=white] (.5,.55)  rectangle  (4,2);
	\draw[dotted] (0,0) .. controls (1.5,0) and (2,1) .. (.5,1); 
	\draw[dotted] (4,0) .. controls (2.5,0) and (3,1) .. (4.5,1); 
	\draw (1.375,.5) .. controls (1.375,1.5) and (3.125,1.5) .. (3.125,.5);
	%
	\draw (0,0) -- (0,2) -- (4,2) -- (4,0);
}})
\tqft\left(
\tikzdiagh[scale=.75]
{
	\draw[dotted] (3,-2) circle(0.707);
	\draw (2.5,-1.5) .. controls (2.75,-1.75) and (3.25,-1.75) .. (3.5,-1.5);
	\draw (2.5,-2.5) .. controls  (2.75,-2.25) and (3.25,-2.25) .. (3.5,-2.5);
}
  \right) 
\xrightarrow{\tqft\bigl( \ \tikzdiagh[scale=.2]{
	\draw (.5,1) -- (.5,3) -- (4.5,3) -- (4.5,1);
	\draw (0,0) .. controls (1.5,0) and (2,1) .. (.5,1); 
	\draw (4,0) .. controls (2.5,0) and (3,1) .. (4.5,1); 
	\filldraw[fill=white, draw=white] (.5,.55)  rectangle  (4,2);
	\draw[dotted] (0,0) .. controls (1.5,0) and (2,1) .. (.5,1); 
	\draw[dotted] (4,0) .. controls (2.5,0) and (3,1) .. (4.5,1); 
	\draw (1.375,.5) .. controls (1.375,1.5) and (3.125,1.5) .. (3.125,.5);
	%
	\draw (0,0) -- (0,2) -- (4,2) -- (4,0);
} \ \bigr)}
\tqft\left(
\tikzdiagh[scale=.75]
{
	\draw[dotted] (-2,-2) circle(0.707);
	\draw (-1.5,-1.5) .. controls (-1.75,-1.75) and (-1.75,-2.25) .. (-1.5,-2.5);
	\draw (-2.5,-1.5) .. controls (-2.25,-1.75) and (-2.25,-2.25) ..  (-2.5,-2.5);
}  
 \right)
\right) \{-1,0\},
\\
\kh\left(
\tikzdiagh[scale=.75]
{
	\draw[dotted] (.5,.5) circle(0.707);
	\draw[->](1,0) -- (0,1);
	\fill[fill=white] (.5,.5) circle (.15);
	\draw[->](0,0) -- (1,1);
}
 \right) 
\ &:= \ 
\cone\left( 
\tqft\left(
\tikzdiagh[scale=.75]
{
	\draw[dotted] (-2,-2) circle(0.707);
	\draw (-1.5,-1.5) .. controls (-1.75,-1.75) and (-1.75,-2.25) .. (-1.5,-2.5);
	\draw (-2.5,-1.5) .. controls (-2.25,-1.75) and (-2.25,-2.25) ..  (-2.5,-2.5);
} 
  \right) 
\xrightarrow{
\tqft\bigl( \ \tikzdiagh[scale=.2]{
	\draw (.5,3) -- (.5,1) -- (4.5,1) -- (4.5,3);
	\filldraw [fill=white, draw=white] (0,0) rectangle (4,2); 
	\draw (0,2) .. controls (1.5,2) and (2,3) .. (.5,3); 
	\draw (4,2) .. controls (2.5,2) and (3,3) .. (4.5,3); 
	\draw (1.375,2.5) .. controls (1.375,1.5) and (3.125,1.5) .. (3.125,2.5);
	%
	\draw (0,2) -- (0,0) -- (4,0) -- (4,2);
} \ \bigr)
\circ \shiftFunct{H}}
\bigl(
\shiftFunct{\tikzdiagh[scale=.2]{
	\draw (.5,3) -- (.5,1) -- (4.5,1) -- (4.5,3);
	\filldraw [fill=white, draw=white] (0,0) rectangle (4,2); 
	\draw (0,2) .. controls (1.5,2) and (2,3) .. (.5,3); 
	\draw (4,2) .. controls (2.5,2) and (3,3) .. (4.5,3); 
	\draw (1.375,2.5) .. controls (1.375,1.5) and (3.125,1.5) .. (3.125,2.5);
	%
	\draw (0,2) -- (0,0) -- (4,0) -- (4,2);
}} 
\bigr)^{-1}
\tqft\left(
\tikzdiagh[scale=.75]
{
	\draw[dotted] (3,-2) circle(0.707);
	\draw (2.5,-1.5) .. controls (2.75,-1.75) and (3.25,-1.75) .. (3.5,-1.5);
	\draw (2.5,-2.5) .. controls  (2.75,-2.25) and (3.25,-2.25) .. (3.5,-2.5);
}
\right)
\right)[1]\{0,1\},
\end{align*}
where $\shiftFunct{H} : \id \Rightarrow \bigl(
\shiftFunct{\tikzdiagh[scale=.2]{
	\draw (.5,3) -- (.5,1) -- (4.5,1) -- (4.5,3);
	\filldraw [fill=white, draw=white] (0,0) rectangle (4,2); 
	\draw (0,2) .. controls (1.5,2) and (2,3) .. (.5,3); 
	\draw (4,2) .. controls (2.5,2) and (3,3) .. (4.5,3); 
	\draw (1.375,2.5) .. controls (1.375,1.5) and (3.125,1.5) .. (3.125,2.5);
	%
	\draw (0,2) -- (0,0) -- (4,0) -- (4,2);
}} 
\bigr)^{-1} \circ \shiftFunct{\tikzdiagh[scale=.2]{
	\draw (.5,3) -- (.5,1) -- (4.5,1) -- (4.5,3);
	\filldraw [fill=white, draw=white] (0,0) rectangle (4,2); 
	\draw (0,2) .. controls (1.5,2) and (2,3) .. (.5,3); 
	\draw (4,2) .. controls (2.5,2) and (3,3) .. (4.5,3); 
	\draw (1.375,2.5) .. controls (1.375,1.5) and (3.125,1.5) .. (3.125,2.5);
	%
	\draw (0,2) -- (0,0) -- (4,0) -- (4,2);
}}$.
Then, we define
\[
\kh(T) := \kh(T_k) \otimes_H \cdots \otimes_H \kh(T_1).
\]

\begin{prop}
For each $T$, there exists a shifting functor $\shiftFunct{\chcob^v}$  and an $\ell \in \bZ$ such that $\shiftFunct{\chcob^v}(\kh(T))[\ell] \cong \tqft(T)$.
\end{prop}

\begin{proof}
If $T$ is an elementary tangle then we have 
\begin{align*}
\kh\left( 
\tikzdiagh[scale=.75]
{
	\draw[dotted] (.5,.5) circle(0.707);
	\draw[->](0,0)  -- (1,1)  ;
	\fill[fill=white] (.5,.5) circle (.15);
	\draw[->](1,0) -- (0,1);
}
 \right) 
 \{1,0\}
 &\cong 
 \tqft\left( 
\tikzdiagh[scale=.75]
{
	\draw[dotted] (.5,.5) circle(0.707);
	\draw[->](0,0) -- (1,1);
	\fill[fill=white] (.5,.5) circle (.15);
	\draw[->](1,0) -- (0,1);
}
 \right) 
 ,
 &
\bigl(
\shiftFunct{\tikzdiagh[scale=.2]{
	\draw (.5,3) -- (.5,1) -- (4.5,1) -- (4.5,3);
	\filldraw [fill=white, draw=white] (0,0) rectangle (4,2); 
	\draw (0,2) .. controls (1.5,2) and (2,3) .. (.5,3); 
	\draw (4,2) .. controls (2.5,2) and (3,3) .. (4.5,3); 
	\draw (1.375,2.5) .. controls (1.375,1.5) and (3.125,1.5) .. (3.125,2.5);
	%
	\draw (0,2) -- (0,0) -- (4,0) -- (4,2);
}} 
\bigr)
 \kh\left(
\tikzdiagh[scale=.75]
{
	\draw[dotted] (.5,.5) circle(0.707);
	\draw[->](1,0) -- (0,1);
	\fill[fill=white] (.5,.5) circle (.15);
	\draw[->](0,0) -- (1,1);
}
 \right) 
 [-1]\{0,-1\}
 &\cong
 \tqft\left(
\tikzdiagh[scale=.75]
{
	\draw[dotted] (.5,.5) circle(0.707);
	\draw[->](1,0) -- (0,1);
	\fill[fill=white] (.5,.5) circle (.15);
	\draw[->](0,0) -- (1,1);
}
 \right). 
\end{align*}
Note that we obtain the second isomorphism thanks to the fact we twist the operator$\tqft\bigl( \ \tikzdiagh[scale=.2]{
	\draw (.5,3) -- (.5,1) -- (4.5,1) -- (4.5,3);
	\filldraw [fill=white, draw=white] (0,0) rectangle (4,2); 
	\draw (0,2) .. controls (1.5,2) and (2,3) .. (.5,3); 
	\draw (4,2) .. controls (2.5,2) and (3,3) .. (4.5,3); 
	\draw (1.375,2.5) .. controls (1.375,1.5) and (3.125,1.5) .. (3.125,2.5);
	%
	\draw (0,2) -- (0,0) -- (4,0) -- (4,2);
} \ \bigr)$ by $\shiftFunct{H}$, so that the following diagram commutes:
\[
\begin{tikzcd}[column sep = 16ex]
\bigl(\shiftFunct{\tikzdiagh[scale=.2]{
	\draw (.5,3) -- (.5,1) -- (4.5,1) -- (4.5,3);
	\filldraw [fill=white, draw=white] (0,0) rectangle (4,2); 
	\draw (0,2) .. controls (1.5,2) and (2,3) .. (.5,3); 
	\draw (4,2) .. controls (2.5,2) and (3,3) .. (4.5,3); 
	\draw (1.375,2.5) .. controls (1.375,1.5) and (3.125,1.5) .. (3.125,2.5);
	%
	\draw (0,2) -- (0,0) -- (4,0) -- (4,2);
}} \bigr)
\tqft\left(
\tikzdiagh[scale=.75]
{
	\draw[dotted] (-2,-2) circle(0.707);
	\draw (-1.5,-1.5) .. controls (-1.75,-1.75) and (-1.75,-2.25) .. (-1.5,-2.5);
	\draw (-2.5,-1.5) .. controls (-2.25,-1.75) and (-2.25,-2.25) ..  (-2.5,-2.5);
} 
  \right) 
  \ar[dash]{r}{\tqft\bigl( \ \tikzdiagh[scale=.2]{
	\draw (.5,3) -- (.5,1) -- (4.5,1) -- (4.5,3);
	\filldraw [fill=white, draw=white] (0,0) rectangle (4,2); 
	\draw (0,2) .. controls (1.5,2) and (2,3) .. (.5,3); 
	\draw (4,2) .. controls (2.5,2) and (3,3) .. (4.5,3); 
	\draw (1.375,2.5) .. controls (1.375,1.5) and (3.125,1.5) .. (3.125,2.5);
	%
	\draw (0,2) -- (0,0) -- (4,0) -- (4,2);
} \ \bigr)
\circ \shiftFunct{H}}
  \ar{r}
&
\bigl( \shiftFunct{\tikzdiagh[scale=.2]{
	\draw (.5,3) -- (.5,1) -- (4.5,1) -- (4.5,3);
	\filldraw [fill=white, draw=white] (0,0) rectangle (4,2); 
	\draw (0,2) .. controls (1.5,2) and (2,3) .. (.5,3); 
	\draw (4,2) .. controls (2.5,2) and (3,3) .. (4.5,3); 
	\draw (1.375,2.5) .. controls (1.375,1.5) and (3.125,1.5) .. (3.125,2.5);
	%
	\draw (0,2) -- (0,0) -- (4,0) -- (4,2);
}} \bigr)
\circ
\bigl(
\shiftFunct{\tikzdiagh[scale=.2]{
	\draw (.5,3) -- (.5,1) -- (4.5,1) -- (4.5,3);
	\filldraw [fill=white, draw=white] (0,0) rectangle (4,2); 
	\draw (0,2) .. controls (1.5,2) and (2,3) .. (.5,3); 
	\draw (4,2) .. controls (2.5,2) and (3,3) .. (4.5,3); 
	\draw (1.375,2.5) .. controls (1.375,1.5) and (3.125,1.5) .. (3.125,2.5);
	%
	\draw (0,2) -- (0,0) -- (4,0) -- (4,2);
}} 
\bigr)^{-1}
\tqft\left(
\tikzdiagh[scale=.75]
{
	\draw[dotted] (3,-2) circle(0.707);
	\draw (2.5,-1.5) .. controls (2.75,-1.75) and (3.25,-1.75) .. (3.5,-1.5);
	\draw (2.5,-2.5) .. controls  (2.75,-2.25) and (3.25,-2.25) .. (3.5,-2.5);
}
\right)
\\
\bigl(\shiftFunct{\tikzdiagh[scale=.2]{
	\draw (.5,3) -- (.5,1) -- (4.5,1) -- (4.5,3);
	\filldraw [fill=white, draw=white] (0,0) rectangle (4,2); 
	\draw (0,2) .. controls (1.5,2) and (2,3) .. (.5,3); 
	\draw (4,2) .. controls (2.5,2) and (3,3) .. (4.5,3); 
	\draw (1.375,2.5) .. controls (1.375,1.5) and (3.125,1.5) .. (3.125,2.5);
	%
	\draw (0,2) -- (0,0) -- (4,0) -- (4,2);
}} \bigr)
\tqft\left(
\tikzdiagh[scale=.75]
{
	\draw[dotted] (-2,-2) circle(0.707);
	\draw (-1.5,-1.5) .. controls (-1.75,-1.75) and (-1.75,-2.25) .. (-1.5,-2.5);
	\draw (-2.5,-1.5) .. controls (-2.25,-1.75) and (-2.25,-2.25) ..  (-2.5,-2.5);
} 
  \right) 
  \ar{u}{1}
  \ar[dash,swap]{r}{\tqft\bigl( \ \tikzdiagh[scale=.2]{
	\draw (.5,3) -- (.5,1) -- (4.5,1) -- (4.5,3);
	\filldraw [fill=white, draw=white] (0,0) rectangle (4,2); 
	\draw (0,2) .. controls (1.5,2) and (2,3) .. (.5,3); 
	\draw (4,2) .. controls (2.5,2) and (3,3) .. (4.5,3); 
	\draw (1.375,2.5) .. controls (1.375,1.5) and (3.125,1.5) .. (3.125,2.5);
	%
	\draw (0,2) -- (0,0) -- (4,0) -- (4,2);
} \ \bigr)}
  \ar{r}
&
\tqft\left(
\tikzdiagh[scale=.75]
{
	\draw[dotted] (3,-2) circle(0.707);
	\draw (2.5,-1.5) .. controls (2.75,-1.75) and (3.25,-1.75) .. (3.5,-1.5);
	\draw (2.5,-2.5) .. controls  (2.75,-2.25) and (3.25,-2.25) .. (3.5,-2.5);
}
\right)
\ar[swap]{u}{\shiftFunct{H}}
\end{tikzcd}
\]
Then, the general result follows from \cref{prop:comptangles}.
\end{proof}

\begin{lem}
Consider two (plane diagrams of) oriented $(n,m)$-tangles $T$  and $T'$ such that $T'$ is obtained from $T$ by a planar isotopy. Then $\kh(T') \cong \kh(T)$ in  $\cD^\gradC({H^n}, H^m)$.
\end{lem}

\begin{proof}
Suppose that we exchange the $i$-th crossing with the $(i+1)$-th one in $T$ to obtain $T'$. 
Take $\xi' := \sigma	_i \xi$ where $\sigma_i \in S_k$ is the simple transposition that acts by exchanging $\xi_i$ with $\xi_{i+1}$ in $\xi$. Then, $T'_{\xi'}$ is equivalent to $T_\xi$. 
Thus, we have
\[
\shiftFunct{H} : C(T)_\xi = \shiftFunct{\chcob_\xi} \bigl(\tqft(T_{\xi})\bigr) \xrightarrow{\simeq}  \shiftFunct{\chcob'_{\xi'}} \bigl(\tqft(T'_{\xi'})\bigr) = C(T')_{\xi'},
\]
where $H : \chcob_{\xi} \Rightarrow \chcob'_{\xi'}$ is a locally vertical change of chronology. For $\xi_j = 0$ and $j' := \sigma_i(j)$, consider the following diagram:
\[
\begin{tikzcd}
C(T)_\xi
\ar{rr}{d_{\xi,j}}
\ar{dr}{\shiftFunct{H_{\xi,j}}}
\ar[swap]{ddd}{\shiftFunct{H}}
 &
 &
  C(T)_{\xi+j}
  \ar{ddd}{\shiftFunct{H'}}
\\
&
\shiftFunct{\chcob_{\xi+j}} \circ \shiftFunct{\chcob_{\xi,j}}(\tqft(T_\xi))
\ar{ur}{\tqft(\chcob_{\xi,j})}
\ar{d}{\shiftFunct{H'}}
&
\\
&
\shiftFunct{\chcob'_{\xi'+j'}} \circ \shiftFunct{\chcob'_{\xi', j'}} (\tqft(T'_{\xi'}))
\ar{dr}{\tqft(\chcob'_{\xi',j'})}
&
\\
C(T')_{\xi'} 
\ar{ur}{\shiftFunct{H'_{\xi',j'}}}
\ar[swap]{rr}{d'_{\xi',j'}}
&
&
C(T')_{\xi'+j'},
\end{tikzcd}
\]
where $H' : \chcob_{\xi+j} \Rightarrow \chcob'_{\xi'+j'}$ is a locally vertical change of chronology. The left part commutes thanks to \cref{prop:locvertchange}. The upper and lower part commute by definition of the differential. The right part commutes since $\chcob_{\xi,j}$ is equivalent to $\chcob'_{\xi',j'}$. Thus, the whole diagram commutes, and $\tqft(T') \cong \tqft(T)$. By consequence, $\kh(T') \cong \kh(T)$.  

Suppose that $T'$ is obtained from $T$ by the following local planar isotopy:
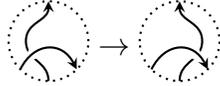
\[
\tikzdiagh[scale=.75]
{
	\draw[dotted] (.5,.5) circle(0.707);
	\draw[->] (.5,.5-.707) .. controls (.5,.5-.35) and (.1,.5-.35) ..(.1,.5)
		.. controls (.1,.5+.35) and (.5,.5+.35) .. (.5,.5+.707);
	\fill[fill=white] (.225,.275) circle (.1);
	\draw[->](0,0) .. controls (.25,.5) and (.75,.5) ..  (1,0);
}
\rightarrow
\tikzdiagh[scale=.75,xscale=-1]
{
	\draw[dotted] (.5,.5) circle(0.707);
	\draw[->] (.5,.5-.707) .. controls (.5,.5-.35) and (.1,.5-.35) ..(.1,.5)
		.. controls (.1,.5+.35) and (.5,.5+.35) .. (.5,.5+.707);
	\fill[fill=white] (.225,.275) circle (.1);
	\draw[<-](0,0) .. controls (.25,.5) and (.75,.5) ..  (1,0);
}
\]
Suppose the crossing in the pictured disk is the $i$-th one. Then we have $T_\xi \cong T'_{\xi}$ for all $\xi$. Moreover, for $\xi_i = 0$, there is a change of chronology $H_\xi : \chcob_{\xi} \Rightarrow \chcob'_\xi$ for all $\xi$, given by reversing the orientation of the saddle above the $i$-th crossing. For $\xi_i = 1$ we have an equivalence $\chcob_{\xi} \cong \chcob'_{\xi'}$. Similarly, there is $H_{\chcob_{\xi,j}} : \chcob_{\xi,j} \Rightarrow \chcob'_{\xi,j}$, which is trivial whenever $j \neq i$.  Then, consider the following diagram: 
\[
\begin{tikzcd}
C(T_\xi)
\ar{rr}{d_{\xi,j}}
\ar{dr}{\shiftFunct{H_{\xi,j}}}
\ar[swap]{ddd}{\shiftFunct{H_\xi}}
 &
 &
  C(T_{\xi+j})
  \ar{ddd}{\shiftFunct{H_{\xi+j}}}
\\
&
\shiftFunct{\chcob_{\xi+j}} \circ \shiftFunct{\chcob_{\xi,j}}(\tqft(T_\xi))
\ar{ur}{\tqft(\chcob_{\xi,j})}
\ar{d}{\shiftFunct{H_{\xi+j} \circ H_{\chcob_{\xi,j}}}}
&
\\
&
\shiftFunct{\chcob'_{\xi+j}} \circ \shiftFunct{\chcob'_{\xi, j}} (\tqft(T'_{\xi}))
\ar{dr}{\tqft(\chcob'_{\xi,j})}
&
\\
C(T'_{\xi}) 
\ar{ur}{\shiftFunct{H'_{\xi,j}}}
\ar[swap]{rr}{d'_{\xi,j}}
&
&
C(T'_{\xi+j}),
\end{tikzcd}
\]
It commutes for $j \neq i$ for the same reasons as above. For $j = i$, it also commutes since $\tqft(\chcob'_{\xi+j}) \circ \shiftFunct{H_{\chcob_{\xi,j}}} = \tqft(\chcob_{\xi,j})$. 
The other cases are similar, finishing the proof. 
\end{proof}

\begin{lem}\label{lem:RI}
Consider two oriented $(n,m)$-tangles $T$  and $T'$. Suppose $T$ differs from $T'$ only in a small region where we have 
\begin{align*}
T &\supset 
\tikzdiagh[scale=.75]
{
	\draw[dotted] (.5,.5) circle(0.707);
	\draw(1,0) .. controls (.5,0) and (-.25,1) .. (.5,1)  ;
	\fill[fill=white] (.5,.25) circle (.15);
	\draw(0,0) .. controls (.5,0) and (1.25,1) .. (.5,1)  ;
} 
  & &\text{and} &
T' &\supset 
\tikzdiagh[scale=.75]
{
	\draw[dotted] (.5,.5) circle(0.707);
	\draw(1,0) .. controls (.5,0) and (1.25,1) .. (.5,1)  ;
	\draw(0,0) .. controls (.5,0) and (-.25,1) .. (.5,1)  ;
}
\end{align*}
Then, there is an isomorphism $\kh(T) \cong \kh(T')$ in $\cD^\gradC(H^n, H^m)$.
\end{lem}

\begin{proof}
We can decompose $T$  such that $T = T_2 T_C T_1$ and $T' = T_2 T_I T_1$ where 
$T_C = \tikzdiagh[scale=.75]
{
	\draw[dotted] (.5,.5) circle(0.707);
	\draw(1,0) .. controls (.5,0) and (-.25,1) .. (.5,1)  ;
	\fill[fill=white] (.5,.25) circle (.15);
	\draw(0,0) .. controls (.5,0) and (1.25,1) .. (.5,1)  ;
} $
 and
  $T_I = \tikzdiagh[scale=.75]
{
	\draw[dotted] (.5,.5) circle(0.707);
	\draw(1,0) .. controls (.5,0) and (1.25,1) .. (.5,1)  ;
	\draw(0,0) .. controls (.5,0) and (-.25,1) .. (.5,1)  ;
} $. 
 We will show that $T_C \cong T_I$ and the result will from fact that $\kh(T_i)$ is cofibrant both as left and as right module, thanks to \cref{prop:projbimod}.

No matter the orientation on $T$, we have
\[
\kh
\tikzdiagh[scale=.75]
{
	\draw[dotted] (.5,.5) circle(0.707);
	\draw(1,0) .. controls (.5,0) and (-.25,1) .. (.5,1)  ;
	\fill[fill=white] (.5,.25) circle (.15);
	\draw(0,0) .. controls (.5,0) and (1.25,1) .. (.5,1)  ;
}
\cong \cone\left(
(\shiftFunct{\tikzdiagh[scale=.2]{
	\begin{scope}
	\clip (-1,.5) rectangle (1,-.25); 
	\draw (0,0) .. controls (-1.5,0) and (-1,1) .. (.5,1); 
	\end{scope}
	\begin{scope}
	\draw  (4,1) -- (4.5,1) -- (4.5,3);
	\end{scope}
	%
	%
	\draw (0,2) .. controls (1.5,2) and (2,3) .. (.5,3); 
	\draw (4,2) .. controls (2.5,2) and (3,3) .. (4.5,3); 
	\draw (1.375,2.5) .. controls (1.375,1.5) and (3.125,1.5) .. (3.125,2.5);
	%
	\draw (0,0) -- (4,0) -- (4,2);
	\draw (0,2) .. controls (-1.5,2) and (-1,3) .. (.5,3); 
	\draw (-.875,.5) -- (-.875,2.5);
}})
\tqft
 \tikzdiagh[scale=.75]
{
	\draw[dotted] (.5,.5) circle(0.707);
	\draw(1,0) .. controls (.5,0) and (1.25,1) .. (.5,1)  ;
	\draw(0,0) .. controls (.5,0) and (-.25,1) .. (.5,1)  ;
}
\xrightarrow{\tqft\bigl( \ \tikzdiagh[scale=.2]{
	\begin{scope}
	\clip (-1,.5) rectangle (1,-.25); 
	\draw (0,0) .. controls (-1.5,0) and (-1,1) .. (.5,1); 
	\end{scope}
	\begin{scope}
	\draw  (4,1) -- (4.5,1) -- (4.5,3);
	\end{scope}
	%
	%
	\draw (0,2) .. controls (1.5,2) and (2,3) .. (.5,3); 
	\draw (4,2) .. controls (2.5,2) and (3,3) .. (4.5,3); 
	\draw (1.375,2.5) .. controls (1.375,1.5) and (3.125,1.5) .. (3.125,2.5);
	%
	\draw (0,0) -- (4,0) -- (4,2);
	\draw (0,2) .. controls (-1.5,2) and (-1,3) .. (.5,3); 
	\draw (-.875,.5) -- (-.875,2.5);
} \ \bigr)}
\tqft 
 \tikzdiagh[scale=.75]
{
	\draw[dotted] (.5,.5) circle(0.707);
	\draw(1,0) .. controls (.75,0) and (.75,.25) .. (.5,.25)  ;
	\draw(0,0) .. controls (.25,0) and (.25,.25) .. (.5,.25)  ;
	\draw (.5,.75) circle (.25);
} 
 \right)\{-1,0 \}.
\]
Then, as in \cite[Lemma 7.3]{putyra14}, we have a quasi-isomorphism $\kh(T_C) \xrightarrow{\simeq} \kh(T_I)$ given by the vertical map in the following diagram:
\[
\begin{tikzcd}[column sep=14ex]
\\
(\shiftFunct{\tikzdiagh[scale=.2]{
	\begin{scope}
	\clip (-1,.5) rectangle (1,-.25); 
	\draw (0,0) .. controls (-1.5,0) and (-1,1) .. (.5,1); 
	\end{scope}
	\begin{scope}
	\draw  (4,1) -- (4.5,1) -- (4.5,3);
	\end{scope}
	%
	%
	\draw (0,2) .. controls (1.5,2) and (2,3) .. (.5,3); 
	\draw (4,2) .. controls (2.5,2) and (3,3) .. (4.5,3); 
	\draw (1.375,2.5) .. controls (1.375,1.5) and (3.125,1.5) .. (3.125,2.5);
	%
	\draw (0,0) -- (4,0) -- (4,2);
	\draw (0,2) .. controls (-1.5,2) and (-1,3) .. (.5,3); 
	\draw (-.875,.5) -- (-.875,2.5);
}})
 \tqft \tikzdiagh[scale=.75]
{
	\draw[dotted] (.5,.5) circle(0.707);
	\draw(1,0) .. controls (.5,0) and (1.25,1) .. (.5,1)  ;
	\draw(0,0) .. controls (.5,0) and (-.25,1) .. (.5,1)  ;
} 
\{-1,0\}
\ar[dash]{r}{\tqft\bigl( \ \tikzdiagh[scale=.2]{
	\begin{scope}
	\clip (-1,.5) rectangle (1,-.25); 
	\draw (0,0) .. controls (-1.5,0) and (-1,1) .. (.5,1); 
	\end{scope}
	\begin{scope}
	\draw  (4,1) -- (4.5,1) -- (4.5,3);
	\end{scope}
	%
	%
	\draw (0,2) .. controls (1.5,2) and (2,3) .. (.5,3); 
	\draw (4,2) .. controls (2.5,2) and (3,3) .. (4.5,3); 
	\draw (1.375,2.5) .. controls (1.375,1.5) and (3.125,1.5) .. (3.125,2.5);
	%
	\draw (0,0) -- (4,0) -- (4,2);
	\draw (0,2) .. controls (-1.5,2) and (-1,3) .. (.5,3); 
	\draw (-.875,.5) -- (-.875,2.5);
} \ \bigr)} 
\ar{r}
&
\tqft 
 \tikzdiagh[scale=.75]
{
	\draw[dotted] (.5,.5) circle(0.707);
	\draw(1,0) .. controls (.75,0) and (.75,.25) .. (.5,.25)  ;
	\draw(0,0) .. controls (.25,0) and (.25,.25) .. (.5,.25)  ;
	\draw (.5,.75) circle (.25);
} \{-1,0\}
\\
0 \ar{r} 
\ar{u}
&
 \tqft \tikzdiagh[scale=.75]
{
	\draw[dotted] (.5,.5) circle(0.707);
	\draw(1,0) .. controls (.5,0) and (1.25,1) .. (.5,1)  ;
	\draw(0,0) .. controls (.5,0) and (-.25,1) .. (.5,1)  ;
}
\ar{u}
\ar[dash,swap]{u}{\tqft\bigl( \ \tikzdiagh[scale=.2]{
	\begin{scope}
	\clip (-1,.5) rectangle (1,-.25); 
	\draw (0,0) .. controls (-1.5,0) and (-1,1) .. (.5,1); 
	\end{scope}
	\begin{scope}
	\draw  (4,1) -- (4.5,1) -- (4.5,4);
	\end{scope}
	%
	%
	\draw (0,3) .. controls (1.5,3) and (2,4) .. (.5,4); 
	\draw (4,3) .. controls (2.5,3) and (3,4) .. (4.5,4); 
	%
	%
	\draw (0,0) -- (4,0) -- (4,3);
	\draw (0,3) .. controls (-1.5,3) and (-1,4) .. (.5,4); 
	%
	\draw (-.875,.5) .. controls (-.875,1.5) and (3.125,1.5) .. (3.125,3.5);
	\draw (-.875,3.5) .. controls (-.875,1.5) and (1.375,1.5) .. (1.375,3.5);
} \ \bigr)}
\end{tikzcd}
\]
 Note that the vertical map is graded. 
\end{proof}

\begin{lem}\label{lem:RIb}
Consider two oriented $(n,m)$-tangles $T$  and $T'$. Suppose $T$ differs from $T'$ only in a small region where we have 
\begin{align*}
T &\supset 
\tikzdiagh[scale=.75]
{
	\draw[dotted] (.5,.5) circle(0.707);
	\draw(0,0) .. controls (.5,0) and (1.25,1) .. (.5,1)  ;
	\fill[fill=white] (.5,.25) circle (.15);
	\draw(1,0) .. controls (.5,0) and (-.25,1) .. (.5,1)  ;
} 
 & &\text{and} &
T' &\supset 
\tikzdiagh[scale=.75]
{
	\draw[dotted] (.5,.5) circle(0.707);
	\draw(1,0) .. controls (.5,0) and (1.25,1) .. (.5,1)  ;
	\draw(0,0) .. controls (.5,0) and (-.25,1) .. (.5,1)  ;
} 
\end{align*}
Then, there is an isomorphism $\kh(T) \cong \kh(T')$ in $\cD^\gradC(H^n, H^m)$.
\end{lem}

\begin{proof}
The proof is similar to the one of \cref{lem:RI} except that
\[
\kh
\tikzdiagh[scale=.75]
{
	\draw[dotted] (.5,.5) circle(0.707);
	\draw(0,0) .. controls (.5,0) and (1.25,1) .. (.5,1)  ;
	\fill[fill=white] (.5,.25) circle (.15);
	\draw(1,0) .. controls (.5,0) and (-.25,1) .. (.5,1)  ;
}
\cong \cone\left(
\tqft
 \tikzdiagh[scale=.75]
{
	\draw[dotted] (.5,.5) circle(0.707);
	\draw(1,0) .. controls (.75,0) and (.75,.25) .. (.5,.25)  ;
	\draw(0,0) .. controls (.25,0) and (.25,.25) .. (.5,.25)  ;
	\draw (.5,.75) circle (.25);
} 
\xrightarrow{\tqft\bigl( \ \tikzdiagh[scale=.2]{
	\draw  (.5,3) -- (4.5,3) -- (4.5,1);
	\begin{scope}
		\clip (-1,.5) rectangle (1,-.5);
		\draw (0,0) .. controls (-1.5,0) and (-1,1) .. (.5,1); 
	\end{scope}
	\draw (0,0) .. controls (1.5,0) and (2,1) .. (.5,1); 
	\draw (4,0) .. controls (2.5,0) and (3,1) .. (4.5,1); 
	\filldraw[fill=white, draw=white] (.5,.55)  rectangle  (4,2);
	\draw[dotted] (0,0) .. controls (1.5,0) and (2,1) .. (.5,1); 
	\draw[dotted] (4,0) .. controls (2.5,0) and (3,1) .. (4.5,1); 
	\draw[dotted] (0,0) .. controls (-1.5,0) and (-1,1) .. (.5,1); 
	\draw (1.375,.5) .. controls (1.375,1.5) and (3.125,1.5) .. (3.125,.5);
	%
	\draw (0,2) -- (4,2) -- (4,0);
	\draw (0,2) .. controls (-1.5,2) and (-1,3) .. (.5,3); 
	\draw (-.875,.5) -- (-.875,2.5);
}\ \bigr) \circ \shiftFunct{H}}
\bigl(\shiftFunct{\tikzdiagh[scale=.2]{
	\draw  (.5,3) -- (4.5,3) -- (4.5,1);
	\begin{scope}
		\clip (-1,.5) rectangle (1,-.5);
		\draw (0,0) .. controls (-1.5,0) and (-1,1) .. (.5,1); 
	\end{scope}
	\draw (0,0) .. controls (1.5,0) and (2,1) .. (.5,1); 
	\draw (4,0) .. controls (2.5,0) and (3,1) .. (4.5,1); 
	\filldraw[fill=white, draw=white] (.5,.55)  rectangle  (4,2);
	\draw[dotted] (0,0) .. controls (1.5,0) and (2,1) .. (.5,1); 
	\draw[dotted] (4,0) .. controls (2.5,0) and (3,1) .. (4.5,1); 
	\draw[dotted] (0,0) .. controls (-1.5,0) and (-1,1) .. (.5,1); 
	\draw (1.375,.5) .. controls (1.375,1.5) and (3.125,1.5) .. (3.125,.5);
	%
	\draw (0,2) -- (4,2) -- (4,0);
	\draw (0,2) .. controls (-1.5,2) and (-1,3) .. (.5,3); 
	\draw (-.875,.5) -- (-.875,2.5);
}}\bigr)^{-1}
\tqft 
 \tikzdiagh[scale=.75]
{
	\draw[dotted] (.5,.5) circle(0.707);
	\draw(1,0) .. controls (.5,0) and (1.25,1) .. (.5,1)  ;
	\draw(0,0) .. controls (.5,0) and (-.25,1) .. (.5,1)  ;
}
 \right)[1]\{0,1\},
\]
and the quasi-isomorphism looks like
\[
\begin{tikzcd}[column sep=14ex]
\tqft 
 \tikzdiagh[scale=.75]
{
	\draw[dotted] (.5,.5) circle(0.707);
	\draw(1,0) .. controls (.5,0) and (1.25,1) .. (.5,1)  ;
	\draw(0,0) .. controls (.5,0) and (-.25,1) .. (.5,1)  ;
}
\ar{r}
& 
0
\\
\tqft
 \tikzdiagh[scale=.75]
{
	\draw[dotted] (.5,.5) circle(0.707);
	\draw(1,0) .. controls (.75,0) and (.75,.25) .. (.5,.25)  ;
	\draw(0,0) .. controls (.25,0) and (.25,.25) .. (.5,.25)  ;
	\draw (.5,.75) circle (.25);
} 
\{0,1\}
\ar[swap,dash]{r}{\tqft\bigl( \ \tikzdiagh[scale=.2]{
	\draw  (.5,3) -- (4.5,3) -- (4.5,1);
	\begin{scope}
		\clip (-1,.5) rectangle (1,-.5);
		\draw (0,0) .. controls (-1.5,0) and (-1,1) .. (.5,1); 
	\end{scope}
	\draw (0,0) .. controls (1.5,0) and (2,1) .. (.5,1); 
	\draw (4,0) .. controls (2.5,0) and (3,1) .. (4.5,1); 
	\filldraw[fill=white, draw=white] (.5,.55)  rectangle  (4,2);
	\draw[dotted] (0,0) .. controls (1.5,0) and (2,1) .. (.5,1); 
	\draw[dotted] (4,0) .. controls (2.5,0) and (3,1) .. (4.5,1); 
	\draw[dotted] (0,0) .. controls (-1.5,0) and (-1,1) .. (.5,1); 
	\draw (1.375,.5) .. controls (1.375,1.5) and (3.125,1.5) .. (3.125,.5);
	%
	\draw (0,2) -- (4,2) -- (4,0);
	\draw (0,2) .. controls (-1.5,2) and (-1,3) .. (.5,3); 
	\draw (-.875,.5) -- (-.875,2.5);
}\ \bigr)  \circ \shiftFunct{H}}
\ar{r}
 \ar{u}
\ar[dash]{u}{\tqft\bigl( \ \tikzdiagh[scale=.2]{
	\draw  (.5,4) -- (4.5,4) -- (4.5,1);
	\begin{scope}
		\clip (-1,.5) rectangle (1,-.5);
		\draw (0,0) .. controls (-1.5,0) and (-1,1) .. (.5,1); 
	\end{scope}
	\draw (0,0) .. controls (1.5,0) and (2,1) .. (.5,1); 
	\draw (4,0) .. controls (2.5,0) and (3,1) .. (4.5,1); 
	\filldraw[fill=white, draw=white] (.5,.55)  rectangle  (4,3);
	\draw[dotted] (0,0) .. controls (1.5,0) and (2,1) .. (.5,1); 
	\draw[dotted] (4,0) .. controls (2.5,0) and (3,1) .. (4.5,1); 
	\draw[dotted] (0,0) .. controls (-1.5,0) and (-1,1) .. (.5,1); 
	%
	%
	\draw (0,3) -- (4,3) -- (4,0);
	\draw (0,3) .. controls (-1.5,3) and (-1,4) .. (.5,4); 
	\draw (-.875,3.5) .. controls (-.875,2.5) and (3.125,2.5) .. (3.125,.5);
	\draw (-.875,.5) .. controls (-.875,2) and (1.375,2) .. (1.375,.5);
} \ \bigr)}
&
\bigl(\shiftFunct{\tikzdiagh[scale=.2]{
	\draw  (.5,3) -- (4.5,3) -- (4.5,1);
	\begin{scope}
		\clip (-1,.5) rectangle (1,-.5);
		\draw (0,0) .. controls (-1.5,0) and (-1,1) .. (.5,1); 
	\end{scope}
	\draw (0,0) .. controls (1.5,0) and (2,1) .. (.5,1); 
	\draw (4,0) .. controls (2.5,0) and (3,1) .. (4.5,1); 
	\filldraw[fill=white, draw=white] (.5,.55)  rectangle  (4,2);
	\draw[dotted] (0,0) .. controls (1.5,0) and (2,1) .. (.5,1); 
	\draw[dotted] (4,0) .. controls (2.5,0) and (3,1) .. (4.5,1); 
	\draw[dotted] (0,0) .. controls (-1.5,0) and (-1,1) .. (.5,1); 
	\draw (1.375,.5) .. controls (1.375,1.5) and (3.125,1.5) .. (3.125,.5);
	%
	\draw (0,2) -- (4,2) -- (4,0);
	\draw (0,2) .. controls (-1.5,2) and (-1,3) .. (.5,3); 
	\draw (-.875,.5) -- (-.875,2.5);
}}\bigr)^{-1}
\tqft 
 \tikzdiagh[scale=.75]
{
	\draw[dotted] (.5,.5) circle(0.707);
	\draw(1,0) .. controls (.5,0) and (1.25,1) .. (.5,1)  ;
	\draw(0,0) .. controls (.5,0) and (-.25,1) .. (.5,1)  ;
}
\{0,1\}
\ar{u}
\end{tikzcd}
\]
which is again graded.
\end{proof}

\begin{lem}\label{lem:RII}
Let $T$ and $T'$ be two oriented $(n,m)$-tangles such that $T'$ differs from $T$ only in a small region where
\begin{align*}
T &\supset 
\tikzdiagh[scale=.75]
{
	\draw[dotted] (.5,.5) circle(0.707);
	\draw[->] (0,0) .. controls (1,.25) and (1,.75) .. (0,1); 
	\fill[fill=white] (.5,.25) circle (.15);
	\fill[fill=white] (.5,.75) circle (.15);
	\draw[->] (1,0) .. controls (0,.25) and (0,.75) .. (1,1);
} 
  & &\text{and} &
T' &\supset 
\tikzdiagh[scale=.75]
{
	\draw[dotted] (.5,.5) circle(0.707);
	\draw[->] (0,0) .. controls (.25,.25) and (.25,.75) .. (0,1); 
	\draw[->] (1,0) .. controls (.75,.25) and (.75,.75) .. (1,1);
}
\end{align*}
Then, there is an isomorphism $ T  \cong  T' \{-1,1\}$ in $\cD^\gradC(H^n, H^m)$.
\end{lem}

\begin{proof}
Using \cref{prop:removeloops}, we obtain that $\kh(\tikzdiagh[scale=.75]
{
	\draw[dotted] (.5,.5) circle(0.707);
	\draw[->] (0,0) .. controls (1,.25) and (1,.75) .. (0,1); 
	\fill[fill=white] (.5,.25) circle (.15);
	\fill[fill=white] (.5,.75) circle (.15);
	\draw[->] (1,0) .. controls (0,.25) and (0,.75) .. (1,1);
} )$ looks like the first row in the diagram below, and we construct a graded map $\kh(T') \rightarrow \kh(T)$ by the vertical arrows:
\[
\begin{tikzcd}[row sep=-2ex, column sep = 4ex]
&
\tqft\tikzdiagh[scale=.75]
{
	\draw[dotted] (.5,.5) circle(0.707);
	\draw (0,0) .. controls (.25,.25) and (.25,.75) .. (0,1); 
	\draw (1,0) .. controls (.75,.25) and (.75,.75) .. (1,1);
}\{-1,1\} 
\ar{dr}
& 
\\
\bigl(\shiftFunct{\tikzdiagh[scale=.2]{
	\draw (.5,1) -- (.5,3) -- (4.5,3) -- (4.5,1);
	\draw (0,0) .. controls (1.5,0) and (2,1) .. (.5,1); 
	\draw (4,0) .. controls (2.5,0) and (3,1) .. (4.5,1); 
	\filldraw[fill=white, draw=white] (.5,.55)  rectangle  (4,2);
	\draw[dotted] (0,0) .. controls (1.5,0) and (2,1) .. (.5,1); 
	\draw[dotted] (4,0) .. controls (2.5,0) and (3,1) .. (4.5,1); 
	\draw (1.375,.5) .. controls (1.375,1.5) and (3.125,1.5) .. (3.125,.5);
	%
	\draw (0,0) -- (0,2) -- (4,2) -- (4,0);
}}^{(-1,1)} \bigr)
\tqft \tikzdiagh[scale=.75]
{
	\draw[dotted] (3,-2) circle(0.707);
	\draw (2.5,-1.5) .. controls (2.75,-1.75) and (3.25,-1.75) .. (3.5,-1.5);
	\draw (2.5,-2.5) .. controls  (2.75,-2.25) and (3.25,-2.25) .. (3.5,-2.5);
}
\ar{ur} 
\ar{dr}
&&
\bigl(\shiftFunct{\tikzdiagh[scale=.2]{
	\draw (.5,1) -- (.5,3) -- (4.5,3) -- (4.5,1);
	\draw (0,0) .. controls (1.5,0) and (2,1) .. (.5,1); 
	\draw (4,0) .. controls (2.5,0) and (3,1) .. (4.5,1); 
	\filldraw[fill=white, draw=white] (.5,.55)  rectangle  (4,2);
	\draw[dotted] (0,0) .. controls (1.5,0) and (2,1) .. (.5,1); 
	\draw[dotted] (4,0) .. controls (2.5,0) and (3,1) .. (4.5,1); 
	\draw (1.375,.5) .. controls (1.375,1.5) and (3.125,1.5) .. (3.125,.5);
	%
	\draw (0,0) -- (0,2) -- (4,2) -- (4,0);
}}^{(0,2)} \bigr)
\tqft \tikzdiagh[scale=.75]
{
	\draw[dotted] (3,-2) circle(0.707);
	\draw (2.5,-1.5) .. controls (2.75,-1.75) and (3.25,-1.75) .. (3.5,-1.5);
	\draw (2.5,-2.5) .. controls  (2.75,-2.25) and (3.25,-2.25) .. (3.5,-2.5);
}
\\
&
\bigl(\shiftFunct{\tikzdiagh[scale=.2]{
	\draw (.5,1) -- (.5,3) -- (4.5,3) -- (4.5,1);
	\draw (0,0) .. controls (1.5,0) and (2,1) .. (.5,1); 
	\draw (4,0) .. controls (2.5,0) and (3,1) .. (4.5,1); 
	\filldraw[fill=white, draw=white] (.5,.55)  rectangle  (4,2);
	\draw[dotted] (0,0) .. controls (1.5,0) and (2,1) .. (.5,1); 
	\draw[dotted] (4,0) .. controls (2.5,0) and (3,1) .. (4.5,1); 
	\draw (1.375,.5) .. controls (1.375,1.5) and (3.125,1.5) .. (3.125,.5);
	%
	\draw (0,0) -- (0,2) -- (4,2) -- (4,0);
}}^{(-1,2)} \bigr)
\tqft \tikzdiagh[scale=.75]
{
	\draw[dotted] (.5,.5) circle(0.707);
	\draw (0,0) .. controls (.25,.25) and (.75,.25) .. (1,0); 
	\draw (.5,.5) circle (.15);
	\draw (0,1) .. controls (.25,.75) and (.75,.75) .. (1,1);
}
\ar{ur}
&
\\[10ex]
0 \ar{r} 
 \ar{uu}
 & 
\tqft\tikzdiagh[scale=.75]
{
	\draw[dotted] (.5,.5) circle(0.707);
	\draw (0,0) .. controls (.25,.25) and (.25,.75) .. (0,1); 
	\draw (1,0) .. controls (.75,.25) and (.75,.75) .. (1,1);
}\{-1,1\} 
\ar{u}
\ar[dash,swap]{u}{-\tqft \tikzdiagh[scale=.2]{
	\draw (.5,4) -- (.5,1) -- (6.5,1) -- (6.5,4);
	\filldraw [fill=white, draw=white] (0,0) rectangle (6,3); 
	\draw (0,3) .. controls (1.5,3) and (2,4) .. (.5,4); 
	\draw (6,3) .. controls (4.5,3) and (5,4) .. (6.5,4); 
	\draw (3,3) .. controls (4.5,3) and (5,4) .. (3.5,4); 
	\draw (3,3) .. controls (1.5,3) and (2,4) .. (3.5,4); 
	\draw (2.125,3.5) .. controls (2.125,1.5) and (4.375,1.5) .. (4.375,3.5);
	\draw (1.375,3.5) .. controls (1.375,.5) and (5.125,.5) .. (5.125,3.5);
	\draw (0,3) -- (0,0) -- (6,0) -- (6,3);
} }
\ar[bend left=60,pos=.35]{uuu}{1}
\ar{r}
 &
  0
 \ar{uu}
\end{tikzcd}
\]
where the middle part is in homological degree zero. Up to composing the vertical maps with a change of chronology, we obtain a commutative diagram. 
By \cite[Lemma 7.5]{putyra14}, we know the resulting vertical maps yield a quasi-isomorphism. 
\end{proof}

\begin{lem}\label{lem:RIIb}
Let $T$ and $T'$ be two oriented $(n,m)$-tangles such that $T'$ differs from $T$ only in a small region where
\begin{align*}
T &\supset 
\tikzdiagh[scale=.75]
{
	\draw[dotted] (.5,.5) circle(0.707);
	\draw[->] (0,0) .. controls (1,.25) and (1,.75) .. (0,1); 
	\fill[fill=white] (.5,.25) circle (.15);
	\fill[fill=white] (.5,.75) circle (.15);
	\draw[<-] (1,0) .. controls (0,.25) and (0,.75) .. (1,1);
} 
  & &\text{and} &
T' &\supset 
\tikzdiagh[scale=.75]
{
	\draw[dotted] (.5,.5) circle(0.707);
	\draw[->] (0,0) .. controls (.25,.25) and (.25,.75) .. (0,1); 
	\draw[<-] (1,0) .. controls (.75,.25) and (.75,.75) .. (1,1);
}
\end{align*}
Then, there is an isomorphism 
$T \cong  
\bigl(
\shiftFunct{\tikzdiagh[scale=.2]{
	\draw (.5,3) -- (.5,1) -- (4.5,1) -- (4.5,3);
	\filldraw [fill=white, draw=white] (0,0) rectangle (4,2); 
	\draw (0,2) .. controls (1.5,2) and (2,3) .. (.5,3); 
	\draw (4,2) .. controls (2.5,2) and (3,3) .. (4.5,3); 
	\draw (1.375,2.5) .. controls (1.375,1.5) and (3.125,1.5) .. (3.125,2.5);
	%
	\draw (0,2) -- (0,0) -- (4,0) -- (4,2);
}^{(0,1)}} 
\bigr) T' $ 
in $\cD^\gradC(H^n, H^m)$.
\end{lem}

\begin{proof}
We obtain that $\kh(\tikzdiagh[scale=.75]
{
	\draw[dotted] (.5,.5) circle(0.707);
	\draw[->] (0,0) .. controls (1,.25) and (1,.75) .. (0,1); 
	\fill[fill=white] (.5,.25) circle (.15);
	\fill[fill=white] (.5,.75) circle (.15);
	\draw[<-] (1,0) .. controls (0,.25) and (0,.75) .. (1,1);
} )$ looks like
\[
\begin{tikzcd}[row sep=-2ex, column sep = 4ex]
&
\bigl(\shiftFunct{\tikzdiagh[scale=.2]{
	\draw (.5,3) -- (.5,1) -- (4.5,1) -- (4.5,3);
	\filldraw [fill=white, draw=white] (0,0) rectangle (4,2); 
	\draw (0,2) .. controls (1.5,2) and (2,3) .. (.5,3); 
	\draw (4,2) .. controls (2.5,2) and (3,3) .. (4.5,3); 
	\draw (1.375,2.5) .. controls (1.375,1.5) and (3.125,1.5) .. (3.125,2.5);
	%
	\draw (0,2) -- (0,0) -- (4,0) -- (4,2);
}}^{(0,1)} \bigr)
\tqft\tikzdiagh[scale=.75]
{
	\draw[dotted] (.5,.5) circle(0.707);
	\draw (0,0) .. controls (.25,.25) and (.25,.75) .. (0,1); 
	\draw (1,0) .. controls (.75,.25) and (.75,.75) .. (1,1);
}
\ar{dr}
& 
\\
\tqft \tikzdiagh[scale=.75]
{
	\draw[dotted] (3,-2) circle(0.707);
	\draw (2.5,-1.5) .. controls (2.75,-1.75) and (3.25,-1.75) .. (3.5,-1.5);
	\draw (2.5,-2.5) .. controls  (2.75,-2.25) and (3.25,-2.25) .. (3.5,-2.5);
}
\{-1,0\}
\ar{ur} 
\ar{dr}
&&
\tqft \tikzdiagh[scale=.75]
{
	\draw[dotted] (3,-2) circle(0.707);
	\draw (2.5,-1.5) .. controls (2.75,-1.75) and (3.25,-1.75) .. (3.5,-1.5);
	\draw (2.5,-2.5) .. controls  (2.75,-2.25) and (3.25,-2.25) .. (3.5,-2.5);
}
\{0,1\}
\\
&
\tqft \tikzdiagh[scale=.75]
{
	\draw[dotted] (.5,.5) circle(0.707);
	\draw (0,0) .. controls (.25,.25) and (.75,.25) .. (1,0); 
	\draw (.5,.5) circle (.15);
	\draw (0,1) .. controls (.25,.75) and (.75,.75) .. (1,1);
}
\{-1,1\}
\ar{ur}
&
\end{tikzcd}
\]
where the middle part is in homological degree zero. Then, we verify we can use a similar underlying quasi-isomorphism as in \cref{lem:RII}, and it is graded. 
\end{proof}

\begin{lem}\label{lem:RIII}
Let $T$ and $T'$ be two oriented $(n,m)$-tangles such that $T'$ differs from $T$ only in a small region where
\begin{align*}
T &\supset 
\tikzdiagh[scale=.75]
{
	\draw[dotted] (.5,.5) circle(0.707);
	\draw (.5,.5-.707) .. controls (.5,.5-.35) and (.1,.5-.35) ..(.1,.5)
		.. controls (.1,.5+.35) and (.5,.5+.35) .. (.5,.5+.707);
	\fill[fill=white] (.25,.75) circle (.1);
	\fill[fill=white] (.25,.25) circle (.1);
	\draw(1,0) -- (0,1);
	\fill[fill=white] (.5,.5) circle (.1);
	\draw(0,0) -- (1,1);
}
 & &\text{and} &
T' &\supset 
\tikzdiagh[scale=.75,xscale=-1]
{
	\draw[dotted] (.5,.5) circle(0.707);
	\draw (.5,.5-.707) .. controls (.5,.5-.35) and (.1,.5-.35) ..(.1,.5)
		.. controls (.1,.5+.35) and (.5,.5+.35) .. (.5,.5+.707);
	\fill[fill=white] (.25,.75) circle (.1);
	\fill[fill=white] (.25,.25) circle (.1);
	\draw(1,0) -- (0,1);
	\fill[fill=white] (.5,.5) circle (.1);
	\draw(0,0) -- (1,1);
}
\end{align*}
Then, there is an isomorphism $ T \cong  T'$ in $\cD^\gradC(H^n, H^m)$.
\end{lem}

\begin{proof}
We will suppose all strands are oriented upward, the other cases being similar. Then, we have
\begin{align*}
\kh \tikzdiagh[scale=.75]
{
	\draw[dotted] (.5,.5) circle(0.707);
	\draw[->] (.5,.5-.707) .. controls (.5,.5-.35) and (.1,.5-.35) ..(.1,.5)
		.. controls (.1,.5+.35) and (.5,.5+.35) .. (.5,.5+.707);
	\fill[fill=white] (.25,.75) circle (.1);
	\fill[fill=white] (.25,.25) circle (.1);
	\draw[->](1,0) -- (0,1);
	\fill[fill=white] (.5,.5) circle (.1);
	\draw[->](0,0) -- (1,1);
}
&\cong
\cone\left(
\tqft \tikzdiagh[scale=.75]
{
	\draw[dotted] (.5,.5) circle(0.707);
	\draw[->] (.5,.5-.707) .. controls (.5,.5-.35) and (.1,.5-.35) ..(.1,.5)
		.. controls (.1,.5+.35) and (.5,.5+.35) .. (.5,.5+.707);
	\fill[fill=white] (.25,.75) circle (.1);
	\fill[fill=white] (.25,.25) circle (.1);
	\draw[->](1,0) .. controls (.75,.25) and (.75,.75) .. (1,1);
	\draw[->](0,0) .. controls (.4,.25) and (.4,.75) .. (0,1);
}
\xrightarrow{\ \tqft\bigl( \ \tikzdiagh[scale=.2]{
	\draw (.5,3) -- (.5,1) -- (4.5,1) -- (4.5,3);
	\filldraw [fill=white, draw=white] (0,0) rectangle (4,2); 
	\draw (0,2) .. controls (1.5,2) and (2,3) .. (.5,3); 
	\draw (4,2) .. controls (2.5,2) and (3,3) .. (4.5,3); 
	\draw (1.375,2.5) .. controls (1.375,1.5) and (3.125,1.5) .. (3.125,2.5);
	%
	\draw (0,2) -- (0,0) -- (4,0) -- (4,2);
	\filldraw [fill=white, draw = black] (-.5,-1) rectangle (3.5,1);
} \ \bigr )  \circ \shiftFunct{H} \ } 
(\shiftFunct{\tikzdiagh[scale=.2]{
	\draw (.5,3) -- (.5,1) -- (4.5,1) -- (4.5,3);
	\filldraw [fill=white, draw=white] (0,0) rectangle (4,2); 
	\draw (0,2) .. controls (1.5,2) and (2,3) .. (.5,3); 
	\draw (4,2) .. controls (2.5,2) and (3,3) .. (4.5,3); 
	\draw (1.375,2.5) .. controls (1.375,1.5) and (3.125,1.5) .. (3.125,2.5);
	%
	\draw (0,2) -- (0,0) -- (4,0) -- (4,2);
	\filldraw [fill=white, draw = black] (-.5,-1) rectangle (3.5,1);
} })^{-1}
\tqft \tikzdiagh[scale=.75]
{
	\draw[dotted] (.5,.5) circle(0.707);
	\draw[->] (.5,.5-.707) .. controls (.5,.5-.35) and (.1,.5-.35) ..(.1,.5)
		.. controls (.1,.5+.35) and (.5,.5+.35) .. (.5,.5+.707);
	\fill[fill=white] (.25,.75) circle (.1);
	\fill[fill=white] (.25,.25) circle (.1);
	\draw[->](0,0) .. controls (.25,.35) and (.75,.35) .. (1,0);
	\draw[<-](0,1) .. controls (.25,.65) and (.75,.65) .. (1,1);
}
\right)[1]\{0,1\},
\\
\kh \tikzdiagh[scale=.75]
{
	\draw[->][dotted] (.5,.5) circle(0.707);
	\draw[->] (.5,.5-.707) .. controls (.5,.5-.35) and (.9,.5-.35) ..(.9,.5)
		.. controls (.9,.5+.35) and (.5,.5+.35) .. (.5,.5+.707);
	\fill[fill=white] (.75,.75) circle (.1);
	\fill[fill=white] (.75,.25) circle (.1);
	\draw[->](1,0) -- (0,1);
	\fill[fill=white] (.5,.5) circle (.1);
	\draw[->](0,0) -- (1,1);
}
&\cong
\cone\left(
\tqft \tikzdiagh[xscale=-.75,yscale=.75]
{
	\draw[dotted] (.5,.5) circle(0.707);
	\draw[->] (.5,.5-.707) .. controls (.5,.5-.35) and (.1,.5-.35) ..(.1,.5)
		.. controls (.1,.5+.35) and (.5,.5+.35) .. (.5,.5+.707);
	\fill[fill=white] (.25,.75) circle (.1);
	\fill[fill=white] (.25,.25) circle (.1);
	\draw[->](1,0) .. controls (.75,.25) and (.75,.75) .. (1,1);
	\draw[->](0,0) .. controls (.4,.25) and (.4,.75) .. (0,1);
}
\xrightarrow{\ \tqft\bigl( \ \tikzdiagh[scale=.2]{
	\filldraw [fill=white, draw = black] (1,2) rectangle (5,4);
	\filldraw [fill=white, draw = white] (0,2) rectangle (4.5,3);
	\draw (.5,3) -- (.5,1) -- (4.5,1) -- (4.5,3);
	\filldraw [fill=white, draw=white] (0,0) rectangle (4,2); 
	\draw (0,2) .. controls (1.5,2) and (2,3) .. (.5,3); 
	\draw (4,2) .. controls (2.5,2) and (3,3) .. (4.5,3); 
	\draw (1.375,2.5) .. controls (1.375,1.5) and (3.125,1.5) .. (3.125,2.5);
	%
	\draw (0,2) -- (0,0) -- (4,0) -- (4,2);
} \ \bigr )  \circ \shiftFunct{H'} \ } 
(\shiftFunct{\tikzdiagh[scale=.2]{
	\filldraw [fill=white, draw = black] (1,2) rectangle (5,4);
	\filldraw [fill=white, draw = white] (0,2) rectangle (4.5,3);
	\draw (.5,3) -- (.5,1) -- (4.5,1) -- (4.5,3);
	\filldraw [fill=white, draw=white] (0,0) rectangle (4,2); 
	\draw (0,2) .. controls (1.5,2) and (2,3) .. (.5,3); 
	\draw (4,2) .. controls (2.5,2) and (3,3) .. (4.5,3); 
	\draw (1.375,2.5) .. controls (1.375,1.5) and (3.125,1.5) .. (3.125,2.5);
	%
	\draw (0,2) -- (0,0) -- (4,0) -- (4,2);
} })^{-1}
\tqft \tikzdiagh[xscale=-.75,yscale=.75]
{
	\draw[dotted] (.5,.5) circle(0.707);
	\draw[->] (.5,.5-.707) .. controls (.5,.5-.35) and (.1,.5-.35) ..(.1,.5)
		.. controls (.1,.5+.35) and (.5,.5+.35) .. (.5,.5+.707);
	\fill[fill=white] (.25,.75) circle (.1);
	\fill[fill=white] (.25,.25) circle (.1);
	\draw[->](0,0) .. controls (.25,.35) and (.75,.35) .. (1,0);
	\draw[<-](0,1) .. controls (.25,.65) and (.75,.65) .. (1,1);
}
\right)[1]\{0,1\}.
\end{align*}
By \cref{lem:RII} and \cref{lem:RIIb}, we obtain 
\begin{align*}
\kh \tikzdiagh[scale=.75]
{
	\draw[dotted] (.5,.5) circle(0.707);
	\draw[->] (.5,.5-.707) .. controls (.5,.5-.35) and (.1,.5-.35) ..(.1,.5)
		.. controls (.1,.5+.35) and (.5,.5+.35) .. (.5,.5+.707);
	\fill[fill=white] (.25,.75) circle (.1);
	\fill[fill=white] (.25,.25) circle (.1);
	\draw[->](1,0) -- (0,1);
	\fill[fill=white] (.5,.5) circle (.1);
	\draw[->](0,0) -- (1,1);
}
&\cong
\cone\left(
\tqft \tikzdiagh[scale=.75]
{
	\draw[dotted] (.5,.5) circle(0.707);
	\draw[->] (.5,.5-.707) --(.5,.5+.707);
	\fill[fill=white] (.25,.75) circle (.1);
	\fill[fill=white] (.25,.25) circle (.1);
	\draw[->](1,0) .. controls (.75,.25) and (.75,.75) .. (1,1);
	\draw[->](0,0) .. controls (.25,.25) and (.25,.75) .. (0,1);
}
\{-1, 1\}
\xrightarrow{F_L}
(\shiftFunct{\tikzdiagh[scale=.2]{
	\draw (.5,3) -- (.5,1) -- (4.5,1) -- (4.5,3);
	\filldraw [fill=white, draw=white] (0,0) rectangle (4,2); 
	\draw (0,2) .. controls (1.5,2) and (2,3) .. (.5,3); 
	\draw (4,2) .. controls (2.5,2) and (3,3) .. (4.5,3); 
	\draw (1.375,2.5) .. controls (1.375,1.5) and (3.125,1.5) .. (3.125,2.5);
	%
	\draw (0,2) -- (0,0) -- (4,0) -- (4,2);
	\filldraw [fill=white, draw = black] (-.5,-1) rectangle (3.5,1);
} })^{-1}
\tqft \tikzdiagh[scale=.75]
{
	\draw[dotted] (.5,.5) circle(0.707);
	\draw[->] (.5,.5-.707) .. controls (.5,.5-.35) and (.1,.5-.35) ..(.1,.5)
		.. controls (.1,.5+.35) and (.5,.5+.35) .. (.5,.5+.707);
	\fill[fill=white] (.25,.75) circle (.1);
	\fill[fill=white] (.25,.25) circle (.1);
	\draw[->](0,0) .. controls (.25,.35) and (.75,.35) .. (1,0);
	\draw[<-](0,1) .. controls (.25,.65) and (.75,.65) .. (1,1);
}
\right)[1]\{0,1\},
\\
\kh \tikzdiagh[scale=.75]
{
	\draw[->][dotted] (.5,.5) circle(0.707);
	\draw[->] (.5,.5-.707) .. controls (.5,.5-.35) and (.9,.5-.35) ..(.9,.5)
		.. controls (.9,.5+.35) and (.5,.5+.35) .. (.5,.5+.707);
	\fill[fill=white] (.75,.75) circle (.1);
	\fill[fill=white] (.75,.25) circle (.1);
	\draw[->](1,0) -- (0,1);
	\fill[fill=white] (.5,.5) circle (.1);
	\draw[->](0,0) -- (1,1);
}
&\cong
\cone\left(
\tqft \tikzdiagh[scale=.75]
{
	\draw[dotted] (.5,.5) circle(0.707);
	\draw[->] (.5,.5-.707) --(.5,.5+.707);
	\fill[fill=white] (.25,.75) circle (.1);
	\fill[fill=white] (.25,.25) circle (.1);
	\draw[->](1,0) .. controls (.75,.25) and (.75,.75) .. (1,1);
	\draw[->](0,0) .. controls (.25,.25) and (.25,.75) .. (0,1);
}
 \{-1,1\}
\xrightarrow{F_R}
(\shiftFunct{\tikzdiagh[scale=.2]{
	\filldraw [fill=white, draw = black] (1,2) rectangle (5,4);
	\filldraw [fill=white, draw = white] (0,2) rectangle (4.5,3);
	\draw (.5,3) -- (.5,1) -- (4.5,1) -- (4.5,3);
	\filldraw [fill=white, draw=white] (0,0) rectangle (4,2); 
	\draw (0,2) .. controls (1.5,2) and (2,3) .. (.5,3); 
	\draw (4,2) .. controls (2.5,2) and (3,3) .. (4.5,3); 
	\draw (1.375,2.5) .. controls (1.375,1.5) and (3.125,1.5) .. (3.125,2.5);
	%
	\draw (0,2) -- (0,0) -- (4,0) -- (4,2);
} })^{-1}
\tqft \tikzdiagh[xscale=-.75,yscale=.75]
{
	\draw[dotted] (.5,.5) circle(0.707);
	\draw[->] (.5,.5-.707) .. controls (.5,.5-.35) and (.1,.5-.35) ..(.1,.5)
		.. controls (.1,.5+.35) and (.5,.5+.35) .. (.5,.5+.707);
	\fill[fill=white] (.25,.75) circle (.1);
	\fill[fill=white] (.25,.25) circle (.1);
	\draw[->](0,0) .. controls (.25,.35) and (.75,.35) .. (1,0);
	\draw[<-](0,1) .. controls (.25,.65) and (.75,.65) .. (1,1);
}
\right)[1]\{0,1\},
\end{align*}
for some maps $F_L$ and $F_R$ obtained by composition. 

Consider the following diagram
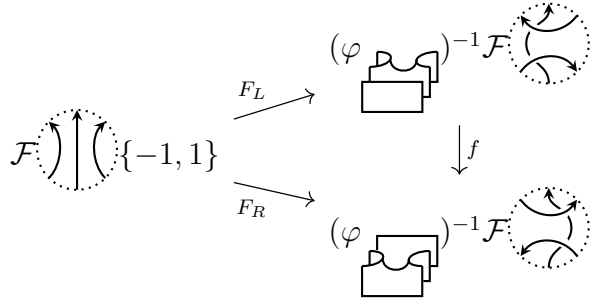
\[
\begin{tikzcd}[row sep=-2ex,column sep= 6ex]
&
(\shiftFunct{\tikzdiagh[scale=.2]{
	\draw (.5,3) -- (.5,1) -- (4.5,1) -- (4.5,3);
	\filldraw [fill=white, draw=white] (0,0) rectangle (4,2); 
	\draw (0,2) .. controls (1.5,2) and (2,3) .. (.5,3); 
	\draw (4,2) .. controls (2.5,2) and (3,3) .. (4.5,3); 
	\draw (1.375,2.5) .. controls (1.375,1.5) and (3.125,1.5) .. (3.125,2.5);
	%
	\draw (0,2) -- (0,0) -- (4,0) -- (4,2);
	\filldraw [fill=white, draw = black] (-.5,-1) rectangle (3.5,1);
} })^{-1}
\tqft \tikzdiagh[scale=.75]
{
	\draw[dotted] (.5,.5) circle(0.707);
	\draw[->] (.5,.5-.707) .. controls (.5,.5-.35) and (.1,.5-.35) ..(.1,.5)
		.. controls (.1,.5+.35) and (.5,.5+.35) .. (.5,.5+.707);
	\fill[fill=white] (.25,.75) circle (.1);
	\fill[fill=white] (.25,.25) circle (.1);
	\draw[->](0,0) .. controls (.25,.35) and (.75,.35) .. (1,0);
	\draw[<-](0,1) .. controls (.25,.65) and (.75,.65) .. (1,1);
}
\ar{dd}{f}
\\
\tqft \tikzdiagh[scale=.75]
{
	\draw[dotted] (.5,.5) circle(0.707);
	\draw[->] (.5,.5-.707) --(.5,.5+.707);
	\fill[fill=white] (.25,.75) circle (.1);
	\fill[fill=white] (.25,.25) circle (.1);
	\draw[->](1,0) .. controls (.75,.25) and (.75,.75) .. (1,1);
	\draw[->](0,0) .. controls (.25,.25) and (.25,.75) .. (0,1);
}
 \{-1,1\}
 \ar{ur}{F_L}
 \ar[swap]{dr}{F_R}
 &
\\
&
(\shiftFunct{\tikzdiagh[scale=.2]{
	\filldraw [fill=white, draw = black] (1,2) rectangle (5,4);
	\filldraw [fill=white, draw = white] (0,2) rectangle (4.5,3);
	\draw (.5,3) -- (.5,1) -- (4.5,1) -- (4.5,3);
	\filldraw [fill=white, draw=white] (0,0) rectangle (4,2); 
	\draw (0,2) .. controls (1.5,2) and (2,3) .. (.5,3); 
	\draw (4,2) .. controls (2.5,2) and (3,3) .. (4.5,3); 
	\draw (1.375,2.5) .. controls (1.375,1.5) and (3.125,1.5) .. (3.125,2.5);
	%
	\draw (0,2) -- (0,0) -- (4,0) -- (4,2);
} })^{-1}
\tqft \tikzdiagh[xscale=-.75,yscale=.75]
{
	\draw[dotted] (.5,.5) circle(0.707);
	\draw[->] (.5,.5-.707) .. controls (.5,.5-.35) and (.1,.5-.35) ..(.1,.5)
		.. controls (.1,.5+.35) and (.5,.5+.35) .. (.5,.5+.707);
	\fill[fill=white] (.25,.75) circle (.1);
	\fill[fill=white] (.25,.25) circle (.1);
	\draw[->](0,0) .. controls (.25,.35) and (.75,.35) .. (1,0);
	\draw[<-](0,1) .. controls (.25,.65) and (.75,.65) .. (1,1);
}
\end{tikzcd}
\]
where $f$ is the obvious graded isomorphism obtained after decomposing 
$(\shiftFunct{\tikzdiagh[scale=.2]{
	\draw (.5,3) -- (.5,1) -- (4.5,1) -- (4.5,3);
	\filldraw [fill=white, draw=white] (0,0) rectangle (4,2); 
	\draw (0,2) .. controls (1.5,2) and (2,3) .. (.5,3); 
	\draw (4,2) .. controls (2.5,2) and (3,3) .. (4.5,3); 
	\draw (1.375,2.5) .. controls (1.375,1.5) and (3.125,1.5) .. (3.125,2.5);
	%
	\draw (0,2) -- (0,0) -- (4,0) -- (4,2);
	\filldraw [fill=white, draw = black] (-.5,-1) rectangle (3.5,1);
} })^{-1}
\tqft \tikzdiagh[scale=.75]
{
	\draw[dotted] (.5,.5) circle(0.707);
	\draw[->] (.5,.5-.707) .. controls (.5,.5-.35) and (.1,.5-.35) ..(.1,.5)
		.. controls (.1,.5+.35) and (.5,.5+.35) .. (.5,.5+.707);
	\fill[fill=white] (.25,.75) circle (.1);
	\fill[fill=white] (.25,.25) circle (.1);
	\draw[->](0,0) .. controls (.25,.35) and (.75,.35) .. (1,0);
	\draw[<-](0,1) .. controls (.25,.65) and (.75,.65) .. (1,1);
}$ 
and
$(\shiftFunct{\tikzdiagh[scale=.2]{
	\filldraw [fill=white, draw = black] (1,2) rectangle (5,4);
	\filldraw [fill=white, draw = white] (0,2) rectangle (4.5,3);
	\draw (.5,3) -- (.5,1) -- (4.5,1) -- (4.5,3);
	\filldraw [fill=white, draw=white] (0,0) rectangle (4,2); 
	\draw (0,2) .. controls (1.5,2) and (2,3) .. (.5,3); 
	\draw (4,2) .. controls (2.5,2) and (3,3) .. (4.5,3); 
	\draw (1.375,2.5) .. controls (1.375,1.5) and (3.125,1.5) .. (3.125,2.5);
	%
	\draw (0,2) -- (0,0) -- (4,0) -- (4,2);
} })^{-1}
\tqft \tikzdiagh[xscale=-.75,yscale=.75]
{
	\draw[dotted] (.5,.5) circle(0.707);
	\draw[->] (.5,.5-.707) .. controls (.5,.5-.35) and (.1,.5-.35) ..(.1,.5)
		.. controls (.1,.5+.35) and (.5,.5+.35) .. (.5,.5+.707);
	\fill[fill=white] (.25,.75) circle (.1);
	\fill[fill=white] (.25,.25) circle (.1);
	\draw[->](0,0) .. controls (.25,.35) and (.75,.35) .. (1,0);
	\draw[<-](0,1) .. controls (.25,.65) and (.75,.65) .. (1,1);
}$, 
and applying the corresponding changes of chronology. 
We observe that the diagram is commutative, since the underlying maps are given by the same cobordisms, which can be thought as all being normalized through changes of chronology. This implies that $\kh(T) \cong \kh(T')$. 
\end{proof}

\begin{thm}\label{thm:isotanglesG}
If $T$ and $T'$ are isotopic $(n,m)$-tangles, then there exists a shifting functor $\shiftFunct{\chcob^v}$ such that there is an isomorphism
\[
\shiftFunct{\chcob^v} \bigl( \kh(T) \bigr) \xrightarrow{\simeq} \kh(T'),
\]
in $\cD^\gradC(H^n, H^m)$.
\end{thm}

\begin{proof}
This follows immediatly from \cref{lem:RI}, \cref{lem:RIb}, \cref{lem:RII}, \cref{lem:RIIb}, and \cref{lem:RIII}.
\end{proof}

\subsection{Tangle invariant}

Consider the \emph{degree collapsing map} $\kappa : \Hom_\gradC \rightarrow \bZ$ given by
\[
\kappa(t,(p_2,p_1)) = p_2 + p_1
\]
for $(t,(p_2,p_1)) \in \Hom_\gradC(a,b)$.

For $t \in B_n^m$, let $\tqft_q(t)$ be the $\bZ \times \gradC$-graded space given by the same elements as $\tqft(t)$ but where $m \in \tqft_q(t)$ has degree
\[
\deg_{\bZ\times\gradC}(m) := (\kappa(\deg_\gradC(m)) + n, \deg_\gradC(m)).
\]
We denote $\deg_q(m)$ for the $\bZ$-degree of $m$, and we refer to it as \emph{quantum degree}.

Note that for a cobordism $\chcob : t \rightarrow t'$, then $\tqft(\chcob) : \tqft_q(t) \rightarrow \tqft_q(t')$ is homogenenous with respect to the quantum degree, and is of degree
\[
\deg_q(\tqft(\chcob)) = \{ \text{\#births} + \text{\#deaths} - \text{\#saddles} \}.
\]
 Moreover, the composition map $\mu[t,t']$ preserves the quantum degree for all $t,t' \in B_\bullet^\bullet$. 

We extend the shifting functor $\shiftFunct{\chcob^{(v_2,v_1)}}$ to a shifting functor for the $\bZ \times \gradC$-grading by setting 
\[
\deg_q\bigl( \shiftFunct{\chcob^{(v_2,v_1)}}(m) ) := \deg_q(m) + v_2 + v_1 + \deg_q(\tqft(\chcob)),
\]
for all $m \in M \in \Mod^{\bZ\times\gradC}$. 
All the results in \cref{sec:gradCat} extend directly to the $\bZ \times \gradC$-case, using the same compatibility maps, in the sense that all isomorphisms involved are graded with respect to the quantum degree. 

All this means $H^n$ is $\bZ$-graded, and $\tqft_q(t)$ is a $\bZ$-graded $H^m$-$H^n$-bimodule (in the $\gradC$-graded sense). Furthermore, the differential $d_{\xi,j}$ of \cref{sec:tangleres} also preserves the quantum grading. Thus, we can define a $\tqft_q(T)$ in a similar fashion, using $\tqft_q(t)$ instead of $\tqft(t)$, and it is a $\bZ$-graded dg-bimodule. We define similarly $\kh_q(T)$. 

Consider the additive subcategory $\BIMOD_{q}(H^m, H^n) \subset \BIMOD^{\gradC}(H^m, H^n)$ with objects being $\bZ \times \gradC$-bimodules and maps are homogeneous with respect to the $\gradC$-grading and preserve the quantum grading. Let $\cKOM_{q}(H^m, H^n)$ be the corresponding homotopy category. We can think of $\kh_q(T)$ as an object of $\cKOM_{q}(H^m, H^n)$. 

\begin{rem}
Note that since we consider maps that do not preserve the $\gradC$-grading, it means the homology of an object in $\cKOM_{q}(H^m, H^n)$ is a $\bZ \times \bZ$-graded (for the homological and the quantum grading) $R$-module. 
\end{rem}

\begin{thm}[\emph{Gluing property}]
Let $T'$ be an oriented $(m',n)$-tangle and $T$ an oriented $(n,m)$-tangle with compatible orientations. There is an isomorphism
\[
\kh_q(T') \otimes_n \kh_q(T) \cong \kh_q(T'T),
\]
in $\cKOM_{q}(H^m, H^n)$.
\end{thm}

\begin{proof}
This follows from \cref{prop:comptangles} together with the fact that the composition maps $\mu[t',t]$ preserve the quantum degree. 
\end{proof}

\begin{thm}[\emph{Invariant property}]
If $T$ and $T'$ are isotopic $(n,m)$-tangles, then there is an isomorphism
\[
\kh_q(T) \xrightarrow{\simeq} \kh_q(T'),
\]
in $\cKOM_{q}(H^m, H^n)$. 
\end{thm}

\begin{proof}
By \cite[\S 7]{putyra14}, we know that the maps in \cref{lem:RI}, \cref{lem:RIb}, \cref{lem:RII}, \cref{lem:RIIb}, and \cref{lem:RIII} are homotopy equivalences (they are actually parts of strong homotopy retract, see \cite[Definition 7.2]{putyra14}). Thus, we only need to verify that they are all graded with respect to the quantum degree. It is trivial for the Reidemeister moves I and III. For the Reidemeister II move, we verify that applying the shifting functor $\{-1,1\}$ is invariant for the quantum grading, and so is 
$\bigl(
\shiftFunct{\tikzdiagh[scale=.2]{
	\draw (.5,3) -- (.5,1) -- (4.5,1) -- (4.5,3);
	\filldraw [fill=white, draw=white] (0,0) rectangle (4,2); 
	\draw (0,2) .. controls (1.5,2) and (2,3) .. (.5,3); 
	\draw (4,2) .. controls (2.5,2) and (3,3) .. (4.5,3); 
	\draw (1.375,2.5) .. controls (1.375,1.5) and (3.125,1.5) .. (3.125,2.5);
	%
	\draw (0,2) -- (0,0) -- (4,0) -- (4,2);
}^{(0,1)}} 
\bigr)$.
\end{proof}

\begin{prop}\label{prop:recoverKh}
If $T$ is a link, then $H(\kh_q(T))$ is isomorphic to the covering Khovanov homology $H^{cov}(T)$ of $T$, as constructed in \cite{putyra14}. 
\end{prop}

\begin{proof} 
The cube of resolution we obtain is the same as the one constructed in~\cite[\S 5]{putyra14},
 up to multiplying each edge by an invertible scalar (if we forget the change of chronology part in the definition of $d_{\xi,i}$, we obtain the same cube). 
 Since the cube we have is anticommutive, it means the scalars we get from the changes of chronology yield a sign assignment (in the sense of \cite[\S 5]{putyra14}). 
 Thus, by \cite[Lemma 5.7]{putyra14}, we conclude our cube is isomorphic to the anticommutative cube of \cite[\S 5]{putyra14}. By consequence, $H(\kh_q(T)) \cong H^{cov}(T)$. 
\end{proof}

In particular, it means that if we take $X=Y=Z=1$ (before computing homology), we recover the usual Khovanov homology of the tangle \cite{khovanovHn}, and if we take $X=Z=1$ and $Y=-1$, we obtain a tangle version of odd Khovanov homology \cite{ORS}, thanks to \cite[Proposition 10.8]{putyra14}.



\section{A chronological half 2-Kac--Moody}\label{sec:2action}

We generalize the level 2 cyclotomic odd half 2-Kac--Moody from~\cite{pedrodd} to one that corresponds with the covering TQFT of \cite{putyra14}. 
Then, we construct a graded 2-action from it on $\BIMOD_q(H^\bullet, H^\bullet)$ similar the one constructed by Brundan--Stroppel~\cite{brundanstroppel3} in the even case. 

\subsection{Chronological cyclotomic half 2-Kac--Moody}

Let $\Bbbk$ be a unital, commutative ring. Fix some natural integer $n > 0$. 
Fix a set of type $A_{n+1}$ simple roots $\simpleRootsSet := \{\alpha_1, \dots, \alpha_n \}$. 
We consider $\mathfrak{gl}_{n+1}$-type weights of the form $\bw = (w_0, \dots, w_{n})$ where $w_i \in \{0,1,2\}$. For $1 \leq i \leq n$, we say there is an arrow
\[
\F_i : \bw \rightarrow \bw',
\]
whenever $w'_{i-1} = w_i-1$, $w'_{i} = w_i+1$  and $w'_j = w_i$ for $j \neq i-1,i$. Also, there is an arrow
\[
\F_i^{(2)} : \bw \rightarrow \bw'',
\]
whenever there are arrows $\F_i : \bw \rightarrow \bw'$ and $\F_i : \bw' \rightarrow \bw''$. These arrows can be composed in the obvious way (but $\F_i\F_i \neq \F_i^{(2)}$). Define
\[
\CSeq := \bigl\{i_r^{(\varepsilon_r)} \cdots i_1^{(\varepsilon_1)} | \varepsilon_\ell \in \{1,2\}, i_\ell \in \{1, \dots, n\} \bigr\}.
\]
To $\bi = i_r^{(\varepsilon_r)} \cdots i_1^{(\varepsilon_1)} \in \CSeq$ we assign a composition of arrows $\F_\bi := \F_{i_r}^{(\varepsilon_r)} \cdots \F_{i_1}^{(\varepsilon_1)}$, where $\F_i^{(1)} := \F_i$. 

We also define a map
\[
p : \simpleRootsSet \times \simpleRootsSet \rightarrow \bZ \times \bZ,
\quad
p(\alpha_i, \alpha_j) := 
\begin{cases}
0, & \text{if $|i-j| > 1$}, \\
(-1,0), & \text{if $j = i +1$}, \\
(0,-1), &\text{if $j = i - 1$}, \\
(1,1),& \text{if $j = i$}.
\end{cases}
\]
In order to shorten notations, we will simply write $p_{ij} := p(\alpha_i, \alpha_j)$. 

\begin{defn}
Let $\cR_0$ be the $R$-algebra generated by string diagrams where there can be simple and double strands, and they both carry a label in $\simpleRootsSet$. Simple strands can also carry decorations in the form of dots and a double strand can split in two strands and vice versa (preserving the labels). 
These diagrams are equipped with a height function such that there can not be two critical points (i.e. crossing, dot or splitter) at the same height. 
 Composition of a diagram $D'$ with a diagram $D$ is given by putting $D'$ on top of $D$, if the labels and the strands corresponds (i.e. we cannot connect a double strand with a single strand) and zero otherwise. Furthermore, these diagrams are $\bZ \times \bZ$-graded where the generators are:
\begin{align*}
\deg_R\left( 
\tikzdiagc{
	\draw (0,0) node[below]{\small $\alpha_i$} -- (0,1) node[midway, tikzdot]{};
}
\right) &= (-1,-1),
&
\deg_R\left(
\tikzdiagc{
	\draw[ds] (0,0) node[below]{\small $\alpha_i$}  -- (0,.5);
	\draw (-.03,.5) .. controls (-.03,.75) and (-.5,.75) .. (-.5,1);
	\draw (.03,.5) .. controls (.03,.75) and (.5,.75) .. (.5,1);
}
\right) &= (1,0),
&
\deg_R\left(
\tikzdiagc[yscale=-1]{
	\draw[ds] (0,0) -- (0,.5);
	\draw (-.03,.5) .. controls (-.03,.75) and (-.5,.75) .. (-.5,1) node[below]{\small $\alpha_i$} ;
	\draw (.03,.5) .. controls (.03,.75) and (.5,.75) .. (.5,1) node[below]{\small $\alpha_i$} ;
}
\right) &= (0,1),
\end{align*}
\begin{align*}
\deg_R\left(
\tikzdiagc{
	\draw (0,0) node[below]{\small $\alpha_i$} .. controls (0,.5) and (1,.5) .. (1,1);
	\draw (1,0) node[below]{\small $\alpha_j$} .. controls (1,.5) and (0,.5) .. (0,1);
}
\right) &= p_{ij},
&
\deg_R\left(
\tikzdiagc{
	\draw[ds] (0,0) node[below]{\small $\alpha_i$} .. controls (0,.5) and (1,.5) .. (1,1);
	\draw (1,0) node[below]{\small $\alpha_j$} .. controls (1,.5) and (0,.5) .. (0,1);
}
\right) &= 2p_{ij},
\\
\deg_R\left(
\tikzdiagc{
	\draw (0,0) node[below]{\small $\alpha_i$} .. controls (0,.5) and (1,.5) .. (1,1);
	\draw[ds] (1,0) node[below]{\small $\alpha_j$} .. controls (1,.5) and (0,.5) .. (0,1);
}
\right) &= 2p_{ij},
&
\deg_R\left(
\tikzdiagc{
	\draw[ds] (0,0) node[below]{\small $\alpha_i$} .. controls (0,.5) and (1,.5) .. (1,1);
	\draw[ds] (1,0) node[below]{\small $\alpha_j$} .. controls (1,.5) and (0,.5) .. (0,1);
}
\right) &= 4p_{ij}.
\end{align*}
We consider these diagrams modulo \emph{graded planar isotopy}, meaning that we can apply ambient isotopies on the diagram at the cost of multiplying by a scalar whenever we exchange the height of critical points:
\[
\tikzdiagc{
	\node at(.5,.25) {\small $\dots$};
	\node at(.5,1.75) {\small $\dots$};
	\draw (0,0) -- (0,2);
	\draw (1,0) -- (1,2);
	\filldraw [fill=white, draw=black,rounded corners] (-.25,.375) rectangle (1.25,.875) node[midway] { $W'$};
}
\dots
\tikzdiagc{
	\node at(.5,.25) {\small $\dots$};
	\node at(.5,1.75) {\small $\dots$};
	\draw (0,0) -- (0,2);
	\draw (1,0) -- (1,2);
	\filldraw [fill=white, draw=black,rounded corners] (-.25,1.125) rectangle (1.25,1.625) node[midway] { $W$};
}
\ = \lambda_R(|W|,|W'|) \ 
\tikzdiagc{
	\node at(.5,.25) {\small $\dots$};
	\node at(.5,1.75) {\small $\dots$};
	\draw (0,0) -- (0,2);
	\draw (1,0) -- (1,2);
	\filldraw [fill=white, draw=black,rounded corners] (-.25,1.125) rectangle (1.25,1.625) node[midway] { $W$};
}
\dots
\tikzdiagc{
	\node at(.5,.25) {\small $\dots$};
	\node at(.5,1.75) {\small $\dots$};
	\draw (0,0) -- (0,2);
	\draw (1,0) -- (1,2);
	\filldraw [fill=white, draw=black,rounded corners] (-.25,.375) rectangle (1.25,.875) node[midway] { $W'$};
}
\]
Finally, let $\cR$ be the quotient of $\cR_0$ by the following local relations:
\begin{align}
\label{eq:doubleR}
\tikzdiagd{
	\draw (0,0) node[below]{\small $\alpha_i$} -- (0,1) node[pos=.33,tikzdot]{} node[pos=.66,tikzdot]{};
}
\ &= 0,
&
\tikzdiagd[yscale=.5]{
	\draw (0,0) node[below]{\small $\alpha_i$} .. controls (0,.5) and (1,.5) .. (1,1) .. controls (1,1.5) and (0,1.5) .. (0,2);
	\draw (1,0) node[below]{\small $\alpha_j$} .. controls (1,.5) and (0,.5) .. (0,1) .. controls (0,1.5) and (1,1.5) .. (1,2);
}
\ &= 
\begin{cases}
\  0 & \text{if $i=j$,} 
\\
-XYZ\ 
\tikzdiagd{
	\draw (0,0) node[below]{\small $\alpha_i$} -- (0,1) node[midway,tikzdot]{};
	\draw (1,0) node[below]{\small $\alpha_{i{+}1}$} -- (1,1) ;
}
\ + XYZ \ 
\tikzdiagd{
	\draw (0,0) node[below]{\small $\alpha_i$} -- (0,1);
	\draw (1,0) node[below]{\small $\alpha_{i{+}1}$} -- (1,1)  node[midway,tikzdot]{};
}
&\text{if $j = i+1$,}
\\
YZ^2\ 
\tikzdiagd{
	\draw (0,0) node[below]{\small $\alpha_i$} -- (0,1) node[midway,tikzdot]{};
	\draw (1,0) node[below]{\small $\alpha_{i{-}1}$} -- (1,1);
}
\ - YZ^2 \ 
\tikzdiagd{
	\draw (0,0) node[below]{\small $\alpha_i$} -- (0,1);
	\draw (1,0) node[below]{\small $\alpha_{i{-}1}$} -- (1,1) node[midway,tikzdot]{};
}
&\text{if $j = i -1$,}
\\
\ 
\tikzdiagd{
	\draw (0,0) node[below]{\small $\alpha_i$} -- (0,1);
	\draw (1,0) node[below]{\small $\alpha_j$} -- (1,1);
}
&\text{otherwise,}
\end{cases}
\end{align}
\begin{align}\label{eq:nilHecke}
\tikzdiagd{
	\draw (0,0) node[below]{\small $\alpha_i$} .. controls (0,.5) and (1,.5) .. (1,1);
	\draw (1,0) node[below]{\small $\alpha_i$} .. controls (1,.5) and (0,.5) .. (0,1) node[near end, tikzdot]{};
}
\ -XY \ 
\tikzdiagd{
	\draw (0,0) node[below]{\small $\alpha_i$} .. controls (0,.5) and (1,.5) .. (1,1);
	\draw (1,0) node[below]{\small $\alpha_i$} .. controls (1,.5) and (0,.5) .. (0,1) node[near start, tikzdot]{};
}
\ = \ 
\tikzdiagd{
	\draw (0,0) node[below]{\small $\alpha_i$} -- (0,1);
	\draw (1,0) node[below]{\small $\alpha_i$}  -- (1,1);
}
\ = \ 
\tikzdiagd{
	\draw (0,0) node[below]{\small $\alpha_i$} .. controls (0,.5) and (1,.5) .. (1,1)node[near start, tikzdot]{};
	\draw (1,0) node[below]{\small $\alpha_i$} .. controls (1,.5) and (0,.5) .. (0,1) ;
}
\ -XY \ 
\tikzdiagd{
	\draw (0,0) node[below]{\small $\alpha_i$} .. controls (0,.5) and (1,.5) .. (1,1)node[near end, tikzdot]{};
	\draw (1,0) node[below]{\small $\alpha_i$} .. controls (1,.5) and (0,.5) .. (0,1) ;
}
\end{align}
\begin{align}\label{eq:dotcommutescrossing}
\begin{split}
\tikzdiagd{
	\draw (0,0) node[below]{\small $\alpha_i$} .. controls (0,.5) and (1,.5) .. (1,1);
	\draw (1,0) node[below]{\small $\alpha_j$} .. controls (1,.5) and (0,.5) .. (0,1) node[near end, tikzdot]{};
}
\ &= \lambda_R((-1,-1), p_{ij})  \ 
\tikzdiagd{
	\draw (0,0) node[below]{\small $\alpha_i$} .. controls (0,.5) and (1,.5) .. (1,1);
	\draw (1,0) node[below]{\small $\alpha_j$} .. controls (1,.5) and (0,.5) .. (0,1) node[near start, tikzdot]{};
}
\\
\tikzdiagd{
	\draw (0,0) node[below]{\small $\alpha_i$} .. controls (0,.5) and (1,.5) .. (1,1) node[near start, tikzdot]{};
	\draw (1,0) node[below]{\small $\alpha_j$} .. controls (1,.5) and (0,.5) .. (0,1);
}
\ &= \lambda_R(p_{ij}, (-1,-1))  \ 
\tikzdiagd{
	\draw (0,0) node[below]{\small $\alpha_i$} .. controls (0,.5) and (1,.5) .. (1,1)  node[near end, tikzdot]{};
	\draw (1,0) node[below]{\small $\alpha_j$} .. controls (1,.5) and (0,.5) .. (0,1);
}
\end{split}
\end{align}
\begin{align}\label{eq:crosssymiszero}
\tikzdiagd{
	\draw (0,0) node[below]{\small $\alpha_i$} .. controls (0,.5) and (1,.5) .. (1,1) ;
	\draw (1,0) node[below]{\small $\alpha_{i\pm 1}$} .. controls (1,.5) and (0,.5) .. (0,1) node[near end, tikzdot]{};
}
\ + XY\  
\tikzdiagd{
	\draw (0,0) node[below]{\small $\alpha_i$} .. controls (0,.5) and (1,.5) .. (1,1) node[near end, tikzdot]{};
	\draw (1,0) node[below]{\small $\alpha_{i\pm 1}$} .. controls (1,.5) and (0,.5) .. (0,1);
}
\ = 0,
\end{align}
\begin{align}  \label{eq:RiiiR3}
\tikzdiagd[scale=1.5]{
	\draw (0,0) node[below]{\small $\alpha_i$} .. controls (0,.5) and (1,.5) .. (1,1);
	\draw (.5,0) node[below]{\small $\alpha_i$} .. controls (.5,.25) and (0,.25) .. (0,.5) .. controls (0,.75) and (.5,.75) .. (.5,1);
	\draw (1,0) node[below]{\small $\alpha_i$} .. controls (1,.5) and (0,.5) .. (0,1);
}
\ &=   \ 
\tikzdiagd[scale=1.5]{
	\draw (0,0) node[below]{\small $\alpha_i$} .. controls (0,.5) and (1,.5) .. (1,1);
	\draw (.5,0) node[below]{\small $\alpha_i$} .. controls (.5,.25) and (1,.25) .. (1,.5) .. controls (1,.75) and (.5,.75) .. (.5,1);
	\draw (1,0) node[below]{\small $\alpha_i$} .. controls (1,.5) and (0,.5) .. (0,1);
}
&
&
\\
\label{eq:RijkR3}
\tikzdiagd[scale=1.5]{
	\draw (0,0) node[below]{\small $\alpha_i$} .. controls (0,.5) and (1,.5) .. (1,1);
	\draw (.5,0) node[below]{\small $\alpha_j$} .. controls (.5,.25) and (0,.25) .. (0,.5) .. controls (0,.75) and (.5,.75) .. (.5,1);
	\draw (1,0) node[below]{\small $\alpha_k$} .. controls (1,.5) and (0,.5) .. (0,1);
}
\ &= \lambda_R(p_{jk}, p_{ik}) \lambda_R(p_{jk}+p_{ik},p_{ij})  \ 
\tikzdiagd[scale=1.5]{
	\draw (0,0) node[below]{\small $\alpha_i$} .. controls (0,.5) and (1,.5) .. (1,1);
	\draw (.5,0) node[below]{\small $\alpha_j$} .. controls (.5,.25) and (1,.25) .. (1,.5) .. controls (1,.75) and (.5,.75) .. (.5,1);
	\draw (1,0) node[below]{\small $\alpha_k$} .. controls (1,.5) and (0,.5) .. (0,1);
}
&
&
\parbox{.25\textwidth}{ for $i\neq k \neq j \neq i$ \\or $|i - k| > 1$,}
\\
 \label{eq:RiijR3}
\tikzdiagd[scale=1.5]{
	\draw (0,0) node[below]{\small $\alpha_i$} .. controls (0,.5) and (1,.5) .. (1,1);
	\draw (.5,0) node[below]{\small $\alpha_j$} .. controls (.5,.25) and (0,.25) .. (0,.5) .. controls (0,.75) and (.5,.75) .. (.5,1);
	\draw (1,0) node[below]{\small $\alpha_k$} .. controls (1,.5) and (0,.5) .. (0,1);
}
\ &= X \lambda_R(p_{jk}, p_{ik}) \lambda_R(p_{jk}+p_{ik},p_{ij})  \ 
\tikzdiagd[scale=1.5]{
	\draw (0,0) node[below]{\small $\alpha_i$} .. controls (0,.5) and (1,.5) .. (1,1);
	\draw (.5,0) node[below]{\small $\alpha_j$} .. controls (.5,.25) and (1,.25) .. (1,.5) .. controls (1,.75) and (.5,.75) .. (.5,1);
	\draw (1,0) node[below]{\small $\alpha_k$} .. controls (1,.5) and (0,.5) .. (0,1);
}
&
&
\parbox{.25\textwidth}{ for $ i = j $ and $ k = \pm i$ \\or $j = k $ and $ i = k \pm 1$,}
\end{align}
\begin{align} \label{eq:RijiR3}
\begin{split}
-YZ^{-2} \tikzdiagd[scale=1.5]{
	\draw (0,0) node[below]{\small $\alpha_i$} .. controls (0,.5) and (1,.5) .. (1,1);
	\draw (.5,0) node[below]{\small $\alpha_{i{+}1}$} .. controls (.5,.25) and (0,.25) .. (0,.5) .. controls (0,.75) and (.5,.75) .. (.5,1);
	\draw (1,0) node[below]{\small $\alpha_i$} .. controls (1,.5) and (0,.5) .. (0,1);
}
\ + Z^{-1}\ 
\tikzdiagd[scale=1.5]{
	\draw (0,0) node[below]{\small $\alpha_i$} .. controls (0,.5) and (1,.5) .. (1,1);
	\draw (.5,0) node[below]{\small $\alpha_{i{+}1}$} .. controls (.5,.25) and (1,.25) .. (1,.5) .. controls (1,.75) and (.5,.75) .. (.5,1);
	\draw (1,0) node[below]{\small $\alpha_i$} .. controls (1,.5) and (0,.5) .. (0,1);
}
 &=  \ 
\tikzdiagd[xscale=.75,yscale=1.5]{
	\draw (0,0) node[below]{\small $\alpha_i$} -- (0,1);
	\draw (1,0) node[below]{\small $\alpha_{i{+}1}$} -- (1,1);
	\draw (2,0) node[below]{\small $\alpha_i$} -- (2,1);
}
\\
XYZ^{-1} \tikzdiagd[scale=1.5]{
	\draw (0,0) node[below]{\small $\alpha_i$} .. controls (0,.5) and (1,.5) .. (1,1);
	\draw (.5,0) node[below]{\small $\alpha_{i{-}1}$} .. controls (.5,.25) and (0,.25) .. (0,.5) .. controls (0,.75) and (.5,.75) .. (.5,1);
	\draw (1,0) node[below]{\small $\alpha_i$} .. controls (1,.5) and (0,.5) .. (0,1);
}
\ - XZ^{-2}\ 
\tikzdiagd[scale=1.5]{
	\draw (0,0) node[below]{\small $\alpha_i$} .. controls (0,.5) and (1,.5) .. (1,1);
	\draw (.5,0) node[below]{\small $\alpha_{i{-}1}$} .. controls (.5,.25) and (1,.25) .. (1,.5) .. controls (1,.75) and (.5,.75) .. (.5,1);
	\draw (1,0) node[below]{\small $\alpha_i$} .. controls (1,.5) and (0,.5) .. (0,1);
}
 &=  \ 
\tikzdiagd[xscale=.75,yscale=1.5]{
	\draw (0,0) node[below]{\small $\alpha_i$} -- (0,1);
	\draw (1,0) node[below]{\small $\alpha_{i{-}1}$} -- (1,1);
	\draw (2,0) node[below]{\small $\alpha_i$} -- (2,1);
}
\end{split}
\end{align}
\begin{align}\label{eq:doublesplitters}
\tikzdiagd{
	\draw[ds] (0,0) -- (0,.25);
	\draw (-.03,.25) .. controls (-.03,.5) and (-.5,.5) .. (-.5,.75);
	\draw (.03,.25) .. controls (.03,.5) and (.5,.5) .. (.5,.75);
	\draw[ds] (0,0) -- (0,-.25);
	\draw (-.03,-.25) .. controls (-.03,-.5) and (-.5,-.5) .. (-.5,-.75) node[below]{\small $\alpha_i$} ;
	\draw (.03,-.25) .. controls (.03,-.5) and (.5,-.5) .. (.5,-.75) node[below]{\small $\alpha_i$} ;
}
\ &= \ 
\tikzdiagd{
	\draw (0,0) node[below]{\small $\alpha_i$}   .. controls (0,.75) and (1,.75) .. (1,1.5);
	\draw (1,0) node[below]{\small $\alpha_i$}   .. controls (1,.75) and (0,.75) .. (0,1.5);
}
&
\tikzdiagd{
	\draw[ds] (0,0)  node[below]{\small $\alpha_i$}  -- (0,.25);
	\draw (-.03,.25) .. controls (-.03,.5) and (-.5,.5) .. (-.5,.75);
	\draw (.03,.25) .. controls (.03,.5) and (.5,.5) .. (.5,.75);
	\draw[ds] (0,1.5) -- (0,1.25);
	\draw (-.03,1.25) .. controls (-.03,1) and (-.5,1) .. (-.5,.75);
	\draw (.03,1.25) .. controls (.03,1) and (.5,1) .. (.5,.75);
}
\ &= 0,
\end{align}
\begin{align}
\tikzdiagd{
	\draw[ds] (0,0)  node[below]{\small $\alpha_i$}  -- (0,.25);
	\draw (-.03,.25) .. controls (-.03,.5) and (-.5,.5) .. (-.5,.75) node[tikzdot,pos=1]{};
	\draw (.03,.25) .. controls (.03,.5) and (.5,.5) .. (.5,.75);
	\draw[ds] (0,1.5) -- (0,1.25);
	\draw (-.03,1.25) .. controls (-.03,1) and (-.5,1) .. (-.5,.75);
	\draw (.03,1.25) .. controls (.03,1) and (.5,1) .. (.5,.75);
}
\ = \ 
\tikzdiagd{
	\draw[ds] (0,0)  node[below]{\small $\alpha_i$} -- (0,1.5)
}
\ = -XY \ 
\tikzdiagd{
	\draw[ds] (0,0)  node[below]{\small $\alpha_i$}  -- (0,.25);
	\draw (-.03,.25) .. controls (-.03,.5) and (-.5,.5) .. (-.5,.75);
	\draw (.03,.25) .. controls (.03,.5) and (.5,.5) .. (.5,.75) node[tikzdot,pos=1]{};
	\draw[ds] (0,1.5) -- (0,1.25);
	\draw (-.03,1.25) .. controls (-.03,1) and (-.5,1) .. (-.5,.75);
	\draw (.03,1.25) .. controls (.03,1) and (.5,1) .. (.5,.75);
}
&&
\tikzdiagd{
	\draw[ds] (0,0)  node[below]{\small $\alpha_i$}  -- (0,.25);
	\draw (-.03,.25) .. controls (-.03,.5) and (-.5,.5) .. (-.5,.75) 
		.. controls (-.5,1.125) and (.5,1.125) .. (.5,1.5);
	\draw (.03,.25) .. controls (.03,.5) and (.5,.5) .. (.5,.75)
		.. controls (.5,1.125) and (-.5,1.125) .. (-.5,1.5);
}
\ = 0 = \ 
\tikzdiagd[yscale=-1]{
	\draw[ds] (0,0)  -- (0,.25);
	\draw (-.03,.25) .. controls (-.03,.5) and (-.5,.5) .. (-.5,.75) 
		.. controls (-.5,1.125) and (.5,1.125) .. (.5,1.5)  node[below]{\small $\alpha_i$} ;
	\draw (.03,.25) .. controls (.03,.5) and (.5,.5) .. (.5,.75)
		.. controls (.5,1.125) and (-.5,1.125) .. (-.5,1.5)  node[below]{\small $\alpha_i$} ;
}
\end{align}
\begin{align}
\label{eq:splitcross}
\begin{split}
\tikzdiagd{
	\draw (0,0) node[below]{\small $\alpha_i$} .. controls (0,1) and (2,1) .. (2,2);
	\draw[ds] (2,0) node[below]{\small $\alpha_j$}  -- (.75,1.25);
	\draw (.75 + 0.02 ,1.25 + 0.02) .. controls (.5 + 0.02, 1.5 + 0.02) and (.75,1.75) .. (.5,2);
	\draw (.75 - 0.02,1.25 - 0.02) .. controls (.5 - 0.02,1.5 - 0.02) and (.25,1.25) .. (0,1.5);
}
\ &= \lambda_R\bigl( (1,0),  2 p_{ij} \bigr) \ 
\tikzdiagd{
	\draw (0,0) node[below]{\small $\alpha_i$} .. controls (0,1) and (2,1) .. (2,2);
	\draw[ds] (2,0) node[below]{\small $\alpha_j$}  -- (1.75,.25);
	\draw (1.75 + 0.02 ,.25 + 0.02) .. controls (1.5 + 0.02, .5 + 0.02) and (1.75,.75) .. (1.5,1) -- (.5,2);
	\draw (1.75 - 0.02,.25 - 0.02) .. controls (1.5 - 0.02,.5 - 0.02) and (1.25,.25) .. (1,.5) -- (0,1.5);
}
\\
\tikzdiagd{
	\draw[ds] (0,0) node[below]{\small $\alpha_i$} .. controls (0,1) and (2,1) .. (2,2);
	\draw[ds] (2,0) node[below]{\small $\alpha_j$}  -- (.75,1.25);
	\draw (.75 + 0.02 ,1.25 + 0.02) .. controls (.5 + 0.02, 1.5 + 0.02) and (.75,1.75) .. (.5,2);
	\draw (.75 - 0.02,1.25 - 0.02) .. controls (.5 - 0.02,1.5 - 0.02) and (.25,1.25) .. (0,1.5);
}
\ &= \lambda_R\bigl( (1,0),  4 p_{ij} \bigr) \ 
\tikzdiagd{
	\draw[ds] (0,0) node[below]{\small $\alpha_i$} .. controls (0,1) and (2,1) .. (2,2);
	\draw[ds] (2,0) node[below]{\small $\alpha_j$}  -- (1.75,.25);
	\draw (1.75 + 0.02 ,.25 + 0.02) .. controls (1.5 + 0.02, .5 + 0.02) and (1.75,.75) .. (1.5,1) -- (.5,2);
	\draw (1.75 - 0.02,.25 - 0.02) .. controls (1.5 - 0.02,.5 - 0.02) and (1.25,.25) .. (1,.5) -- (0,1.5);
}
\end{split}
\\
\label{eq:splitdcross}
\begin{split}
\tikzdiagd[scale=-1]{
	\draw (0,0) .. controls (0,1) and (2,1) .. (2,2) node[below]{\small $\alpha_i$};
	\draw[ds] (2,0) -- (1.75,.25);
	\draw (1.75 + 0.02 ,.25 + 0.02) .. controls (1.5 + 0.02, .5 + 0.02) and (1.75,.75) .. (1.5,1) -- (.5,2)  node[below]{\small $\alpha_j$} ;
	\draw (1.75 - 0.02,.25 - 0.02) .. controls (1.5 - 0.02,.5 - 0.02) and (1.25,.25) .. (1,.5) -- (0,1.5);
}
\ &= \lambda_R\bigl( (0,1),  2 p_{ij} \bigr) \ 
\tikzdiagd[scale=-1]{
	\draw (0,0)  .. controls (0,1) and (2,1) .. (2,2)node[below]{\small $\alpha_i$};
	\draw[ds] (2,0) -- (.75,1.25);
	\draw (.75 + 0.02 ,1.25 + 0.02) .. controls (.5 + 0.02, 1.5 + 0.02) and (.75,1.75) .. (.5,2) node[below]{\small $\alpha_j$} ;
	\draw (.75 - 0.02,1.25 - 0.02) .. controls (.5 - 0.02,1.5 - 0.02) and (.25,1.25) .. (0,1.5);
}
\\
\tikzdiagd[scale=-1]{
	\draw[ds] (0,0) .. controls (0,1) and (2,1) .. (2,2) node[below]{\small $\alpha_i$} ;
	\draw[ds] (2,0)   -- (1.75,.25);
	\draw (1.75 + 0.02 ,.25 + 0.02) .. controls (1.5 + 0.02, .5 + 0.02) and (1.75,.75) .. (1.5,1) -- (.5,2)  node[below]{\small $\alpha_j$};
	\draw (1.75 - 0.02,.25 - 0.02) .. controls (1.5 - 0.02,.5 - 0.02) and (1.25,.25) .. (1,.5) -- (0,1.5);
}
\ &= \lambda_R\bigl( (0,1),  4 p_{ij} \bigr) \ 
\tikzdiagd[scale=-1]{
	\draw[ds] (0,0).. controls (0,1) and (2,1) .. (2,2)  node[below]{\small $\alpha_i$} ;
	\draw[ds] (2,0)  -- (.75,1.25);
	\draw (.75 + 0.02 ,1.25 + 0.02) .. controls (.5 + 0.02, 1.5 + 0.02) and (.75,1.75) .. (.5,2)  node[below]{\small $\alpha_j$};
	\draw (.75 - 0.02,1.25 - 0.02) .. controls (.5 - 0.02,1.5 - 0.02) and (.25,1.25) .. (0,1.5);
}
\end{split}
\end{align}
and the mirror along the vertical axis of \cref{eq:splitcross} and \cref{eq:splitdcross}. 
\end{defn}

\begin{rem}
Since whenever there are two dots on the same strand it is zero, we also have that
\[
\tikzdiagd[xscale=.5]{
	\draw (0,0) node[below]{\small $\alpha_i$} -- (0,1);
	\draw (1,0) node[below]{\small $\alpha_i$} -- (1,1);
	\draw (2,0) node[below]{\small $\alpha_i$} -- (2,1);
}
\ = 0,
\] 
for all $\alpha_i \in \simpleRootsSet$, and similarly if we have two double strands, or a double strand next to a single strand. 
\end{rem}

The algebra $\cR$ also possesses a \emph{quantum  grading} given by collapsing the $R$-grading: $\deg_q(x) = p_2 + p_1$ whenever $\deg_R(x) = (p_2, p_1)$. 
For $\bi \in \CSeq$, we define the idempotent $1_\bi \in \cR$ given by  only vertical strands, with label determined by $i_\ell$ and it is a simple strand whenever $\varepsilon_\ell =1$ or double strand when $\varepsilon_\ell =2$. For example with $\bi = i^{(1)} j^{(1)} i^{(2)} k^{(1)}$ we have:
\[
1_\bi = \ 
\tikzdiagd[xscale=.5]{
	\draw (0,0) node[below]{\small $\alpha_i$} -- (0,1);
	\draw (1,0) node[below]{\small $\alpha_j$} -- (1,1);
	\draw[ds] (2,0) node[below]{\small $\alpha_i$} -- (2,1);
	\draw (3,0) node[below]{\small $\alpha_k$} -- (3,1);
}
\]

Consider a weight $\bw$ and $\bi,\bj \in \CSeq$ such that $\bw' \F_\bj \bw$ and $\bw' \F_\bi \bw$ for the same weight $\bw'$. 
Any diagram $D$ in $1_\bj R 1_\bi$ can be decomposed as a composition of diagrams $\bi_{r} D_r \bi_{r-1} D_{r-1} \bi_{r-2} \cdots \bi_1 D_1 \bi_0$ with $\bi_0 = \bi$ and $\bi_r = \bj$, and where there is exactly one generator (i.e. dot, splitter or crossing) in each $D_\ell$. We say that $D$ is an \emph{illegal diagram} for $\bw',\bw$ whenever there are no arrow $\bw' \F_{\bi_\ell} \bw$ for some $\ell$.  For example with $n=2$, the following diagram is illegal 
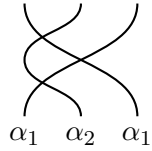
\[ 
\tikzdiagd[scale=1.5]{
	\draw (0,0) node[below]{\small $\alpha_1$} .. controls (0,.5) and (1,.5) .. (1,1);
	\draw (.5,0) node[below]{\small $\alpha_{2}$} .. controls (.5,.25) and (0,.25) .. (0,.5) .. controls (0,.75) and (.5,.75) .. (.5,1);
	\draw (1,0) node[below]{\small $\alpha_1$} .. controls (1,.5) and (0,.5) .. (0,1);
}
\]
for $\bw = (2,1,0)$ and $\bw' = (0,2,1)$, because there are no arrow $\bw'\F_2\F_1\F_1\bw$ (it factorizes through a weight $(3,0,0)$). 
Then, we set 
\[
\HOM(\F_\bi, \F_\bj) := 1_\bj R 1_\bi/I, \quad I := \bigl(\{ D \in  1_\bj R 1_\bi | \text{$D$ is illegal for $\bw, \bw'$} \}\bigr),
\]
where $\F_\bi,\F_\bj : \bw \rightarrow \bw'$. 

We define the category $\Hom(\bw, \bw')$ where objects are given by formal direct sums of formal $\bZ\times\bZ$-shifts of the compositions of the arrows $\F_i$, $\F_i^{(2)}$ and $\id_\bw$ going from $\bw$ to $\bw'$. The morphisms are matrices of elements in $\HOM(\F_\bi, \F_\bj)$ of quantum degree zero, with the composition given by multiplication in $\cR$.

\begin{defn}
Let $\oddKC$ be the 2-category with:
\begin{itemize}
\item objects are $(n+1)$-uples $\bw = (w_0, \dots, w_{n})$ where $w_i \in \{0,1,2\}$;
\item $\Hom_\oddKC(\bw, \bw')$ is given by $\Hom(\bw, \bw')$ as  above;
\item horizontal composition of 1-morphism is concatenation of arrows;
\item horizontal composition of 2-morphisms is given by right-then-left juxtaposition of diagrams:
\[
\tikzdiagh{
	\node at(.5,.1) {\small $\dots$};
	\node at(.5,.9) {\small $\dots$};
	\draw (0,0) -- (0,1);
	\draw (1,0) -- (1,1);
	\filldraw [fill=white, draw=black,rounded corners] (-.25,.25) rectangle (1.25,.75) node[midway] { $W'$};
}
\ \circ\ 
\tikzdiagh{
	\node at(.5,.1) {\small $\dots$};
	\node at(.5,.9) {\small $\dots$};
	\draw (0,0) -- (0,1);
	\draw (1,0) -- (1,1);
	\filldraw [fill=white, draw=black,rounded corners] (-.25,.25) rectangle (1.25,.75) node[midway] { $W$};
}
\ := \ 
\tikzdiagc{
	\node at(.5,.25) {\small $\dots$};
	\node at(.5,1.75) {\small $\dots$};
	\draw (0,0) -- (0,2);
	\draw (1,0) -- (1,2);
	\filldraw [fill=white, draw=black,rounded corners] (-.25,1.125) rectangle (1.25,1.625) node[midway] { $W'$};
}
\tikzdiagc{
	\node at(.5,.25) {\small $\dots$};
	\node at(.5,1.75) {\small $\dots$};
	\draw (0,0) -- (0,2);
	\draw (1,0) -- (1,2);
	\filldraw [fill=white, draw=black,rounded corners] (-.25,.375) rectangle (1.25,.875) node[midway] { $W$};
}
\in \Hom(\F_{\bi'}\F_{\bi}, \F_{\bj'}\F_{\bj}),
\]
for $W' \in \Hom(\F_{\bi'},\F_{\bj'}), W \in \Hom(\F_{\bi}, \F_{\bj})$. 
\end{itemize}
\end{defn}

As usual, specializing $X=Y=Z=1$ recovers a level 2 cyclotomic half 2-Kac--Moody algebra (as in~\cite{khovanovlauda2, brundanstroppel3}), and taking $Y=-1$ instead gives back the construction from~\cite{pedrodd}. 

\subsection{Ladder diagrams}

Given a weight $\bw = (w_0, \dots, w_n)$, we associate to it a string diagram with $n+1$ vertical strands that are labeled by the elements $w_i$. We picture these labels by writing a strand with label $0$ as dashed, label $1$ as solid, and label $2$ as double dashed. We refer to the dashed and double dashed strands as \emph{invisible}, and to the solid strands as \emph{visible}. For example, if we take $\bw = (0,1,1,2,0)$ we obtain:
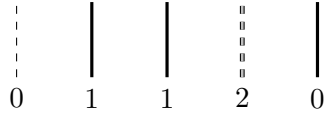
\[
\tikzdiagd[xscale=1]{
	\draw[dt] (0,0) node[below]{\small $0$} -- (0,1);
	\draw[vt] (1,0) node[below]{\small $1$} -- (1,1);
	\draw[vt] (2,0) node[below]{\small $1$} -- (2,1);
	\draw[dds] (3,0) node[below]{\small $2$} -- (3,1);
	\draw[vt] (4,0) node[below]{\small $0$} -- (4,1);
}
\]
Then, to the arrow $\F_i : \bw \rightarrow \bw'$ we associate the following ladder diagram:
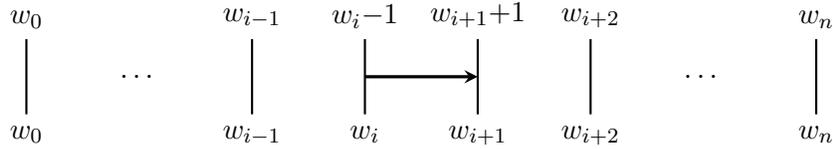
\[
\tikzdiagd[xscale=1.5]{
	\draw (0,0)node[below]{\small $w_0$} -- (0,1) node[above]{\small $w_0$};
	\node at(1,.5) {\small $\dots$};
	\draw (2,0)node[below]{\small $w_{i{-}1}$} -- (2,1)node[above]{\small $w_{i{-}1}$};
	\draw (3,0)node[below]{\small $w_{i}$} -- (3,1) node[above]{\small $w_i{-}1$};
	\draw[->,vt] (3,.5) -- (4,.5);
	\draw (4,0)node[below]{\small $w_{i{+}1}$} -- (4,1)node[above]{\small $w_{i{+}1}{+}1$};
	\draw (5,0)node[below]{\small $w_{i{+}2}$} -- (5,1)node[above]{\small $w_{i{+}2}$};
	\node at(6,.5) {\small $\dots$};
	\draw (7,0)node[below]{\small $w_{n}$} -- (7,1)node[above]{\small $w_n$};
}
\]
where we consider the rung as a solid line, and thus as visible.  
To the arrow $\F_i^{(2)} : \bw \rightarrow \bw'$ we associate
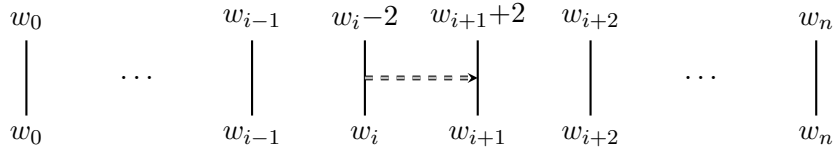
\[
\tikzdiagd[xscale=1.5]{
	\draw (0,0)node[below]{\small $w_0$} -- (0,1) node[above]{\small $w_0$};
	\node at(1,.5) {\small $\dots$};
	\draw (2,0)node[below]{\small $w_{i{-}1}$} -- (2,1)node[above]{\small $w_{i{-}1}$};
	\draw (3,0)node[below]{\small $w_{i}$} -- (3,1) node[above]{\small $w_i{-}2$};
	\draw[->,dds] (3,.5) -- (4,.5);
	\draw (4,0)node[below]{\small $w_{i{+}1}$} -- (4,1)node[above]{\small $w_{i{+}1}{+}2$};
	\draw (5,0)node[below]{\small $w_{i{+}2}$} -- (5,1)node[above]{\small $w_{i{+}2}$};
	\node at(6,.5) {\small $\dots$};
	\draw (7,0)node[below]{\small $w_{n}$} -- (7,1)node[above]{\small $w_n$};
}
\]
where the rung is double dashed, and thus invisible.  
The composition of arrows 
translates to stacking ladder diagram on top of each other, from bottom to top. Let $\ladder(\F_\bi)$ for $\bi \in \CSeq$ be the ladder diagram associated to $\F_\bi$. 

Because $w_i \in \{0,1,2\}$ and $w_i' \in \{0,1,2\}$, there are only 4 possibilities for what the ladder diagram of $\F_i : \bw \rightarrow \bw'$ looks like on the $i$-th and $(i+1)$-th strands:
\begin{align}\label{eq:allFi}
\tikzdiagc{
	\draw[vt] (0,0) -- (0,.5); \draw[dt] (0,.5) -- (0,1);
	\draw[->,vt] (0,.5) -- (1,.5);
	\draw[dt] (1,0) -- (1,.5); \draw[vt] (1,.5) -- (1,1);
}
&&
\tikzdiagc{
	\draw[vt] (0,0) -- (0,.5); \draw[dt] (0,.5) -- (0,1);
	\draw[->, vt] (0,.5) -- (1,.5);
	\draw[vt] (1,0) -- (1,.5); \draw[dds] (1,.5) -- (1,1);
}
&&
\tikzdiagc{
	\draw[dds] (0,0) -- (0,.5); \draw[vt] (0,.5) -- (0,1);
	\draw[->, vt] (0,.5) -- (1,.5);
	\draw[dt] (1,0) -- (1,.5); \draw[vt] (1,.5) -- (1,1);
}
&&
\tikzdiagc{
	\draw[dds] (0,0) -- (0,.5); \draw[vt] (0,.5) -- (0,1);
	\draw[->, vt] (0,.5) -- (1,.5);
	\draw[vt] (1,0) -- (1,.5); \draw[dds] (1,.5) -- (1,1);
}
\end{align}
For $\F_i^{(2)} : \bw \rightarrow \bw'$, there is only one possibility:
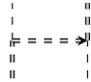
\[
\tikzdiagc{
	\draw[dds] (0,0) -- (0,.5); \draw[dt] (0,.5) -- (0,1);
	\draw[->, dds] (0,.5) -- (1,.5);
	\draw[dt] (1,0) -- (1,.5); \draw[dds] (1,.5) -- (1,1);
}
\]

\begin{exe}
Here is an example of a ladder diagram
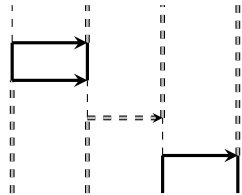
\[
\tikzdiagh[yscale=.5]{
	\draw[dds] (0,0) -- (0,3);
	\draw[vt] (0,3) -- (0,4);
	\draw[dt] (0,4) -- (0,5);
	\draw[dds] (1,0) -- (1,2);
	\draw[dt] (1,2) -- (1,3);
	\draw[vt] (1,3) -- (1,4);
	\draw[dds] (1,4) -- (1,5);
	\draw[vt] (2,0) -- (2,1);
	\draw[dt] (2,1) -- (2,2);
	\draw[dds] (2,2) -- (2,5);
	\draw[vt] (3,0) -- (3,1);
	\draw[dds] (3,1) -- (3,5);
	\draw[vt,->] (2,1) -- (3,1);
	\draw[dds,->] (1,2) -- (2,2);
	\draw[vt,->] (0,3) -- (1,3);
	\draw[vt,->] (0,4) -- (1,4);
}
\]
associated to $\bw' \F_1 \F_1 \F_2^{(2)} \F_3 \bw$ where $\bw = (2,2,1,1)$ and $\bw' = (0,2,2,2)$.
\end{exe}

\subsection{The 2-functor}
Given a weight $\bw$ we write $s(\bw) := \#\{ w_i \in \bw | w_i = 1\}$. We consider the full sub-2-category $\oddKCeven$ of $\oddKC$ with objects $\bw$ being such that $s(\bw) \equiv 0 \mod 2$. 
We also consider the bicategory $\BIMODH$ where objects are $H^n$ for all $n \geq 0$, and $\Hom_{\BIMODH}(H^n, H^m) := \BIMOD_q(H^n, H^m)$. The horizontal composition is given by tensor product $\otimes_H$. 
We will construct a 2-action $\oddKCeven \rightarrow \BIMODH$, similar to the one that can be found in~\cite{brundanstroppel3} (when $X=Y=Z=1$ our 2-action coincides with the one in the reference). 

\begin{rem}
A construction of a generalized covering arc algebra (in the sense of \cite{chenkhovanov, stroppelalgebra}) is possible, allowing to act with the whole $\oddKC$ instead of $\oddKCeven$. For the sake of simplicity, we do not consider this case in this paper. 
\end{rem}

\smallskip

Because of \cref{eq:allFi}, we know that there is at most $2$ visible lines meeting at each vertex of $\ladder(\F_\bi)$ for all $\bi \in \CSeq$. In particular, it means we can turn it into a flat tangle by removing all the invisible lines and smoothing if necessary, so that we turn \eqref{eq:allFi} into:
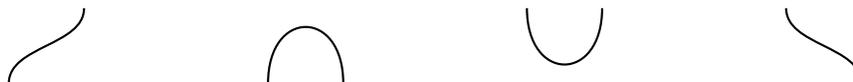
\begin{align*}
\tikzdiagc{
	\draw (0,0) .. controls (0,.5) and (1,.5) .. (1,1);
}
&&
\tikzdiagc{
	\draw (0,0) .. controls (0,1) and (1,1) .. (1,0);	
}
&&
\tikzdiagc{
	\draw (0,1) .. controls (0,0) and (1,0) .. (1,1);
}
&&
\tikzdiagc{
	\draw (0,1) .. controls (0,.5) and (1,.5) .. (1,0);
}
\end{align*}
We write $\ladderToTangle(\F_\bi)$ for the corresponding flat tangle. 

\begin{exe}
For example, we have
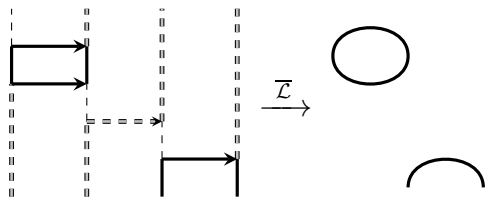
\[
\tikzdiagh[yscale=.5]{
	\draw[dds] (0,0) -- (0,3);
	\draw[vt] (0,3) -- (0,4);
	\draw[dt] (0,4) -- (0,5);
	\draw[dds] (1,0) -- (1,2);
	\draw[dt] (1,2) -- (1,3);
	\draw[vt] (1,3) -- (1,4);
	\draw[dds] (1,4) -- (1,5);
	\draw[vt] (2,0) -- (2,1);
	\draw[dt] (2,1) -- (2,2);
	\draw[dds] (2,2) -- (2,5);
	\draw[vt] (3,0) -- (3,1);
	\draw[dds] (3,1) -- (3,5);
	\draw[vt,->] (2,1) -- (3,1);
	\draw[dds,->] (1,2) -- (2,2);
	\draw[vt,->] (0,3) -- (1,3);
	\draw[vt,->] (0,4) -- (1,4);
}
\ \xrightarrow{\ \ladderToTangle\ } \ 
\tikzdiagh[yscale=.5]{
	\draw[vt] (0,3.5) .. controls (0,4.5) and (1,4.5) .. (1,3.5) .. controls (1,2.5) and (0,2.5) .. (0,3.5);
	\draw[vt] (1,0) .. controls (1,1) and (2,1) .. (2,0);
}
\]
\end{exe}

 This allows us to define a $\bZ\times\bZ$-graded additive 1-map 
\[
\UtoBfunctor_0 : \oddKCeven \rightarrow \BIMODH,
\]
with $\bw \mapsto H^{s(\bw)/2}$, and $\F_\bi : \bw \rightarrow \bw'$ is sent to $\tqft(\ladderToTangle(\F_\bi))$. 

\begin{rem}
We use the term \emph{1-map} to emphasize the fact it is not a functor (the objects and 1-morphisms in $\BIMODH$ do not form a category since $\BIMODH$ is a bicategory). 
\end{rem}

 The idea is to look at $\tqft\bigl(\ladderToTangle(\F_\bi)\bigr)$, and  to assign an homogeneous map of bimodules (in the form of a cobordism) to each generating element of $\cR$, depending on the weight $\bw$. 
First, let us define $\Gamma_\bw(i) \in R$ as 
\begin{align*}
\Gamma_\bw(i) := (-XY)^{\# \{w_j \in \bw | w_j = 1 \text{ for } j \leq i \}},
\end{align*}
where 
the exponent is the number of visible lines on the left or equal to the $i$th vertical line in the ladder diagram associated to $\bw$. 
 Then, we define $\UtoBfunctor_0$ as follow, supposing all diagrams involved to be not illegal:
\begin{itemize}
\item To a dot on a strand labeled $\alpha_i$, 
we associate ${\Gamma_\bw(i)}$ times the cobordism given by the identity on $\ladder(\F_i)$ with a dot above the rung. For example: 
\begin{align*}
\UtoBfunctor_0 & 
\left(
\tikzdiagc{
	\draw (0,0) node[below]{\small $\alpha_i$} -- (0,1) node[midway, tikzdot]{};
}\ : \ 
\tikzdiagh{
	\draw[vt] (0,0) -- (0,.5); \draw[dt] (0,.5) -- (0,1);
	\draw[->,vt] (0,.5) -- (1,.5);
	\draw[dt] (1,0) -- (1,.5); \draw[vt] (1,.5) -- (1,1);
}\ \{-1,-1\}
\quad \rightarrow \quad 
\tikzdiagh{
	\draw[vt] (0,0) -- (0,.5); \draw[dt] (0,.5) -- (0,1);
	\draw[->,vt] (0,.5) -- (1,.5);
	\draw[dt] (1,0) -- (1,.5); \draw[vt] (1,.5) -- (1,1);
}
\right)
 \\
&:=
\Gamma_\bw(i) 
\tqft\left(  \ 
\tikzdiagh{
	\draw (0,1) .. controls (.25,1.25) and (.75,1-.25) .. (1,1);
	\draw (0,0) .. controls (.25,.25) and (.75,-.25) .. (1,0);
	\draw (0,0) -- (0,1);
	\draw (1,0) -- (1,1);
	\node[tikzdot]at(.5,.5){};
} \ 
 \right)
 : 
\tqft\left( \ 
\tikzdiagh{
	\draw[vt] (0,0) .. controls (0,.5) and (1,.5) .. (1,1);
} \ 
\right)
\rightarrow
\tqft\left( \ 
\tikzdiagh{
	\draw[vt] (0,0) .. controls (0,.5) and (1,.5) .. (1,1);
} \ 
\right)
\end{align*}
%
\item To the splitters we associate
\begin{align*}
\UtoBfunctor_0 & \left( \ 
\tikzdiagc{
	\draw[ds] (0,0) node[below]{\small $\alpha_i$}  -- (0,.5);
	\draw (-.03,.5) .. controls (-.03,.75) and (-.5,.75) .. (-.5,1);
	\draw (.03,.5) .. controls (.03,.75) and (.5,.75) .. (.5,1);
} \ : \ 
\tikzdiagh{
	\draw[dds] (0,-.25) -- (0,.5); \draw[dt] (0,.5) -- (0,1.25);
	\draw[->, dds] (0,.5) -- (1,.5);
	\draw[dt] (1,-.25) -- (1,.5); \draw[dds] (1,.5) -- (1,1.25);
}
\rightarrow
\tikzdiagh{
	\draw[dds] (0,-.25) -- (0,.25); 
	\draw [vt] (0,.25) -- (0,.75);
	\draw[dt] (0,.75) -- (0,1.25);
	\draw[->, vt] (0,.25) -- (1,.25);
	\draw[->, vt] (0,.75) -- (1,.75);
	\draw[dt] (1,-.25) -- (1,.25); 
	\draw [vt] (1,.25) -- (1,.75);
	\draw[dds] (1,.75) -- (1,1.25);
} \ 
\right)
\ := \ 
\tqft( \emptyset ) 
\xrightarrow{ \tqft\left( \ \tikzdiagc[scale=.5]{
	\draw (1,2) .. controls (1,1) and (2,1) .. (2,2);
	\draw (1,2) .. controls (1,1.75) and (2,1.75) .. (2,2);
	\draw (1,2) .. controls (1,2.25) and (2,2.25) .. (2,2);
} \ \right)}
\tqft\left( \ 
\tikzdiagh{
	\draw[vt] (0,.5) .. controls (0,1) and (1,1) .. (1,.5) .. controls (1,0) and (0,0) .. (0,.5);
} \ 
\right)
\intertext{and}
\UtoBfunctor_0 & \left( \ 
\tikzdiagc[yscale=-1]{
	\draw[ds] (0,0) -- (0,.5);
	\draw (-.03,.5) .. controls (-.03,.75) and (-.5,.75) .. (-.5,1) node[below]{\small $\alpha_i$} ;
	\draw (.03,.5) .. controls (.03,.75) and (.5,.75) .. (.5,1) node[below]{\small $\alpha_i$} ;
} \ : \ 
\tikzdiagh{
	\draw[dds] (0,-.25) -- (0,.25); 
	\draw [vt] (0,.25) -- (0,.75);
	\draw[dt] (0,.75) -- (0,1.25);
	\draw[->, vt] (0,.25) -- (1,.25);
	\draw[->, vt] (0,.75) -- (1,.75);
	\draw[dt] (1,-.25) -- (1,.25); 
	\draw [vt] (1,.25) -- (1,.75);
	\draw[dds] (1,.75) -- (1,1.25);
}
\rightarrow
\tikzdiagh{
	\draw[dds] (0,-.25) -- (0,.5); \draw[dt] (0,.5) -- (0,1.25);
	\draw[->, dds] (0,.5) -- (1,.5);
	\draw[dt] (1,-.25) -- (1,.5); \draw[dds] (1,.5) -- (1,1.25);
} \ 
\right)
\ := \ 
\tqft\left(  \ 
\tikzdiagh{
	\draw[vt] (0,.5) .. controls (0,1) and (1,1) .. (1,.5) .. controls (1,0) and (0,0) .. (0,.5);
} \ 
\right) 
\xrightarrow{ (-XY){\Gamma_\bw(i)}  \tqft\left( \ \tikzdiagh[xscale=-.5,yscale=.5]{
	\draw (1,0) .. controls (1,1) and (2,1) .. (2,0);
	\draw (1,0) .. controls (1,-.25) and (2,-.25) .. (2,0);
	\draw[dashed] (1,0) .. controls (1,.25) and (2,.25) .. (2,0);
	\draw[->] (1.5,1.5) [partial ellipse=0:270:3ex and 1ex];
} \ \right)}
\tqft\left(\emptyset \right)
\end{align*}
\item The crossings with at least a double strand are all sent to the identity cobordism. 
\item To a crossing between two (single) strands of the same color, we associate:
\begin{align*}
\UtoBfunctor_0 & \left( \ 
\tikzdiagc{
	\draw (0,0) node[below]{\small $\alpha_i$} .. controls (0,.5) and (1,.5) .. (1,1);
	\draw (1,0) node[below]{\small $\alpha_i$} .. controls (1,.5) and (0,.5) .. (0,1);
} \ : \ 
\tikzdiagh{
	\draw[dds] (0,-.25) -- (0,.25); 
	\draw [vt] (0,.25) -- (0,.75);
	\draw[dt] (0,.75) -- (0,1.25);
	\draw[->, vt] (0,.25) -- (1,.25);
	\draw[->, vt] (0,.75) -- (1,.75);
	\draw[dt] (1,-.25) -- (1,.25); 
	\draw [vt] (1,.25) -- (1,.75);
	\draw[dds] (1,.75) -- (1,1.25);
} \ 
\rightarrow
\tikzdiagh{
	\draw[dds] (0,-.25) -- (0,.25); 
	\draw [vt] (0,.25) -- (0,.75);
	\draw[dt] (0,.75) -- (0,1.25);
	\draw[->, vt] (0,.25) -- (1,.25);
	\draw[->, vt] (0,.75) -- (1,.75);
	\draw[dt] (1,-.25) -- (1,.25); 
	\draw [vt] (1,.25) -- (1,.75);
	\draw[dds] (1,.75) -- (1,1.25);
} \ 
\right)
\\
\ &:= \ 
(-XY){\Gamma_\bw(i)}  \tqft\left( \ \tikzdiagc[xscale=-.5,yscale=.5]{
	\draw (1,0) .. controls (1,1) and (2,1) .. (2,0);
	\draw (1,0) .. controls (1,-.25) and (2,-.25) .. (2,0);
	\draw[dashed] (1,0) .. controls (1,.25) and (2,.25) .. (2,0);
	\draw[->] (1.5,1.5) [partial ellipse=0:270:3ex and 1ex];
	\draw (1,4) .. controls (1,3) and (2,3) .. (2,4);
	\draw (1,4) .. controls (1,3.75) and (2,3.75) .. (2,4);
	\draw (1,4) .. controls (1,4.25) and (2,4.25) .. (2,4);
} \right)
: 
\tqft\left( \ 
\tikzdiagh{
	\draw[vt] (0,.5) .. controls (0,1) and (1,1) .. (1,.5) .. controls (1,0) and (0,0) .. (0,.5);
} \ 
\right)
\rightarrow
\tqft\left( \ 
\tikzdiagh{
	\draw[vt] (0,.5) .. controls (0,1) and (1,1) .. (1,.5) .. controls (1,0) and (0,0) .. (0,.5);
} \ 
\right)
\end{align*}
\item To a crossing between two strands with $|i-j| > 1$, we associate the identity cobordism, since at the level of the tangles we are exchanging distant rungs, and thus the source and target are equivalent tangles. 
\item To a crossing between two strands with $j = i - 1$ we associate:
\begin{align*}
\UtoBfunctor_0 & \left( \ 
\tikzdiagc{
	\draw (0,0) node[below]{\small $\alpha_i$} .. controls (0,.5) and (1,.5) .. (1,1);
	\draw[myblue] (1,0) node[below]{\small $\alpha_{i{-}1}$} .. controls (1,.5) and (0,.5) .. (0,1);
} \ : \ 
\tikzdiagh{
	\draw[vt] (0,0) -- (0,.5);
	\draw[dds] (0,.5) -- (0,1);
	\draw[vt] (0,1) -- (0,1.5);
	\draw[->,vt] (-1,.5) -- (0,.5);
	\draw[->,vt] (0,1) -- (1,1);
	\draw[dotted] (-1,0) -- (-1,1.5);
	\draw[dotted] (1,0) -- (1,1.5);
} \ 
\rightarrow
\tikzdiagh{
	\draw[vt] (0,0) -- (0,.5);
	\draw[dt] (0,.5) -- (0,1);
	\draw[vt] (0,1) -- (0,1.5);
	\draw[->,vt] (0,.5) -- (1,.5);
	\draw[->,vt] (-1,1) -- (0,1);
	\draw[dotted] (-1,0) -- (-1,1.5);
	\draw[dotted] (1,0) -- (1,1.5);
} \ 
\right)
\\
\ &\rightsquigarrow \ 
 \tqft\left( \ \tikzdiagh[scale=.2]{
	\draw (.5,1) -- (.5,3) -- (4.5,3) -- (4.5,1);
	\draw (0,0) .. controls (1.5,0) and (2,1) .. (.5,1); 
	\draw (4,0) .. controls (2.5,0) and (3,1) .. (4.5,1); 
	\filldraw[fill=white, draw=white] (.5,.55)  rectangle  (4,2);
	\draw[dotted] (0,0) .. controls (1.5,0) and (2,1) .. (.5,1); 
	\draw[dotted] (4,0) .. controls (2.5,0) and (3,1) .. (4.5,1); 
	\draw (1.375,.5) .. controls (1.375,1.5) and (3.125,1.5) .. (3.125,.5);
	%
	\draw (0,0) -- (0,2) -- (4,2) -- (4,0);
} \right)
: 
\tqft\left( \ 
\tikzdiagh{
	\draw (-1,.5) .. controls (0,.5) .. (0,0);
	\draw (0,1.5) .. controls (0,1) .. (1,1);
	\draw[->] (-.15,-.15+.75) -- (.15,.15+.75);
} \ 
\right)
\rightarrow
\tqft\left( \ 
\tikzdiagh{
	\draw (1,.5) .. controls (0,.5) .. (0,0);
	\draw (0,1.5) .. controls (0,1) .. (-1,1);
} \ 
\right)
\end{align*}
where  the lines we drawn as dotted could be either dashed, solid or double dashed (depending on the weight $\bw$). 
The arrow gives the orientation of the saddle point. 
Similarly, for $j = i + 1$ we put:
\begin{align*}
\UtoBfunctor_0 & \left( \ 
\tikzdiagc{
	\draw[myblue] (0,0) node[below]{\small $\alpha_i$} .. controls (0,.5) and (1,.5) .. (1,1);
	\draw (1,0) node[below]{\small $\alpha_{i{+}1}$} .. controls (1,.5) and (0,.5) .. (0,1);
} \ : \ 
\tikzdiagh{
	\draw[vt] (0,0) -- (0,.5);
	\draw[dt] (0,.5) -- (0,1);
	\draw[vt] (0,1) -- (0,1.5);
	\draw[->,vt] (0,.5) -- (1,.5);
	\draw[->,vt] (-1,1) -- (0,1);
	\draw[dotted] (-1,0) -- (-1,1.5);
	\draw[dotted] (1,0) -- (1,1.5);
} \ 
\rightarrow
\tikzdiagh{
	\draw[vt] (0,0) -- (0,.5);
	\draw[dds] (0,.5) -- (0,1);
	\draw[vt] (0,1) -- (0,1.5);
	\draw[->,vt] (-1,.5) -- (0,.5);
	\draw[->,vt] (0,1) -- (1,1);
	\draw[dotted] (-1,0) -- (-1,1.5);
	\draw[dotted] (1,0) -- (1,1.5);
} \ 
\right)
\\
\ &\rightsquigarrow \ 
(-XY){\Gamma_\bw(i)}  \tqft\left( \ \tikzdiagh[scale=.2]{
	\draw (.5,3) -- (.5,1) -- (4.5,1) -- (4.5,3);
	\filldraw [fill=white, draw=white] (0,0) rectangle (4,2); 
	\draw (0,2) .. controls (1.5,2) and (2,3) .. (.5,3); 
	\draw (4,2) .. controls (2.5,2) and (3,3) .. (4.5,3); 
	\draw (1.375,2.5) .. controls (1.375,1.5) and (3.125,1.5) .. (3.125,2.5);
	%
	\draw (0,2) -- (0,0) -- (4,0) -- (4,2);
} \right)
: 
\tqft\left( \ 
\tikzdiagh{
	\draw (1,.5) .. controls (0,.5) .. (0,0);
	\draw (0,1.5) .. controls (0,1) .. (-1,1);
	\draw[->] (.15,-.15+.75) -- (-.15,.15+.75);
} \ 
\right)
\rightarrow
\tqft\left( \ 
\tikzdiagh{
	\draw (-1,.5) .. controls (0,.5) .. (0,0);
	\draw (0,1.5) .. controls (0,1) .. (1,1);
} \ 
\right)
\end{align*}
\end{itemize}

However, there is an issue here if we want $\UtoBfunctor_0$ to be a map of $R$-algebra. Indeed, if we isotopy two distant crossings with labels $j = i \pm 1$, it will produce a global scalar given by the graded planar isotopy. On the other side, changing the chronology at the level of the cobordisms will be different, and it will depend on the closure of the tangle (the saddle point could be either a split or a merge). Hopefully, by introducing some grading shift on $\tqft(\ladderToTangle(\F_\bi))$, we can fix this. The idea is to observe that if we shift the target
\begin{align*}
\tqft\left(
\tikzdiagh[scale=.75]
{
	\draw[dotted] (3,-2) circle(0.707);
	\draw (2.5,-1.5) .. controls (2.75,-1.75) and (3.25,-1.75) .. (3.5,-1.5);
	\draw (2.5,-2.5) .. controls  (2.75,-2.25) and (3.25,-2.25) .. (3.5,-2.5);
}
  \right) 
\xrightarrow{\tqft\bigl( \ \tikzdiagh[scale=.2]{
	\draw (.5,1) -- (.5,3) -- (4.5,3) -- (4.5,1);
	\draw (0,0) .. controls (1.5,0) and (2,1) .. (.5,1); 
	\draw (4,0) .. controls (2.5,0) and (3,1) .. (4.5,1); 
	\filldraw[fill=white, draw=white] (.5,.55)  rectangle  (4,2);
	\draw[dotted] (0,0) .. controls (1.5,0) and (2,1) .. (.5,1); 
	\draw[dotted] (4,0) .. controls (2.5,0) and (3,1) .. (4.5,1); 
	\draw (1.375,.5) .. controls (1.375,1.5) and (3.125,1.5) .. (3.125,.5);
	%
	\draw (0,0) -- (0,2) -- (4,2) -- (4,0);
} \ \bigr)}
 (\shiftFunct{ \tikzdiagh[scale=.2]{
	\draw (.5,3) -- (.5,1) -- (4.5,1) -- (4.5,3);
	\filldraw [fill=white, draw=white] (0,0) rectangle (4,2); 
	\draw (0,2) .. controls (1.5,2) and (2,3) .. (.5,3); 
	\draw (4,2) .. controls (2.5,2) and (3,3) .. (4.5,3); 
	\draw (1.375,2.5) .. controls (1.375,1.5) and (3.125,1.5) .. (3.125,2.5);
	%
	\draw (0,2) -- (0,0) -- (4,0) -- (4,2);
}^{(1,0)}})
\tqft\left(
\tikzdiagh[scale=.75]
{
	\draw[dotted] (-2,-2) circle(0.707);
	\draw (-1.5,-1.5) .. controls (-1.75,-1.75) and (-1.75,-2.25) .. (-1.5,-2.5);
	\draw (-2.5,-1.5) .. controls (-2.25,-1.75) and (-2.25,-2.25) ..  (-2.5,-2.5);
}  
 \right)
\intertext{or the source}
\bigl(
\shiftFunct{\tikzdiagh[scale=.2]{
	\draw (.5,3) -- (.5,1) -- (4.5,1) -- (4.5,3);
	\filldraw [fill=white, draw=white] (0,0) rectangle (4,2); 
	\draw (0,2) .. controls (1.5,2) and (2,3) .. (.5,3); 
	\draw (4,2) .. controls (2.5,2) and (3,3) .. (4.5,3); 
	\draw (1.375,2.5) .. controls (1.375,1.5) and (3.125,1.5) .. (3.125,2.5);
	%
	\draw (0,2) -- (0,0) -- (4,0) -- (4,2);
}^{(1,0)}} 
\bigr)
\tqft\left(
\tikzdiagh[scale=.75]
{
	\draw[dotted] (-2,-2) circle(0.707);
	\draw (-1.5,-1.5) .. controls (-1.75,-1.75) and (-1.75,-2.25) .. (-1.5,-2.5);
	\draw (-2.5,-1.5) .. controls (-2.25,-1.75) and (-2.25,-2.25) ..  (-2.5,-2.5);
} 
  \right) 
\xrightarrow{\tqft\bigl( \ \tikzdiagh[scale=.2]{
	\draw (.5,3) -- (.5,1) -- (4.5,1) -- (4.5,3);
	\filldraw [fill=white, draw=white] (0,0) rectangle (4,2); 
	\draw (0,2) .. controls (1.5,2) and (2,3) .. (.5,3); 
	\draw (4,2) .. controls (2.5,2) and (3,3) .. (4.5,3); 
	\draw (1.375,2.5) .. controls (1.375,1.5) and (3.125,1.5) .. (3.125,2.5);
	%
	\draw (0,2) -- (0,0) -- (4,0) -- (4,2);
} \ \bigr)}
\tqft\left(
\tikzdiagh[scale=.75]
{
	\draw[dotted] (3,-2) circle(0.707);
	\draw (2.5,-1.5) .. controls (2.75,-1.75) and (3.25,-1.75) .. (3.5,-1.5);
	\draw (2.5,-2.5) .. controls  (2.75,-2.25) and (3.25,-2.25) .. (3.5,-2.5);
}
\right)
\end{align*}
then the first map is of degree $(0,-1)$, like a crossing for $j=i-1$, and the second map is of degree $(-1,0)$, like for $j=i+1$. Also, given a pair of maps $g : M' \rightarrow N', f : M \rightarrow N$ that are purely homogeneous and with degree that does not change the underlying tangle, we have that the following 2-diagram
\begin{equation}\label{eq:fotimesgradedcommutes}
\begin{tikzcd}
& 
N' \otimes_H M 
\ar{dr}{1 \otimes f}
&
\\
M' \otimes_H M  
\ar{ur}{g \otimes 1}
\ar[swap]{dr}{1 \otimes f}
&&
N' \otimes_H N
\\
&
M' \otimes_H N 
\ar[swap]{ur}{g \otimes 1}
 \ar[phantom]{uu}{\quad \rotatebox{90}{$\Rightarrow$}\ \lambda_R(|f|,|g|)}
 &
\end{tikzcd}
\end{equation}
commutes. 
 Moreover, the degree shift introduced preserve the quantum grading, so that the spaces are isomorphic to their shifted version in $\BIMOD_q(H^\bullet, H^\bullet)$. 

\subsubsection{Normalized tangle}

We give an algorithm to construct the shifted version of $\UtoBfunctor_0(\F_\bi)$.
Fix a ladder diagram $\ladder(\F_\bi : \bw \rightarrow \bw')$. 
Let $T_0 := \ladderToTangle(\F_\bi)$ and $\shiftFunct{0} := \id$. Let $k := 0$ be a counting variable (i.e. a variable that will increase at each step of the algorithm).   
We read $\ladder(\F_\bi)$ from bottom to top, decomposing it as $\ladder(\F_\bi) =  \bw_r \F_{i_r} \bw_{r-1} \cdots \bw_1 \F_{i_1} \bw_0$, where $\bw' = \bw_r$ and $\bw = \bw_0$ . Whenever we encounter  the following situation:
\begin{equation}\label{eq:normalization1}
\tikzdiagh{
	\draw[vt] (0,0) -- (0,.5);
	\draw[dt] (0,.5) -- (0,1);
	\draw[->,vt] (0,.5) -- (1,.5);
	\draw[dotted] (-1,0) -- (-1,1);
	\draw[dotted] (1,0) -- (1,1);
}
\end{equation}
i.e. $\bw_\ell \F_{i_\ell}\bw_{\ell-1}$ with $(\bw_{\ell-1})_{i_\ell} = 1$, 
we follow the dashed strand until we possibly encounter 
one of the two following situations:
\begin{align}\label{eq:normalization2}
\tikzdiagh{
	\draw[dt] (0,.5) -- (0,1);
	\draw[vt] (0,1) -- (0,1.5);
	\draw[->,vt] (-1,1) -- (0,1);
	\draw[dotted] (-1,.5) -- (-1,1.5);
	\draw[dotted] (1,.5) -- (1,1.5);
}
&&
\tikzdiagh{
	\draw[dt] (0,.5) -- (0,1);
	\draw[dds] (0,1) -- (0,1.5);
	\draw[->,dds] (-1,1) -- (0,1);
	\draw[dds] (-1,.5) -- (-1,1);
	\draw[dt] (-1,1) -- (-1,1.5);
	\draw[dotted] (1,.5) -- (1,1.5);
}
\end{align}
In the first case, 
we consider the chronological cobordism $\chcob_{k+1} : T_k \rightarrow T_{k+1}$ given by putting a saddle realizing a surgery over the arcs of the tangle diagrams corresponding to \eqref{eq:normalization1} and \eqref{eq:normalization2}, with orientation to the top left, and
we put
\begin{align*}
 \shiftFunct{k+1} &:=  \shiftFunct{\chcob_{k+1}^{(1,0)}} \circ \shiftFunct{k},
 &
 k &:= k+1.
\end{align*}
Then, we continue reading $\ladder(\F_\bi)$ from just above \eqref{eq:normalization1}, that is starting at $\F_{i_\ell +1}$. 
In the second case, we jump one case to the left and continue to follow the dashed strand, until we either attain the end of the ladder diagram or \cref{eq:normalization2}, repeating. 
When we reach the end of $\ladder(\F_\bi)$, we put $\shiftFunct{\F_\bi} := \shiftFunct{k}$. Finally, we define
\[
\UtoBfunctor(\F_\bi) :=  \shiftFunct{\F_\bi}\left( \UtoBfunctor_0(\F_\bi) \right).
\]

\begin{exe}
For example, we have
\[
\tikzdiagh[yscale=.5]{
	\draw[dds] (0,0) -- (0,3);
	\draw[vt] (0,3) -- (0,4);
	\draw[dt] (0,4) -- (0,5);
	\draw[dds] (1,0) -- (1,2);
	\draw[dt] (1,2) -- (1,3);
	\draw[vt] (1,3) -- (1,4);
	\draw[dds] (1,4) -- (1,5);
	\draw[vt] (2,0) -- (2,1);
	\draw[dt] (2,1) -- (2,2);
	\draw[dds] (2,2) -- (2,5);
	\draw[vt] (3,0) -- (3,1);
	\draw[dds] (3,1) -- (3,5);
	\draw[vt,->] (2,1) -- (3,1);
	\draw[dds,->] (1,2) -- (2,2);
	\draw[vt,->] (0,3) -- (1,3);
	\draw[vt,->] (0,4) -- (1,4);
}
\ \xmapsto{\ \UtoBfunctor\ } \ 
\bigl(\shiftFunct{\tikzdiagh[scale=.2]{
	\draw  (.5,3) -- (4.5,3) -- (4.5,1);
	\begin{scope}
		\clip (-1,.5) rectangle (1,-.5);
		\draw (0,0) .. controls (-1.5,0) and (-1,1) .. (.5,1); 
	\end{scope}
	\draw (0,0) .. controls (1.5,0) and (2,1) .. (.5,1); 
	\draw (4,0) .. controls (2.5,0) and (3,1) .. (4.5,1); 
	\filldraw[fill=white, draw=white] (.5,.55)  rectangle  (4,2);
	\draw[dotted] (0,0) .. controls (1.5,0) and (2,1) .. (.5,1); 
	\draw[dotted] (4,0) .. controls (2.5,0) and (3,1) .. (4.5,1); 
	\draw[dotted] (0,0) .. controls (-1.5,0) and (-1,1) .. (.5,1); 
	\draw (1.375,.5) .. controls (1.375,1.5) and (3.125,1.5) .. (3.125,.5);
	%
	\draw (0,2) -- (4,2) -- (4,0);
	\draw (0,2) .. controls (-1.5,2) and (-1,3) .. (.5,3); 
	\draw (-.875,.5) -- (-.875,2.5);
}^{(1,0)}}\bigr)
\tqft\left( \ 
\tikzdiagh[yscale=.5]{
	\draw[vt] (0,3.5) .. controls (0,4.5) and (1,4.5) .. (1,3.5) .. controls (1,2.5) and (0,2.5) .. (0,3.5);
	\draw[vt] (1,0) .. controls (1,1) and (2,1) .. (2,0);
}
 \ \right)
\]
\end{exe}

We define $\UtoBfunctor$ on objects as $\UtoBfunctor_0$. For the 2-morphisms, we need to throw in some extra changes of chronology in the recipe.  
First, we need to modify the 2-morphism for a crossing $\tau_{ij}$ with $|i-j| > 1$. Suppose $\F_\bj$ is obtained from $\F_\bi$ by exchanging $\F_{i_\ell}$ with $\F_{i_{\ell+1}}$ and $|i_{\ell} - i_{\ell+1}| > 1$.  In this situation, we put 
\begin{equation}\label{eq:FLbicrossdistant}
\UtoBfunctor
\left(
\tikzdiagc{
	\draw (0,0) node[below]{\small $\alpha_{i_{\ell{+}1}}$} .. controls (0,.5) and (1,.5) .. (1,1);
	\draw[myblue] (1,0) node[below]{\small $\alpha_{i_\ell}$} .. controls (1,.5) and (0,.5) .. (0,1);
} 
\right)
:= 
\shiftFunct{H_{\tau_{ij}}}(\UtoBfunctor(\F_\bi)) 
\UtoBfunctor_0\left(
\tikzdiagc{
	\draw (0,0) node[below]{\small $\alpha_{i_{\ell{+}1}}$} .. controls (0,.5) and (1,.5) .. (1,1);
	\draw[myblue] (1,0) node[below]{\small $\alpha_{i_\ell}$} .. controls (1,.5) and (0,.5) .. (0,1);
} 
\right),
\end{equation}
where $\shiftFunct{H_{\tau_{ij}}} : \shiftFunct{\F_\bi} \Rightarrow \shiftFunct{\F_\bj}$ is the natural transformation obtained from \cref{prop:nattrfromchangeofch}. 
Furthermore, we need to add some change of chronology for a crossing $\tau_{ij}$ with $j = i + 1$. 
Consider the natural isomorphism
$
\shiftFunct{H_{\tau_{ij}}} : \shiftFunct{F_\bi} \Rightarrow \shiftFunct{\F_\bj} \circ \shiftFunct{\tikzdiagh[scale=.2]{
	\draw (.5,3) -- (.5,1) -- (4.5,1) -- (4.5,3);
	\filldraw [fill=white, draw=white] (0,0) rectangle (4,2); 
	\draw (0,2) .. controls (1.5,2) and (2,3) .. (.5,3); 
	\draw (4,2) .. controls (2.5,2) and (3,3) .. (4.5,3); 
	\draw (1.375,2.5) .. controls (1.375,1.5) and (3.125,1.5) .. (3.125,2.5);
	%
	\draw (0,2) -- (0,0) -- (4,0) -- (4,2);
}^{(1,0)}}
$
given by change of chronology, and where the saddle corresponds with the saddle above the rungs associated to the strands involved in the crossing. 
 Then, we define $\UtoBfunctor(\tau_{ij})$ by making the following diagram commutes:
\[
\begin{tikzcd}
\shiftFunct{\F_\bi} \bigl( \UtoBfunctor_0(\F_\bi)  \bigr)
\ar{rr}{\UtoBfunctor(\tau_{ij})}
\ar[swap,"\sim"'{sloped}]{dr}{\shiftFunct{H_{\tau_{ij}}} }
&
&
\shiftFunct{\F_\bj} \bigl( \UtoBfunctor_0(\F_\bj) \bigr)
\\
&
\shiftFunct{\F_\bj} \circ  \shiftFunct{\tikzdiagh[scale=.2]{
	\draw (.5,3) -- (.5,1) -- (4.5,1) -- (4.5,3);
	\filldraw [fill=white, draw=white] (0,0) rectangle (4,2); 
	\draw (0,2) .. controls (1.5,2) and (2,3) .. (.5,3); 
	\draw (4,2) .. controls (2.5,2) and (3,3) .. (4.5,3); 
	\draw (1.375,2.5) .. controls (1.375,1.5) and (3.125,1.5) .. (3.125,2.5);
	%
	\draw (0,2) -- (0,0) -- (4,0) -- (4,2);
}^{(1,0)}}  \bigl( \UtoBfunctor_0(\F_\bi) \bigr)
\ar[swap]{ur}{\shiftFunct{\F_\bj} ( \UtoBfunctor_0(\tau_{ij}) )}
&
\end{tikzcd}
\]
where we recall that the grading shift functor twists the morphisms since they carry a non-trivial $\gradC$-degree (see \cref{sec:grshiftBIMOD} and \cref{sec:Gcommutativity}). 
Similarly, for $j = i-1$, we put
\begin{equation}\label{eq:defUtoBtauj-1}
\begin{tikzcd}[column sep = 20ex]
\shiftFunct{\F_\bi} \bigl( \UtoBfunctor_0(\F_\bi)  \bigr)
\ar{r}{\UtoBfunctor(\tau_{ij})}
\ar[swap,"\vsim"']{d}{\shiftFunct{H}  }
&
\shiftFunct{\F_\bj} \bigl( \UtoBfunctor_0(\F_\bj) \bigr)
\\
\shiftFunct{\F_\bi} \circ  \shiftFunct{\tikzdiagh[scale=.2]{
	\draw (.5,3) -- (.5,1) -- (4.5,1) -- (4.5,3);
	\filldraw [fill=white, draw=white] (0,0) rectangle (4,2); 
	\draw (0,2) .. controls (1.5,2) and (2,3) .. (.5,3); 
	\draw (4,2) .. controls (2.5,2) and (3,3) .. (4.5,3); 
	\draw (1.375,2.5) .. controls (1.375,1.5) and (3.125,1.5) .. (3.125,2.5);
	%
	\draw (0,2) -- (0,0) -- (4,0) -- (4,2);
}^{(1,0)}} \circ  \shiftFunct{\tikzdiagh[scale=.2]{
	\draw (.5,1) -- (.5,3) -- (4.5,3) -- (4.5,1);
	\draw (0,0) .. controls (1.5,0) and (2,1) .. (.5,1); 
	\draw (4,0) .. controls (2.5,0) and (3,1) .. (4.5,1); 
	\filldraw[fill=white, draw=white] (.5,.55)  rectangle  (4,2);
	\draw[dotted] (0,0) .. controls (1.5,0) and (2,1) .. (.5,1); 
	\draw[dotted] (4,0) .. controls (2.5,0) and (3,1) .. (4.5,1); 
	\draw (1.375,.5) .. controls (1.375,1.5) and (3.125,1.5) .. (3.125,.5);
	%
	\draw (0,0) -- (0,2) -- (4,2) -- (4,0);
}^{(0,1)}}   \bigl( \UtoBfunctor_0(\F_\bi) \bigr)
\ar[swap,dash]{r}{\shiftFunct{\F_\bi}  \circ  \shiftFunct{\tikzdiagh[scale=.2]{
	\draw (.5,3) -- (.5,1) -- (4.5,1) -- (4.5,3);
	\filldraw [fill=white, draw=white] (0,0) rectangle (4,2); 
	\draw (0,2) .. controls (1.5,2) and (2,3) .. (.5,3); 
	\draw (4,2) .. controls (2.5,2) and (3,3) .. (4.5,3); 
	\draw (1.375,2.5) .. controls (1.375,1.5) and (3.125,1.5) .. (3.125,2.5);
	%
	\draw (0,2) -- (0,0) -- (4,0) -- (4,2);
}^{(1,0)}}   ( \UtoBfunctor_0(\tau_{ij}) )  }
\ar{r}
&
\shiftFunct{\F_\bi} \circ  \shiftFunct{\tikzdiagh[scale=.2]{
	\draw (.5,3) -- (.5,1) -- (4.5,1) -- (4.5,3);
	\filldraw [fill=white, draw=white] (0,0) rectangle (4,2); 
	\draw (0,2) .. controls (1.5,2) and (2,3) .. (.5,3); 
	\draw (4,2) .. controls (2.5,2) and (3,3) .. (4.5,3); 
	\draw (1.375,2.5) .. controls (1.375,1.5) and (3.125,1.5) .. (3.125,2.5);
	%
	\draw (0,2) -- (0,0) -- (4,0) -- (4,2);
}^{(1,0)}}  \bigl( \UtoBfunctor_0(\F_\bj) \bigr)
\ar[swap,"\vsim"']{u}{\shiftFunct{H_{\tau_{ij}}}}
\end{tikzcd}
\end{equation}
where $
\shiftFunct{H_{\tau_{ji}}} : \shiftFunct{F_\bi} \circ \shiftFunct{\tikzdiagh[scale=.2]{
	\draw (.5,3) -- (.5,1) -- (4.5,1) -- (4.5,3);
	\filldraw [fill=white, draw=white] (0,0) rectangle (4,2); 
	\draw (0,2) .. controls (1.5,2) and (2,3) .. (.5,3); 
	\draw (4,2) .. controls (2.5,2) and (3,3) .. (4.5,3); 
	\draw (1.375,2.5) .. controls (1.375,1.5) and (3.125,1.5) .. (3.125,2.5);
	%
	\draw (0,2) -- (0,0) -- (4,0) -- (4,2);
}^{(1,0)}} \Rightarrow \shiftFunct{\F_\bj}, 
$
and $\shiftFunct{H} : \id \Rightarrow  \shiftFunct{\tikzdiagh[scale=.2]{
	\draw (.5,3) -- (.5,1) -- (4.5,1) -- (4.5,3);
	\filldraw [fill=white, draw=white] (0,0) rectangle (4,2); 
	\draw (0,2) .. controls (1.5,2) and (2,3) .. (.5,3); 
	\draw (4,2) .. controls (2.5,2) and (3,3) .. (4.5,3); 
	\draw (1.375,2.5) .. controls (1.375,1.5) and (3.125,1.5) .. (3.125,2.5);
	%
	\draw (0,2) -- (0,0) -- (4,0) -- (4,2);
}^{(1,0)}} \circ  \shiftFunct{\tikzdiagh[scale=.2]{
	\draw (.5,1) -- (.5,3) -- (4.5,3) -- (4.5,1);
	\draw (0,0) .. controls (1.5,0) and (2,1) .. (.5,1); 
	\draw (4,0) .. controls (2.5,0) and (3,1) .. (4.5,1); 
	\filldraw[fill=white, draw=white] (.5,.55)  rectangle  (4,2);
	\draw[dotted] (0,0) .. controls (1.5,0) and (2,1) .. (.5,1); 
	\draw[dotted] (4,0) .. controls (2.5,0) and (3,1) .. (4.5,1); 
	\draw (1.375,.5) .. controls (1.375,1.5) and (3.125,1.5) .. (3.125,.5);
	%
	\draw (0,0) -- (0,2) -- (4,2) -- (4,0);
}^{(0,1)}}$. 
If we consider a splitter $\gamma : \F_\bi \rightarrow \F_\bj$, then there exists a natural isomorphism $\shiftFunct{H_{\gamma}} :  \shiftFunct{\F_\bi} \Rightarrow \shiftFunct{\F_\bj}$ (because at least one of the potential saddles above the spawned rungs is a merge and annihilated with the degree shift $(1,0)$), and we put
\[
\begin{tikzcd}
\shiftFunct{\F_\bi} \bigl( \UtoBfunctor_0(\F_\bi)  \bigr)
\ar{rr}{\UtoBfunctor(\gamma)}
\ar[swap]{dr}{\shiftFunct{\F_\bi} ( \UtoBfunctor_0(\gamma) )  }
&
&
\shiftFunct{\F_\bj} \bigl( \UtoBfunctor_0(\F_\bj) \bigr)
\\
&
\shiftFunct{\F_\bi} \bigl( \UtoBfunctor_0(\F_\bj) \bigr)
\ar[swap,"\sim"'{sloped}]{ur}{\shiftFunct{H_{\gamma}}}
&
\end{tikzcd}
\]
Finally, for the other cases where $f : \F_\bi \rightarrow \F_\bi$, we define 
\[
\UtoBfunctor(f) := \shiftFunct{\F_\bi} \bigl( \UtoBfunctor_0(f) \bigr),
\]
where again the grading shift functor twists the morphism. 

\begin{exe} Let us do some detailed example of such a computation. 
Let $ n = 5$, $\bw = (2,1,1,1,1,0)$ and $\bi = 5412$. We have
\begin{align*}
\ladder(\F_\bi) &= \ 
\tikzdiagh[scale=.75, yscale=.5]{
	\draw[dds] (0,0) -- (0,2);
	\draw[vt] (0,2) -- (0,5);
	\draw[vt] (1,0) -- (1,1);
	\draw[dt] (1,1) -- (1,2);
	\draw[vt] (1,2) -- (1,5);
	\draw[vt] (2,0) -- (2,1);
	\draw[dds] (2,1) -- (2,5);
	\draw[vt] (3,0) -- (3,3);
	\draw[dt] (3,3) -- (3,5);
	\draw[vt] (4,0) -- (4,3);
	\draw[dds] (4,3) -- (4,4);
	\draw[vt] (4,4) -- (4,5);
	\draw[dt] (5,0) -- (5,4);
	\draw[vt] (5,4) -- (5,5);
	\draw[vt,->] (1,1) -- (2,1);
	\draw[vt,->] (0,2) -- (1,2);
	\draw[vt,->] (3,3) -- (4,3);
	\draw[vt,->] (4,4) -- (5,4);
}
&
\text{and}&
&
\ladderToTangle (\F_\bi) &\cong \ 
\tikzdiagh[scale=.75,yscale=1]{
	\draw (0,2) .. controls (0,1) and (1,1) .. (1,2);
	\draw (0,0) .. controls (0,1) and (1,1) .. (1,0);
	\draw (2,2) .. controls (2,1) and (3,1) .. (3,2);
	\draw (2,0) .. controls (2,1) and (3,1) .. (3,0);
}
\end{align*}
We want to apply $\UtoBfunctor(\tau_{12})$ and $\UtoBfunctor(\tau_{54})$ in the two possible order, and see what happens. In particular, we want to compare it to the equality $ \tau_{12} \tau_{54}= Z\tau_{54} \tau_{12} $ 
obtained from graded planar isotopy. 
We have
\[
\begin{tikzcd}[row sep = 3ex]
&
 \tqft \left( \tikzdiag[scale=.25]{-1.75ex}{
	\draw (0,0) .. controls (0,.5) and (.25,.5) .. (.25,1) .. controls (.25,1.5) and (0,1.5) .. (0,2);
	\draw (1,0) .. controls (1,.5) and (.75,.5) .. (.75,1) .. controls (.75,1.5) and (1,1.5) .. (1,2);
}\ \tikzdiag[scale=.25]{-1.75ex}{
	\draw (0,2) .. controls (0,1) and (1,1) .. (1,2);
	\draw (0,0) .. controls (0,1) and (1,1) .. (1,0);
} \right) 
\ar{dr}{\UtoBfunctor(\tau_{54})}
&
\\
\shiftFunct{\tikzdiag[scale=.25]{-1.75ex}{
	\draw (0,2) .. controls (0,1) and (1,1) .. (1,2);
	\draw (0,0) .. controls (0,1) and (1,1) .. (1,0);
	\draw[giga thick] (.5,.75) -- (.5,1.25);
}\ \tikzdiag[scale=.25]{-1.75ex}{
	\draw (0,2) .. controls (0,1) and (1,1) .. (1,2);
	\draw (0,0) .. controls (0,1) and (1,1) .. (1,0);
}^{(1,0)}} 
\tqft \left( \tikzdiag[scale=.25]{-1.75ex}{
	\draw (0,2) .. controls (0,1) and (1,1) .. (1,2);
	\draw (0,0) .. controls (0,1) and (1,1) .. (1,0);
}\ \tikzdiag[scale=.25]{-1.75ex}{
	\draw (0,2) .. controls (0,1) and (1,1) .. (1,2);
	\draw (0,0) .. controls (0,1) and (1,1) .. (1,0);
} \right) 
\ar{ur}{\UtoBfunctor(\tau_{12})} 
\ar[swap]{dr}{\UtoBfunctor(\tau_{54})}
&&
\shiftFunct{\tikzdiag[scale=.25]{-1.75ex}{
	\draw (0,0) .. controls (0,.5) and (.25,.5) .. (.25,1) .. controls (.25,1.5) and (0,1.5) .. (0,2);
	\draw (1,0) .. controls (1,.5) and (.75,.5) .. (.75,1) .. controls (.75,1.5) and (1,1.5) .. (1,2);
}\ \tikzdiag[scale=.25]{-1.75ex}{
	\draw (0,0) .. controls (0,.5) and (.25,.5) .. (.25,1) .. controls (.25,1.5) and (0,1.5) .. (0,2);
	\draw (1,0) .. controls (1,.5) and (.75,.5) .. (.75,1) .. controls (.75,1.5) and (1,1.5) .. (1,2);
	\draw[giga thick] (.25,1) -- (.75,1);
}^{(1,0)}} 
\tqft \left( \tikzdiag[scale=.25]{-1.75ex}{
	\draw (0,0) .. controls (0,.5) and (.25,.5) .. (.25,1) .. controls (.25,1.5) and (0,1.5) .. (0,2);
	\draw (1,0) .. controls (1,.5) and (.75,.5) .. (.75,1) .. controls (.75,1.5) and (1,1.5) .. (1,2);
}\ \tikzdiag[scale=.25]{-1.75ex}{
	\draw (0,0) .. controls (0,.5) and (.25,.5) .. (.25,1) .. controls (.25,1.5) and (0,1.5) .. (0,2);
	\draw (1,0) .. controls (1,.5) and (.75,.5) .. (.75,1) .. controls (.75,1.5) and (1,1.5) .. (1,2);
} \right)
\\
&
\shiftFunct{\tikzdiag[scale=.25]{-1.75ex}{
	\draw (0,0) .. controls (0,.5) and (.25,.5) .. (.25,1) .. controls (.25,1.5) and (0,1.5) .. (0,2);
	\draw (1,0) .. controls (1,.5) and (.75,.5) .. (.75,1) .. controls (.75,1.5) and (1,1.5) .. (1,2);
}\ \tikzdiag[scale=.25]{-1.75ex}{
	\draw (0,0) .. controls (0,.5) and (.25,.5) .. (.25,1) .. controls (.25,1.5) and (0,1.5) .. (0,2);
	\draw (1,0) .. controls (1,.5) and (.75,.5) .. (.75,1) .. controls (.75,1.5) and (1,1.5) .. (1,2);
	\draw[giga thick] (.25,1) -- (.75,1);
}^{(1,0)}} 
\circ
\shiftFunct{\tikzdiag[scale=.25]{-1.75ex}{
	\draw (0,2) .. controls (0,1) and (1,1) .. (1,2);
	\draw (0,0) .. controls (0,1) and (1,1) .. (1,0);
	\draw[giga thick] (.5,.75) -- (.5,1.25);
}\ \tikzdiag[scale=.25]{-1.75ex}{
	\draw (0,0) .. controls (0,.5) and (.25,.5) .. (.25,1) .. controls (.25,1.5) and (0,1.5) .. (0,2);
	\draw (1,0) .. controls (1,.5) and (.75,.5) .. (.75,1) .. controls (.75,1.5) and (1,1.5) .. (1,2);
}^{(1,0)}}
\tqft \left( \tikzdiag[scale=.25]{-1.75ex}{
	\draw (0,2) .. controls (0,1) and (1,1) .. (1,2);
	\draw (0,0) .. controls (0,1) and (1,1) .. (1,0);
}\ \tikzdiag[scale=.25]{-1.75ex}{
	\draw (0,0) .. controls (0,.5) and (.25,.5) .. (.25,1) .. controls (.25,1.5) and (0,1.5) .. (0,2);
	\draw (1,0) .. controls (1,.5) and (.75,.5) .. (.75,1) .. controls (.75,1.5) and (1,1.5) .. (1,2);
} \right) 
\ar[swap]{ur}{\UtoBfunctor(\tau_{12})}
&
\end{tikzcd}
\]
where the thick line represents a saddle cobordism as in \cref{fig:chchange}, and they are all oriented to the top or to the left. 
For the rest of the example, we suppose we have chosen an arbitrary closure of the tangles, so that we can make actual computations. 
We start by computing the maps on the top path. For $\tau_{12}$, we have $\shiftFunct{H_{\tau_{12}}} = 1$ and thus  
\[
\UtoBfunctor(\tau_{12}) : 
\shiftFunct{\tikzdiag[scale=.25]{-1.75ex}{
	\draw (0,2) .. controls (0,1) and (1,1) .. (1,2);
	\draw (0,0) .. controls (0,1) and (1,1) .. (1,0);
	\draw[giga thick] (.5,.75) -- (.5,1.25);
}\ \tikzdiag[scale=.25]{-1.75ex}{
	\draw (0,2) .. controls (0,1) and (1,1) .. (1,2);
	\draw (0,0) .. controls (0,1) and (1,1) .. (1,0);
}^{(1,0)}} 
\tqft \left( \tikzdiag[scale=.25]{-1.75ex}{
	\draw (0,2) .. controls (0,1) and (1,1) .. (1,2);
	\draw (0,0) .. controls (0,1) and (1,1) .. (1,0);
}\ \tikzdiag[scale=.25]{-1.75ex}{
	\draw (0,2) .. controls (0,1) and (1,1) .. (1,2);
	\draw (0,0) .. controls (0,1) and (1,1) .. (1,0);
} \right) 
\xrightarrow{(-XY) \tqft\left(\tikzdiag[scale=.25]{-1.75ex}{
	\draw (0,2) .. controls (0,1) and (1,1) .. (1,2);
	\draw (0,0) .. controls (0,1) and (1,1) .. (1,0);
	\draw[giga thick] (.5,.75) -- (.5,1.25);
}\ \tikzdiag[scale=.25]{-1.75ex}{
	\draw (0,2) .. controls (0,1) and (1,1) .. (1,2);
	\draw (0,0) .. controls (0,1) and (1,1) .. (1,0);
} \right)}
 \tqft \left( \tikzdiag[scale=.25]{-1.75ex}{
	\draw (0,0) .. controls (0,.5) and (.25,.5) .. (.25,1) .. controls (.25,1.5) and (0,1.5) .. (0,2);
	\draw (1,0) .. controls (1,.5) and (.75,.5) .. (.75,1) .. controls (.75,1.5) and (1,1.5) .. (1,2);
}\ \tikzdiag[scale=.25]{-1.75ex}{
	\draw (0,2) .. controls (0,1) and (1,1) .. (1,2);
	\draw (0,0) .. controls (0,1) and (1,1) .. (1,0);
} \right).
\]
For $\tau_{54}$, we have again $\shiftFunct{H_{\tau_{54}}} = \id$, and thus we obtain
\[
\UtoBfunctor(\tau_{54}) : 
 \tqft \left( \tikzdiag[scale=.25]{-1.75ex}{
	\draw (0,0) .. controls (0,.5) and (.25,.5) .. (.25,1) .. controls (.25,1.5) and (0,1.5) .. (0,2);
	\draw (1,0) .. controls (1,.5) and (.75,.5) .. (.75,1) .. controls (.75,1.5) and (1,1.5) .. (1,2);
}\ \tikzdiag[scale=.25]{-1.75ex}{
	\draw (0,2) .. controls (0,1) and (1,1) .. (1,2);
	\draw (0,0) .. controls (0,1) and (1,1) .. (1,0);
} \right) 
\xrightarrow{ \shiftFunct{H} }
\shiftFunct{\tikzdiag[scale=.25]{-1.75ex}{
	\draw (0,0) .. controls (0,.5) and (.25,.5) .. (.25,1) .. controls (.25,1.5) and (0,1.5) .. (0,2);
	\draw (1,0) .. controls (1,.5) and (.75,.5) .. (.75,1) .. controls (.75,1.5) and (1,1.5) .. (1,2);
}\ \tikzdiag[scale=.25]{-1.75ex}{
	\draw (0,0) .. controls (0,.5) and (.25,.5) .. (.25,1) .. controls (.25,1.5) and (0,1.5) .. (0,2);
	\draw (1,0) .. controls (1,.5) and (.75,.5) .. (.75,1) .. controls (.75,1.5) and (1,1.5) .. (1,2);
	\draw[giga thick] (.25,1) -- (.75,1);
}^{(1,0)}} 
\circ
\shiftFunct{\tikzdiag[scale=.25]{-1.75ex}{
	\draw (0,0) .. controls (0,.5) and (.25,.5) .. (.25,1) .. controls (.25,1.5) and (0,1.5) .. (0,2);
	\draw (1,0) .. controls (1,.5) and (.75,.5) .. (.75,1) .. controls (.75,1.5) and (1,1.5) .. (1,2);
}\ \tikzdiag[scale=.25]{-1.75ex}{
	\draw (0,2) .. controls (0,1) and (1,1) .. (1,2);
	\draw (0,0) .. controls (0,1) and (1,1) .. (1,0);
	\draw[giga thick] (.5,.75) -- (.5,1.25);
}^{(0,1)}} 
 \tqft \left( \tikzdiag[scale=.25]{-1.75ex}{
	\draw (0,0) .. controls (0,.5) and (.25,.5) .. (.25,1) .. controls (.25,1.5) and (0,1.5) .. (0,2);
	\draw (1,0) .. controls (1,.5) and (.75,.5) .. (.75,1) .. controls (.75,1.5) and (1,1.5) .. (1,2);
}\ \tikzdiag[scale=.25]{-1.75ex}{
	\draw (0,2) .. controls (0,1) and (1,1) .. (1,2);
	\draw (0,0) .. controls (0,1) and (1,1) .. (1,0);
} \right) 
\xrightarrow{ \shiftFunct{\tikzdiag[scale=.25]{-1.75ex}{
	\draw (0,0) .. controls (0,.5) and (.25,.5) .. (.25,1) .. controls (.25,1.5) and (0,1.5) .. (0,2);
	\draw (1,0) .. controls (1,.5) and (.75,.5) .. (.75,1) .. controls (.75,1.5) and (1,1.5) .. (1,2);
}\ \tikzdiag[scale=.25]{-1.75ex}{
	\draw (0,0) .. controls (0,.5) and (.25,.5) .. (.25,1) .. controls (.25,1.5) and (0,1.5) .. (0,2);
	\draw (1,0) .. controls (1,.5) and (.75,.5) .. (.75,1) .. controls (.75,1.5) and (1,1.5) .. (1,2);
	\draw[giga thick] (.25,1) -- (.75,1);
}^{(1,0)}} \tqft\left( \tikzdiag[scale=.25]{-1.75ex}{
	\draw (0,0) .. controls (0,.5) and (.25,.5) .. (.25,1) .. controls (.25,1.5) and (0,1.5) .. (0,2);
	\draw (1,0) .. controls (1,.5) and (.75,.5) .. (.75,1) .. controls (.75,1.5) and (1,1.5) .. (1,2);
}\ \tikzdiag[scale=.25]{-1.75ex}{
	\draw (0,2) .. controls (0,1) and (1,1) .. (1,2);
	\draw (0,0) .. controls (0,1) and (1,1) .. (1,0);
	\draw[giga thick] (.5,.75) -- (.5,1.25);
} \right) }
\shiftFunct{\tikzdiag[scale=.25]{-1.75ex}{
	\draw (0,0) .. controls (0,.5) and (.25,.5) .. (.25,1) .. controls (.25,1.5) and (0,1.5) .. (0,2);
	\draw (1,0) .. controls (1,.5) and (.75,.5) .. (.75,1) .. controls (.75,1.5) and (1,1.5) .. (1,2);
}\ \tikzdiag[scale=.25]{-1.75ex}{
	\draw (0,0) .. controls (0,.5) and (.25,.5) .. (.25,1) .. controls (.25,1.5) and (0,1.5) .. (0,2);
	\draw (1,0) .. controls (1,.5) and (.75,.5) .. (.75,1) .. controls (.75,1.5) and (1,1.5) .. (1,2);
	\draw[giga thick] (.25,1) -- (.75,1);
}^{(1,0)}} 
 \tqft \left( \tikzdiag[scale=.25]{-1.75ex}{
	\draw (0,0) .. controls (0,.5) and (.25,.5) .. (.25,1) .. controls (.25,1.5) and (0,1.5) .. (0,2);
	\draw (1,0) .. controls (1,.5) and (.75,.5) .. (.75,1) .. controls (.75,1.5) and (1,1.5) .. (1,2);
}\ \tikzdiag[scale=.25]{-1.75ex}{
	\draw (0,0) .. controls (0,.5) and (.25,.5) .. (.25,1) .. controls (.25,1.5) and (0,1.5) .. (0,2);
	\draw (1,0) .. controls (1,.5) and (.75,.5) .. (.75,1) .. controls (.75,1.5) and (1,1.5) .. (1,2);
} \right),
\]
where $\shiftFunct{H}$ is computed as the product of the scalars obtained by the sequence of changes of chronology:
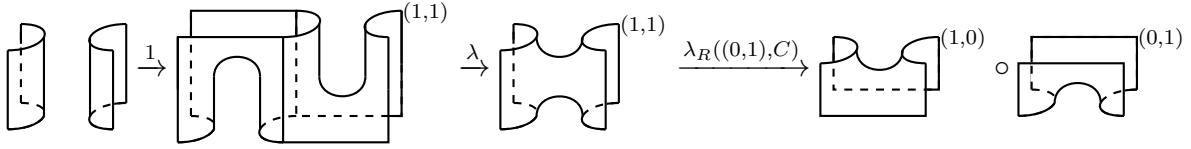
\[
\tikzdiagc[scale=.35]{
	\draw (0,0) .. controls (1.5,0) and (2,1) .. (.5,1); 
	\draw (4,0) .. controls (2.5,0) and (3,1) .. (4.5,1); 
	\filldraw[fill=white, draw=white] (.5,.55)  rectangle  (4,2);
	\draw[dashed] (0,0) .. controls (1.5,0) and (2,1) .. (.5,1); 
	\draw[dashed] (4,0) .. controls (2.5,0) and (3,1) .. (4.5,1); 
	\draw (1.375,.5) -- (1.375,3.5);
	\draw[yshift=1cm]  (0,2) .. controls (1.5,2) and (2,3) .. (.5,3); 
	\draw[yshift=1cm]  (4,2) .. controls (2.5,2) and (3,3) .. (4.5,3); 
	%
	\draw (3.125,.5) -- (3.125,3.5);
	\draw (0,0) -- (0,3);
	\draw (.5,3.25) -- (.5,4);
	\draw[dashed] (.5,1) -- (.5,4);
	\draw (4,0) -- (4,3);
	\draw (4.5,1) -- (4.5,4);
}
\xrightarrow{1}
 \tikzdiagc[scale=.35]{
	\draw (0,0) .. controls (1.5,0) and (2,1) .. (.5,1); 
	\draw (4,0) .. controls (2.5,0) and (3,1) .. (4.5,1); 
	\filldraw[fill=white, draw=white] (.5,.55)  rectangle  (4.5,2);
	\draw[dashed] (0,0) .. controls (1.5,0) and (2,1) .. (.5,1); 
	\draw[dashed] (4,0) .. controls (2.5,0) and (3,1) .. (4.5,1); 
	\draw (0,0) -- (0,4);
	\draw[dashed] (.5,1) -- (.5,5);
	\draw (1.375,.5) -- (1.375,2.5);
	\draw (3.125,.5) -- (3.125,2.5);
	\draw (.5,4) -- (.5,5) -- (4.5,5) -- (4.5,4);
	\draw (1.375,2.5) .. controls (1.375,3.5) and (3.125,3.5) .. (3.125,2.5);
	\draw (0,2) -- (0,4) -- (4,4);
	\draw (8,1) -- (8.5,1) -- (8.5,5);
	\draw (5.375,2.5) .. controls (5.375,1.5) and (7.125,1.5) .. (7.125,2.5);
	\draw[dashed] (4.5,5) -- (4.5,1) -- (8.5,1) -- (8.5,3);
	\draw (4,4) -- (4,0) -- (8,0) -- (8,2);
	\draw (4,4) .. controls (5.5,4) and (6,5) .. (4.5,5); 
	\draw (8,4) .. controls (6.5,4) and (7,5) .. (8.5,5); 
	\draw (8,2) -- (8,4);
	\draw (8.5,3) -- (8.5,5);
	\draw (5.375,2.5) -- (5.375,4.5);
	\draw (7.125,2.5) -- (7.125,4.5);
}
^{(1,1)}
\xrightarrow{\lambda}
\tikzdiagc[scale=.35]{
	\draw (0,0) .. controls (1.5,0) and (2,1) .. (.5,1); 
	\draw (4,0) .. controls (2.5,0) and (3,1) .. (4.5,1); 
	\filldraw[fill=white, draw=white] (.5,.55)  rectangle  (4,2);
	\draw[dashed] (0,0) .. controls (1.5,0) and (2,1) .. (.5,1); 
	\draw[dashed] (4,0) .. controls (2.5,0) and (3,1) .. (4.5,1); 
	\draw (1.375,.5) .. controls (1.375,1.5) and (3.125,1.5) .. (3.125,.5);
	\draw[yshift=1cm]  (0,2) .. controls (1.5,2) and (2,3) .. (.5,3); 
	\draw[yshift=1cm]  (4,2) .. controls (2.5,2) and (3,3) .. (4.5,3); 
	\draw[yshift=1cm]  (1.375,2.5) .. controls (1.375,1.5) and (3.125,1.5) .. (3.125,2.5);
	\draw (0,0) -- (0,3);
	\draw (.5,3.25) -- (.5,4);
	\draw[dashed] (.5,1) -- (.5,4);
	\draw (4,0) -- (4,3);
	\draw (4.5,1) -- (4.5,4);
}^{(1,1)}
\xrightarrow{\lambda_R((0,1),C)}
 \tikzdiagc[scale=.35]{
	\draw (.5,3) -- (.5,1) -- (4.5,1) -- (4.5,3);
	\filldraw [fill=white, draw=white] (0,0) rectangle (4,2); 
	\draw (0,2) .. controls (1.5,2) and (2,3) .. (.5,3); 
	\draw (4,2) .. controls (2.5,2) and (3,3) .. (4.5,3); 
	\draw (1.375,2.5) .. controls (1.375,1.5) and (3.125,1.5) .. (3.125,2.5);
	\draw[dashed] (.5,3) -- (.5,1) -- (4.5,1) -- (4.5,3);
	\draw (0,2) -- (0,0) -- (4,0) -- (4,2);
}^{(1,0)}
\circ
\tikzdiagc[scale=.35]{
	\draw (.5,1) -- (.5,3) -- (4.5,3) -- (4.5,1);
	\draw (0,0) .. controls (1.5,0) and (2,1) .. (.5,1); 
	\draw (4,0) .. controls (2.5,0) and (3,1) .. (4.5,1); 
	\filldraw[fill=white, draw=white] (.5,.55)  rectangle  (4,2);
	\draw[dashed] (0,0) .. controls (1.5,0) and (2,1) .. (.5,1); 
	\draw[dashed] (4,0) .. controls (2.5,0) and (3,1) .. (4.5,1); 
	\draw (1.375,.5) .. controls (1.375,1.5) and (3.125,1.5) .. (3.125,.5);
	\draw[dashed] (.5,1) -- (.5,3) -- (4.5,3) -- (4.5,1);
	\draw (0,0) -- (0,2) -- (4,2) -- (4,0);
}^{(0,1)}
\]
where the saddle is above the pair of arcs at the right, and 
where $\lambda = 1$ and $C = (-1,0)$ if the saddle $\tikzdiag[scale=.25]{-1.75ex}{
	\draw (0,0) .. controls (0,.5) and (.25,.5) .. (.25,1) .. controls (.25,1.5) and (0,1.5) .. (0,2);
	\draw (1,0) .. controls (1,.5) and (.75,.5) .. (.75,1) .. controls (.75,1.5) and (1,1.5) .. (1,2);
}\ \tikzdiag[scale=.25]{-1.75ex}{
	\draw (0,0) .. controls (0,.5) and (.25,.5) .. (.25,1) .. controls (.25,1.5) and (0,1.5) .. (0,2);
	\draw (1,0) .. controls (1,.5) and (.75,.5) .. (.75,1) .. controls (.75,1.5) and (1,1.5) .. (1,2);
	\draw[giga thick] (.25,1) -- (.75,1);
}$ is a merge, and $\lambda=Z$ and $C = (0,-1)$ if it is a split. These are computed using \cref{fig:chchange}. 
Then, note that $\tikzdiag[scale=.25]{-1.75ex}{
	\draw (0,0) .. controls (0,.5) and (.25,.5) .. (.25,1) .. controls (.25,1.5) and (0,1.5) .. (0,2);
	\draw (1,0) .. controls (1,.5) and (.75,.5) .. (.75,1) .. controls (.75,1.5) and (1,1.5) .. (1,2);
}\ \tikzdiag[scale=.25]{-1.75ex}{
	\draw (0,0) .. controls (0,.5) and (.25,.5) .. (.25,1) .. controls (.25,1.5) and (0,1.5) .. (0,2);
	\draw (1,0) .. controls (1,.5) and (.75,.5) .. (.75,1) .. controls (.75,1.5) and (1,1.5) .. (1,2);
	\draw[giga thick] (.25,1) -- (.75,1);
}^{(1,0)}$ has degree $((1,0)+C)$, thus
\[
 \shiftFunct{\tikzdiag[scale=.25]{-1.75ex}{
	\draw (0,0) .. controls (0,.5) and (.25,.5) .. (.25,1) .. controls (.25,1.5) and (0,1.5) .. (0,2);
	\draw (1,0) .. controls (1,.5) and (.75,.5) .. (.75,1) .. controls (.75,1.5) and (1,1.5) .. (1,2);
}\ \tikzdiag[scale=.25]{-1.75ex}{
	\draw (0,0) .. controls (0,.5) and (.25,.5) .. (.25,1) .. controls (.25,1.5) and (0,1.5) .. (0,2);
	\draw (1,0) .. controls (1,.5) and (.75,.5) .. (.75,1) .. controls (.75,1.5) and (1,1.5) .. (1,2);
	\draw[giga thick] (.25,1) -- (.75,1);
}^{(1,0)}} 
\tqft\left( \tikzdiag[scale=.25]{-1.75ex}{
	\draw (0,0) .. controls (0,.5) and (.25,.5) .. (.25,1) .. controls (.25,1.5) and (0,1.5) .. (0,2);
	\draw (1,0) .. controls (1,.5) and (.75,.5) .. (.75,1) .. controls (.75,1.5) and (1,1.5) .. (1,2);
}\ \tikzdiag[scale=.25]{-1.75ex}{
	\draw (0,2) .. controls (0,1) and (1,1) .. (1,2);
	\draw (0,0) .. controls (0,1) and (1,1) .. (1,0);
	\draw[giga thick] (.5,.75) -- (.5,1.25);
} \right)
 = \lambda_R\bigl((0,-1), (1,0)+C)\bigr) \tqft\left( \tikzdiag[scale=.25]{-1.75ex}{
	\draw (0,0) .. controls (0,.5) and (.25,.5) .. (.25,1) .. controls (.25,1.5) and (0,1.5) .. (0,2);
	\draw (1,0) .. controls (1,.5) and (.75,.5) .. (.75,1) .. controls (.75,1.5) and (1,1.5) .. (1,2);
}\ \tikzdiag[scale=.25]{-1.75ex}{
	\draw (0,2) .. controls (0,1) and (1,1) .. (1,2);
	\draw (0,0) .. controls (0,1) and (1,1) .. (1,0);
	\draw[giga thick] (.5,.75) -- (.5,1.25);
} \right).
\]
Putting all this together, we obtain
\[
\UtoBfunctor(\tau_{54}\tau_{12}) = \lambda (-XYZ)  \tqft\left(
 \tikzdiag[scale=.25]{-1.75ex}{
	\draw (0,0) .. controls (0,.5) and (.25,.5) .. (.25,1) .. controls (.25,1.5) and (0,1.5) .. (0,2);
	\draw (1,0) .. controls (1,.5) and (.75,.5) .. (.75,1) .. controls (.75,1.5) and (1,1.5) .. (1,2);
}\ \tikzdiag[scale=.25]{-1.75ex}{
	\draw (0,2) .. controls (0,1) and (1,1) .. (1,2);
	\draw (0,0) .. controls (0,1) and (1,1) .. (1,0);
	\draw[giga thick] (.5,.75) -- (.5,1.25);
} 
\circ
 \tikzdiag[scale=.25]{-1.75ex}{
	\draw (0,2) .. controls (0,1) and (1,1) .. (1,2);
	\draw (0,0) .. controls (0,1) and (1,1) .. (1,0);
	\draw[giga thick] (.5,.75) -- (.5,1.25);
}\ \tikzdiag[scale=.25]{-1.75ex}{
	\draw (0,2) .. controls (0,1) and (1,1) .. (1,2);
	\draw (0,0) .. controls (0,1) and (1,1) .. (1,0);
} \right).
\]
Now, we do the bottom path. We start by $\tau_{54}$ giving
\begin{align*}
&
\UtoBfunctor(\tau_{54}) : 
\shiftFunct{\tikzdiag[scale=.25]{-1.75ex}{
	\draw (0,2) .. controls (0,1) and (1,1) .. (1,2);
	\draw (0,0) .. controls (0,1) and (1,1) .. (1,0);
	\draw[giga thick] (.5,.75) -- (.5,1.25);
}\ \tikzdiag[scale=.25]{-1.75ex}{
	\draw (0,2) .. controls (0,1) and (1,1) .. (1,2);
	\draw (0,0) .. controls (0,1) and (1,1) .. (1,0);
}^{(1,0)}} 
\tqft \left( \tikzdiag[scale=.25]{-1.75ex}{
	\draw (0,2) .. controls (0,1) and (1,1) .. (1,2);
	\draw (0,0) .. controls (0,1) and (1,1) .. (1,0);
}\ \tikzdiag[scale=.25]{-1.75ex}{
	\draw (0,2) .. controls (0,1) and (1,1) .. (1,2);
	\draw (0,0) .. controls (0,1) and (1,1) .. (1,0);
} \right) 
\xrightarrow{\shiftFunct{H'}}
\shiftFunct{\tikzdiag[scale=.25]{-1.75ex}{
	\draw (0,2) .. controls (0,1) and (1,1) .. (1,2);
	\draw (0,0) .. controls (0,1) and (1,1) .. (1,0);
	\draw[giga thick] (.5,.75) -- (.5,1.25);
}\ \tikzdiag[scale=.25]{-1.75ex}{
	\draw (0,2) .. controls (0,1) and (1,1) .. (1,2);
	\draw (0,0) .. controls (0,1) and (1,1) .. (1,0);
}^{(1,0)}} 
\circ
\shiftFunct{\tikzdiag[scale=.25]{-1.75ex}{
	\draw (0,2) .. controls (0,1) and (1,1) .. (1,2);
	\draw (0,0) .. controls (0,1) and (1,1) .. (1,0);
}\ \tikzdiag[scale=.25]{-1.75ex}{
	\draw (0,0) .. controls (0,.5) and (.25,.5) .. (.25,1) .. controls (.25,1.5) and (0,1.5) .. (0,2);
	\draw (1,0) .. controls (1,.5) and (.75,.5) .. (.75,1) .. controls (.75,1.5) and (1,1.5) .. (1,2);
	\draw[giga thick] (.25,1) -- (.75,1);
}^{(1,0)}} 
\circ
\shiftFunct{\tikzdiag[scale=.25]{-1.75ex}{
	\draw (0,2) .. controls (0,1) and (1,1) .. (1,2);
	\draw (0,0) .. controls (0,1) and (1,1) .. (1,0);
}\ \tikzdiag[scale=.25]{-1.75ex}{
	\draw (0,2) .. controls (0,1) and (1,1) .. (1,2);
	\draw (0,0) .. controls (0,1) and (1,1) .. (1,0);
	\draw[giga thick] (.5,.75) -- (.5,1.25);
}^{(0,1)}} 
\tqft \left( \tikzdiag[scale=.25]{-1.75ex}{
	\draw (0,2) .. controls (0,1) and (1,1) .. (1,2);
	\draw (0,0) .. controls (0,1) and (1,1) .. (1,0);
}\ \tikzdiag[scale=.25]{-1.75ex}{
	\draw (0,2) .. controls (0,1) and (1,1) .. (1,2);
	\draw (0,0) .. controls (0,1) and (1,1) .. (1,0);
} \right) 
\\
&\xrightarrow{\shiftFunct{\tikzdiag[scale=.25]{-1.75ex}{
	\draw (0,2) .. controls (0,1) and (1,1) .. (1,2);
	\draw (0,0) .. controls (0,1) and (1,1) .. (1,0);
	\draw[giga thick] (.5,.75) -- (.5,1.25);
}\ \tikzdiag[scale=.25]{-1.75ex}{
	\draw (0,2) .. controls (0,1) and (1,1) .. (1,2);
	\draw (0,0) .. controls (0,1) and (1,1) .. (1,0);
}^{(1,0)}} 
\circ
\shiftFunct{\tikzdiag[scale=.25]{-1.75ex}{
	\draw (0,2) .. controls (0,1) and (1,1) .. (1,2);
	\draw (0,0) .. controls (0,1) and (1,1) .. (1,0);
}\ \tikzdiag[scale=.25]{-1.75ex}{
	\draw (0,0) .. controls (0,.5) and (.25,.5) .. (.25,1) .. controls (.25,1.5) and (0,1.5) .. (0,2);
	\draw (1,0) .. controls (1,.5) and (.75,.5) .. (.75,1) .. controls (.75,1.5) and (1,1.5) .. (1,2);
	\draw[giga thick] (.25,1) -- (.75,1);
}^{(1,0)}} 
\tqft\left(  \tikzdiag[scale=.25]{-1.75ex}{
	\draw (0,2) .. controls (0,1) and (1,1) .. (1,2);
	\draw (0,0) .. controls (0,1) and (1,1) .. (1,0);
}\ \tikzdiag[scale=.25]{-1.75ex}{
	\draw (0,2) .. controls (0,1) and (1,1) .. (1,2);
	\draw (0,0) .. controls (0,1) and (1,1) .. (1,0);
	\draw[giga thick] (.5,.75) -- (.5,1.25);
} \right)} 
\shiftFunct{\tikzdiag[scale=.25]{-1.75ex}{
	\draw (0,2) .. controls (0,1) and (1,1) .. (1,2);
	\draw (0,0) .. controls (0,1) and (1,1) .. (1,0);
	\draw[giga thick] (.5,.75) -- (.5,1.25);
}\ \tikzdiag[scale=.25]{-1.75ex}{
	\draw (0,2) .. controls (0,1) and (1,1) .. (1,2);
	\draw (0,0) .. controls (0,1) and (1,1) .. (1,0);
}^{(1,0)}} 
\circ
\shiftFunct{\tikzdiag[scale=.25]{-1.75ex}{
	\draw (0,2) .. controls (0,1) and (1,1) .. (1,2);
	\draw (0,0) .. controls (0,1) and (1,1) .. (1,0);
}\ \tikzdiag[scale=.25]{-1.75ex}{
	\draw (0,0) .. controls (0,.5) and (.25,.5) .. (.25,1) .. controls (.25,1.5) and (0,1.5) .. (0,2);
	\draw (1,0) .. controls (1,.5) and (.75,.5) .. (.75,1) .. controls (.75,1.5) and (1,1.5) .. (1,2);
	\draw[giga thick] (.25,1) -- (.75,1);
}^{(1,0)}} 
\tqft \left( \tikzdiag[scale=.25]{-1.75ex}{
	\draw (0,2) .. controls (0,1) and (1,1) .. (1,2);
	\draw (0,0) .. controls (0,1) and (1,1) .. (1,0);
}\ \tikzdiag[scale=.25]{-1.75ex}{
	\draw (0,0) .. controls (0,.5) and (.25,.5) .. (.25,1) .. controls (.25,1.5) and (0,1.5) .. (0,2);
	\draw (1,0) .. controls (1,.5) and (.75,.5) .. (.75,1) .. controls (.75,1.5) and (1,1.5) .. (1,2);
} \right)
\xrightarrow{\shiftFunct{H_{\tau_{54}}}}
\shiftFunct{\tikzdiag[scale=.25]{-1.75ex}{
	\draw (0,0) .. controls (0,.5) and (.25,.5) .. (.25,1) .. controls (.25,1.5) and (0,1.5) .. (0,2);
	\draw (1,0) .. controls (1,.5) and (.75,.5) .. (.75,1) .. controls (.75,1.5) and (1,1.5) .. (1,2);
}\ \tikzdiag[scale=.25]{-1.75ex}{
	\draw (0,0) .. controls (0,.5) and (.25,.5) .. (.25,1) .. controls (.25,1.5) and (0,1.5) .. (0,2);
	\draw (1,0) .. controls (1,.5) and (.75,.5) .. (.75,1) .. controls (.75,1.5) and (1,1.5) .. (1,2);
	\draw[giga thick] (.25,1) -- (.75,1);
}^{(1,0)}} 
\circ
\shiftFunct{\tikzdiag[scale=.25]{-1.75ex}{
	\draw (0,2) .. controls (0,1) and (1,1) .. (1,2);
	\draw (0,0) .. controls (0,1) and (1,1) .. (1,0);
	\draw[giga thick] (.5,.75) -- (.5,1.25);
}\ \tikzdiag[scale=.25]{-1.75ex}{
	\draw (0,0) .. controls (0,.5) and (.25,.5) .. (.25,1) .. controls (.25,1.5) and (0,1.5) .. (0,2);
	\draw (1,0) .. controls (1,.5) and (.75,.5) .. (.75,1) .. controls (.75,1.5) and (1,1.5) .. (1,2);
}^{(1,0)}} 
\tqft \left( \tikzdiag[scale=.25]{-1.75ex}{
	\draw (0,2) .. controls (0,1) and (1,1) .. (1,2);
	\draw (0,0) .. controls (0,1) and (1,1) .. (1,0);
}\ \tikzdiag[scale=.25]{-1.75ex}{
	\draw (0,0) .. controls (0,.5) and (.25,.5) .. (.25,1) .. controls (.25,1.5) and (0,1.5) .. (0,2);
	\draw (1,0) .. controls (1,.5) and (.75,.5) .. (.75,1) .. controls (.75,1.5) and (1,1.5) .. (1,2);
} \right)
\end{align*}
We compute $\shiftFunct{H'}$ as we did for $\shiftFunct{H}$ yielding
\[
\shiftFunct{H'} = \lambda' \lambda_R((0,1), C'),
\]
where $\lambda' = 1, C' = (-1,0)$ if the saddle $\tikzdiag[scale=.25]{-1.75ex}{
	\draw (0,2) .. controls (0,1) and (1,1) .. (1,2);
	\draw (0,0) .. controls (0,1) and (1,1) .. (1,0);
}\ \tikzdiag[scale=.25]{-1.75ex}{
	\draw (0,0) .. controls (0,.5) and (.25,.5) .. (.25,1) .. controls (.25,1.5) and (0,1.5) .. (0,2);
	\draw (1,0) .. controls (1,.5) and (.75,.5) .. (.75,1) .. controls (.75,1.5) and (1,1.5) .. (1,2);
	\draw[giga thick] (.25,1) -- (.75,1);
}$ is a merge, and $\lambda' = Z, C'=(0,-1)$ otherwise. Then, we have
\[
\shiftFunct{\tikzdiag[scale=.25]{-1.75ex}{
	\draw (0,2) .. controls (0,1) and (1,1) .. (1,2);
	\draw (0,0) .. controls (0,1) and (1,1) .. (1,0);
	\draw[giga thick] (.5,.75) -- (.5,1.25);
}\ \tikzdiag[scale=.25]{-1.75ex}{
	\draw (0,2) .. controls (0,1) and (1,1) .. (1,2);
	\draw (0,0) .. controls (0,1) and (1,1) .. (1,0);
}^{(1,0)}} 
\circ
\shiftFunct{\tikzdiag[scale=.25]{-1.75ex}{
	\draw (0,2) .. controls (0,1) and (1,1) .. (1,2);
	\draw (0,0) .. controls (0,1) and (1,1) .. (1,0);
}\ \tikzdiag[scale=.25]{-1.75ex}{
	\draw (0,0) .. controls (0,.5) and (.25,.5) .. (.25,1) .. controls (.25,1.5) and (0,1.5) .. (0,2);
	\draw (1,0) .. controls (1,.5) and (.75,.5) .. (.75,1) .. controls (.75,1.5) and (1,1.5) .. (1,2);
	\draw[giga thick] (.25,1) -- (.75,1);
}^{(1,0)}} 
\tqft\left(  \tikzdiag[scale=.25]{-1.75ex}{
	\draw (0,2) .. controls (0,1) and (1,1) .. (1,2);
	\draw (0,0) .. controls (0,1) and (1,1) .. (1,0);
}\ \tikzdiag[scale=.25]{-1.75ex}{
	\draw (0,2) .. controls (0,1) and (1,1) .. (1,2);
	\draw (0,0) .. controls (0,1) and (1,1) .. (1,0);
	\draw[giga thick] (.5,.75) -- (.5,1.25);
} \right)
 = 
\lambda_R\bigl( (0,-1), C''+ C' + (2,0) \bigr)
\tqft\left(  \tikzdiag[scale=.25]{-1.75ex}{
	\draw (0,2) .. controls (0,1) and (1,1) .. (1,2);
	\draw (0,0) .. controls (0,1) and (1,1) .. (1,0);
}\ \tikzdiag[scale=.25]{-1.75ex}{
	\draw (0,2) .. controls (0,1) and (1,1) .. (1,2);
	\draw (0,0) .. controls (0,1) and (1,1) .. (1,0);
	\draw[giga thick] (.5,.75) -- (.5,1.25);
} \right),
\]
where $C'' = \deg\left( \tikzdiag[scale=.25]{-1.75ex}{
	\draw (0,2) .. controls (0,1) and (1,1) .. (1,2);
	\draw (0,0) .. controls (0,1) and (1,1) .. (1,0);
	\draw[giga thick] (.5,.75) -- (.5,1.25);
}\ \tikzdiag[scale=.25]{-1.75ex}{
	\draw (0,2) .. controls (0,1) and (1,1) .. (1,2);
	\draw (0,0) .. controls (0,1) and (1,1) .. (1,0);
}\right)$. 
Then, we compute $\shiftFunct{\tau_{54}}$ as
\[
\shiftFunct{\tau_{54}} = \lambda_R \bigl( (1,0), C) \circ \shiftFunct{H''} \circ \lambda_R\bigl( C''+ (1,0), (1,0) \bigr),
\]
where $\shiftFunct{H''}$ is the change of chronology exchanging the two saddles in $\tikzdiag[scale=.25]{-1.75ex}{
	\draw (0,2) .. controls (0,1) and (1,1) .. (1,2);
	\draw (0,0) .. controls (0,1) and (1,1) .. (1,0);
	\draw[giga thick] (.5,.75) -- (.5,1.25);
}\ \tikzdiag[scale=.25]{-1.75ex}{
	\draw (0,0) .. controls (0,.5) and (.25,.5) .. (.25,1) .. controls (.25,1.5) and (0,1.5) .. (0,2);
	\draw (1,0) .. controls (1,.5) and (.75,.5) .. (.75,1) .. controls (.75,1.5) and (1,1.5) .. (1,2);
	\draw[giga thick] (.25,1) -- (.75,1);
}$.
For $\tau_{12}$ we have
\[
\UtoBfunctor(\tau_{12}) : 
\shiftFunct{\tikzdiag[scale=.25]{-1.75ex}{
	\draw (0,0) .. controls (0,.5) and (.25,.5) .. (.25,1) .. controls (.25,1.5) and (0,1.5) .. (0,2);
	\draw (1,0) .. controls (1,.5) and (.75,.5) .. (.75,1) .. controls (.75,1.5) and (1,1.5) .. (1,2);
}\ \tikzdiag[scale=.25]{-1.75ex}{
	\draw (0,0) .. controls (0,.5) and (.25,.5) .. (.25,1) .. controls (.25,1.5) and (0,1.5) .. (0,2);
	\draw (1,0) .. controls (1,.5) and (.75,.5) .. (.75,1) .. controls (.75,1.5) and (1,1.5) .. (1,2);
	\draw[giga thick] (.25,1) -- (.75,1);
}^{(1,0)}} 
\circ
\shiftFunct{\tikzdiag[scale=.25]{-1.75ex}{
	\draw (0,2) .. controls (0,1) and (1,1) .. (1,2);
	\draw (0,0) .. controls (0,1) and (1,1) .. (1,0);
	\draw[giga thick] (.5,.75) -- (.5,1.25);
}\ \tikzdiag[scale=.25]{-1.75ex}{
	\draw (0,0) .. controls (0,.5) and (.25,.5) .. (.25,1) .. controls (.25,1.5) and (0,1.5) .. (0,2);
	\draw (1,0) .. controls (1,.5) and (.75,.5) .. (.75,1) .. controls (.75,1.5) and (1,1.5) .. (1,2);
}^{(1,0)}} 
\tqft \left( \tikzdiag[scale=.25]{-1.75ex}{
	\draw (0,2) .. controls (0,1) and (1,1) .. (1,2);
	\draw (0,0) .. controls (0,1) and (1,1) .. (1,0);
}\ \tikzdiag[scale=.25]{-1.75ex}{
	\draw (0,0) .. controls (0,.5) and (.25,.5) .. (.25,1) .. controls (.25,1.5) and (0,1.5) .. (0,2);
	\draw (1,0) .. controls (1,.5) and (.75,.5) .. (.75,1) .. controls (.75,1.5) and (1,1.5) .. (1,2);
} \right)
\xrightarrow{(-XY)
\shiftFunct{\tikzdiag[scale=.25]{-1.75ex}{
	\draw (0,0) .. controls (0,.5) and (.25,.5) .. (.25,1) .. controls (.25,1.5) and (0,1.5) .. (0,2);
	\draw (1,0) .. controls (1,.5) and (.75,.5) .. (.75,1) .. controls (.75,1.5) and (1,1.5) .. (1,2);
}\ \tikzdiag[scale=.25]{-1.75ex}{
	\draw (0,0) .. controls (0,.5) and (.25,.5) .. (.25,1) .. controls (.25,1.5) and (0,1.5) .. (0,2);
	\draw (1,0) .. controls (1,.5) and (.75,.5) .. (.75,1) .. controls (.75,1.5) and (1,1.5) .. (1,2);
	\draw[giga thick] (.25,1) -- (.75,1);
}^{(1,0)}} 
\tqft \left( \tikzdiag[scale=.25]{-1.75ex}{
	\draw (0,2) .. controls (0,1) and (1,1) .. (1,2);
	\draw (0,0) .. controls (0,1) and (1,1) .. (1,0);
	\draw[giga thick] (.5,.75) -- (.5,1.25);
}\ \tikzdiag[scale=.25]{-1.75ex}{
	\draw (0,0) .. controls (0,.5) and (.25,.5) .. (.25,1) .. controls (.25,1.5) and (0,1.5) .. (0,2);
	\draw (1,0) .. controls (1,.5) and (.75,.5) .. (.75,1) .. controls (.75,1.5) and (1,1.5) .. (1,2);
} \right)) }
\shiftFunct{\tikzdiag[scale=.25]{-1.75ex}{
	\draw (0,0) .. controls (0,.5) and (.25,.5) .. (.25,1) .. controls (.25,1.5) and (0,1.5) .. (0,2);
	\draw (1,0) .. controls (1,.5) and (.75,.5) .. (.75,1) .. controls (.75,1.5) and (1,1.5) .. (1,2);
}\ \tikzdiag[scale=.25]{-1.75ex}{
	\draw (0,0) .. controls (0,.5) and (.25,.5) .. (.25,1) .. controls (.25,1.5) and (0,1.5) .. (0,2);
	\draw (1,0) .. controls (1,.5) and (.75,.5) .. (.75,1) .. controls (.75,1.5) and (1,1.5) .. (1,2);
	\draw[giga thick] (.25,1) -- (.75,1);
}^{(1,0)}} 
\tqft \left( \tikzdiag[scale=.25]{-1.75ex}{
	\draw (0,0) .. controls (0,.5) and (.25,.5) .. (.25,1) .. controls (.25,1.5) and (0,1.5) .. (0,2);
	\draw (1,0) .. controls (1,.5) and (.75,.5) .. (.75,1) .. controls (.75,1.5) and (1,1.5) .. (1,2);
}\ \tikzdiag[scale=.25]{-1.75ex}{
	\draw (0,0) .. controls (0,.5) and (.25,.5) .. (.25,1) .. controls (.25,1.5) and (0,1.5) .. (0,2);
	\draw (1,0) .. controls (1,.5) and (.75,.5) .. (.75,1) .. controls (.75,1.5) and (1,1.5) .. (1,2);
} \right)
\]
where
\[
\shiftFunct{\tikzdiag[scale=.25]{-1.75ex}{
	\draw (0,0) .. controls (0,.5) and (.25,.5) .. (.25,1) .. controls (.25,1.5) and (0,1.5) .. (0,2);
	\draw (1,0) .. controls (1,.5) and (.75,.5) .. (.75,1) .. controls (.75,1.5) and (1,1.5) .. (1,2);
}\ \tikzdiag[scale=.25]{-1.75ex}{
	\draw (0,0) .. controls (0,.5) and (.25,.5) .. (.25,1) .. controls (.25,1.5) and (0,1.5) .. (0,2);
	\draw (1,0) .. controls (1,.5) and (.75,.5) .. (.75,1) .. controls (.75,1.5) and (1,1.5) .. (1,2);
	\draw[giga thick] (.25,1) -- (.75,1);
}^{(1,0)}} 
\tqft \left( \tikzdiag[scale=.25]{-1.75ex}{
	\draw (0,2) .. controls (0,1) and (1,1) .. (1,2);
	\draw (0,0) .. controls (0,1) and (1,1) .. (1,0);
	\draw[giga thick] (.5,.75) -- (.5,1.25);
}\ \tikzdiag[scale=.25]{-1.75ex}{
	\draw (0,0) .. controls (0,.5) and (.25,.5) .. (.25,1) .. controls (.25,1.5) and (0,1.5) .. (0,2);
	\draw (1,0) .. controls (1,.5) and (.75,.5) .. (.75,1) .. controls (.75,1.5) and (1,1.5) .. (1,2);
} \right) 
=
\lambda_R\bigl((-1,0),C+(1,0) \bigl)
\tqft \left( \tikzdiag[scale=.25]{-1.75ex}{
	\draw (0,2) .. controls (0,1) and (1,1) .. (1,2);
	\draw (0,0) .. controls (0,1) and (1,1) .. (1,0);
	\draw[giga thick] (.5,.75) -- (.5,1.25);
}\ \tikzdiag[scale=.25]{-1.75ex}{
	\draw (0,0) .. controls (0,.5) and (.25,.5) .. (.25,1) .. controls (.25,1.5) and (0,1.5) .. (0,2);
	\draw (1,0) .. controls (1,.5) and (.75,.5) .. (.75,1) .. controls (.75,1.5) and (1,1.5) .. (1,2);
} \right).
\]
Putting all these together, we obtain
\[
\UtoBfunctor(\tau_{12}\tau_{54}) =
\lambda'  \shiftFunct{H''} \lambda_R(C'', (1,1))  (-XYZ^{2})   \tqft\left(
 \tikzdiag[scale=.25]{-1.75ex}{
	\draw (0,2) .. controls (0,1) and (1,1) .. (1,2);
	\draw (0,0) .. controls (0,1) and (1,1) .. (1,0);
	\draw[giga thick] (.5,.75) -- (.5,1.25);
}\ \tikzdiag[scale=.25]{-1.75ex}{
	\draw (0,0) .. controls (0,.5) and (.25,.5) .. (.25,1) .. controls (.25,1.5) and (0,1.5) .. (0,2);
	\draw (1,0) .. controls (1,.5) and (.75,.5) .. (.75,1) .. controls (.75,1.5) and (1,1.5) .. (1,2);
} 
\circ
 \tikzdiag[scale=.25]{-1.75ex}{
	\draw (0,2) .. controls (0,1) and (1,1) .. (1,2);
	\draw (0,0) .. controls (0,1) and (1,1) .. (1,0);
}\ \tikzdiag[scale=.25]{-1.75ex}{
	\draw (0,2) .. controls (0,1) and (1,1) .. (1,2);
	\draw (0,0) .. controls (0,1) and (1,1) .. (1,0);
	\draw[giga thick] (.5,.75) -- (.5,1.25);
} \right).
\]
In order to be able to compare the two results, we will now assume the tangle is closed by putting two arcs 
$\tikzdiag[scale=.25]{-1ex}{
	\draw (0,0) .. controls (0,-.5) and (1,-.5) .. (1,0);
	\draw (2,0) .. controls (2,-.5) and (3,-.5) .. (3,0);
}$
at the bottom and at the top. Then, we have
\[
\tqft\left(
 \tikzdiag[scale=.25]{-1.75ex}{
	\draw (0,2) .. controls (0,1) and (1,1) .. (1,2);
	\draw (0,0) .. controls (0,1) and (1,1) .. (1,0);
	\draw[giga thick] (.5,.75) -- (.5,1.25);
}\ \tikzdiag[scale=.25]{-1.75ex}{
	\draw (0,0) .. controls (0,.5) and (.25,.5) .. (.25,1) .. controls (.25,1.5) and (0,1.5) .. (0,2);
	\draw (1,0) .. controls (1,.5) and (.75,.5) .. (.75,1) .. controls (.75,1.5) and (1,1.5) .. (1,2);
} 
\circ
 \tikzdiag[scale=.25]{-1.75ex}{
	\draw (0,2) .. controls (0,1) and (1,1) .. (1,2);
	\draw (0,0) .. controls (0,1) and (1,1) .. (1,0);
}\ \tikzdiag[scale=.25]{-1.75ex}{
	\draw (0,2) .. controls (0,1) and (1,1) .. (1,2);
	\draw (0,0) .. controls (0,1) and (1,1) .. (1,0);
	\draw[giga thick] (.5,.75) -- (.5,1.25);
} \right)
=
X
\tqft\left(
 \tikzdiag[scale=.25]{-1.75ex}{
	\draw (0,0) .. controls (0,.5) and (.25,.5) .. (.25,1) .. controls (.25,1.5) and (0,1.5) .. (0,2);
	\draw (1,0) .. controls (1,.5) and (.75,.5) .. (.75,1) .. controls (.75,1.5) and (1,1.5) .. (1,2);
}\ \tikzdiag[scale=.25]{-1.75ex}{
	\draw (0,2) .. controls (0,1) and (1,1) .. (1,2);
	\draw (0,0) .. controls (0,1) and (1,1) .. (1,0);
	\draw[giga thick] (.5,.75) -- (.5,1.25);
} 
\circ
 \tikzdiag[scale=.25]{-1.75ex}{
	\draw (0,2) .. controls (0,1) and (1,1) .. (1,2);
	\draw (0,0) .. controls (0,1) and (1,1) .. (1,0);
	\draw[giga thick] (.5,.75) -- (.5,1.25);
}\ \tikzdiag[scale=.25]{-1.75ex}{
	\draw (0,2) .. controls (0,1) and (1,1) .. (1,2);
	\draw (0,0) .. controls (0,1) and (1,1) .. (1,0);
} \right),
\]
and $\lambda = Z$, $\lambda' = Z$, $C'' = (-1,0)$ and $\shiftFunct{H''} = Z$. Thus, we obtain $\UtoBfunctor(\tau_{12}\tau_{54}) = Z \UtoBfunctor(\tau_{54}\tau_{12})$, which agree as expected with $ \tau_{12} \tau_{54}= Z\tau_{54} \tau_{12} $. 
We leave as an exercise to the reader the computation of other cases. 
\end{exe}

\subsubsection{$\cR$-relations}

Fix two weights $\bw$ and $\bw'$ with $s(\bw) \equiv s(\bw') \equiv 0 \mod 2$. 

\begin{lem}\label{lem:onefunctorUtoB}
We have that
\[
\UtoBfunctor  : \Hom_\oddKC(\bw, \bw') \rightarrow \BIMOD_q(H^{s(\bw)/2}, H^{s(\bw')/2})
\]
is a well-defined 1-functor.
\end{lem}

\begin{proof}
We need to show that the image of $\UtoBfunctor$ respects the relations \cref{eq:doubleR} to \cref{eq:splitdcross}, and the graded planar isotopies. 
Thanks to \cref{lem:FLijisnatural} below and \cref{prop:fotimesgcommutes}, we can work locally since for $f : \F_\bi \rightarrow \F_\bj$ the following diagram $\gradC$-graded commutes:
\[
\begin{tikzcd}[column sep=6ex]
\UtoBfunctor(\F_{\bi_2} \F_{\bi} \F_{\bi_1}) 
\ar{r}{\UtoBfunctor(\bi_2, \bi\bi_1)^{-1}}
\ar[swap]{dd}{\UtoBfunctor(1\circ f \circ 1)}
&
\UtoBfunctor(\F_{\bi_2}) \otimes_H
\UtoBfunctor(\F_{\bi} \F_{\bi_1}) 
\ar{r}{1 \otimes \UtoBfunctor(\bi,\bi_1)^{-1}}
\ar{dd}{\UtoBfunctor(1) \otimes \UtoBfunctor(f \circ 1)}
&
\UtoBfunctor(\F_{\bi_2}) 
\otimes
\bigl(
\UtoBfunctor(\F_{\bi} )
\otimes
\UtoBfunctor(\F_{\bi_1}) 
\bigr)
\ar{dd}{\UtoBfunctor(1) \otimes \UtoBfunctor(f) \otimes \UtoBfunctor(1)}
\\
{}
&
{}
\ar[phantom]{l}{\Leftarrow}
&
{}
\ar[phantom]{l}{\Leftarrow}
\\
\UtoBfunctor(\F_{\bi_2} \F_{\bj} \F_{\bi_1}) 
&
\UtoBfunctor(\F_{\bi_2}) \otimes_H
\UtoBfunctor(\F_{\bj} \F_{\bi_1}) 
\ar{l}{\UtoBfunctor(\bi_2, \bj\bi_1)}
&
\UtoBfunctor(\F_{\bi_2}) 
\otimes
\bigl(
\UtoBfunctor(\F_{\bi} )
\otimes
\UtoBfunctor(\F_{\bi_1}) 
\bigr)
\ar{l}{1 \otimes \UtoBfunctor(\bj,\bi_1)}
\end{tikzcd}
\]
Therefore, if we can show that $\UtoBfunctor(f) = \UtoBfunctor(g)$ for some $g$ and $|\UtoBfunctor(f)|_\gradC = |\UtoBfunctor(g)|_\gradC$, then we have $\UtoBfunctor(1 \circ f \circ 1) = \UtoBfunctor(1 \circ g \circ 1)$. 
Furthermore, we also obtain that $\UtoBfunctor$ is stable under graded planar isotopy, thanks to \cref{eq:fotimesgradedcommutes}. 
Also, for similar reasons as in \cite[Proof of Proposition 2.11]{pedrodd}, it is enough to show that $\UtoBfunctor$ is stables under the relations \cref{eq:doubleR} - \cref{eq:doublesplitters}, and the remaining ones will follow. It appears that we can already verify (most of) the relations at the level of $\dotchcobcat$, so that we can do most of our computations using dotted cobordisms. 

We have that
\begin{align*}
\UtoBfunctor\left(
\tikzdiagc{
	\draw (0,0) node[below]{\small $\alpha_i$} -- (0,1) node[pos=.33,tikzdot]{} node[pos=.66,tikzdot]{};
}\right)
&= 0,
&
\UtoBfunctor\left(
\tikzdiagc[yscale=.5]{
	\draw (0,0) node[below]{\small $\alpha_i$} .. controls (0,.5) and (1,.5) .. (1,1) .. controls (1,1.5) and (0,1.5) .. (0,2);
	\draw (1,0) node[below]{\small $\alpha_i$} .. controls (1,.5) and (0,.5) .. (0,1) .. controls (0,1.5) and (1,1.5) .. (1,2);
}\right)
&= 0,
\end{align*}
immediately from \eqref{eq:dotcob}. 

We obtain
\[
\UtoBfunctor\left(
\tikzdiagc[yscale=.5]{
	\draw (0,0) node[below]{\small $\alpha_i$} .. controls (0,.5) and (1,.5) .. (1,1) .. controls (1,1.5) and (0,1.5) .. (0,2);
	\draw (1,0) node[below]{\small $\alpha_i$} .. controls (1,.5) and (0,.5) .. (0,1) .. controls (0,1.5) and (1,1.5) .. (1,2);
}\right)
=
\UtoBfunctor\left(
\tikzdiagc[yscale=.5]{
	\draw (0,0) node[below]{\small $\alpha_i$} -- (0,2);
	\draw (1,0) node[below]{\small $\alpha_i$} -- (1,2);
}\right)
\]
for $|i-j| > 1$ since the underlying cobordisms are related through a horizontal diffeotopy (we exchange the position of distant rungs before putting them at the same position again). Moreover, the composition of changes of chronology involved in \cref{eq:FLbicrossdistant} is homotopic to the identity, thus $\shiftFunct{\tau_{ji}} \circ \shiftFunct{\tau_{ij}} = 1$. 

In order to compute the case $j = i -1$, we first suppose the diagrams involved are not illegal. Then, we compute using \cref{eq:dotcob} that in $\dotchcobcat$:
\begin{align*}
 \lambda \ 
\tikzdiagc[scale=.35]{
	\draw (0,0) .. controls (1.5,0) and (2,1) .. (.5,1); 
	\draw (4,0) .. controls (2.5,0) and (3,1) .. (4.5,1); 
	\filldraw[fill=white, draw=white] (.5,.55)  rectangle  (4,2);
	\draw[dashed] (0,0) .. controls (1.5,0) and (2,1) .. (.5,1); 
	\draw[dashed] (4,0) .. controls (2.5,0) and (3,1) .. (4.5,1); 
	\draw (1.375,.5) .. controls (1.375,1.5) and (3.125,1.5) .. (3.125,.5);
	\draw[yshift=1cm]  (0,2) .. controls (1.5,2) and (2,3) .. (.5,3); 
	\draw[yshift=1cm]  (4,2) .. controls (2.5,2) and (3,3) .. (4.5,3); 
	\draw[yshift=1cm]  (1.375,2.5) .. controls (1.375,1.5) and (3.125,1.5) .. (3.125,2.5);
	\draw (0,0) -- (0,3);
	\draw (.5,3.25) -- (.5,4);
	\draw[dashed] (.5,1) -- (.5,4);
	\draw (4,0) -- (4,3);
	\draw (4.5,1) -- (4.5,4);
}
\ = \ 
 \tikzdiagc[scale=.35]{
	\draw (0,0) .. controls (1.5,0) and (2,1) .. (.5,1); 
	\draw (4,0) .. controls (2.5,0) and (3,1) .. (4.5,1); 
	\filldraw[fill=white, draw=white] (.5,.55)  rectangle  (4.5,2);
	\draw[dashed] (0,0) .. controls (1.5,0) and (2,1) .. (.5,1); 
	\draw[dashed] (4,0) .. controls (2.5,0) and (3,1) .. (4.5,1); 
	\draw (0,0) -- (0,4);
	\draw[dashed] (.5,1) -- (.5,5);
	\draw (1.375,.5) -- (1.375,2.5);
	\draw (3.125,.5) -- (3.125,2.5);
	\draw (.5,4) -- (.5,5) -- (4.5,5) -- (4.5,4);
	\draw (1.375,2.5) .. controls (1.375,3.5) and (3.125,3.5) .. (3.125,2.5);
	\draw (0,2) -- (0,4) -- (4,4);
	\draw (8,1) -- (8.5,1) -- (8.5,5);
	\draw (5.375,2.5) .. controls (5.375,1.5) and (7.125,1.5) .. (7.125,2.5);
	\draw[dashed] (4.5,5) -- (4.5,1) -- (8.5,1) -- (8.5,3);
	\draw (4,4) -- (4,0) -- (8,0) -- (8,2);
	\draw (4,4) .. controls (5.5,4) and (6,5) .. (4.5,5); 
	\draw (8,4) .. controls (6.5,4) and (7,5) .. (8.5,5); 
	\draw (8,2) -- (8,4);
	\draw (8.5,3) -- (8.5,5);
	\draw (5.375,2.5) -- (5.375,4.5);
	\draw (7.125,2.5) -- (7.125,4.5);
}
\ =  XZ \ 
\tikzdiagc[scale=.35]{
	\draw (0,0) .. controls (1.5,0) and (2,1) .. (.5,1); 
	\draw (4,0) .. controls (2.5,0) and (3,1) .. (4.5,1); 
	\filldraw[fill=white, draw=white] (.5,.55)  rectangle  (4,2);
	\draw[dashed] (0,0) .. controls (1.5,0) and (2,1) .. (.5,1); 
	\draw[dashed] (4,0) .. controls (2.5,0) and (3,1) .. (4.5,1); 
	\draw (1.375,.5) -- (1.375,3.5);
	\draw[yshift=1cm]  (0,2) .. controls (1.5,2) and (2,3) .. (.5,3); 
	\draw[yshift=1cm]  (4,2) .. controls (2.5,2) and (3,3) .. (4.5,3); 
	%
	\draw (3.125,.5) -- (3.125,3.5);
	\draw (0,0) -- (0,3);
	\draw (.5,3.25) -- (.5,4);
	\draw[dashed] (.5,1) -- (.5,4);
	\draw (4,0) -- (4,3);
	\draw (4.5,1) -- (4.5,4);
	\node [tikzdot] at(.5,1.75) {};
}
\ + YZ \ 
\tikzdiagc[scale=.35]{
	\draw (0,0) .. controls (1.5,0) and (2,1) .. (.5,1); 
	\draw (4,0) .. controls (2.5,0) and (3,1) .. (4.5,1); 
	\filldraw[fill=white, draw=white] (.5,.55)  rectangle  (4,2);
	\draw[dashed] (0,0) .. controls (1.5,0) and (2,1) .. (.5,1); 
	\draw[dashed] (4,0) .. controls (2.5,0) and (3,1) .. (4.5,1); 
	\draw (1.375,.5) -- (1.375,3.5);
	\draw[yshift=1cm]  (0,2) .. controls (1.5,2) and (2,3) .. (.5,3); 
	\draw[yshift=1cm]  (4,2) .. controls (2.5,2) and (3,3) .. (4.5,3); 
	%
	\draw (3.125,.5) -- (3.125,3.5);
	\draw (0,0) -- (0,3);
	\draw (.5,3.25) -- (.5,4);
	\draw[dashed] (.5,1) -- (.5,4);
	\draw (4,0) -- (4,3);
	\draw (4.5,1) -- (4.5,4);
	\node [tikzdot] at(3.5,1.75) {};
}
\end{align*}
where $\lambda = Z$ if the first saddle in the left-most term is a merge, and $\lambda = 1$ if it is a split, and all saddles are oriented with arrows going either left or forward. 
Moreover, it is not hard to see by drawing the underlying ladder diagrams that we always have
\[
\Gamma_\bw(j) = (-XY)\Gamma_\bw(i). 
\]
Therefore, we conclude that
\[
\UtoBfunctor_0(\tau_{ji}\tau_{ij}) = \lambda^{-1}\bigl( (-XY) XZ.\UtoBfunctor_0(x_j)+ YZ.\UtoBfunctor_0(x_i) \bigr),
\]
where we write $\tau_{ji}$ and $\tau_{ij}$ for the two crossings and $x_i, x_j$ for the dots on the strands labeled $\alpha_i$ and $\alpha_j$ respectively. 
Looking at \cref{eq:defUtoBtauj-1}, we first note that we can compute $H$ as the composition of changes of chronology
\[
\tikzdiagc[scale=.35]{
	\draw (0,0) .. controls (1.5,0) and (2,1) .. (.5,1); 
	\draw (4,0) .. controls (2.5,0) and (3,1) .. (4.5,1); 
	\filldraw[fill=white, draw=white] (.5,.55)  rectangle  (4,2);
	\draw[dashed] (0,0) .. controls (1.5,0) and (2,1) .. (.5,1); 
	\draw[dashed] (4,0) .. controls (2.5,0) and (3,1) .. (4.5,1); 
	\draw (1.375,.5) -- (1.375,3.5);
	\draw[yshift=1cm]  (0,2) .. controls (1.5,2) and (2,3) .. (.5,3); 
	\draw[yshift=1cm]  (4,2) .. controls (2.5,2) and (3,3) .. (4.5,3); 
	%
	\draw (3.125,.5) -- (3.125,3.5);
	\draw (0,0) -- (0,3);
	\draw (.5,3.25) -- (.5,4);
	\draw[dashed] (.5,1) -- (.5,4);
	\draw (4,0) -- (4,3);
	\draw (4.5,1) -- (4.5,4);
}
\xrightarrow{1}
 \tikzdiagc[scale=.35]{
	\draw (0,0) .. controls (1.5,0) and (2,1) .. (.5,1); 
	\draw (4,0) .. controls (2.5,0) and (3,1) .. (4.5,1); 
	\filldraw[fill=white, draw=white] (.5,.55)  rectangle  (4.5,2);
	\draw[dashed] (0,0) .. controls (1.5,0) and (2,1) .. (.5,1); 
	\draw[dashed] (4,0) .. controls (2.5,0) and (3,1) .. (4.5,1); 
	\draw (0,0) -- (0,4);
	\draw[dashed] (.5,1) -- (.5,5);
	\draw (1.375,.5) -- (1.375,2.5);
	\draw (3.125,.5) -- (3.125,2.5);
	\draw (.5,4) -- (.5,5) -- (4.5,5) -- (4.5,4);
	\draw (1.375,2.5) .. controls (1.375,3.5) and (3.125,3.5) .. (3.125,2.5);
	\draw (0,2) -- (0,4) -- (4,4);
	\draw (8,1) -- (8.5,1) -- (8.5,5);
	\draw (5.375,2.5) .. controls (5.375,1.5) and (7.125,1.5) .. (7.125,2.5);
	\draw[dashed] (4.5,5) -- (4.5,1) -- (8.5,1) -- (8.5,3);
	\draw (4,4) -- (4,0) -- (8,0) -- (8,2);
	\draw (4,4) .. controls (5.5,4) and (6,5) .. (4.5,5); 
	\draw (8,4) .. controls (6.5,4) and (7,5) .. (8.5,5); 
	\draw (8,2) -- (8,4);
	\draw (8.5,3) -- (8.5,5);
	\draw (5.375,2.5) -- (5.375,4.5);
	\draw (7.125,2.5) -- (7.125,4.5);
}
^{(1,1)}
\xrightarrow{\lambda}
\tikzdiagc[scale=.35]{
	\draw (0,0) .. controls (1.5,0) and (2,1) .. (.5,1); 
	\draw (4,0) .. controls (2.5,0) and (3,1) .. (4.5,1); 
	\filldraw[fill=white, draw=white] (.5,.55)  rectangle  (4,2);
	\draw[dashed] (0,0) .. controls (1.5,0) and (2,1) .. (.5,1); 
	\draw[dashed] (4,0) .. controls (2.5,0) and (3,1) .. (4.5,1); 
	\draw (1.375,.5) .. controls (1.375,1.5) and (3.125,1.5) .. (3.125,.5);
	\draw[yshift=1cm]  (0,2) .. controls (1.5,2) and (2,3) .. (.5,3); 
	\draw[yshift=1cm]  (4,2) .. controls (2.5,2) and (3,3) .. (4.5,3); 
	\draw[yshift=1cm]  (1.375,2.5) .. controls (1.375,1.5) and (3.125,1.5) .. (3.125,2.5);
	\draw (0,0) -- (0,3);
	\draw (.5,3.25) -- (.5,4);
	\draw[dashed] (.5,1) -- (.5,4);
	\draw (4,0) -- (4,3);
	\draw (4.5,1) -- (4.5,4);
}^{(1,1)}
\xrightarrow{\lambda_R((0,1),C)}
 \tikzdiagc[scale=.35]{
	\draw (.5,3) -- (.5,1) -- (4.5,1) -- (4.5,3);
	\filldraw [fill=white, draw=white] (0,0) rectangle (4,2); 
	\draw (0,2) .. controls (1.5,2) and (2,3) .. (.5,3); 
	\draw (4,2) .. controls (2.5,2) and (3,3) .. (4.5,3); 
	\draw (1.375,2.5) .. controls (1.375,1.5) and (3.125,1.5) .. (3.125,2.5);
	\draw[dashed] (.5,3) -- (.5,1) -- (4.5,1) -- (4.5,3);
	\draw (0,2) -- (0,0) -- (4,0) -- (4,2);
}^{(1,0)}
\circ
\tikzdiagc[scale=.35]{
	\draw (.5,1) -- (.5,3) -- (4.5,3) -- (4.5,1);
	\draw (0,0) .. controls (1.5,0) and (2,1) .. (.5,1); 
	\draw (4,0) .. controls (2.5,0) and (3,1) .. (4.5,1); 
	\filldraw[fill=white, draw=white] (.5,.55)  rectangle  (4,2);
	\draw[dashed] (0,0) .. controls (1.5,0) and (2,1) .. (.5,1); 
	\draw[dashed] (4,0) .. controls (2.5,0) and (3,1) .. (4.5,1); 
	\draw (1.375,.5) .. controls (1.375,1.5) and (3.125,1.5) .. (3.125,.5);
	\draw[dashed] (.5,1) -- (.5,3) -- (4.5,3) -- (4.5,1);
	\draw (0,0) -- (0,2) -- (4,2) -- (4,0);
}^{(0,1)}
\]
where $C = (0,-1)$ if $\lambda=Z$, and $C = (-1,0)$ if $\lambda=1$. Recall that by definition of $\shiftFunct{H}$, we multiply by 
the coefficients written on the arrows. 
 Moreover, the change of chronology involved in the defintion of $\shiftFunct{\F_\bi}  \circ  \shiftFunct{\tikzdiagh[scale=.2]{
	\draw (.5,3) -- (.5,1) -- (4.5,1) -- (4.5,3);
	\filldraw [fill=white, draw=white] (0,0) rectangle (4,2); 
	\draw (0,2) .. controls (1.5,2) and (2,3) .. (.5,3); 
	\draw (4,2) .. controls (2.5,2) and (3,3) .. (4.5,3); 
	\draw (1.375,2.5) .. controls (1.375,1.5) and (3.125,1.5) .. (3.125,2.5);
	%
	\draw (0,2) -- (0,0) -- (4,0) -- (4,2);
}^{(1,0)}}   ( \UtoBfunctor_0(\tau_{ij}) )$ gives 
\[
\lambda_R\bigl(\deg(\UtoBfunctor_0(\tau_{ij})),(1,0)+C)\bigr) = \lambda_R\bigl((0,-1),(1,0)+C)\bigr) = 
\begin{cases}
YZ &\text{if C = (0,-1)}, \\
1 &\text{if C = (-1,0)}.
\end{cases} 
\]
Finally, $\shiftFunct{\tau_{ij}} = \shiftFunct{\tau_{ji}} = 1$, so that by putting all this together we obtain
\begin{equation*}
\UtoBfunctor(\tau_{ji}\tau_{ij}) = -YZ^2.\UtoBfunctor(x_j)+ YZ^2.\UtoBfunctor(x_i).
\end{equation*}
We now suppose that the diagram involving two crossings in \cref{eq:doubleR} is illegal, but the ones on the right part are not. We must show that the right part acts as zero. In this case, the ladder diagram must looks like one of the followings:
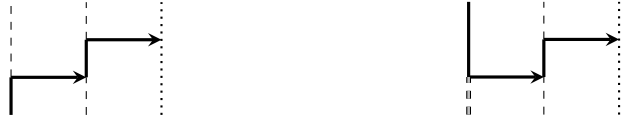
\begin{align*}
\tikzdiagh{
	\draw[dt] (0,0) -- (0,.5);
	\draw[vt] (0,.5) -- (0,1);
	\draw[dt] (0,1) -- (0,1.5);
	\draw[->,vt] (-1,.5) -- (0,.5);
	\draw[->,vt] (0,1) -- (1,1);
	\draw[vt] (-1,0) -- (-1,.5);
	\draw[dt] (-1,.5) -- (-1,1.5);
	\draw[dotted] (1,0) -- (1,1.5);
} 
&&
\tikzdiagh{
	\draw[dt] (0,0) -- (0,.5);
	\draw[vt] (0,.5) -- (0,1);
	\draw[dt] (0,1) -- (0,1.5);
	\draw[->,vt] (-1,.5) -- (0,.5);
	\draw[->,vt] (0,1) -- (1,1);
	\draw[dds] (-1,0) -- (-1,.5);
	\draw[vt] (-1,.5) -- (-1,1.5);
	\draw[dotted] (1,0) -- (1,1.5);
} 
\end{align*}
In either ways, we have $\Gamma_\bw(i) = \Gamma_\bw(j)$ and $\UtoBfunctor_0(x_i) = \UtoBfunctor_0(x_j)$ because the two rungs are connected by a visible strand. 
Thus $\UtoBfunctor(x_i) = \UtoBfunctor(x_j)$ and the right part of \cref{eq:doubleR} is zero. 

The case $j = i + 1$ is similar, but still carries enough differences to be worth explaining a bit.
First, we have 
\[
\UtoBfunctor(x_i) =  \shiftFunct{\tikzdiagh[scale=.2]{
	\draw (.5,3) -- (.5,1) -- (4.5,1) -- (4.5,3);
	\filldraw [fill=white, draw=white] (0,0) rectangle (4,2); 
	\draw (0,2) .. controls (1.5,2) and (2,3) .. (.5,3); 
	\draw (4,2) .. controls (2.5,2) and (3,3) .. (4.5,3); 
	\draw (1.375,2.5) .. controls (1.375,1.5) and (3.125,1.5) .. (3.125,2.5);
	%
	\draw (0,2) -- (0,0) -- (4,0) -- (4,2);
}^{(1,0)}}
 \bigl(\UtoBfunctor_0(x_i)\bigr), 
\]
giving a scalar $1$ when the first saddle is a merge, and $XYZ^2$ when it is a split. The same applies for $x_j$. 
Moreover, keeping the same conventions as before, the saddles carry now a different orientation, so that $\lambda = Z$ when the first saddle is a merge, and $\lambda = XY$ when it is a split. The factors $\Gamma_\bw(i)$ are also different so that now
\[
\UtoBfunctor_0(\tau_{ji}\tau_{ij}) = \lambda^{-1}\bigl( XYZ.\UtoBfunctor_0(x_j) + (-XY)Z.\UtoBfunctor_0(x_i) \bigr).
\]
 Also, $\shiftFunct{H}$ multiplies by $\lambda' \lambda_R(C',(0,1)))^{-1}$, where $\lambda' = 1$ and $C' = (-1,0)$ if the first saddle is a merge, and $\lambda'= Z$ and $C'=(0,-1)$  in the other case. Putting all this together, we obtain
\begin{equation*}
\UtoBfunctor(\tau_{ji}\tau_{ij}) = XYZ.\UtoBfunctor(x_j) - XYZ.\UtoBfunctor(x_i).
\end{equation*}
The proof for the illegal diagrams is essentially the same, and we leave the details to the reader. 

We note that \cref{eq:nilHecke} immediately follows from \cref{eq:dotcob} and the definition of $\Gamma_\bw(i)$.

For \cref{eq:dotcommutescrossing}, if $|i-j| > 1$, then it is immediate. If $j = i + 1$ and considering the first case, then we first observe that $\Gamma_\bw(j) = \Gamma_{\bw'}(j)$ where $\bw' \F_j \bw$. Each side of \cref{eq:dotcommutescrossing} is sent by $ \UtoBfunctor$ to one of the two path in the following diagram:
\[
\begin{tikzcd}[column sep=12ex, row sep = 10ex]
\shiftFunct{\tikzdiagh[scale=.2]{
	\draw (.5,3) -- (.5,1) -- (4.5,1) -- (4.5,3);
	\filldraw [fill=white, draw=white] (0,0) rectangle (4,2); 
	\draw (0,2) .. controls (1.5,2) and (2,3) .. (.5,3); 
	\draw (4,2) .. controls (2.5,2) and (3,3) .. (4.5,3); 
	\draw (1.375,2.5) .. controls (1.375,1.5) and (3.125,1.5) .. (3.125,2.5);
	%
	\draw (0,2) -- (0,0) -- (4,0) -- (4,2);
}^{(1,0)}}
\tqft\left(
\tikzdiagh[scale=.5]{
	\draw (1,.5) .. controls (0,.5) .. (0,0);
	\draw (0,1.5) .. controls (0,1) .. (-1,1);
}
\right)
\ar[dash]{r}{\tqft\left( \ \tikzdiagh[scale=.2]{
	\draw (.5,3) -- (.5,1) -- (4.5,1) -- (4.5,3);
	\filldraw [fill=white, draw=white] (0,0) rectangle (4,2); 
	\draw (0,2) .. controls (1.5,2) and (2,3) .. (.5,3); 
	\draw (4,2) .. controls (2.5,2) and (3,3) .. (4.5,3); 
	\draw (1.375,2.5) .. controls (1.375,1.5) and (3.125,1.5) .. (3.125,2.5);
	%
	\draw (0,2) -- (0,0) -- (4,0) -- (4,2);
} \right)}
\ar{r}
\ar[dash,swap]{d}{
\shiftFunct{\tikzdiagh[scale=.2]{
	\draw (.5,3) -- (.5,1) -- (4.5,1) -- (4.5,3);
	\filldraw [fill=white, draw=white] (0,0) rectangle (4,2); 
	\draw (0,2) .. controls (1.5,2) and (2,3) .. (.5,3); 
	\draw (4,2) .. controls (2.5,2) and (3,3) .. (4.5,3); 
	\draw (1.375,2.5) .. controls (1.375,1.5) and (3.125,1.5) .. (3.125,2.5);
	%
	\draw (0,2) -- (0,0) -- (4,0) -- (4,2);
}^{(1,0)}}
\tqft\left(
\tikzdiagh[scale=.5]{
	\draw (1,.5) .. controls (0,.5) .. (0,0); \node[tikzdot] at (.5,.5){};
	\draw (0,1.5) .. controls (0,1) .. (-1,1);
}
\right)
}
\ar{d}
&
\tqft\left(
\tikzdiagh[scale=.5]{
	\draw (-1,.5) .. controls (0,.5) .. (0,0);
	\draw (0,1.5) .. controls (0,1) .. (1,1);
}
\right)
\ar[dash]{d}{
\tqft\left(
\tikzdiagh[scale=.5]{	
	\draw (-1,.5) .. controls (0,.5) .. (0,0);
	\draw (0,1.5) .. controls (0,1) .. (1,1); \node[tikzdot] at (.5,1){};
}
\right)
}
\ar{d}
\\
\shiftFunct{\tikzdiagh[scale=.2]{
	\draw (.5,3) -- (.5,1) -- (4.5,1) -- (4.5,3);
	\filldraw [fill=white, draw=white] (0,0) rectangle (4,2); 
	\draw (0,2) .. controls (1.5,2) and (2,3) .. (.5,3); 
	\draw (4,2) .. controls (2.5,2) and (3,3) .. (4.5,3); 
	\draw (1.375,2.5) .. controls (1.375,1.5) and (3.125,1.5) .. (3.125,2.5);
	%
	\draw (0,2) -- (0,0) -- (4,0) -- (4,2);
}^{(1,0)}}
\tqft\left(
\tikzdiagh[scale=.5]{
	\draw (1,.5) .. controls (0,.5) .. (0,0);
	\draw (0,1.5) .. controls (0,1) .. (-1,1);
}
\right)
\ar[dash,swap]{r}{\tqft\left( \ \tikzdiagh[scale=.2]{
	\draw (.5,3) -- (.5,1) -- (4.5,1) -- (4.5,3);
	\filldraw [fill=white, draw=white] (0,0) rectangle (4,2); 
	\draw (0,2) .. controls (1.5,2) and (2,3) .. (.5,3); 
	\draw (4,2) .. controls (2.5,2) and (3,3) .. (4.5,3); 
	\draw (1.375,2.5) .. controls (1.375,1.5) and (3.125,1.5) .. (3.125,2.5);
	%
	\draw (0,2) -- (0,0) -- (4,0) -- (4,2);
} \right)}
\ar{r}
&
\tqft\left(
\tikzdiagh[scale=.5]{
	\draw (-1,.5) .. controls (0,.5) .. (0,0);
	\draw (0,1.5) .. controls (0,1) .. (1,1);
}
\right)
\end{tikzcd}
\]
By \cref{eq:cobcommute1}, we have that at the level of the underlying cobordims there is a difference by a factor of $\lambda_R(C,(-1,-1))$ where $C$ depends on whether the saddle is a merge or a split after closing the tangle. Moreover, the grading shift functor on the left downard arrow multiply by a factor of $\lambda_R((-1,-1),(1,0)+C)$, so that in the end we have \cref{eq:dotcommutescrossing}. 
The other case and the cases $j = i-1$ are similar, and we leave the details to the reader.

Clearly, we have \eqref{eq:crosssymiszero} because we are adding a dot at the same height on the same surface. 

For \cref{eq:RijkR3}, we are permuting saddles points that we can treat as having $\bZ \times \bZ$-degrees given by $p_{ij}$ because of the changes of chronology.

Both \cref{eq:RiiiR3} and \cref{eq:RiijR3} are trivially verified because they factorize through invalid weights.

For \cref{eq:RijiR3} and $j = i +1$, we first note that the ladder diagram must have one of the two following forms (otherwise all diagrams are illegal):
\begin{align}\label{eq:twoladdersforR3}
\tikzdiagh[yscale=.5]{
	\draw[dds] (0,0) -- (0,1);
	\draw[vt] (0,1) -- (0,3);
	\draw[dt] (0,3) -- (0,4);
	\draw[dt] (1,0) -- (1,1);
	\draw[vt] (1,1) -- (1,2);
	\draw[dt] (1,2) -- (1,3);
	\draw[vt] (1,3) -- (1,4);
	\draw[dotted] (2,0) -- (2,4);
	\draw[vt, ->] (0,1) -- (1,1);
	\draw[vt, ->] (1,2) -- (2,2);
	\draw[vt, ->] (0,3) -- (1,3);
}
&&\text{ or }&&
\tikzdiagh[yscale=.5]{
	\draw[dds] (0,0) -- (0,1);
	\draw[vt] (0,1) -- (0,3);
	\draw[dt] (0,3) -- (0,4);
	\draw[vt] (1,0) -- (1,1);
	\draw[dds] (1,1) -- (1,2);
	\draw[vt] (1,2) -- (1,3);
	\draw[dds] (1,3) -- (1,4);
	\draw[dotted] (2,0) -- (2,4);
	\draw[vt, ->] (0,1) -- (1,1);
	\draw[vt, ->] (1,2) -- (2,2);
	\draw[vt, ->] (0,3) -- (1,3);
}
\end{align}
In the first case, the right part of the left term in \cref{eq:RijiR3} acts by zero and the other one decomposes as
\begin{align*}
\bigl(\shiftFunct{\chcob^{(1,0)}}\bigr)
&
\UtoBfunctor_0\left(
\tikzdiagh[yscale=.35,xscale=.5]{
	\draw[dds] (0,0) -- (0,1);
	\draw[vt] (0,1) -- (0,3);
	\draw[dt] (0,3) -- (0,4);
	\draw[dt] (1,0) -- (1,1);
	\draw[vt] (1,1) -- (1,2);
	\draw[dt] (1,2) -- (1,3);
	\draw[vt] (1,3) -- (1,4);
	\draw[dotted] (2,0) -- (2,4);
	\draw[vt, ->] (0,1) -- (1,1);
	\draw[vt, ->] (1,2) -- (2,2);
	\draw[vt, ->] (0,3) -- (1,3);
}
\right)
\xrightarrow{\Gamma . \UtoBfunctor_0(\tau_{ij}) }
\UtoBfunctor_0\left(
\tikzdiagh[yscale=.35,xscale=.5]{
	\draw[dds] (0,0) -- (0,1);
	\draw[vt] (0,1) -- (0,2);
	\draw[dt] (0,2) -- (0,4);
	\draw[dt] (1,0) -- (1,1);
	\draw[vt] (1,1) -- (1,2);
	\draw[dds] (1,2) -- (1,3);
	\draw[vt] (1,3) -- (1,4);
	\draw[dotted] (2,0) -- (2,4);
	\draw[vt, ->] (0,1) -- (1,1);
	\draw[vt, ->] (1,3) -- (2,3);
	\draw[vt, ->] (0,2) -- (1,2);
}
\right)
\\
&
\xrightarrow{ (-XY)\Gamma. \UtoBfunctor_0(\tau_{ii})}
\UtoBfunctor_0\left(
\tikzdiagh[yscale=.35,xscale=.5]{
	\draw[dds] (0,0) -- (0,1);
	\draw[vt] (0,1) -- (0,2);
	\draw[dt] (0,2) -- (0,4);
	\draw[dt] (1,0) -- (1,1);
	\draw[vt] (1,1) -- (1,2);
	\draw[dds] (1,2) -- (1,3);
	\draw[vt] (1,3) -- (1,4);
	\draw[dotted] (2,0) -- (2,4);
	\draw[vt, ->] (0,1) -- (1,1);
	\draw[vt, ->] (1,3) -- (2,3);
	\draw[vt, ->] (0,2) -- (1,2);
}
\right)
\xrightarrow{ \shiftFunct{\chcob^{(1,0)}} \UtoBfunctor_0(\tau_{ji}) \circ \shiftFunct{H}}
\bigl(\shiftFunct{\chcob^{(1,0)}}\bigr)
\UtoBfunctor_0\left(
\tikzdiagh[yscale=.35,xscale=.5]{
	\draw[dds] (0,0) -- (0,1);
	\draw[vt] (0,1) -- (0,3);
	\draw[dt] (0,3) -- (0,4);
	\draw[dt] (1,0) -- (1,1);
	\draw[vt] (1,1) -- (1,2);
	\draw[dt] (1,2) -- (1,3);
	\draw[vt] (1,3) -- (1,4);
	\draw[dotted] (2,0) -- (2,4);
	\draw[vt, ->] (0,1) -- (1,1);
	\draw[vt, ->] (1,2) -- (2,2);
	\draw[vt, ->] (0,3) -- (1,3);
}
\right)
\end{align*}
where $\Gamma := \Gamma_\bw(i)$ and $\chcob := \tikzdiagh[scale=.2]{
	\begin{scope}
	\clip (-1,.5) rectangle (1,-.25); 
	\draw (0,0) .. controls (-1.5,0) and (-1,1) .. (.5,1); 
	\end{scope}
	\begin{scope}
	\draw  (4,1) -- (4.5,1) -- (4.5,3);
	\end{scope}
	%
	%
	\draw (0,2) .. controls (1.5,2) and (2,3) .. (.5,3); 
	\draw (4,2) .. controls (2.5,2) and (3,3) .. (4.5,3); 
	\draw (1.375,2.5) .. controls (1.375,1.5) and (3.125,1.5) .. (3.125,2.5);
	%
	\draw (0,0) -- (4,0) -- (4,2);
	\draw (0,2) .. controls (-1.5,2) and (-1,3) .. (.5,3); 
	\draw (-.875,.5) -- (-.875,2.5);
}$. At the level of the underlying cobordism, we have by \cref{eq:reverseorientation} and \cref{eq:chcobrelbirthmerge} that
\[
\tikzdiagh[scale=.3]{
	\begin{scope}
	\clip (-1,.5) rectangle (1,-.25); 
	\draw (0,0) .. controls (-1.5,0) and (-1,1) .. (.5,1); 
	\end{scope}
	\begin{scope}
	\draw  (4,1) -- (4.5,1) -- (4.5,7);
	\end{scope}
	%
	%
	%
	\draw (1.375,2.5) .. controls (1.375,1.5) and (3.125,1.5) .. (3.125,2.5);
	%
	\draw (0,0) -- (4,0) -- (4,6);
	%
	%
	%
	\draw (-.875,.5) -- (-.875,2.5);
	\draw (-.875,2.5) .. controls (-.875,4.5) and (1.375,4.5) .. (1.375,2.5); 
	\draw (3.125,2.5) --  (3.125,6.5);
	\node at(1.5,6){
		\tikzdiagh[scale=.3]{
		\draw  (.5,3) -- (4.5,3) -- (4.5,1);
		\begin{scope}
			\clip (-1,.5) rectangle (1,-.5);
		\end{scope}
		%
		%
		\filldraw[fill=white, draw=white] (.5,.55)  rectangle  (4,2);
		%
		%
		\draw (1.375,.5) .. controls (1.375,1.5) and (3.125,1.5) .. (3.125,.5);
		%
		\draw (0,2) -- (4,2) -- (4,0);
		\draw (0,2) .. controls (-1.5,2) and (-1,3) .. (.5,3); 
		\draw (-.875,.5) -- (-.875,2.5);
		\draw (-.875,.5) .. controls (-.875,-1.5) and (1.375,-1.5) .. (1.375,.5); 
		}
	};
	\draw [->] (2.25,2) -- (2.75,2.75);
	\draw [<-] (2.75,5) -- (1.5,5);
} 
\ = X\ 
\tikzdiagh[scale=.3]{
	\begin{scope}
	\clip (-1,.5) rectangle (1,-.25); 
	\draw (0,0) .. controls (-1.5,0) and (-1,1) .. (.5,1); 
	\end{scope}
	\begin{scope}
	\draw  (4,1) -- (4.5,1) -- (4.5,7);
	\end{scope}
	\draw (0,0) -- (4,0) -- (4,6);
	\draw (-.875,.5) -- (-.875,6.5);
	\node at(1.5,6){
		\tikzdiagh[scale=.3]{
		\draw  (.5,3) -- (4.5,3) -- (4.5,1);
		\begin{scope}
			\clip (-1,.5) rectangle (1,-.5);
		\end{scope}
		\filldraw[fill=white, draw=white] (.5,.55)  rectangle  (4,2);
		\draw (0,2) -- (4,2) -- (4,0);
		\draw (0,2) .. controls (-1.5,2) and (-1,3) .. (.5,3); 
		\draw (-.875,.5) -- (-.875,2.5);
		}
	};
}
\]
where the death is negative as usual. Then, we compute using the same kind of arguments as before that $\shiftFunct{H}$ multiplies by a factor $YZ$ and the grading shift functor $\shiftFunct{\chcob^{(1,0)}}$ twists $\UtoBfunctor_0(\tau_{ji})$ by $\lambda_R((0,-1),(1,-1)) = YZ$. In conclusion, we get
\[
\UtoBfunctor(\tau_{ji}\tau_{ii}\tau_{ij}) = -YZ^{-2} 1_\bw.
\]
In the second case of ladder diagram from \cref{eq:twoladdersforR3}, the left part of the left term  in \cref{eq:RijiR3} gives zero. For the right part, we get
 \begin{align*}
\UtoBfunctor_0\left(
\tikzdiagh[yscale=.35,xscale=.5]{
	\draw[dds] (0,0) -- (0,1);
	\draw[vt] (0,1) -- (0,3);
	\draw[dt] (0,3) -- (0,4);
	\draw[vt] (1,0) -- (1,1);
	\draw[dds] (1,1) -- (1,2);
	\draw[vt] (1,2) -- (1,3);
	\draw[dds] (1,3) -- (1,4);
	\draw[dotted] (2,0) -- (2,4);
	\draw[vt, ->] (0,1) -- (1,1);
	\draw[vt, ->] (1,2) -- (2,2);
	\draw[vt, ->] (0,3) -- (1,3);
}
\right)
&
\xrightarrow{ \UtoBfunctor_0(\tau_{ij})  \circ \shiftFunct{H} }
\bigl(\shiftFunct{\chcob^{(1,0)}}\bigr)
\UtoBfunctor_0\left(
\tikzdiagh[yscale=.35,xscale=.5]{
	\draw[dds] (0,0) -- (0,2);
	\draw[vt] (0,2) -- (0,3);
	\draw[dt] (0,3) -- (0,4);
	\draw[vt] (1,0) -- (1,1);
	\draw[dt] (1,1) -- (1,2);
	\draw[vt] (1,2) -- (1,3);
	\draw[dds] (1,3) -- (1,4);
	\draw[dotted] (2,0) -- (2,4);
	\draw[vt, ->] (0,2) -- (1,2);
	\draw[vt, ->] (1,1) -- (2,1);
	\draw[vt, ->] (0,3) -- (1,3);
}
\right)
\\
&
\xrightarrow{ \Gamma. \shiftFunct{\chcob^{(1,0)}} \UtoBfunctor_0(\tau_{ii})}
\bigl(\shiftFunct{\chcob^{(1,0)}}\bigr)
\UtoBfunctor_0\left(
\tikzdiagh[yscale=.35,xscale=.5]{
	\draw[dds] (0,0) -- (0,2);
	\draw[vt] (0,2) -- (0,3);
	\draw[dt] (0,3) -- (0,4);
	\draw[vt] (1,0) -- (1,1);
	\draw[dt] (1,1) -- (1,2);
	\draw[vt] (1,2) -- (1,3);
	\draw[dds] (1,3) -- (1,4);
	\draw[dotted] (2,0) -- (2,4);
	\draw[vt, ->] (0,2) -- (1,2);
	\draw[vt, ->] (1,1) -- (2,1);
	\draw[vt, ->] (0,3) -- (1,3);
}
\right)
\xrightarrow{ \Gamma.  \UtoBfunctor_0(\tau_{ji})}
\UtoBfunctor_0\left(
\tikzdiagh[yscale=.35,xscale=.5]{
	\draw[dds] (0,0) -- (0,1);
	\draw[vt] (0,1) -- (0,3);
	\draw[dt] (0,3) -- (0,4);
	\draw[vt] (1,0) -- (1,1);
	\draw[dds] (1,1) -- (1,2);
	\draw[vt] (1,2) -- (1,3);
	\draw[dds] (1,3) -- (1,4);
	\draw[dotted] (2,0) -- (2,4);
	\draw[vt, ->] (0,1) -- (1,1);
	\draw[vt, ->] (1,2) -- (2,2);
	\draw[vt, ->] (0,3) -- (1,3);
}
\right)
\end{align*}
For the underlying cobordism, we have
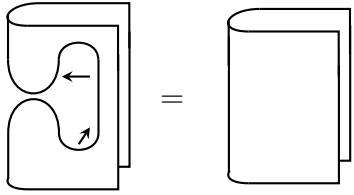
\[
\tikzdiagh[scale=.3]{
	\begin{scope}
	\clip (-1,.5) rectangle (1,-.25); 
	\draw (0,0) .. controls (-1.5,0) and (-1,1) .. (.5,1); 
	\end{scope}
	\begin{scope}
	\draw  (4,1) -- (4.5,1) -- (4.5,7);
	\end{scope}
	%
	%
	%
	\draw (1.375,2.5) .. controls (1.375,1.5) and (3.125,1.5) .. (3.125,2.5);
	%
	\draw (0,0) -- (4,0) -- (4,6);
	%
	%
	%
	\draw (-.875,.5) -- (-.875,2.5);
	\draw (-.875,2.5) .. controls (-.875,4.5) and (1.375,4.5) .. (1.375,2.5); 
	\draw (3.125,2.5) --  (3.125,6.5);
	\node at(1.5,6){
		\tikzdiagh[scale=.3]{
		\draw  (.5,3) -- (4.5,3) -- (4.5,1);
		\begin{scope}
			\clip (-1,.5) rectangle (1,-.5);
		\end{scope}
		%
		%
		\filldraw[fill=white, draw=white] (.5,.55)  rectangle  (4,2);
		%
		%
		\draw (1.375,.5) .. controls (1.375,1.5) and (3.125,1.5) .. (3.125,.5);
		%
		\draw (0,2) -- (4,2) -- (4,0);
		\draw (0,2) .. controls (-1.5,2) and (-1,3) .. (.5,3); 
		\draw (-.875,.5) -- (-.875,2.5);
		\draw (-.875,.5) .. controls (-.875,-1.5) and (1.375,-1.5) .. (1.375,.5); 
		}
	};
	\draw [->] (2.25,2) -- (2.75,2.75);
	\draw [->] (2.75,5) -- (1.5,5);
} 
\ = \ 
\tikzdiagh[scale=.3]{
	\begin{scope}
	\clip (-1,.5) rectangle (1,-.25); 
	\draw (0,0) .. controls (-1.5,0) and (-1,1) .. (.5,1); 
	\end{scope}
	\begin{scope}
	\draw  (4,1) -- (4.5,1) -- (4.5,7);
	\end{scope}
	\draw (0,0) -- (4,0) -- (4,6);
	\draw (-.875,.5) -- (-.875,6.5);
	\node at(1.5,6){
		\tikzdiagh[scale=.3]{
		\draw  (.5,3) -- (4.5,3) -- (4.5,1);
		\begin{scope}
			\clip (-1,.5) rectangle (1,-.5);
		\end{scope}
		\filldraw[fill=white, draw=white] (.5,.55)  rectangle  (4,2);
		\draw (0,2) -- (4,2) -- (4,0);
		\draw (0,2) .. controls (-1.5,2) and (-1,3) .. (.5,3); 
		\draw (-.875,.5) -- (-.875,2.5);
		}
	};
}
\]
Moreover, $\shiftFunct{H}$ gives a factor $Z$ and $\shiftFunct{W^{(1,0)}}$ basically does nothing since $W$ is a merge. In conclusion, we obtain
\[
\UtoBfunctor(\tau_{ji}\tau_{ii}\tau_{ij}) = Z^{-1} 1_\bw.
\]
The case $j = i - 1$ is similar and we leave the details to the reader. 

Finally, 
\cref{eq:doublesplitters} is obvious by definition of $\UtoBfunctor$. 
\end{proof}

\subsubsection{Horizontal composition}

Fix $\F_\bj : \bw' \rightarrow \bw''$ and $\F_\bi : \bw \rightarrow \bw'$.  We write $\compMap{\F_\bj}{\F_\bi}$ for the compatibility map obtained from $\shiftFunct{\F_\bj}$ and $\shiftFunct{\F_\bi}$. 

Consider the map $\UtoBfunctor(\bj, \bi)$ defined by making the following diagram commutes:
\[
\begin{tikzcd}[column sep=14ex]
\UtoBfunctor(\F_\bj) \otimes_H \UtoBfunctor(\F_\bi)
\ar{r}{\UtoBfunctor(\bj, \bi)}
\ar[equals]{d}
&
\UtoBfunctor(\F_\bj \F_\bi)
\\
\shiftFunct{\F_\bj}\left( \tqft\bigl( \ladder (\F_\bj) \bigr) \right) \otimes_H \shiftFunct{\F_\bi}\left( \tqft\bigl( \ladder (\F_\bi) \bigr) \right) 
\ar[swap]{d}{\compMap{\F_\bj}{\F_\bi}}
&
\shiftFunct{\F_{\bj\bi}} \left( \tqft\bigl( \ladder (\F_\bj)\ladder (\F_\bj) \bigr) \right) 
\ar[equals]{u}
\\
\shiftFunct{\F_\bj} \bullet \shiftFunct{\F_\bi} \left( \tqft\bigl( \ladder (\F_\bj) \bigr)  \otimes_H ( \tqft\bigl( \ladder (\F_\bi) \bigr) \right) 
\ar[swap]{r}{\mu[\ladder(\F_\bj), \ladder(\F_\bi)]}
&
\shiftFunct{\F_\bj} \bullet \shiftFunct{\F_\bi} \left( \tqft\bigl( \ladder (\F_\bj)\ladder (\F_\bi) \bigr) \right) 
\ar[swap]{u}{h_{\bj,\bi}}
\end{tikzcd}
\]
where $h_{\bj,\bi}$ is given by 
is given by first adding a grading shift $\shiftFunct{\chcob_{\bj,\bi}^{v_{\bj,\bi}}}$ corresponding with the fact we potentially generated new pairs of \eqref{eq:normalization1}/\eqref{eq:normalization2}  by gluing the two ladder diagrams together, followed by a change of chronology to put it in normalized form. 
Note that $h_{\bj,\bi}$ is an isomorphism in  $\BIMOD_q(H^\bullet, H^\bullet)$. 
Therefore, since all the maps involved are isomorphism, we have that 
$\UtoBfunctor(\bj, \bi)$ 
is an isomorphism (potentially carrying a non-trivial $\gradC$-degree though).

\begin{lem}\label{lem:FLijisnatural}
The isomorphism
\[
\UtoBfunctor(\bj, \bi) : 
\UtoBfunctor(\F_\bj) \otimes_H \UtoBfunctor(\F_\bi) \xrightarrow{\simeq} \UtoBfunctor(\F_\bj \F_\bi),
\]
is $\gradC$-graded natural in $\F_\bi$ and $\F_\bj$.  
\end{lem}

\begin{proof}
Suppose we have $g : \F_\bj \rightarrow \F_{\bj'}$ and $f : \F_\bi \rightarrow \F_{\bi'}$. 
We claim the following diagram commutes:
\begin{equation}\label{eq:FLisonat}
\begin{tikzcd}[column sep=-8ex]
&[-6ex]
\shiftFunct{\F_{\bi'} \bullet \F_\bi} \bigl( \UtoBfunctor_0(\F_{\bi'}) \otimes_H \UtoBfunctor_0(\F_\bi) \bigr)
\ar{r}{\mu}
&[12 ex]
\shiftFunct{\F_{\bi'} \bullet \F_\bi} \bigl( \UtoBfunctor_0(\F_{\bi'\bi}) \bigr)
\ar{dr}{h_{\bi', \bi}}
&
\\
\shiftFunct{\F_{\bi'}} \UtoBfunctor_0(\F_{\bi'}) \otimes_H \shiftFunct{\F_{\bi}} \UtoBfunctor_0(\F_\bi) 
\ar[pos=.4]{ur}{\compMap{\F_\bi'}{\F_\bi}}
\ar[swap]{d}{\UtoBfunctor(g) \otimes_H \UtoBfunctor(f)}
&[-10ex]
&
&
\shiftFunct{\F_{\bi'\bi}} \bigl( \UtoBfunctor_0(\F_{\bi'\bi}) \bigr)
\ar{d}{\UtoBfunctor(g \circ f)}
\\
\shiftFunct{\F_{\bj'}} \UtoBfunctor_0(\F_{\bj'}) \otimes_H \shiftFunct{\F_{\bj}} \UtoBfunctor_0(\F_\bj) 
\ar[swap, pos=.4]{dr}{\compMap{\F_\bj'}{\F_\bj}}
&[-8ex]
&
&
\shiftFunct{\F_{\bj'\bj}} \bigl( \UtoBfunctor_0(\F_{\bj'\bj}) \bigr)
\\
&
\shiftFunct{\F_{\bj'} \bullet \F_\bj} \bigl( \UtoBfunctor_0(\F_{\bj'}) \otimes_H \UtoBfunctor_0(\F_\bj) \bigr)
\ar[swap]{r}{\mu}
&[10 ex]
\shiftFunct{\F_{\bj'} \bullet \F_\bj} \bigl( \UtoBfunctor_0(\F_{\bj'\bj}) \bigr)
\ar[swap]{ur}{h_{\bj',\bj}}
&
\end{tikzcd}
\end{equation}
up to $\shiftFunct{H} : |h_{\bj',\bj}| \circ (|g| \bullet |f|) \Rightarrow |g \circ f| \circ |h_{\bi',\bi}|$. 
It is easier to restrict to the case where $f$ or $g$ is a generator (and the other one is the identity), and the general result will follow. 

The difference of contribution of $\mu$ and of the $\UtoBfunctor_0$ part of $f$ and $g$ in the upper and lower path of \cref{eq:FLisonat} is given by 
the following equation:
\begin{align*}
\tikzdiag[yscale=.5,xscale=.75]{0}{
	\draw (1,-.5) node[below]{\small $y$}
		--
		(1,-.5)
		-- 
		(1,.5)   node[near end,rrect]{\small $\chcob_g$}
		--
		(1,1)
		.. controls (1,1.5) and (1.5,1.5) ..
		(1.5,2);
	\draw (2,-2.5) node[below]{\small $x$}
		--
		(2, -1)   node[near end,rrect]{\small $\chcob_f$}
		--
		(2,1)
		.. controls (2,1.5) and (1.5,1.5) ..
		(1.5,2)
		--
		(1.5,2.5);
}
\ = \compMap{\chcob_g}{\chcob_f}(|y|,|x|) \ 
\tikzdiag[yscale=.5,xscale=.75]{0}{
	\draw (1,-1) node[below]{\small $y$}
		--
		(1,-1)
		.. controls (1,-.5) and (1.5,-.5) ..
		(1.5,0);
	\draw (2,-2) node[below]{\small $x$}
		--
		(2, -1)   
		.. controls (2,-.5) and (1.5,-.5) ..
		(1.5,0)
		--
		(1.5,2.5) node[midway,rrect]{\small $\chcob_g \bullet \chcob_f$};
}
\end{align*}
where $\chcob_f$ (resp. $\chcob_g$) is the underlying cobordism of $\UtoBfunctor_0(f)$ (resp. $\UtoBfunctor_0(g)$). 

Moreover, by coherence of the compatibility maps, the following diagram commutes:
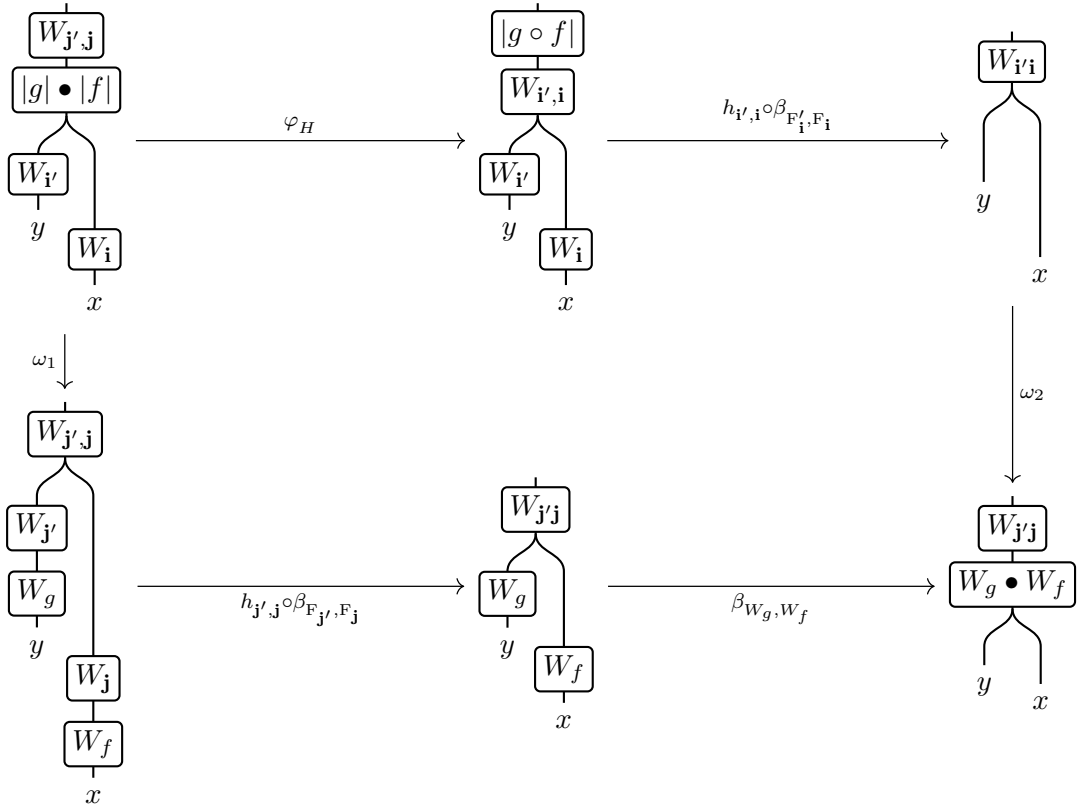
\[
\begin{tikzcd}[column sep = 24ex]
\tikzdiag[yscale=.5,xscale=.75]{0}{
	\draw (1,-.5) node[below]{\small $y$}
		--
		(1,-.5)
		-- 
		(1,.5)   node[near end,rrect]{\small $\chcob_{\bi'}$}
		--
		(1,1)
		.. controls (1,1.5) and (1.5,1.5) ..
		(1.5,2);
	\draw (2,-2.5) node[below]{\small $x$}
		--
		(2, -1.5)   node[near end,rrect]{\small $\chcob_{\bi}$}
		--
		(2,1)
		.. controls (2,1.5) and (1.5,1.5) ..
		(1.5,2)
		--
		(1.5,3) node[midway,rrect]{\small $|g| \bullet |f|$}
		-- 
		(1.5,5)  node[midway,rrect]{\small $\chcob_{\bj',\bj}$};
}
\ar{r}{\shiftFunct{H}}
\ar[swap]{d}{\omega_1}
&
\tikzdiag[yscale=.5,xscale=.75]{0}{
	\draw (1,-.5) node[below]{\small $y$}
		--
		(1,-.5)
		-- 
		(1,.5)   node[near end,rrect]{\small $\chcob_{\bi'}$}
		--
		(1,1)
		.. controls (1,1.5) and (1.5,1.5) ..
		(1.5,2);
	\draw (2,-2.5) node[below]{\small $x$}
		--
		(2, -1.5)   node[near end,rrect]{\small $\chcob_{\bi}$}
		--
		(2,1)
		.. controls (2,1.5) and (1.5,1.5) ..
		(1.5,2)
		--
		(1.5,3) node[midway,rrect]{\small $\chcob_{\bi',\bi}$}
		-- 
		(1.5,5)  node[midway,rrect]{\small $|g \circ f|$};
}
\ar{r}{h_{\bi',\bi} \circ \compMap{\F_\bi'}{\F_\bi}}
&
\tikzdiag[yscale=.5,xscale=.75]{0}{
	\draw (1,-.5) node[below]{\small $y$}
		--
		(1,-.5)
		-- 
		(1,.5)   
		--
		(1,1)
		.. controls (1,1.5) and (1.5,1.5) ..
		(1.5,2);
	\draw (2,-2.5) node[below]{\small $x$}
		--
		(2, -1.5)   
		--
		(2,1)
		.. controls (2,1.5) and (1.5,1.5) ..
		(1.5,2)
		--
		(1.5,3) node[midway,rrect]{\small $\chcob_{\bi'\bi}$}
		-- (1.5,3.5);
}
\ar{d}{\omega_2}
\\
\tikzdiag[yscale=.5,xscale=.75]{0}{
	\draw (1,-2.5) node[below]{\small $y$}
		-- 
		(1,-1.5)   node[near end,rrect]{\small $\chcob_g$}
		-- 
		(1,.5)   node[near end,rrect]{\small $\chcob_{\bj'}$}
		--
		(1,1)
		.. controls (1,1.5) and (1.5,1.5) ..
		(1.5,2);
	\draw (2,-6.5) node[below]{\small $x$}
		--
		(2, -5.5)   node[near end,rrect]{\small $\chcob_f$}
		--
		(2, -3.5)   node[near end,rrect]{\small $\chcob_{\bj}$}
		--
		(2,1)
		.. controls (2,1.5) and (1.5,1.5) ..
		(1.5,2)
		--
		(1.5,3) node[midway,rrect]{\small $\chcob_{\bj',\bj}$}
		-- (1.5,3.5);
}
\ar[swap]{r}{h_{\bj',\bj} \circ \compMap{\F_{\bj'}}{\F_{\bj}}} 
&
\tikzdiag[yscale=.5,xscale=.75]{0}{
	\draw (1,-.5) node[below]{\small $y$}
		--
		(1,-.5)
		-- 
		(1,.5)   node[near end,rrect]{\small $\chcob_g$}
		--
		(1,1)
		.. controls (1,1.5) and (1.5,1.5) ..
		(1.5,2);
	\draw (2,-2.5) node[below]{\small $x$}
		--
		(2, -1.5)   node[near end,rrect]{\small $\chcob_f$}
		--
		(2,1)
		.. controls (2,1.5) and (1.5,1.5) ..
		(1.5,2)
		--
		(1.5,3) node[midway,rrect]{\small $\chcob_{\bj'\bj}$}
		--  (1.5,3.5);
}
\ar[swap]{r}{\compMap{\chcob_g}{\chcob_f}}
&
\tikzdiag[yscale=.5,xscale=.75]{0}{
	\draw (1,.5) node[below]{\small $y$}
		--
		(1,1)
		.. controls (1,1.5) and (1.5,1.5) ..
		(1.5,2);
	\draw (2,0) node[below]{\small $x$}
		--
		(2,1)
		.. controls (2,1.5) and (1.5,1.5) ..
		(1.5,2)
		--
		(1.5,3) node[midway,rrect]{\small $\chcob_g \bullet \chcob_f$}
		-- (1.5,3.5)
		-- (1.5,4.5) node[midway,rrect]{\small $\chcob_{\bj'\bj}$}
		-- (1.5,5);
}
\end{tikzcd}
\]
Note that for a dot, splitter or crossing with $j \neq \pm 1$, then $\omega_1$ coincides with the change of chronology given by the definition of tensor product of maps in $\BIMOD^{\gradC}$ together with the twist in the definition of the grading shift functors on $\shiftFunct{\F_{\bi'}}(\UtoBfunctor_0(g))$ and $\shiftFunct{\F_{\bi}}(\UtoBfunctor_0(g))$. Similarly, $\omega_2$ coincides with the twist in $\shiftFunct{\F_{\bi'\bi}}((\UtoBfunctor_0(g \circ f)))$. Therefore, we conclude that \cref{eq:FLisonat} $\gradC$-graded commutes in these cases. 
When $f$ (or $g$) is a crossing with $j = i \pm 1$, then $\omega_1$ and $\omega_2$ also carry a coefficient coming from non-trivial changes of chronology that coincides with the $\shiftFunct{H_{\tau_{ji}}}$. 
\end{proof}

Inspired by the definition of a weak 2-functor (also called pseudofunctor), we introduce the following:
\begin{defn}
We say that a $R$-linear bicategory $\cD$ is $\gradC$-graded if any 2-morphism decomposes as a finite sum of 2-morphisms carrying a degree $\chcob^v \in I$, and it is compatible with respect to the horizontal and vertical composition. \\
A \emph{$\gradC$-graded 2-functor}  $P : \cC \rightarrow \cD$ from a bicategory $\cC$ to a $\gradC$-graded bicategory $\cD$ consists of
\begin{itemize}
\item for each object $X \in \cC$, and object $P_X \in \cD$;
\item for each hom-category $\Hom_\cC(Y,X)$, a functor $P_{Y,X} : \Hom_\cC(Y,X) \rightarrow \Hom_\cD(P_Y, P_X)$;
\item for each object $X \in \cC$, an invertible 2-morphism $P_{\id_X} : \id_{P_X} \rightarrow P_{X,X}(\id_X)$;
\item for each each triple $Z,Y,X \in \cC$ an isomorphism 
\[
P_{Z,Y,X}(f,g) : P_{Z,Y}(g) \circ P_{Y,X}(f) \Rightarrow P_{Z,X}(g \circ f),
\]
$\gradC$-graded natural in $g : Z \rightarrow X$ and in $f : Y \rightarrow X$, meaning that the following diagram 
\[
\begin{tikzcd}[column sep = 10ex]
P_{Z,Y}(g) \circ P_{Y,X}(f) 
\ar{r}{P_{Z,Y,X}(f,g)}
\ar[swap]{d}{P_{Z,Y}(\theta) \circ P_{Y,X}(\eta) }
&
P_{Z,X}(g \circ f)
\ar{d}{P_{Z,X}(\theta \circ \eta) }
\\
P_{Z,Y}(g') \circ P_{Y,X}(f') 
\ar[swap]{r}{P_{Z,Y,X}(f',g')}
\ar[Rightarrow,shorten <= 3ex, shorten >= 3ex]{ur}
&
P_{Z,X}(g' \circ f')
\end{tikzcd}
\]
$\gradC$-graded commutes for any pair of 2-morphisms $\theta : g \rightarrow g'$ and $\eta : f \rightarrow f'$. 
\end{itemize}
such that 
\begin{itemize}
\item the diagrams
\begin{equation} \label{eq:diagpseudofunctorL}
\begin{tikzcd}[column sep =10ex]
\id_{P_Y} \circ P_{Y,X}(f) \ar{r}{\lambda_\cD}  \ar[swap]{d}{P_{\id_Y} \circ \id}
&
 P_{Y,X}(f)
 \\
 P_{Y,Y}(\id_Y) \circ P_{Y,X}(f) \ar[swap]{r}{P_{Y,Y,X}(\id_Y, f)}
\ar[Rightarrow,shorten <= 3ex, shorten >= 3ex]{ur}
 &
 P_{Y,X}(\id_Y \circ f) \ar[swap]{u}{P_{Y,X}(\lambda_\cC)}
\end{tikzcd}
\end{equation}
and
\begin{equation} \label{eq:diagpseudofunctorR}
\begin{tikzcd}[column sep =10ex]
P_{Y,X}(f) \circ \id_{P_X} 
\ar{r}{\rho_\cD}
\ar[swap]{d}{P_{Y,X}(f) \circ P_{\id_X}}
&
P_{Y,X}(f)
\\
P_{Y,X}(f) \circ P_{X,X}(\id_X) 
\ar[swap]{r}{P_{Y,X,X}(f,\id_X)}
\ar[Rightarrow,shorten <= 3ex, shorten >= 3ex]{ur}
&
P_{Y,X}(f \circ \id_X)
\ar[swap]{u}{P_{Y,X}(\rho_\cC)}
\end{tikzcd}
\end{equation}
both $\gradC$-graded commute;
\item the diagram
\begin{equation} \label{eq:diagpseudofunctorHex}
\begin{tikzcd}
\bigl( P_{Z,Y}(h) \circ P_{Y,X}(g) \bigr) \circ P_{X,W}(f)
\ar{r}{\assoc_\cD}
\ar[swap]{d}{P_{Z,Y,X}(h,g) \circ \id}
&
P_{Z,Y}(h) \circ \bigl( P_{Y,X}(g) \circ P_{X,W}(f) \bigr) 
\ar{d}{\id \circ P_{Y,X,W}(g,f)}
\\
P_{Z,X}(h \circ g) \circ P_{X,W}(f)
\ar[swap]{d}{P_{Z,X,W}(h \circ g, f)}
& 
P_{Z,Y}(h) \circ P_{Y,W}(g \circ f)
\ar{d}{P_{Z,Y,W}(h, g \circ f)}
\\
P_{Z,W}((h \circ g) \circ f)
\ar[swap]{r}{P_{Z,W}(\assoc_\cC)}
\ar[Rightarrow,shorten <= 8ex, shorten >= 8ex]{uur}
&
P_{Z,W}(h \circ (g \circ h))
\end{tikzcd}
\end{equation}
$\gradC$-graded commutes for each quadruple $Z,Y,X,W \in \cC$.
\end{itemize}
with $\lambda_\cC$, $\rho_\cC$ and $\assoc_\cC$ are the left unitor, the right unitor and the associator in $\cC$ respectively.
\end{defn}

\begin{thm}
The assignment
\[
\UtoBfunctor : \oddKCeven \rightarrow \BIMODH
\] is a  $\gradC$-graded 2-functor. 
\end{thm}

\begin{proof}
If we denote $\UtoBfunctor$ as $P$ in the definition of a $\gradC$-graded 2-functor, then we have 
\begin{itemize}
\item $P_\bw := H^{s(\bw)/2}$, 
\item $P_{\bw',\bw}$ is the 1-functor from \cref{lem:onefunctorUtoB}, 
\item $P_{\id_\bw} : H^{s(\bw)/2} \rightarrow \UtoBfunctor(1_\bw) \cong H^{s(\bw)/2}$ is basically the identity,
\item $P_{\bw'',\bw',\bw}(\F_\bj, \F_\bi) := \UtoBfunctor(\bj, \bi)$. 
\end{itemize}
Commutativity of \cref{eq:diagpseudofunctorL} and of \cref{eq:diagpseudofunctorR} is immediate by definition of left and right action on a shifted bimodule (see \cref{defn:shiftedbim}).

By the coherence conditions of the compatibility maps, we obtain that the following diagram commutes:
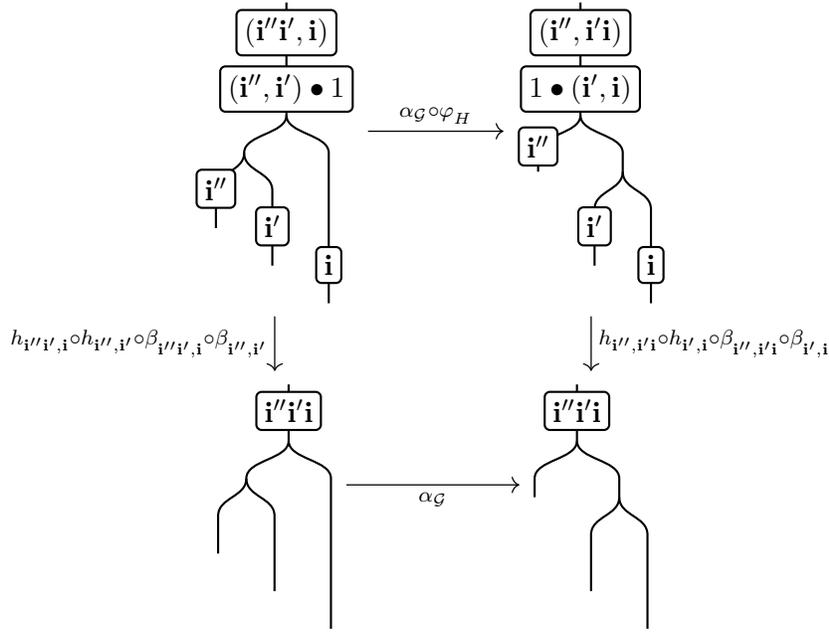
\[
\begin{tikzcd}[column sep = 10ex]
\tikzdiag[yscale=.5,xscale=.75]{0}{
	\draw (0,0)
		--
		(0,0)
		--
		(0,1)   node[near end,rrect]{\small ${\bi''}$}
		.. controls (0,1.5) and (.5,1.5) ..
		(.5,2)
		.. controls (.5,2.5) and (1.25,2.5) ..
		(1.25,3);
	\draw (1,-1)
		--
		(1,-1)
		-- 
		(1,0)   node[near end,rrect]{\small ${\bi'}$}
		--
		(1,1)
		.. controls (1,1.5) and (.5,1.5) ..
		(.5,2);
	\draw (2,-2)
		--
		(2, -1)   node[near end,rrect]{\small ${\bi}$}
		--
		(2,2)
		.. controls (2,2.5) and (1.25,2.5) ..
		(1.25,3) 
		-- 
		(1.25,4) node[midway,rrect]{\small $({\bi'',\bi'}) \bullet 1$}
		-- 
		(1.25,6) node[midway,rrect]{\small $({\bi''\bi',\bi}) $};
}
\ar[swap]{d}{h_{\bi''\bi',\bi} \circ  h_{\bi'',\bi'} \circ \compMap{\bi''\bi'}{\bi} \circ \compMap{\bi''}{\bi'} }
\ar{r}{ \assoc_\gradC \circ \shiftFunct{H}}
&
\tikzdiag[yscale=.5,xscale=.75]{0}{
	\draw (0,2) 
		--
		(0,2.5)  node[near end,rrect]{\small $\bi''$}
		.. controls (0,3) and (.75,3) ..
		(.75,3.5);
	\draw (1,-.5)
		-- 
		(1,.5)   node[near end,rrect]{\small $\bi'$}
		--
		(1,1)
		.. controls (1,1.5) and (1.5,1.5) ..
		(1.5,2);
	\draw (2,-1.5)
		--
		(2, -.5)   node[near end,rrect]{\small $\bi$}
		--
		(2,1)
		.. controls (2,1.5) and (1.5,1.5) ..
		(1.5,2)
		--
		(1.5,2.5)
		.. controls (1.5,3) and (.75,3) .. 
		(.75,3.5)
		-- 
		(.75,4.5) node[midway,rrect]{\small $1 \bullet ({\bi',\bi})$}
		-- 
		(.75,6.5) node[midway,rrect]{\small $({\bi'',\bi'\bi}) $};
}
\ar{d}{h_{\bi'',\bi'\bi} \circ h_{\bi',\bi} \circ \compMap{\bi''}{\bi'\bi} \circ \compMap{\bi'}{\bi}}
\\
\tikzdiag[yscale=.5,xscale=.75]{0}{
	\draw (0,0)
		--
		(0,0)
		--
		(0,1) 
		.. controls (0,1.5) and (.5,1.5) ..
		(.5,2)
		.. controls (.5,2.5) and (1.25,2.5) ..
		(1.25,3);
	\draw (1,-1)
		--
		(1,-1)
		-- 
		(1,0)   
		--
		(1,1)
		.. controls (1,1.5) and (.5,1.5) ..
		(.5,2);
	\draw (2,-2)
		--
		(2, -1)   
		--
		(2,2)
		.. controls (2,2.5) and (1.25,2.5) ..
		(1.25,3) 
		-- 
		(1.25,4) node[midway,rrect]{\small $\bi''\bi'\bi$}
		-- (1.25,4.5);
}
\ar[swap]{r}{\assoc_\gradC}
&
\tikzdiag[yscale=.5,xscale=.75]{0}{
	\draw (0,2) 
		--
		(0,2.5)  
		.. controls (0,3) and (.75,3) ..
		(.75,3.5);
	\draw (1,-.5)
		-- 
		(1,.5)   
		--
		(1,1)
		.. controls (1,1.5) and (1.5,1.5) ..
		(1.5,2);
	\draw (2,-1.5)
		--
		(2, -.5)  
		--
		(2,1)
		.. controls (2,1.5) and (1.5,1.5) ..
		(1.5,2)
		--
		(1.5,2.5)
		.. controls (1.5,3) and (.75,3) .. 
		(.75,3.5)
		-- 
		(.75,4.5) node[midway,rrect]{\small $\bi''\bi'\bi$}
		-- 
		(.75,5);
}
\end{tikzcd}
\]
where we denoted by $\bi$ the grading shift $\shiftFunct{\F_\bi}$ and $\bi',\bi$ for $\chcob_{\bi',\bi}^{v_{\bi',\bi}}$ etc, and where 
\[
\shiftFunct{H} : \chcob_{\bi''\bi',\bi}^{v_{\bi''\bi',\bi}} \circ (\chcob_{\bi',\bi}^{v_{\bi',\bi}} \bullet 1) 
\Rightarrow
\chcob_{\bi'',\bi'\bi}^{v_{\bi'',\bi'\bi}} \circ (1 \bullet \chcob_{\bi',\bi}^{v_{\bi',\bi}}).
\]
Since these are all the coefficients appearing in the definition of $\UtoBfunctor(\bj, \bi)$, we conclude  \cref{eq:diagpseudofunctorHex} $\gradC$-graded commutes. 
\end{proof}

\begin{prop}
The $\gradC$-graded 2-functor $\UtoBfunctor$ is faithful (at the level of 2-hom spaces). 
\end{prop}

\begin{proof}
By the results in~\cite{brundanstroppel3}, we know that $\UtoBfunctor$ is fully faithful after specializing $X=Y=Z=1$. Thus, it is injective at the level of 2-morphisms before specialization. 
\end{proof}

\subsubsection{Odd tangle invariant}

In particular, by specializing $X = Z=1$ and $Y=-1$, we obtain a 2-action of the cyclotomic odd half 2-Kac--Moody from~\cite{pedrodd} on the bicategory of quasi-associative bimodules over the odd arc algebra. Moreover, the complex of 1-morphisms constructed in the reference for an $(m,n)$-tangle $T$ coincides with $\kh(T)$ after applying the 2-functor $\UtoBfunctor$. 
Thus, showing that $\UtoBfunctor$ is full would be enough to prove that the invariant constructed in~\cite{pedrodd} coincides with ours, and in particular with the odd Khovanov homology~\cite{ORS} for links thanks to \cref{prop:recoverKh}.



\bibliographystyle{bibliography/habbrv}



\end{document}